\documentclass[10pt]{amsart}
\usepackage{amssymb,amsmath,amsfonts,amsthm,graphics,mathrsfs,amscd,hyperref}
\usepackage[hmargin=1in,vmargin=1in]{geometry}
\usepackage[all,cmtip]{xy}

\theoremstyle{plain}

\newtheorem{definition}[equation]{Definition}

\newtheorem{lemma}[equation]{Lemma}
\newtheorem{proposition}[equation]{Proposition}
\newtheorem{theorem}[equation]{Theorem}

\newtheorem{example}{Example}
\newtheorem{notation}{Notation}
\newtheorem{remark}[equation]{Remark}

\numberwithin{equation}{subsection}

\title{Deformations of compact holomorphic Poisson submanifolds}
\author{Chunghoon Kim}

\thanks{The author was partially supported by NRF grant 2011-0027969}

\email{ckim042@gmail.com}            

\begin{document}

\maketitle

\begin{abstract}
In this paper, we study deformations of compact holomorphic Poisson submanifolds  which extend Kodaira's series of papers on semi-regularity (deformations of compact complex submanifolds of codimension $1$) (\cite{Kod59}), deformations of compact complex submanifolds of arbitrary codimensions (\cite{Kod62}), and stability of compact complex submanifolds (\cite{Kod63}) in the context of holomorphic Poisson deformations. We also study simultaneous deformations of holomorphic Poisson structures and holomorphic Poisson submanifolds on a fixed underlying compact complex manifold. In appendices, we present deformations of Poisson closed subschemes in the language of functors of Artin rings which is the algebraic version of deformations of holomorphic Poisson submanifolds. We identify first-order deformations and obstructions.
\end{abstract}

\tableofcontents

\section{Introduction}

In this paper, we study deformations of compact holomorphic Poisson submanifolds which extend Kodaira's series of papers on semi-regularity (deformations of compact complex submanifols of codimension $1$) (\cite{Kod59}), deformations of compact complex submanifolds of arbitrary codimensions (\cite{Kod62}), and stability of compact complex submanifolds (\cite{Kod63}) in the context of holomorphic Poisson deformations. We will review deformation theory of compact complex submainfolds presented in \cite{Kod59},\cite{Kod62},\cite{Kod63}, and explain how the theory can be extended in the context of holomorphic Poisson deformations.

Let us review deformations of compact complex submanifolds of codimension $1$ of a complex manifold presented in \cite{Kod59} where  Kodaira-Spencer proved the theorem of completeness of characteristic systems of complete continuous systems of semi-regular complex submanifolds of codimension $1$. For the precise statement, we recall the definitions of a complex analytic family of compact complex submanifolds of codimension $1$, and maximality (or completeness) of a complex analytic family:
\begin{definition}\label{003}
Let $W$ be a complex manifold\footnote{In this paper, all manifolds under consideration are paracompact and connected.} of dimension $n+1$. We denote a point in $W$ by $w$ and a local coordinate of $w$ by $(w^1,...,w^{n+1})$. By a complex analytic family of compact complex submanifolds of codimension $1$ of $W$, we mean a complex submanifold $\mathcal{V}\subset W\times M$ of codimension $1$ where $M$ is a complex manifold, such that $V_t\times t:=\omega^{-1}(t)=\mathcal{V}\cap \pi^{-1}(t)$ for each point $t\in M$ is a connected compact complex submanifold of $W\times t$, where $\omega:\mathcal{V}\to M$ is the map induced from the canonical projection $\pi:W\times M\to M$, and for each point $p\in \mathcal{V}$, there is a holomorphic function $S(w,t)$ on a neighborhood $\mathcal{U}_p$ of $p$ in $W\times M$ such that $\sum_{\alpha=1}^{n+1}|\frac{\partial S(w,t)}{\partial w^\alpha}|^2\ne 0$ at each point in $\mathcal{U}_p\cap \mathcal{V}$, and $\mathcal{U}_p\cap \mathcal{V}$ is defined by $S(w,t)=0$.
\end{definition}

\begin{definition}\label{004}
Let $\mathcal{V}\subset W\times M \xrightarrow{\omega} M$ be a complex analytic family of compact complex submanifolds of $W$ of codimension $1$ and let $t_0$ be a point on $M$. We say that $\mathcal{V}\xrightarrow{\omega} M$ is maximal at $t_0$ if, for any complex analytic family $\mathcal{V}'\subset W\times M'\xrightarrow{\omega'} M'$ of compact complex submanifolds of $W$ of codimension $1$ such that $\omega^{-1}(t_0)=\omega'^{-1}(t_0'),t_0'\in M'$, there exists a holomorphic map $h$ of a neighborhood $N'$ of $t_0'$ on $M'$ into $M$ which maps $t_0'$ to $t_0$ such that $\omega'^{-1}(t')=\omega^{-1}(h(t'))$ for $t'\in N'$. We note that if we set a holomorphic map $\hat{h}:W\times N'\to W\times M$ defined by $(w,t')\to (w,h(t'))$, then the restriction map of $\hat{h}$ to $\mathcal{V}'|_{N'}=\omega'^{-1}(N')\subset W\times N'$ defines a holomorphic map $\mathcal{V}'|_{N'}\to \mathcal{V}$ so that $\mathcal{V}'|_{N'}$ is the family induced from $\mathcal{V}$ by $h$, which means $\mathcal{V}\xrightarrow{\omega} M$ is complete at $t_0$.
\end{definition}
Given a complex analytic family $\mathcal{V}\subset W\times M\xrightarrow{\omega} M$ of compact complex submanifolds of codimension $1$, each fibre $V_t=\omega^{-1}(t)$ of $\mathcal{V}$ for $t\in M$ defines a complex line bundle $\mathcal{N}_t$ on $W$. Then infinitesimal deformations of $V_t$ in the family $\mathcal{V}$ are encoded in the cohomology group $H^0(V_t, \mathcal{N}_t|_{V_t})$, and we can define the characteristic map (see \cite{Kod59} p.479-480)
\begin{align*}
\rho_{d,t}:T_t M\to H^0(V_t, \mathcal{N}_t|_{V_t})
\end{align*}
In \cite{Kod59}, Kodaira-Spencer defined a concept of semi-regularity, and showed that the semi-regularity is `the right condition' for `theorem of existence' and thus   `theorem of completeness' for deformations of compact complex submanifolds of codimension $1$ as follows.

\begin{definition}
Let $V_0$ be a compact complex submanifold of $W$ of codimension $1$ and $\mathcal{N}_0$ be the complex line bundle over $W$ determined by $V_0$. Let $r_0:\mathcal{N}_0\to \mathcal{N}_0|_{V_0}$ be the restriction map which induces a homomorphism $r_0^*:H^1(W,\mathcal{N}_0)\to H^1(V_0, \mathcal{N}_0|_{V_0})$. We say that $V_0$ is semi-regular if $r_0^* H^1(W,\mathcal{N}_0)$ is zero.
\end{definition}

\begin{theorem}[theorem of existence]\label{001}
If $V_0$ is semi-regular, then there exists a complex analytic family $\mathcal{V}\subset W\times M\xrightarrow{\omega} M$ of compact complex submanifolds of $W$ containing $V_0$ as the fibre $\omega^{-1}(0)$ over $0\in M$ such that the characteristic map
\begin{align*}
\rho_{d,0}:T_0 M \to H^0(V_0,\mathcal{N}_0|_{V_0})
\end{align*}
is an isomorphism.
\end{theorem}

\begin{theorem}[theorem of completeness]\label{002}
Let $\mathcal{V}\subset W\times M \xrightarrow{\omega} M$ be a complex analytic family of compact complex submanifolds of $W$ of codimension $1$. If the characteristic map
\begin{align*}
\rho_{d,0}:T_0 M\to H^0(V_0,\mathcal{N}_0|_{V_0})
\end{align*}
 is an isomorphism, then the family $\mathcal{V}\xrightarrow{\omega} M$ is maximal at the point $t=0$.
\end{theorem}

In section \ref{section2}, we extend the concept of semi-regularity and prove an analogue of theorem of existence (Theorem \ref{001}) and an analogue of theorem of completeness (Theorem \ref{002})  in the context of holomorphic Poisson deformations. A holomorphic Poisson manifold $W$ is a complex manifold whose structure sheaf is a sheaf of Poisson algebras.\footnote{We refer to \cite{Lau13} for general information on Poisson geometry} A holomorphic Poisson structure on $W$ is encoded in a holomorphic section (a holomorphic bivector field) $\Lambda_0\in H^0(W,\wedge^2 T_W)$ with $[\Lambda_0,\Lambda_0]=0$, where $T_W$ is the sheaf of germs of holomorphic vector fields, and the bracket $[-,-]$ is the Schouten bracket on $W$. In the sequel a holomorphic Poisson manifold will be denoted by $(W,\Lambda_0)$. Let $V$ be a complex submanifold of a holomorphic Poisson manifold $(W,\Lambda_0)$ and let $i:V\hookrightarrow W$ be the embedding. Then $V$ is called a holomorphic Poisson submanifold of $(W,\Lambda_0)$ if $V$ is a holomorphic Poisson manifold  and the embedding $i$ is a Poisson map with respect to the holomorphic Poisson structures. Then the holomorphic Poisson structure on $V$ is unique. Equivalently a holomorphic Poisson submanifold $V$ of $(W,\Lambda_0)$ can be characterized in the following way: let $V$ be covered by coordinate neighborhoods $W_i, i\in I$ in $W$. We choose a local coordinate $(w_i, z_i):=(w_i^1,...,w_i^r, z_i^1,...,z_i^d)$ on each neighborhood $W_i$ such that $w_i^1=\cdots=w_i^r=0$ defines $V\cap W_i$. Then $V$ is a holomorphic Poisson submanifold of $(W,\Lambda_0)$ if the restriction $[\Lambda_0,w_i^\alpha]|_{V\cap U_i}$ of $[\Lambda_0, w_i^\alpha]$ to $V\cap U_i$ is $0$, i.e. $[\Lambda_0, w_i^\alpha]|_{V\cap U_i}:=[\Lambda_0, w_i^\alpha]|_{w_i=0}=0,\alpha=1,...,r$, or $[\Lambda_0, w_i^\alpha]$ is of the form: $[\Lambda_0, w_i^\alpha]=\sum_{\beta=1}^r w_i^\beta T_{i\alpha}^\beta(w_i,z_i)$ for some $T_{i\alpha}^\beta(w_i,z_i)\in \Gamma(W_i, T_W)$.

Now we explain how we can extend the theory of deformations of compact complex submanifolds of codimension $1$ to the theory of deformations of compact holomorphic Poisson submanifolds of a holomorphic Poisson manifold of codimension $1$. We extend Definition \ref{003} and Definition \ref{004} to define concepts of a family of compact holomorphic Poisson submanifolds of codimension $1$ (see Definition \ref{2d}), and maximality (or completeness) of a family of compact holomorphic Poisson submanifolds (see Definition \ref{005}).

\begin{definition}\label{00023}
Let $(W,\Lambda_0)$ be a holomorphic Poisson manifold of dimension $n+1$. We denote a point in $W$ by $w$ and a local coordinate of $w$ by $(w^1,...,w^{n+1})$. By a Poisson analytic family of compact holomorphic Poisson submanifolds of codimension $1$ of $(W,\Lambda_0)$, we mean a holomorphic Poisson submanifold $\mathcal{V}\subset (W\times M,\Lambda_0)$ of codimension $1$ where $M$ is a complex manifold and $\Lambda_0$ is the holomorphic Poisson structure on $W\times M$ induced from $(W,\Lambda_0)$, such that $V_t:=\omega^{-1}(t)=\mathcal{V}\cap \pi^{-1}(t)$ for each point $t\in M$ is a connected compact holomorphic Poisson submanifold of $(W\times t,\Lambda_0)$, where $\omega:\mathcal{V}\to M$ is the map induced from the canonical projection $\pi:W\times M\to M$, and for each point $p\in \mathcal{V}$, there is a holomorphic function $S(w,t)$ on a neighborhood $\mathcal{U}_p$ of $p$ in $W\times M$ such that $\sum_{\alpha=1}^{n+1}|\frac{\partial S(w,t)}{\partial w^\alpha}|^2\ne 0$ at each point in $ \mathcal{U}_p \cap \mathcal{V}$, and $\mathcal{U}_p\cap \mathcal{V}$ is defined by $S(w,t)=0$.

\end{definition}

\begin{definition}\label{020}
Let $\mathcal{V}\subset (W\times M,\Lambda_0)\xrightarrow{\omega} M$ be a Poisson analytic family of compact holomorphic Poisson submanifolds of $(W,\Lambda_0)$ of codimension $1$ and let $t_0$ be a point on $M$. We say that $\mathcal{V}\xrightarrow{\omega} M$ is maximal at $t_0$ if, for any Poisson analytic family $\mathcal{V}'\subset (W\times M',\Lambda_0)\xrightarrow{\omega'} M'$ of compact holomorphic Poisson submanifolds of $(W,\Lambda_0)$ of codimension $1$ such that $\omega^{-1}(t_0)=\omega'^{-1}(t_0'),t_0'\in M'$, there exists a holomorphic map $h$ of a neighborhood $N'$ of $t_0'$ on $M'$ into $M$ which maps $t_0'$ to $t_0$ such that $\omega'^{-1}(t')=\omega^{-1}(h(t'))$ for $t'\in N'$. We note that if we set a Poisson map $\hat{h}:(W\times N',\Lambda_0)\to (W\times M ,\Lambda_0)$ defined by $(w,t')\to (w,h(t'))$, then the restriction map of $\hat{h}$ to $\mathcal{V}'|_{N'}=\omega'^{-1}(N')\subset (W\times N',\Lambda_0)$ defines a Poisson map $\mathcal{V}'|_{N'}\to \mathcal{V}$ so that $\mathcal{V}'|_{N'}$ is the family induced from $\mathcal{V}$ by $h$, which means $\mathcal{V}\xrightarrow{\omega} M$ is complete at $t_0$. 
\end{definition}
Given a Poisson analytic family $\mathcal{V}\subset (W\times M, \Lambda_0)\to M$ of compact holomorphic Poisson submanifolds of codimension $1$, each fibre $V_t=\omega^{-1}(t)$ of $\mathcal{V}$ for $t\in M$ defines a Poisson line bundle $(\mathcal{N}_t,\nabla_t)$ on $(W,\Lambda_0)$, where $\nabla_t$ is the Poisson connection on $\mathcal{N}_t$ which defines the Poisson line bundle structure (see \cite{Kim16}) so that we have a complex of sheaves on $W$ (see \cite{Kim16})
\begin{align*}
\mathcal{N}_t^\bullet:\mathcal{N}_t\xrightarrow{\nabla_t} \mathcal{N}_t\otimes T_W\xrightarrow{\nabla_t} \mathcal{N}_t\otimes \wedge^2 T_W\xrightarrow{\nabla_t} \cdots
\end{align*}
We will denote the $i$-th hypercohomology group by $\mathbb{H}^i(W,N_t^\bullet)$. We note that $\mathcal{N}_t^\bullet$ induces, by restriction on $V_t$, the complex of sheaves on $V_t$
\begin{align*}
\mathcal{N}_t^\bullet|_{V_t}:\mathcal{N}_t|_{V_t}\xrightarrow{\nabla_t|_{V_t}} \mathcal{N}_t|_{V_t}\otimes T_W|_{V_t}\xrightarrow{\nabla_t|_{V_t}} \mathcal{N}_t|_{V_t}\otimes \wedge^2 T_W|_{V_t}\xrightarrow{\nabla_t|_{V_t}} \cdots
\end{align*}
We will denote the $i$-th hypercohomology group by $\mathbb{H}^i(V_t,\mathcal{N}_t^\bullet |_{V_t})$. Then infinitesimal deformations of $V_t$ in the family $\mathcal{V}$ are encoded in the cohomology group $\mathbb{H}^0(V_t, \mathcal{N}_t^\bullet |_{V_t})$, and we can define the characteristic map (see subsection \ref{009})
\begin{align*}
\rho_{d,t}: T_t M\to \mathbb{H}^0(V_t, \mathcal{N}_t^\bullet |_{V_t})
\end{align*}

As in the concept of semi-regularity in \cite{Kod59}, we similarly define a concept of Poisson semi-regularity and show that Poisson semi-regularity (see Definition \ref{010}) implies `theorem of existence' (see Theorem \ref{jj20}) and thus `theorem of completeness'  (see Theorem \ref{012}) for deformations of compact holomorphic Poisson submanifolds of codimension $1$ as follows.
\begin{definition}
Let $V_0$ be a compact holomorphic Poisson submanifold of a holomorphic Poisson manifold $(W,\Lambda_0)$ of codimension $1$ and let $(\mathcal{N}_0,\nabla_0)$ be the Poisson line bundle over $(W,\Lambda_0)$ determined by $V_0$. We denote by $r_0:\mathcal{N}^\bullet\to \mathcal{N}^\bullet|_{V_0}$ the restriction map of the following complex of sheaves
\begin{center}
$\begin{CD}
\mathcal{N}_0^\bullet: @.\mathcal{N}_0@>\nabla_0>> \mathcal{N}_0\otimes T_W@>\nabla_0>> \mathcal{N}_0\otimes \wedge^2 T_W@>\nabla_0>>\cdots\\
@.@VVV @VVV @VVV \\
\mathcal{N}_0^\bullet |_{V_0}: @. \mathcal{N}_0|_{V_0}@>\nabla_0|_{V_0}>> \mathcal{N}_0|_{V_0}\otimes T_W|_{V_0}@>\nabla_0|_{V_0}>> \mathcal{N}_0|_{V_0}\otimes \wedge^2 T_W|_{V_0}@>\nabla_0|_{V_0}>>\cdots
\end{CD}$
\end{center}
which induces a homomorphism $r_0^*:\mathbb{H}^1(W,\mathcal{N}_0^\bullet)\to \mathbb{H}^1(V_0,\mathcal{N}_0^\bullet |_{V_0})$. We say that $V_0$ is Poisson semi-regular if the image $r_0^* \mathbb{H}^1(W,\mathcal{N}_0^\bullet)$ is zero.
\end{definition}

\begin{theorem}[Theorem of existence]
If $V_0$ is Poisson semi-regular, then there exists a Poisson analytic family $\mathcal{V}\subset (W\times M,\Lambda_0)\xrightarrow{\omega} M$ of compact holomorphic Poisson submanifolds of $(W,\Lambda_0)$ containing $V_0$ as the fibre $\omega^{-1}(0)$ over $0\in M$ such that the characteristic map
\begin{align*}
\rho_{d,0}:T_0 M\to \mathbb{H}^0(V_0,\mathcal{N}_0^\bullet |_{V_0})
\end{align*}
ia an isomorphism.
\end{theorem}

\begin{theorem}[Theorem of completeness]
Let $\mathcal{V}\subset (W\times M, \Lambda_0)\xrightarrow{\omega} M$ be a Poisson analytic family of compact Poisson submanifolds of $(W,\Lambda_0)$ of codimension $1$. If the characteristic map
\begin{align*}
\rho_{d,0}:T_0 M\to\mathbb{H}^0(V_0,\mathcal{N}_0^\bullet|_{V_0})
\end{align*}
 is an isomorphism, then the family $\mathcal{V}\xrightarrow{\omega} M$ is maximal at the point $t=0$.
\end{theorem}

Next we review deformation theory of compact complex submanifolds of arbitrary codimensions presented in \cite{Kod62} and explain how the theory can be extended to the theory of compact holomorphic Poisson submanifolds of arbitrary codimensions in terms of holomorphic Poisson deformations. In \cite{Kod62}, Kodaira showed that deformations of a compact complex submanifold $V$ of a complex manifold $W$ is controlled by the normal bundle $\mathcal{N}_{V/W}$ of $V$ in $W$ so that infinitesimal deformations are encoded in the cohomology group $H^0(V, \mathcal{N}_{V/W})$ and obstructions are encoded in the cohomology group $H^1(V, \mathcal{N}_{V/W})$. For the precise statement, we recall the following definition which generalize Definition \ref{003} to arbitrary codimensions.
\begin{definition}[\cite{Kod62}]\label{017}
Let $W$ be a complex manifold of dimension $d+r$. We denote a point in $W$ by $w$ and a local coordinate of $w$ by $(w^1,...,w^{r+d})$. By a complex analytic family of compact complex submanifolds of dimension $d$ of $W$, we mean a complex submanifold $\mathcal{V}\subset W\times M$ of codimension $r$, where $M$ is a complex manifold, such that 
\begin{enumerate}
\item for each point $t\in M$, $V_t\times t:=\omega^{-1}(t)=\mathcal{V}\cap \pi^{-1}(t)$ is a connected compact complex submanifold of $W\times t$ of dimension $d$, where $\omega:\mathcal{V}\to M$ is the map induced from the canonical projection $\pi:W\times M\to M$.
\item for each point $p\in \mathcal{V}$, there exist $r$ holomorphic functions $f_\alpha(w,t),\alpha=1,...,r$ defined on a neighborhood $\mathcal{U}_p$ of $p$ in $W\times M$ such that $\textnormal{rank} \frac{\partial ( f_1,...,f_r)}{\partial (w^1,...w^{r+d})}=r$, and $\mathcal{U}_p\cap \mathcal{V}$ is defined by the simultaneous equations $f_\alpha(w,t)=0,\alpha=1,...,r$.
\end{enumerate}
We call $\mathcal{V}\subset W\times M$ a complex analytic family of compact complex submanifolds $V_t,t\in M$ of $W$. We also call $\mathcal{V}\subset W\times M$ a complex analytic family of deformations of a compact complex submanifold $V_{t_0}$ of $W$ for each fixed point $t_0\in M$.
\end{definition}
We can define the concept of maximality (or completeness) of a complex analytic family of compact complex submanifolds of arbitrary codimenions as in Definition \ref{004}. Given a complex analytic family $\mathcal{V}\subset W\times M\xrightarrow{\omega} M$, for each fibre $V_t=\omega^{-1}(t)$ of $\mathcal{V}$ for $t\in M$,  infinitesimal deformations of $V_t$ in the family $\mathcal{V}$ are encoded in the cohomology group $H^0(V_t,\mathcal{N}_{V_t/W})$, and we can define the characteristic map (see \cite{Kod62} p.147-150)
\begin{align*}
\sigma_t: T_t M\to H^0(V_t, \mathcal{N}_{V_t/W})
\end{align*}
In \cite{Kod62}, Kodaira showed that given a compact complex submanifold $V$ of a complex manifold $W$, obstructions of deformations of $V$ in $W$ are encoded in $H^1(V, \mathcal{N}_{V/W})$ and so when obstructions vanish, we can prove `theorem of existence' and `theorem of completeness' as follows.

\begin{theorem}[theorem of existence]\label{015}
Let $V$ be a compact complex submanifold of a complex manifold $W$. If $H^1(V, \mathcal{N}_{V/W})=0$, then there exists a complex analytic family $\mathcal{V}$ of compact complex submanifolds $V_t$, $t\in M_1$ of $W$ such that $V_0=V$ and the characteristic map 
\begin{align*}
\sigma_0:T_0(M_1)&\to H^0(V,\mathcal{N}_{V/W})
\end{align*}
is an isomorphism.
\end{theorem}

\begin{theorem}[theorem of completeness]\label{016}
Let $\mathcal{V}$ be a complex analytic family of compact holomorphic Poisson submanifolds $V_t$, $t\in M_1$, of $W$. If the characteristic map
\begin{align*}
\sigma_0:T_0(M_1)\to H^0(V_0, \mathcal{N}_{V_0/W})
\end{align*}
is an isomorphism, then the family $\mathcal{V}$ is maximal at $t=0$.
\end{theorem}

In section \ref{section3}, we extend Definition \ref{017} to define a concept of compact holomorphic Poisson submanifolds of arbitrary codimensions, and prove an analogue of theorem of existence (Theorem \ref{015}) and an analogue of theorem of completeness (Theorem \ref{016}) in the context of holomorphic Poisson deformations. Let $V$ be a holomorphic Poisson submanifold of a holomorphic Poisson manifold $(W,\Lambda_0)$. Then we can define a complex of sheaves associated with the normal bundle $\mathcal{N}_{V/W}$ (see subsection \ref{subsection3.1} and Definition \ref{332a} )
\begin{align*}
\mathcal{N}_{V/W}^\bullet:\mathcal{N}_{V/W}\xrightarrow{\nabla}\mathcal{N}_{V/W}\otimes T_W|_V\xrightarrow{\nabla}\mathcal{N}_{V/W}\otimes \wedge^2 T_W|_V\xrightarrow{\nabla}\mathcal{N}_{V/W}\otimes \wedge^3 T_W|_V\xrightarrow{\nabla}\cdots
\end{align*}
which is called the complex associated with the normal bundle $\mathcal{N}_{V/W}$ of a holomorphic Poisson submanifold $V$ of a holomorphic Poisson manifold $(W,\Lambda_0)$, and denote its $i$-th hypercohomology group by $\mathbb{H}^i(V,\mathcal{N}_{V/W}^\bullet)$. Then holomorphic Poisson deformations of $V$ in $(W, \Lambda_0)$ is controlled by the complex of sheaves $\mathcal{N}_{V/W}^\bullet$ so that infinitesimal deformations are encoded in the cohomology group $\mathbb{H}^0(V, \mathcal{N}_{V/W}^\bullet)$ and obstructions are encoded in the cohomology group $\mathbb{H}^1(V, \mathcal{N}_{V/W}^\bullet)$. For the precise statement, we define a concept of compact holomorphic Poisson submanifolds of an arbitrary codimension which generalize Definition \ref{00023} and Definition \ref{017} as follows.
\begin{definition}[compare Definition \ref{017}]\label{024}
Let $(W,\Lambda_0)$ be a holomorphic Poisson manifold of dimension $d+r$. We denote a point in $W$ by $w$ and a local coordinate of $w$ by $(w^1,...,w^{r+d})$. By a Poisson analytic family of compact holomorphic Poisson submanifolds of dimension $d$ of $(W,\Lambda_0)$, we mean a holomorphic Poisson submanifold $\mathcal{V}\subset(W\times M, \Lambda_0)$ of codimension $r$, where $M$ is a complex manifold and $\Lambda_0$ is the holomorphic Poisson structure on $W\times M$ induced from $(W,\Lambda_0)$, such that
\begin{enumerate}
\item for each point $t\in M$, $V_t\times t:=\omega^{-1}(t)=\mathcal{V}\cap \pi^{-1}(t)$ is a connected compact holomorphic Poisson submanifold of $(W\times t,\Lambda_0)$ of dimension $d$, where $\omega:\mathcal{V}\to M$ is the map induced from the canonical projection $\pi:W\times M\to M$.
\item for each point $p\in \mathcal{V}$, there exist $r$ holomorphic functions $f_\alpha(w,t),\alpha=1,...,r$ defined on a neighborhood $\mathcal{U}_p$ of $p$ in $W\times M$ such that $\textnormal{rank} \frac{\partial ( f_1,...,f_r)}{\partial (w^1,...w^{r+d})}=r$, and $\mathcal{U}_p\cap \mathcal{V}$ is defined by the simultaneous equations $f_\alpha(w,t)=0,\alpha=1,...,r$.
\end{enumerate}
We call $\mathcal{V}\subset (W\times M, \Lambda_0)$ a Poisson analytic family of compact holomorphic Poisson submanifolds $V_t,t\in M$ of $(W,\Lambda_0)$. We also call $\mathcal{V}\subset( W\times M,\Lambda_0)$ a Poisson analytic family of deformations of a compact holomorphic Poisson submanifold $V_{t_0}$ of $(W,\Lambda_0)$ for each fixed point $t_0\in M$.
\end{definition}
We can define the concept of maximality (or completeness) of a Poisson analytic family of compact holomorphic Poisson submanifolds of arbitrary codimensions as in Definition \ref{020}. Given a Poisson analytic family $\mathcal{V}\subset (W\times M, \Lambda_0)\xrightarrow{\omega} M$ of compact holomorphic Poisson submanifolds, for each fibre $V_t=\omega^{-1}(t)$ of $\mathcal{V}$ for $t\in M$, infinitesimal deformations of $V_t$ in the family $\mathcal{V}$ are encoded in the cohomology group $\mathbb{H}^0(V_t, \mathcal{N}_{V_t/W}^\bullet)$, and we can define the characteristic map (see subsection \ref{3infinitesimal})
\begin{align*}
\sigma_t:T_t M\to \mathbb{H}^0(V_t, \mathcal{N}_{V_t/W}^\bullet)
\end{align*}
Given a compact holomorphic Poisson submanifold $V$ of a holomorphic Poisson manifold $(W,\Lambda_0)$, obstructions of holomorphic Poisson deformations of $V$ in $(W,\Lambda_0)$ are encoded in $\mathbb{H}^1(V,\mathcal{N}_{V/W}^\bullet)$ and so when obstructions vanish, we can prove `theorem of existence' (see Theorem \ref{33a}) and `theorem of completeness' (see Theorem \ref{3.5a}) as fallows.
\begin{theorem}[theorem of existence]
Let $V$ be a compact holomorphic Poisson submanifold of a holomorphic Poisson manifold $(W,\Lambda_0)$. If $\mathbb{H}^1(V, \mathcal{N}_{V/W}^\bullet)=0$, then there exists a Poisson analytic family $\mathcal{V}$ of compact holomorphic Poisson submanifolds $V_t$, $t\in M_1$ of $(W,\Lambda_0)$ such that $V_0=V$ and the characteristic map 
\begin{align*}
\sigma_0:T_0(M_1)&\to \mathbb{H}^0(V,\mathcal{N}_{V/W}^\bullet)
\end{align*}
is an isomorphism.
\end{theorem}

\begin{theorem}[theorem of completeness]
Let $\mathcal{V}$ be a Poisson analytic family of compact holomorphic Poisson submanifolds $V_t$, $t\in M_1$, of $(W,\Lambda_0)$. If the characteristic map
\begin{align*}
\sigma_0:T_0(M_1)\to \mathbb{H}^0(V_0, \mathcal{N}_{V_0/W}^\bullet)
\end{align*}
is an isomorphism, then the family $\mathcal{V}$ is maximal at $t=0$.
\end{theorem}

In section \ref{section4}, we study simultaneous deformations of holomorphic Poisson structures and compact holomorphic Poisson submanifolds. Let $V$ be a holomorphic Poisson submanifold of a holomorphic Poisson manifold $(W,\Lambda_0)$. As we saw as above, the complex associated with the normal bundle
\begin{align}\label{031}
\mathcal{N}_{V/W}^\bullet:\mathcal{N}_{V/W}\xrightarrow{\nabla} \mathcal{N}_{V/W}\otimes T_W|_V\xrightarrow{\nabla}\mathcal{N}_{V/W}\otimes \wedge^2 T_W|_V\xrightarrow{\nabla}\cdots
\end{align}
controls holomorphic Poisson deformations of $V$ of $(W,\Lambda_0)$, and the complex
\begin{align}\label{032}
\wedge^2 T_W^\bullet:\wedge^2 T_W\xrightarrow{-[-,\Lambda_0]} \wedge^3 T_W \xrightarrow{-[-,\Lambda_0]} \cdots
\end{align}
controls deformations of the holomorphic Poisson structure $\Lambda_0$ on the fixed underlying complex manifold $W$ (see Appendix \ref{appendixa}).
By combining two complexes $(\ref{031})$ and $(\ref{032})$, we can define a complex of sheaves on $W$ (see subsection \ref{43a} and Definition \ref{033})
\begin{align*}
(\wedge^2 T_W\oplus i_*\mathcal{N}_{V/W})^\bullet:\wedge^2 T_W\oplus i_* \mathcal{N}_{V/W}\xrightarrow{\tilde{\nabla}}\wedge^3 T_W\oplus i_*(\mathcal{N}_{V/W}\otimes T_W|_V)\xrightarrow{\tilde{\nabla}}\wedge^4 T_W\oplus i_*(\mathcal{N}_{V/W}\otimes \wedge^2 T_W|_V)\xrightarrow{\tilde{\nabla}}\cdots
\end{align*}
where $i:V\hookrightarrow W$ is the embedding, which is called the extended complex associated with the normal bundle $\mathcal{N}_{V/W}$ of a holomorphic Poisson submanifold $V$ of a holomorphic Poisson manifold $W$, and denote its $i$-th hypercohomology group by $\mathbb{H}^i(V,(\wedge^2 T_W\oplus i_*\mathcal{N}_{V/W})^\bullet)$. 
Then $(\wedge^2 T_W\oplus i_*\mathcal{N}_{V/W})^\bullet$ controls simultaneous deformations of $\Lambda_0$ and $V$ in $(W,\Lambda_0)$ so that infinitesimal deformations are encoded in the cohomology group $\mathbb{H}^0(W, (\wedge^2 T_W\oplus i_* \mathcal{N}_{V/W})^\bullet)$ and obstructions are encoded in $\mathbb{H}^1(W, (\wedge^2 T_W\oplus i_*\mathcal{N}_{V/W})^\bullet)$. For the precise statements, we extend Definition \ref{024} of a Poisson analytic family of compact holomorphic submanifolds of a holomorphic Poisson manifold $(W,\Lambda_0)$ by deforming $\Lambda_0$ on the fixed complex manifold $W$ as well as a holomorphic Poisson submanifold $V$ of $(W,\Lambda_0)$ as follows (see Definition \ref{4a}).
\begin{definition}
Let $W$ be a complex manifold of dimension $d+r$. We denote a point in $W$ by $w$ and a local coordinate of $w$ by $(w^1,...,w^{r+d})$. By an extended Poisson analytic family of compact holomorphic Poisson submanifolds of dimension $d$ of $W$, we mean a holomorphic Poisson submanifold $\mathcal{V}\subset(W\times M, \Lambda)$ of codimension $r$, where $M$ is a complex manifold and $\Lambda$ is a holomorphic Poisson structure on $W\times M$, such that
\begin{enumerate}
\item the canonical projection $\pi:(W\times M,\Lambda)\to M$ is a Poisson analytic family in the sense of \cite{Kim15} $($but we allow non-compact fibres$)$ so that $\Lambda\in H^0(W\times M, \wedge^2 T_{W\times M/M})$ and $\pi^{-1}(t):=(W\times t,\Lambda_t)$ is a holomorphic Poisson subsmanifold of $(W\times M, \Lambda)$ for each point $t\in M$.
\item for each point $t\in M$, $V_t\times t:=\omega^{-1}(t)=\mathcal{V}\cap \pi^{-1}(t)$ is a connected compact holomorphic Poisson submanifold of $(W\times t,\Lambda_t)$ of dimension $d$, where $\omega:\mathcal{V}\to M$ is the map induced from $\pi$.
\item for each point $p\in \mathcal{V}$, there exist $r$ holomorphic functions $f_\alpha(w,t),\alpha=1,...,r$ defined on a neighborhood $\mathcal{U}_p$ of $p$ in $W\times M$ such that $\textnormal{rank} \frac{\partial ( f_1,...,f_r)}{\partial (w^1,...w^{r+d})}=r$, and $\mathcal{U}_p\cap \mathcal{V}$ is defined by the simultaneous equations $f_\alpha(w,t)=0,\alpha=1,...,r$.
\end{enumerate}
We call $\mathcal{V}\subset (W\times M, \Lambda)$ an extended Poisson analytic family of compact holomorphic Poisson submanifolds $V_t,t\in M$ of $(W,\Lambda_t)$. We also call $\mathcal{V}\subset (W\times M,\Lambda)$ an extended Poisson analytic family of simultaneous deformations of a holomorphic Poisson submanifold $V_{t_0}$ of $(W,\Lambda_{t_0})$ for each fixed point $t_0\in M$.
\end{definition}

As in Definition \ref{004}, we can similarly define the concept of maximality (or completeness) of an extended Poisson analytic family (see Definition \ref{025}). Given an extended Poisson analytic family $\mathcal{V}\subset (M\times M, \Lambda)\xrightarrow{\omega} M$ of compact holomorphic Poisson submanifolds, for each fibre $V_t=\omega^{-1}(t)\subset (W,\Lambda_t)$ of $\mathcal{V}$ for $t\in M$, infinitesimal deformations of $(\Lambda_t, V_t)$ in the family $\mathcal{V}$ are encoded in the cohomology group $\mathbb{H}^0(W, (\wedge^2 T_W\oplus i_*\mathcal{N}_{V_t/W} )^\bullet)$, and we can define the characteristic map (see subsection \ref{027})
\begin{align*}
\sigma_t:T_t M\to \mathbb{H}^0(W, (\wedge^2 T_W\oplus i_*\mathcal{N}_{V_t/W} )^\bullet)
\end{align*}
Given a compact holomorphic Poisson manifold $V$ of a compact holomorphic Poisson manifold $(W,\Lambda_0)$, obstructions of simultaneous deformations of $\Lambda_0$ and $V$ are encoded in $\mathbb{H}^1(W,(\wedge^2 T_W\oplus i_*\mathcal{N}_{V/W} )^\bullet)$ and so when obstructions vanish, we can prove `theorem of  existence' (see Theorem \ref{43c}) and `theorem of completeness' (see Theorem \ref{3.5b}) as follows.
\begin{theorem}[theorem of existence]
Let $V$ be a holomorphic Poisson submanifold of a compact holomorphic Poisson manifold $(W,\Lambda_0)$. If $\mathbb{H}^1(W, (\wedge^2 T_W\oplus i_* \mathcal{N}_{W/V})^\bullet ) =0$, then  there exists an extended Poisson analytic family $\mathcal{V}\subset (W\times M_1,\Lambda)$ of compact holomorphic Poisson submanifolds $V_t,t\in M_1,$ of $(W,\Lambda_t)$ such that $V=V_0\subset (W,\Lambda_0)$ and the characteristic map 
\begin{align*}
\sigma_0:T_0(M_1)&\to \mathbb{H}^0(W, (\wedge^2 T_W\oplus i_*\mathcal{N}_{V/W})^\bullet)
\end{align*}
is an isomorphism.
\end{theorem}

\begin{theorem}[theorem of completeness]
Let $\mathcal{V}\subset(W\times M_1,\Lambda)$ be an extended Poisson analytic family of compact holomorphic Poisson submanifolds $V_t$ of $(W,\Lambda_t)$. If the characteristic map
\begin{align*}
\rho_0:T_0(M_1)&\to \mathbb{H}^0(W, (\wedge^2 T_W\oplus i_*\mathcal{N}_{W/V_0})^\bullet)
\end{align*}
is an isomorphism, then the family $\mathcal{V}$ is maximal at $t=0$.
\end{theorem}

Lastly we review stability of compact complex submanifolds presented in \cite{Kod63} and explain how we can extend the concept of stability in the context of holomorphic Poisson deformations. In \cite{Kod63}, Kodaira defined a concept of stability of compact complex submanifolds as follow:
\begin{definition}
Let $V$ be a compact complex submanifold of a complex manifold $W$. We call $V$ a stable complex submanifold of $W$ if and only if, for any complex fibre manifold \footnote{ for the definition of a complex fibre manifold and a complex fibre submanifold, see \cite{Kod63} p.79} $(\mathcal{W},B,p)$ with $p:\mathcal{W}\to B$ such that $p^{-1}(0)=W$ for a point $0\in B$, there exist a neighborhood $N$ of $0$ in $B$ and a complex fibre submanifold $\mathcal{V}$ with compact fibres of the complex fibre manifold $\mathcal{W}|_N$ such that $\mathcal{V}\cap W=V$, where $\mathcal{W}|_N$ is the restriction $(p^{-1}(N), N, p)$ of $(\mathcal{W},B,p)$ to $N$.
\end{definition}
and he proved 
\begin{theorem}
Let $V$ be a compact complex submanifold of a complex manifold $W$. If the first cohomology group $H^1(V,\mathcal{N}_{V/W})$ vanishes, then $V$ is a stable complex submanifold of $W$.
\end{theorem}

In section \ref{section5}, we similarly define a concept of stable compact holomorphic Poisson submanifolds as follows (see Definition \ref{029}):
\begin{definition}
Let $V$ be a compact holomorphic Poisson submanifold of a holomorphic Poisson manifold $(W,\Lambda_0)$. We call $V$ a stable holomorphic Poisson submanifold of $(W,\Lambda_0)$ if and only if, for any holomorphic Poisson fibre manifold $(\mathcal{W},\Lambda,B,p)$ $($see Definition $\ref{fibre}$$)$ such that $p^{-1}(0)=(W,\Lambda_0)$ for a point $0\in B$, there exist a neighborhood $N$ of $0$ in $B$ and a holomorphic Poisson fibre submanifold $\mathcal{V}$ with compact fibres of the holomorphic Poisson fibre manifold $(\mathcal{W},\Lambda)|_N$ such that $\mathcal{V}\cap W=V$, where $(\mathcal{W},\Lambda)|_N$ is the restriction $(p^{-1}(N), \Lambda|_{p^{-1}(N)}, N,p)$ of $(\mathcal{W},\Lambda, B,p)$ to $N$.
\end{definition}
and we prove (see Theorem \ref{stability})
\begin{theorem}
Let $V$ be a compact holomorphic Poisson submanifold of a holomorphic Poisson manifold $(W,\Lambda_0)$. If the first cohomology group $\mathbb{H}^1(V,\mathcal{N}_{V/W}^\bullet)$ vanishes, then $V$ is a stable holomorphic Poisson submanifold of $(W,\Lambda_0)$.
\end{theorem}

In appendices \ref{appendixb} and \ref{appendixc}, we present deformations of Poisson closed subschemes in the language of functors of Artin rings which is the algebraic version of deformations of holomorphic Poisson submanifolds. We identity first-order deformations and obstructions (see Proposition \ref{b15} and Proposition \ref{cp}).

\section{Deformations of compact holomorphic Poisson submanifolds of codimension $1$ and Poisson semi-regularity}\label{section2}

\begin{definition}\label{2d}
Let $(W,\Lambda_0)$ be a holomorphic Poisson manifold. We denote a point in $W$ by $w$ and a local coordinate of $w$ by $(w^1,...,w^{n+1})$. By a Poisson analytic family of compact holomorphic Poisson submanifolds of codimension $1$ of $(W,\Lambda_0)$, we mean a holomorphic Poisson submanifold $\mathcal{V}\subset (W\times M,\Lambda_0)$ of codimension $1$ where $M$ is a complex manifold and $\Lambda_0$ is the holomorphic Poisson structure on $W\times M$ induced from $(W,\Lambda_0)$, such that $V_t\times t:=\omega^{-1}(t)=\mathcal{V}\cap \pi^{-1}(t)$ for each point $t\in M$ is a connected compact  holomorphic Poisson submanifold of $(W\times t,\Lambda_0)$, where $\omega:\mathcal{V}\to M$ is the map induced from the canonical projection $\pi:W\times M\to M$, and for each point $p\in \mathcal{V}$, there is a holomorphic function $S(w,t)$ on a neighborhood $\mathcal{U}_p$ of $p$ in $W\times M$ such that $\sum_{\alpha=1}^{n+1}|\frac{\partial S(w,t)}{\partial w^\alpha}|^2\ne 0$ at each point in $\mathcal{U}_p\cap \mathcal{V}$, and $ \mathcal{U}_p\cap \mathcal{V}$ is defined by $S(w,t)=0$.
\end{definition}

\subsection{Infinitesimal deformations}\label{009}\

Let $(W,\Lambda_0)$ be a holomorphic Poisson manifold. We denote by $w$ a point on $(W,\Lambda_0)$ and by $(w^1,...,w^{n+1})$ the local holomorphic coordinates (not specified) of $w$. Consider a small spherical neighborhood $N$ of a point on $M$ and let $\{U_i\}$ be a locally finite covering of $W$ by sufficiently small coordinate neighborhoods $U_i$. 

Consider a Poisson analytic family $\mathcal{V}\subset (W\times M,\Lambda_0)$ of compact holomorphic Poisson submanifolds of codimension $1$ as in Definition \ref{2d}. Let $N$ be a small spherical neighborhood of a point on $M$. Then the holomorphic Poisson submanifold $\mathcal{V}$ of $(W\times M, \Lambda_0)$ is defined in each neighborhood $U_i\times N$ by a holomorphic equation $S_i(w,t)=0$ where $S_i(w,t)$ is a holomorphic function on $U_i\times N$ such that $\sum_{\alpha=1}^{n+1} \left|\frac{\partial S_i(w,t)}{\partial w^\alpha}\right|^2\ne 0$ at each point $(w,t)$ of $\mathcal{V}\cap (U_i\times N)$ (if $\mathcal{V}\cap (U_i\times N)$ is empty, we set $S_i(w,t)=1$). By letting 
\begin{align}\label{jj6}
S_i(w,t)=f_{ik}(w,t)S_k(w,t),\,\,\,\,\,w\in U_i\cap U_k, 
\end{align}
we get a system $\{f_{ik}(w,t)\}$ of non-vanishing holomorphic functions $f_{ik}(w,t)$ defined, respectively, on $(U_i\times N)\cap (U_k\times N)$ satisfying 
\begin{align}\label{2a}
f_{ik}(w,t)=f_{ij}(w,t)f_{jk}(w,t),\,\,\,\,\,\,w\in U_i\cap U_j\cap U_k.
\end{align}
 On the other hand, since $S_i(w,t)=0$ define a holomorphic Poisson submanifold, $[\Lambda_0,S_i(w,t)]$ is of the form 
 \begin{align}\label{jj5}
[\Lambda_0, S_i(w_i,t)]= S_i(w,t)T_i(w,t)
 \end{align}
for some $T_i(w,t)=\sum_{i=1}^{n+1} T_i^\alpha(w,t)\frac{\partial}{\partial w^\alpha}$, where $T_i^\alpha(w,t)$ is a holomorphic function on $U_i\times N$. By taking $[\Lambda_0,-]$ on $(\ref{jj5})$, we have $0=-[\Lambda_0, S_i(w,t)]\wedge T_i(w,t)]+S_i(w,t)[\Lambda_0, T_i(w,t)]=-S_i(w,t)T_i(w,t)\wedge T_i(w,t)+S_i(w,t)[\Lambda_0, T_i(w,t)]=S_i(w,t)[\Lambda_0,T_i(w,t)]$ so that
 \begin{align}\label{2b}
  [\Lambda_0,T_i(w,t)]=0,\,\,\,\,w\in U_i
\end{align}  
From $(\ref{jj6})$ and $(\ref{jj5})$, $f_{ik}(w,t)S_k(w,t)T_i(w,t)=S_i(w,t)T_i(w,t)=[\Lambda_0,S_i(w,t)]=[\Lambda_0,f_{ik}(w,t)S_k(w,t)]=S_k(w,t)[\Lambda_0,f_{ik}(w,t)]+f_{ik}(w,t)[\Lambda_0, S_k(w,t)]=S_k(w,t)[\Lambda_0,f_{ik}(w,t)]+f_{ik}(w,t)S_k(w,t)T_k(w,t)$ so that
\begin{align}\label{2c}
f_{ik}(w,t)T_i(w,t)-f_{ik}(w,t)T_k(w,t)=[\Lambda_0,f_{ik}(w,t)],\,\,\,\,w\in U_i\cap U_k
\end{align}
Then the conditions (\ref{2a}),(\ref{2b}),(\ref{2c}) show that $(\{f_{ik}(w,t)\},\{T_i(w,t)\})$ defines the Poisson line bundle over $(W,\Lambda_0)$ for each $t\in N$. We will denote the Poisson line bundle by $(\mathcal{N}_t,\nabla_t)$ for $t\in N$, where $\nabla_t$ is the Poisson connection on $\mathcal{N}_t$ which defines the Poisson line bundle structure (see \cite{Kim16}). Then we have a complex of sheaves on $W$ (see \cite{Kim16})
\begin{align*}
\mathcal{N}_t^\bullet:\mathcal{N}_t\xrightarrow{\nabla_t} \mathcal{N}_t\otimes T_W\xrightarrow{\nabla_t} \mathcal{N}_t\otimes \wedge^2 T_W\xrightarrow{\nabla_t} \cdots
\end{align*}
We will denote the $i$-th hypercohomology group by $\mathbb{H}^i(W,N_t^\bullet)$. We note that $\mathcal{N}_t^\bullet$ induces, by restriction on $V_t$, the complex of sheaves on $V_t$
\begin{align*}
\mathcal{N}_t^\bullet|_{V_t}:\mathcal{N}_t|_{V_t}\xrightarrow{\nabla_t|_{V_t}} \mathcal{N}_t|_{V_t}\otimes T_W|_{V_t}\xrightarrow{\nabla_t|_{V_t}} \mathcal{N}_t|_{V_t}\otimes \wedge^2 T_W|_{V_t}\xrightarrow{\nabla_t|_{V_t}} \cdots
\end{align*}
We will denote the $i$-th hypercohomology group by $\mathbb{H}^i(V_t,\mathcal{N}_t^\bullet|_{V_t})$. 

Denote by $(t_1,..,t_m)$ a system of holomorphic coordinates on $N$. For any tangent vector $v=\sum_{r=1}^m v_r\frac{\partial}{\partial t_r}$ of $M$ at $t\in N$, we set $\psi_i(w,t)=-\sum_{r=1}^m v_r\frac{\partial S_i(w,t)}{\partial t_r}$. Since $S_i(w,t)=0$ for $w\in V_t$, by taking the derivative of $(\ref{jj6})$ and $(\ref{jj5})$ with respect to $t$, we obtain
\begin{align}
& \psi_i(w,t)=f_{ik}(w,t)\psi_k(w,t),\,\,\,\,\,w\in V_t\cap U_i\cap U_k,\label{jj10}\\
  [\Lambda_0, \psi_i(w,t)]|_{V_t}=&[\Lambda_0,\psi_i(w,t)]|_{S_i(w,t)=0}=\psi_i(w,t) T_i(w,t),\,\,\,\,\,w\in V_t\cap U_i\label{jj11}
\end{align}
From $(\ref{jj10})$ and $(\ref{jj11})$, $\{\psi_i(w,t)\}$ define an element in $\mathbb{H}^0(V_t, \mathcal{N}_t^\bullet|_{V_t})$ so that we have a linear map
\begin{align*}
\rho_{d,t}:T_t M&\to \mathbb{H}^0(V_t, \mathcal{N}_t^\bullet|_{V_t})\\
\frac{\partial}{\partial t}&\mapsto \frac{\partial V_t}{\partial t}:=\{\psi_i(z_i,t)\}
\end{align*}
We call $\rho_{d,t}$ the characteristic map.

\begin{definition}\label{010}
Let $V_0$ be a compact holomorphic Poisson submanifold of a holomorphic Poisson manifold $(W,\Lambda_0)$ of codimension $1$ and let $(\mathcal{N}_0,\nabla_0)$ be the Poisson line bundle over $(W,\Lambda_0)$ determined by $V_0$. We denote by $r_0:\mathcal{N}^\bullet\to \mathcal{N}^\bullet|_{V_0}$ the restriction map of the following complex of sheaves
\begin{center}
$\begin{CD}
\mathcal{N}_0^\bullet: @.\mathcal{N}_0@>\nabla_0>> \mathcal{N}_0\otimes T_W@>\nabla_0>> \mathcal{N}_0\otimes \wedge^2 T_W@>\nabla_0>>\cdots\\
@.@VVV @VVV @VVV \\
\mathcal{N}_0^\bullet |_{V_0}: @. \mathcal{N}_0|_{V_0}@>\nabla_0|_{V_0}>> \mathcal{N}_0|_{V_0}\otimes T_W|_{V_0}@>\nabla_0|_{V_0}>> \mathcal{N}_0|_{V_0}\otimes \wedge^2 T_W|_{V_0}@>\nabla_0|_{V_0}>>\cdots
\end{CD}$
\end{center}
which induces a homomorphism $r_0^*:\mathbb{H}^1(W, \mathcal{N}_0^\bullet)\to \mathbb{H}^1(V_0,\mathcal{N}_0^\bullet|_{V_0})$. We say that $V_0$ is Poisson semi-regular if the image $r_0^* \mathbb{H}^1(W,\mathcal{N}_0^\bullet)$ is zero.

\end{definition}

\subsection{Theorem of existence}\label{mm18}\

We extend the argument in \cite{Kod59} in the context of Poisson deformations (see \cite{Kod59} p.484-493). We tried to maintain notational consistency with \cite{Kod59}. Let $(W,\Lambda_0)$ be a  holomorphic Poisson manifold of dimension $n+1\geq 2$ and let $V_0$ be a compact holomorphic Poisson submanifold of $(W,\Lambda_0)$ of dimension $n$. In what follows we denote by $p$ a point on $W$ and by $(w^1(p),,,,, w^{n+1}(p))$ the coordinate of $p$ with respect to a system of local holomorphic coordinate $(w^1,...,w^{n+1})$ on $W$. We choose a locally finite covering $\mathcal{U}=\{U_i\}$ of $W$ such that 
\begin{enumerate}
\item each neighborhood $U_i$ is a polycylinder: $U_i=\{p||w_i^1(p)|<1,...,|w_i^{n+1}(p)|<1\}$ where $(w_i^1,...,w_i^{n+1})$ is a system of local holomorphic coordinates which covers the closure of $U_i$.
\item $V_0\cap U_i$ coincides with the coordinate plane $w_i^{n+1}=0$ if $V_0\cap U_i$ is not empty,
\item $V_0\cap U_i\cap U_k$ is not empty if $V_0\cap U_i,V_0\cap U_k$ and $U_i\cap U_k$ are not empty.
\item if $V_0\cap U_i$ is not empty,  we write $w_i=w^{n+1}_i, z^1_i=w_i^1,...,z^n_i=w^n_i$.
\item $\{U_i^\delta\}$ covers $W$, where $U_i^\delta=\{p||w_i(p)|<1-\delta,...,|z_i^1(p)|<1-\delta,...,|z_i^n(p)|<1-\delta, |w_i(p)|<1\}$ for a sufficiently small $\delta$.
\end{enumerate}

Let $S_{i|0}(p)=w_i(p)$ if $V_0\cap U_i\ne \emptyset$ and $S_{i|0}(p)=1$ if $V_0\cap U_i=\emptyset$ for $p\in U_i$. If we let
\begin{align*}
f_{ik|0}(p)&:=\frac{S_{1|0}(p)}{S_{k|0}(p)},\,\,\,\,\,\text{for $p\in U_i\cap U_k$},\,\,\,\,\,\,\,\,\,\, \\
T_{i|0}(p)&:=[\Lambda_0, \log S_{i|0}(p)](\,\,\text{i.e}\,\, [\Lambda_0, S_{i|0}(p)]=S_{i|0}(p)T_{i|0}(p)),\,\,\,\,\,\text{for $p\in U_i$},
\end{align*}
then the system $(\{f_{ik|0}(p)\},\{T_{i|0}(p)\})$ defines the Poisson line bundle $\mathcal{N}_0$. Let $\{\beta_1,...,\beta_m\}$ be a basis of  $\mathbb{H}^0(V_0,\mathcal{N}_0^\bullet|_{V_0})$. Each $\beta_r$ is a holomorphic section of $\mathcal{N}_0$ over $V_0$ and $\beta_r$ is written in the form via the identification $\mathcal{N}_0|_{U_i}\cong U_i\times \mathbb{C}$,
\begin{align*}
\beta_r:p\mapsto (p,\beta_{ri}(p))
\end{align*}
where $\beta_{ri}(p)$ are holomorphic functions on $V_0\cap U_i$ satisfying
\begin{align}
&\beta_{ri}(p)=f_{ik|0}(p)\beta_{rk}(p),\,\,\,\,\,\text{for $p\in V_0\cap U_i\cap U_k$.}\\
-&[\Lambda_0, \beta_{ri}(p)]|_{V_0}+\beta_{ri}(p)  T_{i|0}(p)=0,\,\,\,\,\,\text{for $p\in V_0\cap U_i$}\label{22g}
\end{align}
With this preparation, we prove
\begin{theorem}[Theorem of existence]\label{jj20}
If $V_0$ is Poisson semi-regular, then there exists a Poisson analytic family $\mathcal{V}\subset (W\times M,\Lambda_0)\xrightarrow{\omega} M$ of compact holomorphic Poisson submanifolds of $(W,\Lambda_0)$ containing $V_0$ as the fibre $\omega^{-1}(0)$ over $0\in M$ such that the characteristic map
\begin{align*}
\rho_{d,0}: T_0 M&\to \mathbb{H}^0(V_0,\mathcal{N}_0|_{V_0}^\bullet)\\
\frac{\partial}{\partial t}&\mapsto \left( \frac{\partial V_t}{\partial t}\right)_{t=0}
\end{align*}
is an isomorphism.
\end{theorem}
\begin{proof}
Let $N$ be a spherical neighborhood of $0$ on the space of $m$ complex variables $t_1,...,t_m$, where $m=\dim \mathbb{H}^0(V_0,\mathcal{N}_0^\bullet |_{V_0})$. In order to prove Theorem \ref{jj20}, it suffices to construct a system $\{S_i(p,t)\}$ of holomorphic functions $S_i(p,t)$ defined on $U_i\times N$ and a system $\{f_{ik}(p,t)\}$ of non-vanishing holomorphic functions $f_{ik}(p,t)$ defined on $(U_i\cap U_k)\times N$ and a system $\{T_i(p,t)\}$ of  holomorphic vector fields $T_i(p,t)$ defined on $U_i\times N$ such that
\begin{align}
&S_i(p,t)=f_{ik}(p,t)S_k(p,t),\,\,\,\,\,\text{for $p\in U_i\cap U_k$},\label{22a}\\
&[\Lambda_0, S_i(p,t)]=S_i(p,t)T_i(p,t),\,\,\,\,\,\text{for $p\in U_i$}\label{22b},\\
&S_i(p,0)=S_{i|0}(p),\,\,\,f_{ik}(p,0)=f_{ik|0}(p),\,\,\,T_i(p,0)=T_{i|0}(p)\label{23i}\\
&S_i(p,t)\ne 0,\,\,\,\,\,\text{if $V_0\cap U_i=\emptyset$}\label{23j}\\
&\frac{\partial S_i(p,t)}{\partial t_r}|_{t=0}=\beta_{ri}(p),\,\,\,r=1,...,m.\label{23k}
\end{align}

We write $S_i(p,t),f_{ik}(p,t)$ and $T_i(p,t)$ in the forms

\begin{align*}
S_i(p,t)=S_{i|0}(p)+\sum_{\mu=1}^\infty S_{i|\mu}(p,t),\,\,\,\,\,
f_{ik}(p,t)=f_{ik|0}(p)+\sum_{\mu=1}^\infty f_{ik|\mu}(p,t),\,\,\,\,\,
T_i(p,t)=T_{i|0}(p)+\sum_{\mu=1}^\infty T_{i|\mu}(p,t)\,\,\,\,\,
\end{align*}
where $S_{i|\mu}(p,t),f_{ik|\mu}(p,t)$ and $T_{i|\mu}(p,t)$ are homogenous polynomials in $t=(t_1,...,t_m)$ of degree $\mu$ whose coefficients are holomorphic functions on $U_i$, on $U_i\cap U_k$, and holomorphic vector fields on $U_i$, respectively. Let
\begin{align}
S_i^\mu(p,t)&=S_{i|0}(p)+\sum_{\lambda=1}^\mu S_{i|\lambda}(p,t)\label{22f}\\
f_{ik}^\mu(p,t)&=f_{ik|0}(p)+\sum_{\lambda=1}^\mu f_{ik|\lambda}(p,t)\label{22x}\\
T_i^\mu(p,t)&=T_{i|0}(p)+\sum_{\lambda=1}^\mu T_{i|\lambda}(p,t)\label{22y}
\end{align}

\begin{notation}\label{notation1}
We write the power series expansion of a holomorphic function $P(t)$ in $t_1,...,t_m$ defined on a neighborhood of the origin $0$ in the form : $P(t)=P(t)+P_1(t)+\cdots+P_\mu(t)+\cdots$, where each $P_\mu(t)$ denotes a homogenous polynomial of degree $\mu$ in $t_1,...,t_m$. We set $P^\mu(t):=P(0)+P_1(t)+\cdots+P_\mu(t)$.  We write $[P(t)]_\mu$ for $P_\mu(t)$ when we substitute a complicated expression for $P(t)$. For any power series $P(t)$, and $Q(t)$ in $t=(t_1,...,t_m)$, we indicate $P(t)\equiv_\mu Q(t)$ that $P(t)-Q(t)$ contains no term of degree $\leq \mu$ in $t$.
\end{notation}
By Notation \ref{notation1}, $(\ref{22a})$ and $(\ref{22b})$ are equivalent to the system of congruences
\begin{align}
S_i^\mu(p,t)&\equiv_\mu f^\mu_{ik}(p,t)S_k^\mu(p,t)\label{22c}\\
[\Lambda_0, S_i^\mu(p,t)]&\equiv_\mu S_i^\mu(p,t)T_i^\mu(p,t),\,\,\,\,\,\mu=1,2,3,\cdots\label{22d}
\end{align}

We will construct $S_{i|\mu}(p,t),f_{ik|\mu}(p,t)$ by induction on $\mu$ satisfying $(\ref{22c})$ and $(\ref{22d})$. We assume the following special forms for $S_{i|\mu}(p,t),\mu\geq 1$:
\begin{equation}\label{22e}
S_{i|\mu}(p,t)=\begin{cases}
\psi_{i|\mu}(z_i(p),t),\,\,\,\,\,\text{if $V_0\cap U_i\ne \emptyset$}\\
0,\,\,\,\,\,\,\,\,\,\,\,\,\,\,\,\,\,\,\,\,\,\,\,\,\,\,\,\,\,\,\,\,\,\,\text{if $V_0\cap U_i= \emptyset$}
\end{cases}
\end{equation}
where $z_i(p)=(z_i^1(p),...,z_i^n(p))$ and $\psi_{i|\mu}(z_i,t)$ is a homogeneous polynomial of degree $\mu$ in $t$ whose coefficients are holomorphic functions of $z_i=(z_i^1,...,z_i^n)$ defined on the polycylinder: $|z_i^1|<1,...,|z_i^n|<1$.

We define $\psi_{i|1}(z_i,t)$ by 
\begin{align}
\psi_{i|1}(z_i(p),t)=\sum_{r=1}^m t_r \beta_{ri}(p),\,\,\,\,p\in V_0\cap U_i.
\end{align}
and determine $S_{i|0}(p,t),S_i^1(p,t)$ by $(\ref{22e})$ and $(\ref{22f})$. By letting $f_{ik|1}(p,t)=\frac{S_{i|1}(p,t)-f_{ik|0}(p)S_{k|1}(p,t)}{S_{k|0}(p)}$, $f_{ik}^1(p,t)=f_{ik|0}(p)+f_{ik|1}(p,t)$ is holomorphic in $p$ and satisfies $(\ref{22c})_1$ (for the detail, see \cite{Kod59} p.486).

On the other hand, we note that $[\Lambda_0, S_i^1(p,t)]\equiv_1 S_i^1 (p,t)T_i^1(p,t)$ is equivalent to $[\Lambda_0, S_{i|1}(p,t)]=S_{i|1}(p,t)T_{i|0}(p)+S_{i|0}(p)T_{i|1}(p,t)$.
By letting
\begin{align*}
T_{i|1}(p,t)=\frac{[\Lambda_0, S_{i|1}(p,t)]-S_{i|1}(p,t) T_{i|0}(p)}{S_{i|0}(p)},
\end{align*}
$T_i^1(p,t)=T_{i|0}(p)+T_{i|1}(p,t)$ is holomorphic in $p$ by $(\ref{22g})$ and satisfies $(\ref{22d})_1$.

Now suppose that we have already constructed $S_i^\mu(p,t),f_{ik}^\mu(p,t), T_i^\mu(p,t)$ satisfying $(\ref{22c})_\mu$ and $(\ref{22d})_\mu$ which imply that
\begin{align}
& f_{ik}^\mu(p,t)\equiv_\mu f_{ij}^\mu(p,t) f_{jk}^\mu (p,t)\label{22i} \\
&[\Lambda_0, T_i^\mu(p,t)]=0 \label{22j}\\
&[\Lambda_0 ,f_{ik}^{\mu}(p,t)]+f_{ik}^\mu(p,t)T_k^{\mu}(p,t)-f_{ik}^{\mu}(p,t)T_i^{\mu}(p,t)\equiv_\mu 0\label{22k}
\end{align}

We define homogenous polynomials $\psi_{ik|\mu+1}(p,t)$ in $t$ of degree $\mu+1$ whose coefficients are in $\Gamma(V_0\cap U_i\cap U_k, \mathcal{O}_{V_0})$  and homogenous polynomials $W_{i|\mu+1}(p,t)$ of degree $\mu+1$ whose coefficient are in $\Gamma(U_i\cap V_0, T_W|_{V_0})$ by
\begin{align}
\psi_{ik|\mu+1}(p,t)&\equiv_{\mu+1} f_{ik}^\mu(p,t)S_k^\mu(p,t)-S_i^\mu(p,t)\,\,\,\,\,\text{for $p\in V_0\cap U_i\cap U_k$}\\
W_{i|\mu+1}(p,t)&\equiv_{\mu+1} [\Lambda_0, S_i^\mu(p,t)]|_{V_0}-S_i^\mu(p,t)T_i^\mu(p,t),\,\,\,\,\,\text{for $p\in V_0\cap U_i$}\label{22h}
\end{align}

Then we have (for the detail, see \cite{Kod59} p.487)
\begin{align}\label{22m}
\psi_{ik|\mu+1}(p,t)=\psi_{ij|\mu+1}(p,t)+f_{ij|0}(p)\psi_{jk|\mu+1}(p,t),\,\,\,\,\,\text{for $p\in V_0\cap U_i\cap U_j\cap U_k$}
\end{align}

On the other hand, let $\tilde{W}_{i|\mu+1}(p,t)\equiv_{\mu+1} [\Lambda_0, S_i^\mu(p,t)]-S_i^\mu(p,t)T_i^\mu(p,t)$ for $p\in U_i$. Then from (\ref{22j}), we obtain\begin{align*}
[\Lambda_0, \tilde{W}_{i|\mu+1}(p,t)]&\equiv_{\mu+1}-[\Lambda_0, S_i^\mu(p,t) T_i^\mu(p,t)]\equiv_{\mu+1}[\Lambda_0, S_i^\mu(p,t)]\wedge T_i^\mu(p,t)-S_i^\mu(p,t)[\Lambda_0, T_i^\mu(p,t)]\\
&\equiv_{\mu+1}\tilde{W}_{i|\mu+1}(p,t)\wedge T_i^\mu(p,t)+S_i^\mu(p,t) T_i^\mu(p,t)\wedge T_i^\mu(p,t)-S_i^\mu(p,t)[\Lambda_0,T_i^\mu(p,t)]\\
&\equiv_{\mu+1}\tilde{W}_{i|\mu+1}(p,t)\wedge T_{i|0}(p)
\end{align*}
By restricting to $V_0$, equivalently by taking $S_{i|0}(p)=0$, we get
\begin{align}\label{22n}
-[ W_{i|\mu+1}(p,t), \Lambda_0]|_{V_0}+(-1)^1 W_{i|\mu+1}(p,t)\wedge T_{i|0}(p)=0
\end{align}

Lastly, let $G_{ik|\mu+1}(p,t)\equiv_{\mu+1} f_{ik}^\mu(p,t)(T_i^\mu(p,t)-T_k^\mu(p,t))-[\Lambda_0, f_{ik}^\mu(p,t)]$. We note that $G_{ik|\mu+1}(p,t)\equiv_\mu 0$ from (\ref{22k}). Then we have, by restricting to $V_0$, equivalently by taking $S_{i|0}(p)=0$,
\begin{align*}
&[\Lambda_0,\psi_{ik|\mu+1}(p,t)]|_{V_0}\equiv_{\mu+1} [\Lambda_0, f_{ik}^\mu(p,t)S_k^\mu(p,t)]|_{V_0}-[\Lambda_0,S_i^\mu(p,t)]|_{V_0}\\
&\equiv_{\mu+1} f_{ik}^\mu(p,t)[\Lambda_0, S_k^\mu(p,t)]|_{V_0}+S_k^\mu(p,t)[\Lambda_0, f_{ik}^\mu(p,t)]|_{V_0}-[\Lambda_0, S_i^\mu(p,t)]|_{V_0}\\
&\equiv_{\mu+1}f_{ik}^\mu(p,t) W_{k|\mu+1}(p,t) + f_{ik}^\mu(p,t) S_k^\mu(p,t) T_k^\mu(p,t) +S_k^\mu(p,t)[\Lambda_0, f_{ik}^\mu(p,t)]|_{V_0}-W_{i|\mu+1}(p,t)-S_i^\mu(p,t)T_i^\mu(p,t)\\
&\equiv_{\mu+1} f_{ik}^\mu(p,t) W_{k|\mu+1}(p,t) +f_{ik}^\mu(p,t) S_k^\mu(p,t) T_k^\mu(p,t) +S_k^\mu(p,t)f_{ik}^\mu(p,t)(T_i^\mu(p,t)-T_k^\mu(p,t))-S_k^\mu (p,t)G_{ik|\mu+1}(p,t)\\
&-W_{i|\mu+1}(p,t) -S_i^\mu(p,t) T_i^\mu(p,t)\\
&\equiv_{\mu+1} f_{ik}^\mu(p,t) W_{k|\mu+1}(p,t) -W_{i|\mu+1}(p,t)+(f_{ik}^\mu(p,t) S_k^\mu(p,t)-S_i^\mu(p,t)) T_i^\mu(p,t)-S_{k|0}(p)G_{ik|\mu+1}(p,t)\\
&\equiv_{\mu+1}  f_{ik|0}(p) W_{k|\mu+1}(p,t) -W_{i|\mu+1}(p,t)+\psi_{ik|\mu+1}(p,t)T_{i|0}(p)
\end{align*}
Hence we obtain
\begin{align}\label{22l}
-[\psi_{ik|\mu+1}(p,t),\Lambda_0]|_{V_0}+\psi_{ik|\mu+1}(p,t)T_{i|0}(p)+f_{ik|0}(p) W_{k|\mu+1}(p,t) -W_{i|\mu+1}(p,t)=0
\end{align}

Hence from $(\ref{22m})$, $(\ref{22n})$ and $(\ref{22l})$, $(\psi_{\mu+1}(t),W_{\mu+1}(t)):=(\{\psi_{ik|\mu+1}(p,t)\}, \{W_{i|\mu+1}(p,t) \})$ defines a $1$-cocycle in the following \v{C}ech resolution of $\mathcal{N}_0^\bullet|_{V_0}$:
\\
\\
\begin{center}
$\begin{CD}
C^0(\mathcal{U}\cap V_0,\mathcal{N}_0|_{V_0}\otimes \wedge^2 T_W|_{V_0})\\
@A\nabla_0|_{V_0}AA \\
C^0(\mathcal{U}\cap V_0,\mathcal{N}_0|_{V_0}\otimes T_W|_{V_0})@>\delta>> C^1(\mathcal{U}\cap V_0,\mathcal{N}_0|_{V_0}\otimes T_W|_{V_0})\\
@A\nabla_0|_{V_0}AA @A\nabla_0|_{V_0}AA\\
C^0(\mathcal{U}\cap V_0,\mathcal{N}_0|_{V_0})@>-\delta>> C^1(\mathcal{U}\cap V_0,\mathcal{N}_0|_{V_0})@>\delta>> C^2(\mathcal{U}\cap V_0, \mathcal{N}_0|_{V_0})
\end{CD}$
\end{center}\
\\

Next we show that if $(\psi_{\mu+1}(t), W_{\mu+1}(t))$ vanishes identically, we can construct $S_i^{\mu+1}(p,t), f_{ik}^{\mu+1}(p,t)$ and $T_i^{\mu+1}(p,t)$ satisfying $(\ref{22c})_{\mu+1}$ and $(\ref{22d})_{\mu+1}$. Indeed, let us assume that $(\psi_{\mu+1}(t),W_{\mu+1}(t))$ vanishes identically. Then there exists homogenous polynomials $\phi_{i|\mu+1}(p,t)$ in $t$ of degree $\mu+1$ whose coefficients are holomorphic functions on $V_0\cap U_i$ such that
\begin{align}
\psi_{ik|\mu+1}(p,t)&=\phi_{i|\mu+1}(p,t)-f_{ik|0}(p)\phi_{k|\mu+1}(p,t),\,\,\,p\in V_0\cap U_i\cap U_k, \label{ii62}\\
W_{i|\mu+1}(p,t)&=-[\Lambda_0,\phi_{i|\mu+1}(p,t)]|_{V_0}+\phi_{i|\mu+1}(p,t)T_{i|0}(p),\,\,\,p\in V_0\cap U_i.\label{22z}
\end{align}
We define $\psi_{i|\mu+1}(z_i,t):=\psi_{i|\mu+1}(z_i(p),t)=\phi_{i|\mu+1}(p,t),p\in V_0\cap U_i$ and determine $S_{i|\mu+1}(p,t)$ by (\ref{22e}), and $S_i^{\mu+1}(p,t)$ by (\ref{22f}). We define $f_{ik|\mu+1}(p,t)\equiv_{\mu+1}\frac{S_i^\mu(p,t)-f_{ik}^\mu(p,t)S_k^\mu(p,t)+S_{i|\mu+1}(p,t)-f_{ik|0}(p)S_{k|\mu+1}(p,t)}{S_{k|0}(p)}$ and determine $f_{ik}^{\mu+1}(p,t)$ by (\ref{22x}). Then $S_i^{\mu+1}(p,t)$ and $f_{ik}^{\mu+1}(p,t)$ satisfy $(\ref{22c})_{\mu+1}$ (for the detail, see \cite{Kod59} p.488).

On the other hand, we note that from (\ref{22z}), we have
\begin{align}\label{23a}
W_{i|\mu+1}(p,t)=-[\Lambda_0, S_{i|\mu+1}(p,t)]|_{V_0}+S_{i|\mu+1}(p,t)T_{i|0}(p),\,\,\,p\in V_0\cap U_i.
\end{align}
We set
\begin{align}\label{23b}
T_{i|\mu+1}(p,t) :\equiv_{\mu+1} \frac{\Phi^\mu(p,t)}{S_{i|0}(p)}= \frac{[\Lambda_0, S_i^\mu(p,t)]-S_i^{\mu}(p,t)T_i^\mu(p,t) +[\Lambda_0, S_{i|\mu+1}(p,t)]-S_{i|\mu+1}(p,t)T_{i|0}(p) }{S_{i|0}(p)}
\end{align}
Then $\Phi^\mu(p,t)\equiv_{\mu+1} 0$ for $p\in V_0\cap U_i$ from $(\ref{22h})$ and $(\ref{23a})$ so that $T_{i|\mu+1}(p,t)$ are holomorphic in $p$ and we have, from (\ref{23b}),
\begin{align*}
[\Lambda_0, S_i^\mu(p,t)+S_{i|\mu+1}(p,t)]&\equiv_{\mu+1} S_i^\mu(p,t) T_i^\mu(p,t)+ S_{i|\mu+1}(p,t)T_{i|0}(p)+S_{i|0}(p) T_{i|\mu+1}(p,t)\\
&\equiv_{\mu+1} (S_i^\mu(p,t)+S_{i|\mu+1}(p,t))(T_i^\mu(p,t)+T_{i|\mu+1}(p,t))
\end{align*}
so that $S_i^{\mu+1}(p,t)$ and $T_i^{\mu+1}(p,t)$ satisfy $(\ref{22d})_{\mu+1}$.

Now we prove that $(\psi_{\mu+1}(t),W_{\mu+1}(t))$ vanishes identically if $V_0$ is Poisson semi-regular. For this purposes it suffices to construct a polynomial $(\{\eta_{ki}(p,t)\}, \{\omega_i(p,t)\})$ in $t$ of degree $\mu+1$ with coefficients in $\mathbb{H}^1(W,\mathcal{N}_0^\bullet)$  such that
\begin{align}
\psi_{ik|\mu+1}(p,t)&=\eta_{ik}(p,t)\,\,\,\,\,\text{for $p\in V_0\cap U_i\cap U_k$},\label{our1}\\
W_{i|\mu+1}(p,t)&=\omega_{i}(p,t)\,\,\,\,\,\,\text{for $p\in V_0 \cap U_i$}.\label{our2}
\end{align}

In fact, $\{\eta_{ik}(p,t),\omega_i(p,t)\}$ represents a polynomial $(\eta(t),\omega(t))$ in $t$ with coefficients in $\mathbb{H}^1(W,\mathcal{N}_0^\bullet)$, and $(\ref{our1})$ and $(\ref{our2})$ imply that $(\psi_{\mu+1}(t), W_{\mu+1}(t))=r_0^* (\eta(t),\omega(t))$. Hence we obtain $(\psi_{\mu+1}(t), W_{\mu+1}(t))$ vanishes if $V_0$ is Poisson semi-regular.

\begin{lemma}[see \cite{Kod59} Lemma $1$ p.488]
For each integer $\lambda\leq \mu$, there exist polynomials $g_{ik}:=g_{ik}^\lambda(p,t)$ in $t$ of degree $\lambda$ whose coefficients are holomorphic functions in $p$ defined on $U_i\cap U_k$ such that
\begin{align}
f_{ik|0}(p)\exp g_{ik}^\lambda(p,t)\equiv_\lambda f_{ik}^\lambda(p,t)\label{23c}
\end{align}
\end{lemma}

We define polynomials $\hat{f}_{ik}^{\mu+1}:=\hat{f}_{ik}^{\mu+1}(p,t)\equiv_{\mu+1} f_{ik|0}(p)\exp g_{ik}^\mu(p,t)$ in $t$ of degree $\mu+1$. Then $\hat{f}_{ik}^{\mu+1}(p,t)\equiv_\mu f_{ik}^\mu(p,t)$ and $S_i^{\mu}(p,t)\equiv_\mu \hat{f}^{\mu+1}(p,t)S_k^\mu(p,t)$. By letting $\eta_{ik}(p,t)\equiv_{\mu+1} \hat{f}_{ik}^{\mu+1}(p,t)S_k^\mu(p,t)-S_i^\mu(p,t)$, we get (for the detail, see \cite{Kod59} p.489)
\begin{align}\label{23f}
\eta_{ik}(p,t)=\eta_{ij}(p,t)+f_{ij|0}(p)\eta_{jk}(p,t)
\end{align}

On the other hand, writing $f_{ik}:=f_{ik}(p,t)$ and $T_i:=T_i(p,t)$, we note that from $(\ref{22k})$, we have $[\Lambda_0, f_{ik}^\mu]+f_{ik}^\mu(T_k^\mu-T_i^\mu)\equiv_{\mu} 0$. Since $f_{ik|0}\exp(g_{ik}^\mu)\equiv_\mu f_{ik}^\mu$ by (\ref{23c}), we have
\begin{align}
&[\Lambda_0, f_{ik|0} \exp g^\mu_{ik}]+f_{ik|0}\exp(g_{ik}^\mu)(T_k^\mu-T_i^\mu)\equiv_\mu 0\notag\\
&\iff [\Lambda_0, \log (f_{ik|0} \exp g_{ik}^\mu)]+(T_k^\mu-T_i^\mu)\equiv_\mu 0\notag\\
&\iff [\Lambda_0, \log f_{ik|0}]+[\Lambda_0, g_{ik}^\mu]+(T_k^\mu-T_i^\mu)=0\notag\\
&\iff [\Lambda_0, \hat{f}_{ik}^{\mu+1}]+\hat{f}_{ik}^{\mu+1}(T_k^\mu-T_i^\mu)= 0 \label{23d}
\end{align}

From $(\ref{22c})$, we have $[\Lambda_0, S_i^{\mu}(p,t)]\equiv_{\mu} S_i^\mu(p,t)T_i^\mu(p,t)$. We define $\omega_i:=\omega_i(p,t)\equiv_{\mu+1} [\Lambda_0, S_i^{\mu}(p,t)]-S_i^\mu(p,t)T_i^\mu(p,t)$. Then we have from (\ref{23d})
\begin{align*}
[\Lambda_0, \eta_{ik}]&\equiv_{\mu+1} [\Lambda_0, \hat{f}_{ik}^{\mu+1}S_k^\mu]-[\Lambda_0, S_i^\mu]\equiv_{\mu+1} S_k^\mu[\Lambda_0, f_{ik}^{\mu+1}]+\hat{f}_{ik}^{\mu+1}[\Lambda_0, S_k^\mu]-[\Lambda_0, S_i^\mu]\\
&\equiv_{\mu+1} S_k^\mu[\Lambda_0, \hat{f}_{ik}^{\mu+1}]+\hat{f}_{ik}^{\mu+1}S_k^\mu T_k^\mu+ f_{ik|0}\omega_k-S_i^\mu T_i^\mu-\omega_i\\
&\equiv_{\mu+1} S_k^\mu[\Lambda_0, \hat{f}_{ik}^{\mu+1}]+\hat{f}_{ik}^{\mu+1}S_k^\mu T_k^\mu+ f_{ik|0}\omega_k+(\eta_{ik}- \hat{f}_{ik}^{\mu+1} S_k^\mu)T_i^\mu-\omega_i\\
&\equiv_{\mu+1} S_k^\mu([\Lambda_0,\hat{f}_{ik}^{\mu+1}]+  \hat{f}_{ik}^{\mu+1}T_k^\mu- \hat{f}_{ik}^{\mu+1}T_i^\mu)+ \eta_{ik} T_{i|0} +f_{ik|0}\omega_k -\omega_i\\
&\equiv_{\mu+1}\eta_{ik}T_{i|0}+ f_{ik|0} \omega_k -\omega_i
\end{align*}
so that we obtain the equality
\begin{align}\label{23e}
-[\eta_{ik}(p,t),\Lambda_0]+ \eta_{ik}(p,t)T_{i|0}(p) +f_{ik|0}(p)\omega_k(p,t) -\omega_i(p,t)=0,\,\,\,p\in U_i\cap U_k.
\end{align}

Lastly, we have from (\ref{22j}),
\begin{align*}
[\Lambda_0,\omega_i]\equiv_{\mu+1} -[\Lambda_0, S_i^\mu T_i^\mu]&\equiv_{\mu+1} [\Lambda_0, S_i^\mu]\wedge T_i^\mu-S_i^\mu[\Lambda_0, T_i^\mu]\equiv_{\mu+1} \omega_i\wedge T_i^\mu+S_i^\mu T_i^\mu\wedge T_i^\mu-S_i^\mu[\Lambda_0, T_i^\mu]\equiv_{\mu+1} \omega_i\wedge T_{i|0}
\end{align*}
so that we obtain the equality
\begin{align}\label{23g}
-[\omega_i(p,t),\Lambda_0]+(-1)^1 \omega_i(p,t)\wedge T_{i|0}(p)=0,\,\,\,p\in U_i.
\end{align}

Form (\ref{23f}), (\ref{23e}), and (\ref{23g}), $(\{\eta_{ik}(p,t)\},\{\omega_i(p,t)\})$ is a polynomial in $t$ whose coefficients in $\mathbb{H}^1(W,\mathcal{N}_0)$. Then since $S_k^\mu(p,t)\equiv_0 0$ for $p\in V_0\cap U_k$, and $\hat{f}_{ik}^{\mu+1}(p,t)\equiv_\mu f_{ik}^\mu(p,t)$, we obtain
\begin{align*}
\eta_{ik}(p,t)&\equiv_{\mu+1} f_{ik}^\mu(p,t)S_k^\mu(p,t)-S_i^\mu(p,t)\equiv_{\mu+1} \psi_{ik|\mu+1}(p,t),\,\,\,p\in V_0\cap U_k,\\
\omega_i(p,t)&\equiv_{\mu+1} [\Lambda_0,S_i^\mu(p,t)]|_{V_0}-S_i^{\mu}(p,t)T_{i}^\mu(p,t)\equiv_{\mu+1} W_{i|\mu+1}(p,t),\,\,\, p\in V_0\cap U_i.
\end{align*}
Hence when $V_0$ is Poisson semi-regular, we can construct $S_i^\mu(p,t),f_{ik}^\mu(p,t)$ and $T_i^\mu(p,t)$ satisfying $(\ref{22c})_\mu$ and $(\ref{22d})_\mu$ by induction on $\mu$, and therefore we obtain formal power sereis $S_i(p,t),f_{ik}(p,t)$ and $T_i(p,t)$ satisfying (\ref{22a}),(\ref{22b}),(\ref{23i}),(\ref{23j}) and (\ref{23k}).

\subsection{Proof of convergence}\label{sectionp}\
\begin{notation}\label{notation2}
Consider a formal power series $f(t)=f(p,t)=\sum f_{h_1h_2\cdots h_m}(p)(t_1)^{h_1}(t_2)^{h_2}\cdots (t_m)^{h_m}$ whose coefficients $f_{h_1h_2\cdots h_m}(p)$ are vector-valued holomorphic functions in $p$ defined on a domain and a power series $a(t)=\sum a_{h_1h_2\cdots h_m}(t_1)^{h_1}(t_2)^{h_2}\cdots (t_m)^{h_m},a_{h_1h_1\cdots h_m}\geq 0$. We indicate by $f(p,t)\ll a(t)$ that $|f_{h_1\cdots h_m}(p)|<a_{h_1\cdots h_m}$. Let
\begin{align*}
A(t)=\frac{b}{64c}\sum_{\mu=1}^\infty \frac{c^\mu(t_1+\cdots +t_m)^\mu}{\mu^2},
\end{align*}
where $b,c$ are positive constants. Then we have
\begin{align}\label{ii0}
A(t)^v\ll \left(\frac{b}{c}\right)^{v-1}A(t),\,\,\,\,v=2,3,\cdots.
\end{align}
\end{notation}

We will show that the formal power series $S_i(t):=S_i(p,t),f_{ik}(t):=f_{ik}(p,t)$ and $T_i(t):=T_i(p,t)$ constructed in the previous subsection satisfy
\begin{align}
S_i(t)-S_{i|0}&\ll A(t),\,\,\,\,\,p\in U_i\cap U_k, \label{our5}\\
f_{ik}(t)-f_{ik|0}& \ll c_1 A(t),\,\,\,\,\,p\in U_i,\label{our6}\\
T_i(t)-T_{i|0}&\ll d_1A(t),\,\,\,\,\,p=(z_i(p),w_i(p))\in U_i\,\,\textnormal{with}\,\,|w_i(p)|<1,|z_i(p)|<1-\delta\iff p\in U^\delta_i, \label{our7}
\end{align}
for some constants $b,c,c_1,d$ from Notation \ref{notation2} and a sufficiently small number $\delta>0$ in the beginning of subsection $\ref{mm18}$. Here we write $T_i(t)=T_i(p,t)$ by the form $T_i(p,t)=T_i^1(p,t)\frac{\partial}{\partial z_i^1}+\cdots T_i^n (p,t)\frac{\partial}{\partial z_i^n}+T_i^{n+1}(p,t)\frac{\partial}{\partial w_i}$ by which we consider $T_i(p,t)$ a power series in $t$ whose coefficients are vector-valued holomorphic functions on $U_i$.

We may assume that $|f_{ik}(p)|<c_2$, $p\in U_i\cap U_k$ for some constant $c_2>0$. Then $S_i^1(t)-S_{i|0}\ll A(t)$ if $b$ is sufficiently large.

Suppose that 
\begin{align}
f^{\mu-1}_{ik}(t)-f_{ik|0}&\ll c_1 A(t),\,\,\,\,\, p\in U_i\cap U_k \label{ii1}\\
S_i^\mu(t)-S_{i|0}&\ll A(t),\,\,\,\,\,p\in U_i \label{ii2}\\
T_i^{\mu-1}(t)-T_{i|0}&\ll d_1A(t),\,\,\,\,\,p=(z_i(p),w_i(p))\in U_i\,\,\textnormal{with}\,\,|w_i(p)|<1,|z_i(p)|<1-\delta\iff p\in U^\delta_i \label{ii3}
\end{align}

First we show that $f_{ik}^\mu(t)-f_{ik|0}\ll c_1A(t),p\in U_i\cap U_k$ for some constant $c_1>0$. We briefly summarize Kodaira's result in the following (see \cite{Kod59} p.491-492): by setting $c_1=\frac{2(c_2+2)(c_2+c_2^3)}{\epsilon}$ for some sufficiently small constant $0<\epsilon<1$, and assuming 
\begin{align}\label{ik1}
c>2bc_1(1+c_2^2), 
\end{align}
we get $f_{ik}^\mu(t)-f_{ik|0}\ll c_1A(t),p\in U_i\cap U_k$.

Next we show that 
\begin{align}\label{ij1}
T_i^\mu(t)-T_{i|0}\ll d_1 A(t),\,\,\,\,\,p\in U_i^\delta
\end{align}
for some constant $d_1>0$. We may assume that 
\begin{align}\label{ii33}
|T_{i|0}(p)|<d_2, \,\,\,\,\, p\in U_i
\end{align}
for some constant $d_2>0$. We recall from (\ref{23b}) that
\begin{align}\label{ii55}
T_{i|\mu}(p,t)\equiv_\mu \frac{[\Lambda_0, S_i^{\mu-1}] -T_i^{\mu-1} S_i^{\mu-1}+[\Lambda_0, S_{i|\mu}]-T_{i|0} S_{i|\mu}}{S_{i|0}}\equiv_{\mu} \frac{[\Lambda_0,S_i^\mu]-T_i^{\mu-1} S_i^\mu}{S_{i|0}}
\end{align}

We estimate $[\Lambda_0, S_i^\mu]-T_i^{\mu-1}S_i^\mu$. We note that
\begin{align}
[\Lambda_0,S_i^\mu]-T_i^{\mu-1}S_i^\mu=[\Lambda_0,&S_i^\mu -S_{i|0}]-(T_i^{\mu-1}-T_{i|0})(S_i^\mu-S_{i|0})+[\Lambda_0, S_{i|0}]-T_{i|0}(S_i^\mu-S_{i|0})-(T_i^{\mu-1} -T_{i|0})S_{i|0}-T_{i|0}S_{i|0}\notag\\
&=[\Lambda_0,S_i^\mu -S_{i|0}]-(T_i^{\mu-1}-T_{i|0})(S_i^\mu-S_{i|0})-T_{i|0}(S_i^\mu-S_{i|0})-(T_i^{\mu-1} -T_{i|0})S_{i|0} \label{ii4}
\end{align}

We note that since $[\Lambda_0,S_i^\mu]-T_i^{\mu-1}S_i^\mu\equiv_{\mu-1} 0$, $(T_i^{\mu-1} -T_{i|0})S_{i|0}$ contributes nothing to $[\Lambda_0,S_i^{\mu}]-S_i^{\mu}T_i^{\mu-1}$. Let us estimate $[\Lambda_0, S_i^\mu-S_{i|0}]$ in $(\ref{ii4})$. Let $\Lambda_0=\sum_{\alpha,\beta=1}^{n+1} \Lambda^i_{\alpha\beta}(x_i)\frac{\partial}{\partial x_i^\alpha}\wedge \frac{\partial}{\partial x_i^\beta}$ with $\Lambda_{\alpha\beta}^i(x_i)=-\Lambda_{\beta\alpha}^i(x_i)$, where $x_i=(w_i,z_i)$. We may assume that $|\Lambda^i_{\alpha\beta}(x_i)|< M$ for some positive constant $M>0$. $[\Lambda_0, S_i^\mu-S_{i|0}]=\sum_{\alpha,\beta=1}^{n+1} 2\Lambda_{\alpha\beta}^i(x_i)\frac{\partial (S_i^\mu-S_{i|0})}{\partial x_i^\alpha}\frac{\partial}{\partial x_i^\beta}$. Let $B_i(z_i):=S_i^\mu-S_{i|0}$. Then $\frac{\partial B_i(z_i)}{\partial w_i}=0$ and $\frac{\partial B_i}{\partial z_i^\alpha}=\frac{1}{2\pi i}\int_{|\xi-z_i^\alpha|=\delta} \frac{B_i(z_i^1,...,\overset{\alpha-th}{\xi},...,z_i^n)}{(\xi-z_i^\alpha)^2}d\xi\ll \frac{A(t)}{\delta}$ for $(z_i,w_i)$ for $|z_i|<1-\delta, |w_i|<1$. Then we obtain
\begin{align}\label{ii5}
[\Lambda_0,S_i^\mu-S_{i|0}]=\sum_{\alpha,\beta=1}^{n+1} 2 \Lambda_{\alpha\beta}^i(x_i) \frac{\partial B_i(z_i)}{\partial x_i^\alpha}\frac{\partial}{\partial x_i^\beta}\ll 2(n+1)^2M\frac{A(t)}{\delta}
\end{align}
Hence we have, from $(\ref{ii4})$, $(\ref{ii3})$, $(\ref{ii5})$, $(\ref{ii33})$ and $(\ref{ii0})$,
\begin{align}\label{ii77}
 [\Lambda_0, S_i^\mu]-T_i^{\mu-1} S_i^\mu\ll \frac{2(n+1)^2 M}{\delta} A(t) +d_1A(t)^2+d_2 A(t)\ll d_3A(t),\,\,\,\,\,p\in U_i^\delta,
 \end{align}
where 
\begin{align}\label{ii56}
d_3:=\frac{2(n+1)^2M}{\delta}+\frac{d_1b}{c}+d_2=d_4+\frac{d_1 b}{c},\,\,\, \text{where} \,\,\,\,d_4:=\frac{2(n+1)^2M}{\delta}+d_2.
\end{align}
 We claim that
\begin{align}\label{ii57}
T_{i|\mu}(t)\ll \frac{d_3}{\epsilon} A(t),\,\,\,\,\,p\in U_i^\delta
\end{align}
Indeed, from $S_{i|0}(p)=w_i(p)$, $(\ref{ii55})$ and $(\ref{ii77})$, if $|w_i(p)|\geq\epsilon $, $T_{i|\mu}(t)\ll \frac{d_3}{\epsilon}A(t)$. If $|w_i(p)|<\epsilon$, we get $T_{i|\mu}(t)\ll \frac{d_3}{\epsilon}A(t)$ by the maximum principle. 

On the other hand, from $(\ref{ii56})$, we have $\frac{d_3}{\epsilon}=\frac{d_4}{\epsilon}+\frac{d_1b}{\epsilon c}$. Now we set $d_1=\frac{2d_4}{\epsilon}$. If we take 
\begin{align}\label{ik2}
c>\frac{2b}{\epsilon}, 
\end{align}
then we get $\frac{d_3}{\epsilon}< \frac{d_4}{\epsilon}+\frac{d_4}{\epsilon}=d_1$ so that we obtain $(\ref{ij1})$ from $(\ref{ii3})$ and $(\ref{ii57})$.

Lastly we show that 
\begin{align}\label{ij2}
S_i^{\mu+1}(t)-S_{i|0}\ll A(t),\,\,\,\,\,p\in U_i
\end{align}
We note that (for the detail, see \cite{Kod59} p.493)
\begin{align}
\psi_{ik|\mu+1}(t)\ll \frac{bc_1}{c}A(t),\,\,\,\,\,p\in V_0\cap U_i\cap U_k
\end{align}

Recall from $(\ref{22z})$ that $W_{i|\mu+1}(p,t)\equiv_{\mu+1} [\Lambda_0, S_i^{\mu}(p,t)]|_{V_0}-S_i^\mu(p,t)T_i^\mu(p,t)$. Since $[\Lambda_0,S_i^\mu]-S_i^\mu T_i^\mu\equiv_\mu 0$, we get $[\Lambda_0, S_i^\mu]-S_i^{\mu}T_i^\mu\equiv_{\mu+1} [\Lambda_0, S_i^{\mu}]-(S_i^\mu-S_{i|0})(T_i^\mu-T_{i|0})-S_i^\mu T_{i|0}-S_{i|0}(T_i^\mu-T_{i|0})$. Since $[\Lambda_0,S_i^\mu]-S_i^\mu T_i^\mu\equiv_\mu 0$, we obtain, from $(\ref{ii2})$ and $(\ref{ij1})$,
\begin{align}
 W_{i|\mu+1}(p,t) \ll \frac{bd_1}{c}A(t),\,\,\,\,\,\,p\in V_0\cap U_i^\delta.
\end{align}

\begin{lemma}[compare \cite{Kod59} p.499]\label{1lemma}
We can choose $\phi_{i|\mu+1}(p,t)$ satisfying 
\begin{align*}
\psi_{ik|\mu+1}(p,t)&=\phi_{i|\mu+1}(p,t)-f_{ik|0}(p)\phi_{k|\mu+1}(p,t)\\
W_{i|\mu+1}(p,t)&= -[\Lambda_0, \phi_{i\mu+1}(p,t)]|_{V_0}+\phi_{i|\mu+1}T_{i|0}(p)
\end{align*}
such that $\phi_{i|\mu+1}\ll c_4\left(\frac{bc_1}{c}+\frac{bd_1}{c} \right)A(t)$, where the constant $c_4$ is independent of $\mu$.
\end{lemma}

\begin{proof}
For simplicity, we write $U_i$ for $V_0\cap U_i$, $U_i^\delta$ for $V_0\cap U_i^\delta$, and let $\mathcal{U}=\{U_i\}$ be the covering of $V_0$. For any $0$-cochain $\phi=\{\phi_i(p)\}$, $1$-cochain $\psi=\{\psi_{ik}(p), W_i(p) \}$ on $\mathcal{U}$, we define the norms of $\phi,(\psi,W)$ by
\begin{align*}
&||\phi ||:=\max_i \sup_{p\in U_i} |\phi_i(p)|,\\
&||(\psi,W)||:=\max_{i,k} \sup_{p\in U_i\cap U_k} |\psi_{ik}(p)|+\max_i \sup_{p\in U_i^\delta} |W_i(p)|
\end{align*}
The coboundary $\phi$ is defined by
\begin{align*}
f_{ik}(p)\phi_k(p)-\phi_i(p), \,\,\,\,\,\, p\in U_i\cap U_k,\,\,\,\,\, -[\phi_i(p),\Lambda_0]|_{V_0}+\phi_i(p) T_{i|0}(p),\,\,\, p\in U_i
\end{align*}

For any $(\psi,W)$, we define 
\begin{align*}
\iota(\psi,W)=\inf_{\delta(\phi)=(\psi,W)} ||\phi||
\end{align*}
It suffices to prove the existence of constant $c$ such that $\iota(\psi,W)\leq c||(\psi,W)||$. Assume that such a constant $c$ does not exist. Then we can find a sequence $(\psi',W'),(\psi'',W''),\cdots,(\psi^{\mu}, W^{\mu}),\cdots$ such that
\begin{align*}
\iota(\psi^{(\mu)},W^{(\mu)})=1,\,\,\,\,\,\,\,\,||(\psi^{(\mu)},W^{(\mu)})||<\frac{1}{\mu}
\end{align*}

We take a covering $\{\bar{U}_i^\delta\}$ of $V_0$. Since $\phi_k^{\mu}(p)<2$ for $p\in U_k$, there exists a subsequence $\phi^{(\mu_1)},\phi^{(\mu_2)},\cdots, \phi^{(\mu_v)}$ of $\phi',\phi'',\cdots$ such that $\phi_k^{(\mu_v)}$ converges absolutely and uniformly on $\bar{U}_k^\delta$ for each $k$. Since $V_0$ is compact, we can choose a subsequence that works for all $k$. On the other hand, since $||(\psi,W)||<\frac{1}{\mu}$, we have in particular
\begin{align}\label{ii65}
|f_{ik|0}(p)\phi_k^{(\mu)}(p)-\phi_i^{(\mu)}(p)|<\frac{1}{\mu},\,\,\,\,\,p\in U_i \cap U_k,\,\,\,\,\,\,|-[\phi_i^{(\mu)}( p),\Lambda_0]|_{V_0}+\phi_i^{(\mu)}(p) T_{i|0}(p)|<\frac{1}{\mu},\,\,\, p\in U_i^\delta
\end{align}

Then $\phi_i^{(\mu_v)}$ converges absolutely and uniformly on the whole $U_i$. Let $\phi_i(p)=\lim_v \phi_i^{(\mu_v)}(p)$ and let $\phi=\{\phi_i(p)\}$. Then we have $||\phi^{(\mu_v)}-\phi ||\to 0$ as $n\to \infty$. On the other hand, from $(\ref{ii65})$, we have $\delta \phi=(0,W_\phi)$, where $W_\phi(p)=0$ for $p\in U_i^\delta$. By identity theorem, $W_\phi(p)=0$ for $p\in U_i$. Hence we have $\delta(\phi^{(\mu_v)}-\phi)=(\psi^{(\mu_v)},W^{(\mu_v)})$ which contradicts to $\iota(\psi^{(\mu_v)},W^{(\mu_v)})=1$.
\end{proof}

By Lemma \ref{1lemma}, we can choose $S_{i|\mu+1}(t)\ll c_4\left(\frac{bc_1}{c}+\frac{bd_1}{c}\right)A(t)$. We note $(\ref{ik1})$ and $(\ref{ik2})$. By setting $c>\max\{2bc_1(1+c_2^2),\frac{2b}{\epsilon} ,c_4bc_1+c_4bd_1  \}$, we get $(\ref{ij2})$. Since the constants $b,c,c_1,d_1$ are independent of $\mu$, we have $(\ref{our5}),(\ref{our6})$, and $(\ref{our7})$.
By letting $N=\{t|\sum_{r=1}^m |t_r|^2<\frac{c^2}{m}\}$, the power series $S_i(p,t),f_{ik}(p,t)$ and $T_i(p,t)$ converges absolutely and uniformly for $t\in N$ so that $S_i(p,t),T_i(p,t)$ and $f_{ik}(p,t)$ are holomorphic on $U_i^\delta\times N$ and $U_i^\delta\cap U_k^\delta\times N$, respectively, and satisfy $(\ref{22a}), (\ref{22b}),(\ref{23i}),(\ref{23j})$ and $(\ref{23k})$ by replacing $U_i$ by $U_i^\delta$. This completes the proof of Theorem \ref{jj20}.

\end{proof} 

\subsection{Maximal families: Theorem of completeness }

\begin{definition}\label{005}
Let $\mathcal{V}\subset (W\times M,\Lambda_0)\xrightarrow{\omega} M$ be a Poisson analytic family of compact holomorphic Poisson submanifolds of $(W,\Lambda)$ of codimension $1$ and let $t_0$ be a point on $M$. We say that $\mathcal{V}\xrightarrow{\omega} M$ is maximal at $t_0$ if, for any Poisson analytic family $\mathcal{V}'\subset (W\times M',\Lambda)\xrightarrow{\omega'} M'$ of compact holomorphic Poisson submanifolds of $(W,\Lambda_0)$ of codimension $1$ such that $\omega^{-1}(t_0)=\omega'^{-1}(t_0'),t_0'\in M'$, there exists a holomorphic map $h$ of a neighborhood $N'$ of $t_0'$ on $M'$ into $M$ which maps $t_0'$ to $t_0$ such that $\omega'^{-1}(t')=\omega^{-1}(h(t'))$ for $t'\in N'$. We note that if we set a Poisson map $\hat{h}:(W\times N',\Lambda_0)\to (W\times M ,\Lambda_0)$ defined by $(w,t')\to (w,h(t'))$, then the restriction map of $\hat{h}$ to $\mathcal{V}'|_{N'}=\omega'^{-1}(N')\subset (W\times N',\Lambda)$ defines a Poisson map $\mathcal{V}'|_{N'}\to \mathcal{V}$ so that $\mathcal{V}'|_{N'}$ is the family induced from $\mathcal{V}$ by $h$, which means $\mathcal{V}\xrightarrow{\omega} M$ is complete at $t_0$.  
\end{definition}

\begin{theorem}[Theorem of completeness]\label{012}
Let $\mathcal{V}\subset (W\times M,\Lambda_0)\xrightarrow{\omega} M$ be a Poisson analytic family of compact holomorphic Poisson submanifolds of $(W,\Lambda_0)$ of codimension $1$. If the characteristic map 
\begin{align*}
\rho_{d,0}:T_0 M\to\mathbb{H}^0(V_0,\mathcal{N}_0^\bullet |_{V_0})
\end{align*}
 is an isomorphism, then the family $\mathcal{V}\xrightarrow{\omega} M$ is maximal at the point $t=0$.
\end{theorem}

\begin{proof}
We extend the arguments in \cite{Kod59} p.494-496 in the context of holomorphic Poisson deformations. We tried to maintain notational consistency with \cite{Kod59}. 

Suppose that $M=\{t|\sum_{r=1}^m |t_r|^2<1\}$ and that $\rho_{d,0}:T_0 M\to \mathbb{H}^0(V_0, \mathcal{N}_0^\bullet |_{V_0})$ is an isomorphism. Let $\mathcal{V}'\subset(W\times M, \Lambda_0)\xrightarrow{\omega'} M'$ be an arbitrary Poisson analytic family of holomorphic Poisson submanifolds of $(W,\Lambda_0)$ of codimension $1$ such that $\omega'^{-1}(0)=V_0$, where $M'=\{s|\sum_{r=1}^l |s_r|^2<1\}$. We will construct a holomorphic map $h:s\to t=h(s)$ of $N'$ into $M$ with $h(0)=0$ such that $\omega'^{-1}(s)=\omega^{-1}(h(s))$ where $N'=\{s|\sum_{r=1}^l |s_r|^2 <\delta\} \subset M'$ for a sufficiently small number $\delta>0$.

We keep the notations of subsection \ref{mm18} so that $\{S_i(p,t)\},\{f_{ik}(p,t)\}$, and $\{T_i(p,t)\}$ determine the Poisson analytics family $\mathcal{V}$. Let $\{R_i(p,s)\},\{e_{ik}(p,s)\}$ and $\{Q_i(p,t)\}$ and  be the corresponding system defining $\mathcal{V}'\subset (W\times M',\Lambda_0)$ so that we have
\begin{align}
S_i(p,t)&=f_{ik}(p,t) S_k(p,t),\,\,\,\,\, \,\,[\Lambda_0, S_i(p,t)]=S_i(p,t)T_i(p,t) \label{ab1}\\
R_i(p,s)&=e_{ik}(p,s)R_k(p,s),\,\,\,\,\, [\Lambda_0, R_i(p,s)]=R_i(p,s)Q_i(p,s)\label{ab2}
\end{align}
We may assume that 
\begin{align}\label{55c}
S_i(p,0)=R_i(p,0)=w_i(p), \,\,\,\,\, f_{ik}(p,0)=e_{ik}(p,0)=f_{ik|0}(p),\,\,\,\,\, T_i(p,0)=R_i(p,0)=T_{i|0}(p).
\end{align}
We expand $S_i(p,t)=w_i(p)+S_{i|0}(p,t)+S_{i|2}(p,t)+\cdots$ and let $S_{i|1}(p,t)=\sum_{r=1}^m B_{ir}(p)t_r$. Then the restriction $\beta_{ir}(p)$ of $B_{ir}(p)$ to $V_0$ satisfy
\begin{align}
\beta_{ir}(p)=f_{ik|0}(p)\beta_{kr}(p), \,\,\,\,\,p\in V_0\cap U_i\cap U_k,\\
-[\beta_{ir}(p), \Lambda_0]|_{V_0}+\beta_{ir}(p)T_{i|0}(p)=0,\,\,\,\,\,p\in V_0\cap U_i
\end{align}
and $\{\beta_1,...,\beta_m\}$ forms a basis of $\mathbb{H}^0(V_0, \mathcal{N}_0^\bullet |_{V_0})$ by the hypothesis.

If there exist non-vanishing holomorphic functions $f_i(p,s)$ defined on $U_i\times N'$ satisfying
\begin{align}\label{55a}
f_i(p,s)R_i(p,s)=S_i(p,h(s)),
\end{align}
we get $\omega'^{-1}(s)=\omega^{-1}(h(s))$. Recall Notation \ref{notation1} and let us write $h(s)$ and $f_i(p,s)$ in the following form:
\begin{align*}
h(s)=(h_1(s)=\sum_{\mu=1}^\infty h_{r|\mu}(s),...,h_m(s)=\sum_{\mu=1}^\infty h_{r|\mu}(s)),\,\,\,\,\,f_i(p,s)=1+\sum_{\mu=1}^\infty f_{i|\mu}(p,s),
\end{align*}
 We will construct such $f_i(p,s)$ and $h(s)$ satisfying $(\ref{55a})$ by solving the system of congruences by induction on $\mu$
\begin{align}\label{55b}
f_i^\mu(p,s)R_i(p,s)\equiv_\mu S_i(p,h^\mu(s)),\,\,\,\,\,\,\,\,\mu=0,1,2,\cdots.
\end{align}
$(\ref{55b})_0$ follows from $(\ref{55c})$. Now assume that $h^{\mu-1}(s)$ and $f_i^{\mu-1}(p,s)$ satisfying $(\ref{55b})_{\mu-1}$ are already determined. We will find $h_\mu(s)$ and $f_{i|\mu}(p,s)$ such that $h^\mu:=h^{\mu-1}(s)+h_\mu(s)$ and $f_i^\mu(p,s):=f_i^{\mu-1}(p,s)+f_{i|\mu}(p,s)$ satisfy $(\ref{55b})_\mu$. We can define homogenous polynomials $\Gamma_{i|\mu}(p,s)$ of degree $\mu$ in $s$ by
\begin{align}\label{ab3}
\Gamma_{i|\mu}(p,s)\equiv_\mu f_i^{\mu-1}(p,s)R_i(p,s)-S_i(p,h^{\mu-1}(s)),\,\,\,\,\,p\in U_i
\end{align}
Then we claim that
\begin{align}
\Gamma_{i|\mu}(p,s)=f_{ik|0}(p)\Gamma_{k|\mu}(p,s),\,\,\,\,\,p\in V_0\cap U_i\cap U_k \label{55d}\\
[\Lambda_0, \Gamma_{i|\mu}(p,s)]|_{w_i(p)=0}=\Gamma_{i|\mu}(p,s)T_{i|0}(p),\,\,\,\,\,\,p\in V_0\cap U_i \label{55e}
\end{align}
Indeed, (\ref{55d}) follows from \cite{Kod59} p.496. On the other hand, to prove $(\ref{55e})$, we remark that
\begin{align}\label{55g}
[\Lambda_0, f_i^{\mu-1}(p,s)]+Q_i(p,s) f_i^{\mu-1}(p,s)-f_i^{\mu-1}(p,s) T_i(p, h^{\mu-1}(p,s))\equiv_{\mu-1} 0
\end{align}
Indeed, by applying $[\Lambda_0,-]$ on $f_i^{\mu-1}(p,s)R_i(p,s)\equiv_{\mu-1} S_i(p,h^{\mu-1}(s))$ in $(\ref{55b})_{\mu-1}$, we get, from $(\ref{ab1})$, $(\ref{ab2})$, and $(\ref{55b})_{\mu-1}$,
\begin{align*}
&[\Lambda_0, f_i^{\mu-1}(p,s)]R_i(p,s)+[\Lambda_0, R_i(p,s)]f_i^{\mu-1}(p,s)\equiv_{\mu-1} [\Lambda_0, S_i(p,h^{\mu-1}(s))] \iff \\
&[\Lambda_0, f_i^{\mu-1}(p,s)]R_i(p,s)+R_i(p,s) Q_i(p,s) f_i^{\mu-1}(p,s)\equiv_{\mu-1} S_i(p,h^{\mu-1}(s)) T_i(p,h^{\mu-1}(s))\equiv_{\mu-1} f_i^{\mu-1}(p,s) R_i(p,s) T_i(p,h^{\mu-1}(s))\\
&\iff ([\Lambda_0, f_i^{\mu-1}(p,s)]+Q_i(p,s) f_i^{\mu-1}(p,s)-f_i^{\mu-1}(p,s) T_i(p, h^{\mu-1}(p,s)))R_i(p,s)\equiv_{\mu-1} 0
\end{align*}
which proves $(\ref{55g})$. By using $(\ref{55g})$, we get, from $(\ref{ab3})_\mu$, $(\ref{ab1})$, and $(\ref{ab2})$,
\begin{align}\label{ab5}
&[\Lambda_0, \Gamma_{i|\mu}(p,s)]\equiv_\mu[\Lambda_0, f_i^{\mu-1}(p,s)R_i(p,s)]-[\Lambda_0, S_i(p,h^{\mu-1}(s))]\\ 
&\equiv_\mu [\Lambda_0, f_i^{\mu-1}(p,s)]R_i(p,s)+[\Lambda_0, R_i(p,s)] f_i^{\mu-1}(p,s)-[\Lambda_0, S_i(p,h^{\mu-1}(s))]\notag\\
&\equiv_\mu [\Lambda_0, f_i^{\mu-1}(p,s)]R_i(p,s)+R_i(p,s) Q_i(p,s)f_i^{\mu-1}(p,s)-S_i(p,h^{\mu-1}(s))T_i(p,h^{\mu-1}(s)) \notag \\
&\equiv_\mu [\Lambda_0, f_i^{\mu-1}(p,s)]R_i(p,s)+R_i(p,s) Q_i(p,s)f_i^{\mu-1}(p,s)+\Gamma_{i|\mu}(p,s) T_{i|0}(p)-f_i^{\mu-1}(p,s) R_i(p,s) T_i(p,h^{\mu-1}(s))\notag\\
&\equiv_\mu ([\Lambda_0, f_i^{\mu-1}(p,s)]+Q_i(p,s) f_i^{\mu-1}(p,s)-f_i^{\mu-1}(p,s) T_i(p, h^{\mu-1}(p,s)))R_i(p,0)+\Gamma_{i|\mu}(p,s) T_{i|0}(p)\notag
\end{align}
By restricting $(\ref{ab5})$ to $V_0$ (i.e by setting $S_i(p,0)=R_i(p,0)=w_i(p)$), we obtain
\begin{align}
[\Lambda_0, \Gamma_{i|\mu}(p,s)]|_{w_i(p)=0}\equiv_\mu \Gamma_{i|\mu}(p,s)T_{i|0}(p),\,\,\,\,\,p\in V_0\cap U_i.
\end{align}
This proves $(\ref{55e})$. From $(\ref{55d})$ and $(\ref{55e})$, there exist homogenous polynomials $h_{r|\mu}$ of degree $\mu$ in $s$ such that
\begin{align}
\sum_{r=1}^m \beta_{ir}(p)h_{r|\mu}(s)=\Gamma_{i|\mu}(p,s),\,\,\,\,\, p\in V_0\cap U_i
\end{align}
 From $h_r^\mu(s)=h_r^{\mu-1}(s)+h_{r|\mu}(s)$, the congruence $(\ref{55b})_\mu$ is equivalent to (for the detail, see \cite{Kod59} p.496)
\begin{align}\label{55f}
\sum_{r=1}^m B_{ir}(p)h_{r|\mu}(s)=w_i(p)f_{i|\mu}(p,s)+\Gamma_{i|\mu}(p,s)
\end{align}
By setting $f_{i|\mu}(p,s):=\frac{\sum_{r=1}^m B_{ir}(p)h_{r|\mu}(s)-\Gamma_{i|\mu}(p,s)}{w_i(p)}$, we get $(\ref{55f})$, which completes the inductive construction of $h^\mu(s)$ and $f_i^\mu(p,s)$ satisfying $(\ref{55b})_\mu$.
\end{proof}

\subsection{Proof of convergence}\

The convergence of $h_r(s), f_i(p,s)$ follows from the same arguments in \cite{Kod59} p.497-498, which completes Theorem \ref{012}.

\section{Deformations of compact holomorphic Poisson submanifolds of arbitrary codimensions }\label{section3}

We extend Definition $\ref{2d}$ to arbitrary codimensions.
\begin{definition}[compare \cite{Kod62}]\label{3a}
Let $(W,\Lambda_0)$ be a holomorphic Poisson manifold of dimension $d+r$. We denote a point in $W$ by $w$ and a local coordinate of $w$ by $(w^1,...,w^{r+d})$. By a Poisson analytic family of compact holomorphic Poisson submanifolds of $(W,\Lambda_0)$ we mean a holomorphic Poisson submanifold $\mathcal{V}\subset(W\times M, \Lambda_0)$ of codimension $r$, where $M$ is a complex manifold and $\Lambda_0$ is the holomorphic Poisson structure on $W\times M$ induced from $(W,\Lambda_0)$, such that
\begin{enumerate}
\item for each point $t\in M$, $V_t\times t:=\omega^{-1}(t)=\mathcal{V}\cap \pi^{-1}(t)$ is a connected compact holomorphic Poisson submanifold of $(W\times t,\Lambda_0)$ of dimension $d$, where $\omega:\mathcal{V}\to M$ is the map induced from the canonical projection $\pi:W\times M\to M$.
\item for each point $p\in \mathcal{V}$, there exist $r$ holomorphic functions $f_\alpha(w,t),\alpha=1,...,r$ defined on a neighborhood $\mathcal{U}_p$ of $p$ in $W\times M$ such that $\textnormal{rank} \frac{\partial ( f_1,...,f_r)}{\partial (w^1,...w^{r+d})}=r$, and $\mathcal{U}_p\cap \mathcal{V}$ is defined by the simultaneous equations $f_\alpha(w,t)=0,\alpha=1,...,r$.
\end{enumerate}
We call $\mathcal{V}\subset (W\times M, \Lambda_0)$ a Poisson analytic family of compact holomorphic Poisson submanifolds $V_t,t\in M$ of $(W,\Lambda_0)$. We also call $\mathcal{V}\subset( W\times M,\Lambda_0)$ a Poisson analytic family of deformations of a compact holomorphic Poisson submanifold $V_{t_0}$ of $(W,\Lambda_0)$ for each fixed point $t_0\in M$.
\end{definition}

\subsection{The complex associated with the normal bundle of a holomorphic Poisson submanifold in a holomorphic Poisson manifold}\label{subsection3.1}\

Let $(W,\Lambda_0)$ be a holomorphic Poisson manifold and $V$ be a holomorphic Poisson submanifold of $(W,\Lambda_0)$. Let $\mathcal{U}=\{W_i\}$ be a open covering such that $W_i$ is a polycylinder with a local coordinate $(w_i,z_i)=(w_i^1,...,w_i^r,z_i^1,...,z_i^d)$ such that $W_i=\{(w_i,z_i)||w_i|<1,|z_i|<1\}$, where $|w_i|=\max_\lambda |w_i^\lambda|,|z_i|=\max_\alpha |z_i^\alpha|$, the local coordinate $(w_i,z_i)$ can be extended to a domain containing the closure of $W_i$ and on each neighborhood $W_i$, $V\cap W_i$ coincides with the subspace of $W_i$ determined by $w_i^1=\cdots=w_i^r=0$. On the intersection $W_i\cap W_k$, the coordinates $w_i^1,...,w_i^r,z_i^1,...,z_i^d$ are holomorphic functions of $w_k$ and $z_k$: $w_i^\alpha=f_{ik}^\alpha(w_k,z_k),\alpha=1,...,r,z_i^\lambda=g_{ik}^\lambda(w_k,z_k),\lambda=1,...,d$. We set $f_{ik}(w_k,z_k):=(f_{ik}^1(w_k,z_k),\cdots f_{ik}^r(w_k,z_k))$ and $g_{ik}(w_k,z_k)=(g_{ik}^1(w_k,z_k),\cdots,g_{ik}^d(w_k,z_k))$ so that we write the formula in the form $w_i=f_{ik}(w_k,z_k),z_i=g_{ik}(w_k,z_k)$. Then we have $f_{ik}(0,z_k)=0$ and so $w_i^\alpha=f_{ik}^\alpha(w_k,z_k)$ has the following form:
\begin{align}\label{3b}
w_i^\alpha=f_{ik}^\alpha(w_k,z_k)=\sum_{\beta=1}^r w_k^\beta F_{ik\beta}^\alpha(w_k,z_k)
\end{align}

We set $U_i=V\cap W_i=\{(0,z_i)||z_i|<1\}$. We denote a point of $V$ by $z$ and if $z=(0,z_i)\in U_i$, we consider $z_i=(z_i^1,...,z_i^d)$ as the coordinate of $z$ on $U_i$. We indicate by writing $z=(0,z_k)\in U_k\cap U_i$ that $z$ is a point in $U_k\cap U_i$ whose coordinate on $U_k$ is $z_k$. We note that
\begin{align}\label{3c}
F_{ik\beta}^\alpha(0,z_k)= \frac{\partial f_{ik}^\alpha(w_k,z_k)}{\partial w_k^\beta}|_{w_k=0},\,\,\,\,\,\,\,\,\,z=(0,z_k)\in U_i\cap U_k,
\end{align}
and let $F_{ik}(z):=(F_{ik\beta}^\alpha(0,z))_{\alpha, \beta=1,...,r}$. Then the matrix valued functions $F_{ik}(z)$ satisfy $F_{ik}(z)=F_{ij}(z)F_{jk}(z)$ for $z\in U_i\cap U_j\cap U_k$ so that they define the normal bundle $\mathcal{N}_{V/W}$.

On the other hand, since $V$ is a holomorphic Poisson submanifold of $(W,\Lambda_0)$, $[\Lambda_0,w_i^\alpha]$ is of the form
\begin{align}\label{3d}
[\Lambda_0,w_i^\alpha]=\sum_{\beta=1}^r w_i^\beta T_{i\alpha}^\beta(w_i,z_i)
\end{align}
where $T_{i\alpha}^\beta(w_i,z_i)=\sum_{\gamma=1}^r P_{i\alpha\beta}^1(w_i,z_i)\frac{\partial}{\partial w_i^1}+\cdots +P_{i\alpha\beta}^r(w_i,z_i)\frac{\partial}{\partial w_i^r}+Q_{i\alpha\beta}^1(w_i,z_i)\frac{\partial}{\partial z_i^1}+\cdots +Q_{i\alpha\beta}^d(w_i,z_i) \frac{\partial}{\partial z_i^d} \in \Gamma(W_i, T_W)$ by which we consider $T_{i\alpha}^\beta(w_i,z_i)$ a vector-valued holomorphic function on $W_i$.
Then on $W_i\cap W_k$, we have

\begin{align}\label{a123}
[\Lambda_0,w_i^\alpha]=\sum_{\beta=1}^r w_i^\beta T_{i\alpha}^\beta (w_i,z_i)=\sum_{\beta=1}^r f_{ik}^\beta(w_k,z_k) T_{i\alpha}^\beta(w_i,z_i)
\end{align}

On the other hand, from $(\ref{3b})$ and $(\ref{3d})$,
\begin{align}\label{3e}
[\Lambda_0, w_i^\alpha]&=\sum_{\beta=1}^r w_k^\beta[\Lambda_0,F_{ik\beta}^\alpha(w_k,z_k)]+\sum_{\beta=1}^r F_{ik\beta}^\alpha(w_k,z_k)[\Lambda_0,w_k^\beta]\\
&=\sum_{\beta=1}^r w_k^\beta[\Lambda_0, F_{ik\beta}^\alpha(w_k,z_k)]+\sum_{\beta,\gamma=1}^rF_{ik\beta}^\alpha (w_k,z_k)w_k^\gamma  T_{k\beta}^\gamma(w_k,z_k)\notag
\end{align}
By taking the derivative of  $(\ref{a123})$ and $(\ref{3e})$ with respect to $w_k^\gamma$ and setting $w_k=0$, we get from $(\ref{3c})$, on $\Gamma(U_i\cap U_k, T_W|_V)$,
\begin{align}\label{3i}
 \sum_{\beta=1}^r F_{ik\gamma }^\beta(0,z_k) T_{i\alpha}^\beta(0,z_i)=[\Lambda_0,F_{ik\gamma }^\alpha(0,z_k)]|_{w_k=0}+\sum_{\beta=1}^rF_{ik\beta}^\alpha(0,z_k)T_{k\beta}^\gamma(0,z_k)
\end{align}

On the other hand, from $(\ref{3d})$, we have
\begin{align}\label{3f}
0&=[\Lambda_0,[\Lambda_0, w_i^\alpha]]=\sum_{\beta=1}^r [\Lambda_0,w_i^\beta T_{i\alpha}^\beta(w_i,z_i)]=\sum_{\beta=1}^r w_i^\beta [\Lambda_0, T_{i\alpha}^\beta(w_i,z_i)]-[\Lambda_0, w_i^\beta]\wedge T_{i\alpha}^\beta(w_i,z_i)\notag\\
&=\sum_{\beta=1}^r w_i^\beta[\Lambda_0, T_{i\alpha}^\beta(w_i,z_i)]-\sum_{\beta,\gamma=1}^r w_i^\gamma T_{i\beta}^\gamma(w_i,z_i)\wedge T_{i\alpha}^\beta(w_i,z_i)\notag\\
&=\sum_{\beta=1}^r w_i^\beta\left( [\Lambda_0, T_{i\alpha}^\beta(w_i,z_i)]-\sum_{\gamma=1}^r T_{i\gamma}^\beta(w_i,z_i)\wedge T_{i\alpha}^\gamma(w_i,z_i)\right)
\end{align}
By taking the derivative of $(\ref{3f})$ with respect to $w_i^\beta$ and setting $w_i=0$,  we get, on $\Gamma(U_i, T_W|_V)$,
\begin{align}\label{3g}
[\Lambda_0,T_{i\alpha}^\beta(0,z_i)]|_{w_i=0}-\sum_{\gamma=1}^r T_{i\gamma}^\beta(0,z_i)\wedge T_{i\alpha}^\gamma(0,z_i)=0
\end{align}

Now we define a complex of sheaves associated with the normal bundle $\mathcal{N}_{V/W}$:
\begin{align*}
\mathcal{N}_{V/W}\xrightarrow{\nabla}\mathcal{N}_{V/W}\otimes T_W|_V\xrightarrow{\nabla} \mathcal{N}_{V/W}\otimes \wedge^2 T_W|_V\xrightarrow{\nabla}\mathcal{N}_{V/W}\otimes \wedge^3 T_W|_V\xrightarrow{\nabla}\cdots
\end{align*}

First we define $\nabla:\mathcal{N}_{V/W}\to \mathcal{N}_{V/W}\otimes T_V|_W$ and then extend to $\nabla:\mathcal{N}_{V/W}\otimes \wedge^{p}T_W|_V\to \mathcal{N}_{V/W}\otimes \wedge^{p+1}T_W|_V$ in the following. We note that $\Gamma(U_i,\mathcal{N}_{W/V})\cong \oplus^r \Gamma(U_i,\mathcal{O}_V)$ and $\Gamma(U_i,\mathcal{N}_{V/W}\otimes T_W|_V)\cong \oplus^r \Gamma(U_i, T_W|_V)$. Using these isomorphism, we define $\nabla$ on $\mathcal{N}_{V/W}$ by the rule\\
\begin{align*}
\nabla(e_i^\alpha):=\sum_{\beta=1}^rT_{i\beta}^\alpha (0,z_i)e_i^\beta
\end{align*}
where $e_i^\alpha=(0,\cdots,\overset{\alpha-th}{1},\cdots,0)\in \oplus_{i=1}^r \Gamma(U_i,\mathcal{O}_V)$. In general, we define
\begin{align*}
\nabla:\oplus^r \Gamma(U_i,\mathcal{O}_V)&\to \oplus^r\Gamma(U_i,T_W|_V)\\
\sum_{\alpha=1}^r g_i^\alpha e_i^\alpha &\mapsto \sum_{\alpha=1}^r-[g_i^\alpha,\Lambda_0]|_{w_i=0}\cdot e_i^\alpha+\sum_{\alpha=1}^rg_i^\alpha \nabla(e_i^\alpha)=\sum_{\alpha=1}^r  \left(-[g_i^\alpha,\Lambda_0]|_{w_i=0}+\sum_{\beta=1}^r g_i^\beta  T_{i\alpha}^\beta(0,z_i)  \right) e_i^\alpha
\end{align*}
where $g_i^\alpha\in \Gamma(U_i,\mathcal{O}_V)$. 

We extend $\nabla$ on $\mathcal{N}_{V/W}\otimes \wedge^p T_W|_V$.
$\mathcal{N}_{V/W}\otimes \wedge^p T_W|_V\xrightarrow{\nabla}\mathcal{N}_{V/W}\otimes \wedge^{p+1}T_W|_V$ is locally defined in the following way: we note that $\Gamma(U_i,\mathcal{N}_{V/W}\otimes \wedge^p T_W|_V)\cong \oplus^r\Gamma(U_i, \wedge^p T_W|_V)$ and $\Gamma(U_i,\mathcal{N}_{V/W}\otimes \wedge^{p+1} T_W|_V)\cong \oplus^r\Gamma(U_i, \wedge^{p+1} T_W|_V)$. From these isomorphism, we define $\nabla$ by the rule
\begin{align*}
\nabla:\oplus^r \Gamma(U_i,\wedge^p T_W|_V)&\to \oplus^r\Gamma(U_i,\wedge^{p+1}T_W|_V)\\
\sum_{\alpha=1}^r g_i^\alpha e_i^\alpha &\mapsto \sum_{\alpha=1}^r -[g_i^\alpha, \Lambda_0]|_{w_i=0}\cdot e_i^\alpha+(-1)^p\sum_{\alpha=1}^rg_i^\alpha\wedge \nabla(e_i^\alpha)\\&=\sum_{\alpha=1}^r  \left(-[g_i^\alpha,\Lambda_0]|_{w_i=0}+(-1)^p\sum_{\beta=1}^r g_i^\beta \wedge T_{i\alpha}^\beta(0,z_i)  \right) e_i^\alpha
\end{align*}
where $g_i^\alpha\in \Gamma(U_i,\wedge^p T_W|_V)$.

First we show that $\nabla$ defines a complex, i.e. $\nabla\circ \nabla=0$. Indeed, 
\begin{align*}
&\nabla\circ \nabla(\sum_{\alpha=1}^r g_i^\alpha e_i^\alpha) =\nabla(\sum_{\alpha=1}^r  \left(-[g_i^\alpha,\Lambda_0]|_{w_i=0}+(-1)^p\sum_{\beta=1}^r g_i^\beta \wedge T_{i\alpha}^\beta(0,z_i)  \right) e_i^\alpha)\\
&=-[ (-1)^p\sum_{\beta=1}^r g_i^\beta\wedge T_{i\alpha}^\beta(0,z_i),\Lambda_0]|_{w_i=0}\cdot e_i^\alpha+ (-1)^{p+1}\sum_{\alpha=1}^r \left(-[g_i^\alpha,\Lambda_0]|_{w_i=0}+(-1)^p\sum_{\beta=1}^r g_i^\beta \wedge T_{i\alpha}^\beta(0,z_i)  \right)\nabla(e_i^\alpha)\\
&=(-1)^{p+1}\sum_{\alpha,\beta=1}^r [g_i^\beta\wedge T_{i\alpha}^\beta(0,z_i),\Lambda_0]|_{w_i=0} e_i^\alpha+(-1)^p\sum_{\alpha,\beta=1}^r[g_i^\beta,\Lambda_0]|_{w_i=0} \wedge T_{i\alpha}^\beta(0,z_i) e_i^\alpha-\sum_{\alpha,\beta,\gamma=1}^r g_i^\beta \wedge T_{i\gamma}^\beta(0,z_i)\wedge T_\alpha^\gamma(0,z_i) e_i^\alpha
\end{align*}
Hence in order to show $\nabla\circ\nabla=0$, we have to show that
\begin{align}\label{3h}
(-1)^{p+1}\sum_{\beta=1}^r [g_i^\beta\wedge T_{i\alpha}^\beta(0,z_i),\Lambda_0]|_{w_i=0}+(-1)^p\sum_{\beta=1}^r [g_i^\beta,\Lambda_0]|_{w_i=0}\wedge T_{i\alpha}^\beta(0,z_i)-\sum_{\beta,\gamma=1}^r g_i^\beta \wedge T_{i\gamma}^\beta(0,z_i)\wedge T_{i\alpha}^\gamma(0,z_i)=0
\end{align}
Indeed, from (\ref{3g}), (\ref{3h}) becomes
\begin{align*}
(-1)^{p+1}\sum_{\beta=1}^r [g_i^\beta,\Lambda_0]|_{w_i=0}\wedge T_{i\alpha}^\beta(0,z_i)+(-1)^{p+1+p}\sum_{\beta=1}^r g_i^\beta\wedge [T_{i\alpha}^\beta(0,z_i),\Lambda_0]|_{w_i=0}\\
+(-1)^p\sum_{\beta=1}^r [g_i^\beta,\Lambda_0]|_{w_i=0}\wedge T_{i\alpha}^\beta(0,z_i)-\sum_{\beta,\gamma=1}^r g_i^\beta\wedge T_{i\gamma}^\beta(0,z_i)\wedge T_{i\alpha}^\gamma(0,z_i)\\
= \sum_{\beta=1}^rg_i^\beta\wedge [\Lambda_0, T_{i\alpha}^\beta(0,z_i)]|_{w_i=0}-\sum_{\beta,\gamma=1}^r g_i^\beta \wedge T_{i\gamma}^\beta(0,z_i)\wedge T_{i\alpha}^\gamma(0,z_i)\\
=\sum_{\beta=1}^rg_i^\beta \left( [\Lambda_0, T_{i\alpha}^\beta(0,z_i)]|_{w_i=0}- \sum_{\gamma=1}^r T_{i\gamma}^\beta(0,z_i) \wedge T_{i\alpha}^\gamma (0,z_i)\right)=0
\end{align*}
Hence $\nabla\circ \nabla=0$.

Next we show $\nabla$ is well-defined. In other words, on $U_i\cap U_k$, the following diagram commutes
\begin{center}
\begin{equation}\label{mm16}
\begin{CD}
\Gamma(U_k,\mathcal{N}_{W/V}\otimes \wedge^p T_{W/V})\cong \oplus^r \Gamma(U_k,\wedge^pT_W|_V)@>\text{on $U_i\cap U_k$}>\cong> \Gamma(U_i,\mathcal{N}_{W/V}\otimes \wedge^p T_{W/V})\cong \oplus^r \Gamma(U_i, \wedge^p T_W|_V)\\
@V\nabla VV @VV\nabla V\\
\Gamma(U_k,\mathcal{N}_{W/V}\otimes \wedge^{p+1} T_{W/V})\cong \oplus^r \Gamma(U_k, \wedge^{p+1}T_{W/V})@>\text{on $U_i\cap U_k$}>\cong> \Gamma(U_i,\mathcal{N}_{W/V}\otimes \wedge^{p+1} T_{W/V})\cong \oplus^r \Gamma(U_i, \wedge^{p+1}T_{W/V})
\end{CD}
\end{equation}
\end{center}
Let $\sum_{\alpha=1}^r g_k^\alpha e_k^\alpha\in \oplus^r \Gamma(U_k,\wedge^p T_W|_V)$. Then $\nabla(\sum_{\alpha=1}^r g_k^\alpha e_k^\alpha)=\sum_{\alpha=1}^r(-[g_k^\alpha,\Lambda_0]|_{w_k=0}+(-1)^p\sum_{\beta=1}^r g_k^\beta \wedge T_{k\alpha}^\beta(0,z_k))e_k^\alpha$ is identified on $U_i\cap U_k$ with
\begin{align}\label{mm10}
\sum_{\gamma=1}^r \left(\sum_{\alpha=1}^r -F_{ik\alpha }^\gamma(0,z_k)[g_k^\alpha,\Lambda_0]|_{w_k=0}+(-1)^p\sum_{\alpha,\beta=1}^r F_{ik \alpha }^\gamma(0,z_k)g_k^\beta\wedge T_{k\alpha}^\beta(0,z_k)\right)e_i^\gamma
\end{align}
On the other hand, $\sum_{\alpha=1}^r g_k^\alpha e_k^\alpha$ is identified on $U_i\cap U_k$ with $\sum_{\alpha,\gamma=1}^r(F_{ik\alpha}^\gamma(0,z_k) g_k^\alpha) e_i^\gamma\in\oplus^r \Gamma(U_i,\wedge^pT_W|_V) $ and
\begin{align}\label{mm11}
\nabla(\sum_{\alpha,\gamma=1}^r (F_{ik\alpha}^\gamma(0,z_k) g_k^\alpha) e_i^\gamma)=\sum_{\gamma=1}^r\left(\sum_{\alpha=1}^r-[F_{ik\alpha}^\gamma(0,z_k) g_k^\alpha,\Lambda_0]|_{w_k=0} +(-1)^p \sum_{\alpha,\beta=1}^r F_{ik\alpha}^\beta(0,z_k)g_k^\alpha \wedge T_{i\gamma}^\beta(0,z_i)\right)e_i^\gamma
\end{align}
Hence in order for the diagram (\ref{mm16}) to commute, we have to show that (\ref{mm10}) coincides with (\ref{mm11}):
\begin{align*}
(-1)^p\sum_{\alpha,\beta=1}^r F_{ik\alpha}^\gamma(0,z_k)g_k^\beta\wedge T_{k\alpha}^\beta(0,z_k)&=\sum_{\alpha=1}^r -[ F_{ik\alpha}^\gamma(0,z_k),\Lambda_0]|_{w_k=0}\wedge g_k^\alpha+(-1)^p \sum_{\alpha,\beta=1}^r F_{ik\alpha}^\beta(0,z_k) g_k^\alpha \wedge T_{i\gamma}^\beta(0,z_i)=0\\
\iff \sum_{\alpha,\beta=1}^r F_{ik\beta}^\gamma(0,z_k)T_{k\beta}^\alpha(0,z_k)\wedge g_k^\alpha&=\sum_{\alpha=1}^r -[ F_{ik\alpha}^\gamma(0,z_k),\Lambda_0]|_{w_k=0}\wedge g_k^\alpha+\sum_{\alpha,\beta=1}^r F_{ik\alpha}^\beta(0,z_k)  T_{i\gamma}^\beta(0,z_i) \wedge g_k^\alpha=0\\
\end{align*}
which follows from $(\ref{3i})$. Therefore $\nabla$ is well-defined.

\begin{definition}\label{332a}
We call the complex defined as above
\begin{align*}
\mathcal{N}_{V/W}^\bullet:\mathcal{N}_{V/W}\xrightarrow{\nabla}\mathcal{N}_{V/W}\otimes T_W|_V\xrightarrow{\nabla}\mathcal{N}_{V/W}\otimes \wedge^2 T_W|_V\xrightarrow{\nabla}\mathcal{N}_{V/W}\otimes \wedge^3 T_W|_V\xrightarrow{\nabla}\cdots
\end{align*}
the complex associated with the normal bundle $\mathcal{N}_{V/W}$ of a holomorphic Poisson submanifold $V$ of a holomorphic Poisson manifold $W$ and denote its $i$-th hypercohomology group by $\mathbb{H}^i(V,\mathcal{N}_{V/W}^\bullet)$.
\end{definition}

\begin{example}
On $W=\mathbb{C}^3$, let $(z,w_1,w_2)$ be a coordinate and $\Lambda_0=w_1z\frac{\partial}{\partial w_1}\wedge \frac{\partial}{\partial w_2}$ be a holomorphic Poisson structure on $\mathbb{C}^3$. Then $w_1=w_2=0$ is a holomorphic Poisson submanifold $V=\mathbb{C}$ with coordinate $z$ so that the normal bundle is $\mathcal{N}_{V/W}\cong \mathcal{O}_\mathbb{C}\oplus \mathcal{O}_\mathbb{C}$. Then $[\Lambda_0,w_1]=w_1z\frac{\partial}{\partial w_2}$ and $[\Lambda_0, w_2]=w_1(-z\frac{\partial}{\partial w_1})$ so that we get $T_1^1(z)=z\frac{\partial}{\partial w_2}, T_1^2(z)=0$ and $T_2^1(z)=-z\frac{\partial}{\partial w_1}, T_2^2(z)=0$. Since $[\Lambda_0, f_i(z)]|_{w_1=w_2=0}=0$ for entire functions $f_i(z),i=1,2$, we have $\nabla(f_1e_1+f_2e_2)=(f_1 T_1^1(z)+f_2T_1^2(z))e_1+(f_1T_2^1(z)+f_2T_2^2(z))e_2=(f_1z\frac{\partial}{\partial w_2},-f_1z\frac{\partial}{\partial w_1})$ so that $\mathbb{H}^0(V,\mathcal{N}_{V/W}^\bullet)=\{(0,f_2(z))|\text{$f_2(z)$ is an entire function}\}$.
\end{example}

\subsection{Infinitesimal deformations}\label{3infinitesimal}\

Let $M_1=\{t=(t_1,...,t_l)\in \mathbb{C}^l ||t|<1\}$. Consider a Poisson analytic family $\mathcal{V}\subset(W\times M_1,\Lambda_0)$ of compact holomorphic Poisson submanifolds $V_t,t \in M_1$ of $(W,\Lambda_0)$ and let $V=V_0$ as in Definition $\ref{3a}$. We keep the notations in subsection \ref{subsection3.1}. Let $\epsilon$ be a sufficiently small positive number. Then for $|t|<\epsilon$, the holomorphic Poisson submanifold $V_t$ of $W$ is defined in each neighborhood $W_i$ by simultaneous equations of the form $w_i^\lambda=\varphi_i^\lambda(z_i,t),\lambda=1,...,r$ where the $\varphi_i^\lambda(z_i,t)$ are holomorphic functions of $z_i,|z_i|<1$, depending holomorphically on $t$, $|t|<\epsilon$, and satisfying the boundary conditions $\varphi_i^\lambda(z_i,0)=0,\lambda=1,...,r.$ By setting $\varphi_i(z_i,t)=(\varphi_i^1(z_i,t),\cdots,\varphi_i^r(z_i,t))$, we write the simultaneous equation as $w_i=\varphi_i(z_i,t)$. Then we have $\varphi_i(g_{ik}(\varphi_k(z_k,t),z_k),t)=f_{ik}(\varphi_k(z_k,t),z_k)$. For each $t$, $|t|<\epsilon$, we set $w_{ti}^\lambda=w_i^\lambda-\varphi_i^\lambda(z_i,t),\lambda=1,..,r$ so that $(w_{ti}^1,...,w_{ti}^r,z_i^1,...,z_i^d)$ form a local coordinate defined on $W_i$. We define $F_{tik}(z_t):=\left(\frac{\partial w_{ti}^\lambda}{\partial w_{tk}^\mu}(z_t)\right)_{\lambda,\mu=1,...,r}$ $\text{for $z_t\in V_t\cap W_i\cap W_k$}$, where we denote by $(\frac{\partial w_{ti}^\lambda}{\partial w_{tk}^\mu}(z_t))$ the value of the partial derivative $\frac{\partial w_{ti}^\lambda}{\partial w_{tk}^\mu}$ at a point $z_t$ on $V_t$. Then the collection $\{F_{tik}(z_t)\}$ of $F_{tik}(z_t)$ forms a system of transition matrices defining the normal bundle $F_t$ of $V_t$ in $W$. Note that $F_{0ik}(z)=F_{ik}(z)$ $\text{for $z\in U_i\cap U_k$}$ from $(\ref{3c})$. 

Take an arbitrary tangent vector $\frac{\partial}{\partial t}=\sum_\rho \gamma_\rho \frac{\partial}{\partial t_\rho}$ of $M_1$ at $t$, $|t|<\epsilon$, and let $\psi_i(z_t,t)=\frac{\partial \varphi_i(z_i,t)}{\partial t}$ $\text{for $z_t=(\varphi_i(z_i,t),z_i)$}$.
Then we obtain the equality
\begin{align}\label{ss3}
\psi_i(z_t,t)=F_{tik}(z_t)\cdot \psi_k(z_t,t),\,\,\,\,\,\text{for $z_t\in V_t\cap W_i\cap W_k$}.
\end{align}

On the other hand, $w_i-\varphi_i(z_i,t)=0$ define a holomorphic Poisson submanifold, we have
\begin{align}\label{ss1}
[\Lambda_0,w_i^\lambda -\varphi_i^\lambda(z_i,t)]=\sum_{\mu=1}^r (w_i^\mu-\varphi_i^\mu(z_i,t))T_{i\lambda}^\mu(w,z_i,t) 
\end{align}
for some $T_i^\mu(w_i,z_i,t)$ which is of the form
\begin{align*}
T_{i\lambda}^\mu(w_i,z_i,t)=P_{i1}^\mu(w_i,z_i,t)\frac{\partial}{\partial w_i^1} +\cdots +P_{ir}^\mu(w_i,z_i,t)\frac{\partial}{\partial w_i^r}+Q_{i1}^\mu(w_i,z_i,t)\frac{\partial}{\partial z_i^1}+\cdots+Q_{id}^\mu(w_i,z_i,t)\frac{\partial}{\partial z_i^d}
\end{align*}
by which we consider $T_{i\lambda}^\mu(w_i,z_i,t)$ as a vector valued holomorphic function of $(w_i,z_i,t)$.

By taking the derivative of (\ref{ss1}) with respect to $t$, we get $[\Lambda_0, -\frac{\partial \varphi_i^\lambda(z_i,t)}{\partial t}]=\sum_{\mu=1}^r -\frac{\partial \varphi_i^\mu(z_i,t)}{\partial t}T_{i\lambda}^\mu(w,z_i,t)+\sum_{\mu=1}^r (w_i^\lambda- \varphi_i^\mu(z_i,t))\frac{\partial T_{i\lambda}^\mu(w,z_i,t)}{\partial t}$. By restricting to $V_t$, equivalently setting $w_i-\varphi_i(z_i,t)=0$, we get, on $\Gamma(W_i\cap V_t, T_W|_{V_t})$,
\begin{align}\label{ss2}
-[\Lambda_0,\frac{\partial \varphi_i^\lambda(z_i,t)}{\partial t}]|_{V_t}+\sum_{\mu=1}^r \frac{\partial \varphi_i^\mu(z_i,t)}{\partial t}T_{i\lambda}^\mu(\varphi_i(z_i,t),z_i,t)=0\iff-[\psi_i^\lambda(z_i,t),\Lambda_0]|_{V_t}+\sum_{\mu=1}^n \psi_i^\mu(z_i,t)T_{i\lambda}^\mu(\varphi_i(z_i,t),z_i,t)=0
\end{align}
From (\ref{ss3}) and (\ref{ss1}), $\{\psi_i(z_i,t)\}$ defines an element in $\mathbb{H}^0(V_t,\mathcal{N}_{V/W_t}^\bullet)$ so that we have a linear map
\begin{align*}
\sigma_t:T_t(M_1)&\to \mathbb{H}^0(V_t,\mathcal{N}_{V_t/W}^\bullet)\\
\frac{\partial}{\partial t}&\mapsto \frac{\partial V_t}{\partial t}:=\{\psi_i(z_t,t)\}
\end{align*}
We call $\sigma_t$ the characteristic map.

\begin{example}
Let $[\xi_0,\xi_1,\xi_2,\xi_3]$ be the homogenous coordinate on $\mathbb{P}^3_\mathbb{C}$. Let $[1,z_1,z_2,z_2]=[1,\frac{\xi_1}{\xi_0},\frac{\xi_2}{\xi_0},\frac{\xi_3}{\xi_0}]$. Then $\mathcal{V}\subset (\mathbb{P}_\mathbb{C}^3\times \mathbb{C},\Lambda_0=z_1\frac{\partial}{\partial z_1}\wedge\frac{\partial}{\partial z_2})$ defined by $\xi_1=\xi_3-t\xi_0=0$ is a Poisson analytic family of deformations of a holomorphic Poisson submanifold $\mathbb{P}_\mathbb{C}^1\cong V:\xi_1=\xi_3=0$ of $(\mathbb{P}^3, \Lambda_0)$. We note that $\mathcal{N}_{\mathbb{P}^1_\mathbb{C}/\mathbb{P}^3_\mathbb{C}}\cong \mathcal{O}_{\mathbb{P}_\mathbb{C}^1} (1)\oplus \mathcal{O}_{\mathbb{P}_\mathbb{C}^1}(1)$ so that we have the characteristic map
\begin{align*}
T_0\mathbb{C}&\to \mathbb{H}^0(V,\mathcal{N}_{V/\mathbb{P}^3_\mathbb{C}}^\bullet)=\mathbb{H}^0(\mathbb{P}^1_\mathbb{C}, (\mathcal{O}_{\mathbb{P}_\mathbb{C}^1}(1)\oplus \mathcal{O}_{\mathbb{P}_\mathbb{C}^1}(1))^\bullet)\\
a&\mapsto (0,a)=(0+0\cdot z_2, a+0\cdot z_2)
\end{align*}
\end{example}

\subsection{Theorem of existence}

\begin{theorem}[theorem of existence]\label{33a}
Let $V$ be a compact holomorphic Poisson submanifold of a holomorphic Poisson manifold $(W,\Lambda_0)$. If $\mathbb{H}^1(V, \mathcal{N}_{V/W}^\bullet)=0$, then there exists a Poisson analytic family $\mathcal{V}$ of compact holomorphic Poisson submanifolds $V_t$, $t\in M_1$ of $(W,\Lambda_0)$ such that $V_0=V$ and the characteristic map 
\begin{align*}
\sigma_0:T_0(M_1)&\to \mathbb{H}^0(V,\mathcal{N}_{V/W}^\bullet)\\
\frac{\partial}{\partial t}&\mapsto \left(\frac{\partial V_t}{\partial t}\right)_{t=0}
\end{align*}
is an isomorphism.
\end{theorem}
\begin{proof}
We extend the arguments in \cite{Kod62} p.150-158 in the context of holomorphic Poisson deformations. We tried to keep notational consistency with \cite{Kod62}. We also keep the notations in subsection \ref{subsection3.1}.

Let $\{\gamma_1,...,\gamma_\rho,...,\gamma_l\}$ be a basis of $\mathbb{H}^0(V,\mathcal{N}_{V/W}^\bullet)$, where $l=\dim \mathbb{H}^0(V,\mathcal{N}_{V/W}^\bullet)$. On each neighborhood $U_i=V\cap W_i$, $\gamma_\rho$ is represented as a vector-valued holomorphic function 
\begin{align*}
\gamma_{\rho i}=(\gamma_{\rho i}^1(z_i),...,\gamma_{\rho i}^\alpha(z_i),...,\gamma_{\rho i}^r(z_i) )\in \oplus^r\Gamma(U_i,\mathcal{O}_V)
\end{align*}
such that 
\begin{align}
&\gamma_{\rho i}(z)=F_{ik}(z)\cdot \gamma_{\rho k}(z),\,\,\,\,\, z\in U_i\cap U_k\label{33d}\\
-[\Lambda_0,&\gamma_{\rho i}^\alpha(z_i)]|_{w_i=0}+\sum_{\beta=1}^r \gamma_{\rho i}^\beta (z_i) T_{i\alpha}^\beta(0,z_i)=0,\,\,\,\,\,z_i\in U_i,\,\,\,\alpha=1,...,r.\label{33e}
\end{align}
Let $\epsilon$ be a small positive number. In order to prove Theorem \ref{33a}, it suffices to construct vector-valued holomorphic functions
\begin{align*}
\varphi_i(z_i,t)=(\varphi_i^1(z_i,t),...,\varphi_i^r(z_i,t))
\end{align*}
in $z_i$, and $t$ with $|z_i|<1,|t|<\epsilon$, with $|\varphi_i(z_i,t)|<1$ satisfying the boundary condition
\begin{align*}
\varphi_i(z_i,0)=0,\,\,\,\,\,
\frac{\partial \varphi_i(z_i,t)}{\partial t_\rho}|_{t=0}=\gamma_{\rho i}(z)
\end{align*}
such that 
\begin{align}
\varphi_i(g_{ik}(\varphi_k(z_k,t),z_k),t)&=f_{ik}(\varphi(z_k,t),z_k),\,\,\,\,\,(\varphi_k(z_k,t),z_k)\in W_k\cap W_i, \label{ppp1}\\
[\Lambda_0, w_i^\alpha-\varphi_i^\alpha(z_i,t)]&|_{w_i=\varphi_i(z_i,t)}=0,\,\,\,\,\,\alpha=1,...,r\label{ppp2}
\end{align}

Recall Notation \ref{notation1}. Then the equalities (\ref{ppp1}) and (\ref{ppp2}) are equivalent to the system of congruences
\begin{align}
\varphi_i^m(g_{ik}(\varphi_k^m(z_k,t),z_k),t)&\equiv_m f_{ik}(\varphi_k^m(z_k,t),z_k),\,\,\,\,\,m=1,2,3,\cdots \label{33b}\\
[\Lambda_0,w_i^\alpha-\varphi_i^{\alpha m}(z_i,t)]&|_{w_i=\varphi_i^m(z_i,t)}\equiv_m 0,\,\,\,\,\,\,\,\,\,\,\,\,\,\,\,\,\,m=1,2,3,\cdots,\,\,\,\alpha=1,...,r. \label{33c}
\end{align}

We will construct the formal power series $\varphi_i^m(z_i,t)$ satisfying $(\ref{33b})_m$ and $(\ref{33c})_m$ by induction on $m$. 

We define $\varphi_{i|1}^\alpha(z_i,t)=\sum_{\rho=1}^l t_\rho \gamma^\alpha_{\rho i}(z_i),\alpha=1,...,r$. Then from $(\ref{33d})$, $\varphi_{i|1}(z_i,t)$ satisfies $(\ref{33b})_1$. On the other hand, since $[\Lambda_0, \gamma^\alpha_{\rho i}(z_i)]|_{w_i=0}=\sum_{\beta=1}^r \gamma_{\rho i}^\beta(z_i) T_{i \alpha}^\beta(z_i,0),\alpha=1,...,r$ from $(\ref{33e})$, we get
\begin{align*}
[\Lambda_0, \varphi_{i|1}^\alpha(z_i,t)]|_{w_i=0}=[\Lambda_0, \sum_{\rho=1}^l t_\rho \gamma^\alpha_{\rho i}(z_i)]|_{w_i=0}=\sum_{\rho=1}^l\sum_{\beta=1}^r t_{\rho}\gamma_{\rho i}^\beta(z_i) T_{i\alpha}^\beta(z_i,0)=\sum_{\beta=1}^r\varphi_{i|1}^\beta(z_i,t) T_{i\alpha}^\beta(z_i,0)
\end{align*}
Then we have $[\Lambda_0,\varphi_{i|1}^\alpha(z_i,t)]=\sum_{\beta=1}^r \varphi_{i|1}^\beta(z_i,t) T_{i\alpha}^\beta(w_i,z_i) -\sum_{\beta=1}^r w_i^\beta P_{i\alpha}^\beta(w_i,z_i,t)$ for some $P_{i\alpha}^\beta(w_i,z_i,t)$ which are homogenous polynomials in $t_1,...,t_l$ of degree $1$ with coefficients in $\Gamma(W_i, T_W)$. Hence from $(\ref{3d})$, we have
\begin{align*}
[\Lambda_0, w_i^\alpha-\varphi^\alpha_{i|1}(z_i,t)]=\sum_{\beta=1}^r (w_i^\beta-\varphi_{i|1}^\beta (z_i,t)) T_{i\alpha}^\beta(w_i,z_i)+\sum_{\beta=1}^r w_i^\beta P_{i\alpha}^\beta(w_i,z_i,t)
\end{align*}
so that we obtain $[\Lambda_0,w_i^\alpha-\varphi_{i|1}(z_i,t)]|_{w_i-\varphi_{i|1}(z_i,t)}\equiv_1 0$ and so $\varphi_{i|1}(z_i,t)$ satisfies $(\ref{33c})_1$. Hence the induction holds for $m=1$.

Now we assume that we have already constructed $\varphi_i^m(z_i,t)=(\varphi_i^{1m}(z_i,t),\cdots, \varphi_i^{\alpha m}(z_i,t),\cdots,\varphi_i^{rm}(z_i,t))$ satisfying $(\ref{33b})_m$ and $(\ref{33c})_m$ such that $[\Lambda_0,w_i^\alpha- \varphi_i^{\alpha m}(z_i,t)]$ is of the form
\begin{align}\label{33f}
[\Lambda_0,w_i^\alpha-\varphi_i^{\alpha m}(z_i,t)]=\sum_{\beta=1}^r (w_i^\beta-\varphi_i^{\beta m}(z_i,t))T_{i\alpha}^\beta(w_i,z_i)+Q_i^{\alpha m}(z_i,t)+\sum_{\beta=1}^r w_i^\beta P_{i\alpha}^{\beta m}(w_i,z_i,t)
\end{align}
such that the degree of $P_{i\alpha}^{\beta m}(w_i,z_i,t)$ is at least $1$ in $t_1,...,t_l$. We note that we can rewrite (\ref{33f})  in the following way.
\begin{align*}
\sum_{\beta=1}^r(w_i^\beta-\varphi_i^{\beta m}(z_i,t))T_{i\alpha}^\beta(w_i,z_i)+Q_i^{\alpha m}(z_i,t)+\sum_{\beta=1}^r \phi_i^{\beta m}(z_i,t)P_{i\alpha}^{\beta m}(\varphi_i^{ m}(z_i,t),z_i,t)+\sum_{\beta=1}^r (w_i^\beta-\varphi_i^{\beta m}(z_i,t) )L_{i\alpha}^\beta(w_i,z_i,t)
\end{align*}
such that the degree of $L_{i\alpha}^\beta(w_i,z_i,t)$ in $t_1,...,t_l$ is at least $1$ so that (\ref{33f}) becomes the following form:
\begin{align}\label{33g}
[\Lambda_0, w_i^\alpha-\varphi_i^{\alpha m}(z_i,t)]=\sum_{\beta=1}^r(w_i^\beta-\varphi_i^{\beta m}(z_i,t))T_{i\alpha}^\beta(w_i,z_i)+K_i^{\alpha m}(z_i,t)+\sum_{\beta=1}^r (w_i^\beta-\varphi_i^{\beta m}(z_i,t) )L_{i\alpha}^\beta(w_i,z_i,t)
\end{align}
where $K_i^{\alpha m}(z_i,t):=Q_i^{\alpha m}(z_i,t)+\sum_{\beta=1}^r \phi_i^{\beta m}(z_i,t)P_{i\alpha}^{\beta m}(\varphi_i^{ m}(z_i,t),z_i,t)$. 

We set
\begin{align}
\psi_{ik}(z_k,t)&:=[\varphi_i^m(g_{ik}(\varphi_k^m(z_k,t),z_k),t)-f_{ik}(\varphi_k^m(z_k,t),z_k)]_{m+1}\label{33o}\\
G_i^\alpha(z_i,t)&:=[[\Lambda_0,w_i^\alpha-\varphi_i^{\alpha m}(z_i,t)]|_{w_i=\varphi_i^m(z_i,t)}]_{m+1},\,\,\,\,\,\alpha=1,...,r.\label{33p}
\end{align}

We claim that $\{(\psi_{ik}^1(z_k,t),...,\psi_{ik}^r(z_k,t))\}\oplus \{ (-G_i^1(z_i,t),...,-G_i^r(z_i,t)) \}$ defines a $1$-cocycle in the following \v{C}ech reolution of $\mathcal{N}_{V/W}^\bullet$:
\begin{center}
$\begin{CD}
C^0(\mathcal{U}\cap V,\mathcal{N}_{W/V}\otimes\wedge^2 T_W|_V)\\
@A\nabla AA\\
C^0(\mathcal{U}\cap V,\mathcal{N}_{W/V}\otimes T_W|_V)@>\delta>> C^1(\mathcal{U}\cap V,\mathcal{N}_{W/V}\otimes T_W|_V)\\
@A\nabla AA @A\nabla AA\\
C^0(\mathcal{U}\cap V,\mathcal{N}_{W/V})@>-\delta>>C^1(\mathcal{U}\cap V,\mathcal{N}_{W/V})@>\delta>> C^2(\mathcal{U}\cap V,\mathcal{N}_{W/V})
\end{CD}$
\end{center}

By defining $\psi_{ik}(z,t)=\psi_{ik}(z_k,t)$ for $(0,z_k)\in U_k\cap U_i$, we have the equality (see \cite{Kod62} p.152-153)
\begin{align}\label{33n}
\psi_{ik}(z,t)=\psi_{ij}(z,t)+F_{ij}\cdot \psi_{jk}(z,t),\,\,\,\,\,\text{for $z\in U_i\cap U_j\cap U_k$}
\end{align}
On the other hand, by applying $[\Lambda_0,-]$ on $(\ref{33g})$, we have
\begin{align}\label{33h}
0=[\Lambda_0,[\Lambda_0, w_i^\alpha-\varphi_i^{\alpha m}]]=\sum_{\beta=1}^r-[\Lambda_0,w_i^{\beta}-\varphi_i^{\beta m} (z_i,t)]\wedge T_{i\alpha}^\beta(w_i,z_i)+\sum_{\beta=1}^r(w_i^\beta-\varphi_i^{\beta m}(z_i,t))[\Lambda_0,T_{i\alpha}^\beta(w_i,z_i)]\\
+[\Lambda_0,K_i^{\alpha m} (z_i,t)]+\sum_{\beta=1}^r-[\Lambda_0,w_i^\beta-\varphi_i^\beta(z_i,t)]\wedge L_{i\alpha}^\beta(w_i,z_i,t)+\sum_{\beta=1}^r (w_i^\beta-\varphi_i^\beta(z_i,t))[\Lambda_0,L_{i\alpha}^\beta(w_i,z_i,t)]\notag
\end{align}
By restricting $(\ref{33h})$ to $w_i=\varphi_i^m(z_i,t)$, since $G_i^\alpha(z_i,t)\equiv_m 0,\alpha=1,...,r$, we get
\begin{align*}
&0\equiv_{m+1}\sum_{\beta=1}^r -G_i^\beta(z_i,t)\wedge T_{i\alpha}^\beta(0,z_i)+[\Lambda_0,K_i^{\alpha m}(z_i,t)]|_{w_i=\varphi_i^m(z_i,t)}+\sum_{\beta=1}^r -G_i^\beta(z_i,t)\wedge L_{i\alpha}^\beta(\varphi_i^m(z_i,t),z_i,t)
\end{align*}
Since the degree $L(w_i,z_i,t)$ is at least $1$ in $t_1,...,t_l$ and we have, from $(\ref{33g})$,
\begin{align}\label{iii1}
G_i^\alpha(z_i,t)\equiv_{m+1}K_i^{\alpha m}(z_i,t)=Q_i^{\alpha m}(z_i,t)+\sum_{\beta=1}^r \phi_i^{\beta m}(z_i,t)P_{i\alpha}^{\beta m}(\varphi_i^{\beta m}(z_i,t),z_i,t),
\end{align}
we obtain
\begin{align}
0\equiv_{m+1}\sum_{\beta=1}^r -G_i^\beta(z_i,t)\wedge T_{i\alpha}^\beta(0,z_i)+[\Lambda_0, G_i^\alpha(z_i,t)]|_{w_i=0}
\end{align}

Next, since $f_{ik}^\alpha(w_k,z_k)-\varphi_i^{\alpha m}(g_{ik}(w_k,z_k),t)-[f_{ik}^\alpha(\varphi_k^m(z_k,t),z_k)-\varphi_i^{\alpha m}(g_{ik}(\varphi_k^m(z_k,t),z_k),t)]=\sum_{\beta=1}^r(w_k^\beta-\varphi_k^{\beta m}(z_k,t))\cdot S_{k\alpha}^\beta(w_k,z_k,t)$ for some $S_{k\alpha}^\beta(w_k,z_k,t)$. By setting $t=0$, we get $f_{ik}^\alpha(w_k,z_k)-f_{ik}^\alpha(0,z_k)=\sum_{\beta=1}^r w_k^\beta\cdot S_{k\alpha}^\beta(w_k,z_k,0)$, and then by taking the derivative with respect to $w_k^\gamma$ and setting $w_k=0$, we obtain 
\begin{align}\label{ijk1}
\frac{\partial f_{ik}^\alpha(w_k,z_k)}{\partial w_k^\gamma}|_{w_k=0}=S_{k\alpha}^\gamma(0,z_k,0). 
\end{align}
Then we have
\begin{align*}
&[\Lambda_0,f_{ik}^\alpha(w_k,z_k)-\varphi_i^{\alpha m}(g_{ik}(w_k,z_k),t)]|_{w_k=\varphi_k^m(z_k,t)}+[\Lambda_0, \psi_{ik}^\alpha(z_k,t)]|_{w_k=0}\\
&\equiv_{m+1}[\Lambda_0,f_{ik}^\alpha(w_k,z_k)-\varphi_i^{\alpha m}(g_{ik}(w_k,z_k),t)]|_{w_k=\varphi_k^m(z_k,t)}+[\Lambda_0, \psi_{ik}^\alpha(z_k,t)]|_{w_k=\varphi_k^m(z_k,t)}\\
&\equiv_{m+1}[\Lambda_0, f_{ik}^\alpha(w_k,z_k)-\varphi_i^{\alpha m}(g_{ik}(w_k,z_k),t)+\psi_{ik}^\alpha(z_k,t)]|_{w_k=\varphi_k^m(z_k,t)}\\
&\equiv_{m+1}[\Lambda_0,\sum_{\beta=1}^r (w_k^\beta-\varphi_k^{\beta m}(z_k,t))S_{k\alpha}^\beta(w_k,z_k,t)]|_{w_k=\varphi_k^m(z_k,t)}\\
&\equiv_{m+1}\sum_{\beta=1}^r[\Lambda_0, w_k^\beta-\varphi_k^{\beta m}(z_k,t)]|_{w_k=\varphi_k^m(z_k,t)}\cdot S_{k\alpha}^\beta(\varphi_k^m(z_k,t),z_k,t)\\
&\equiv_{m+1}\sum_{\beta=1}^r G_k^\beta(z_k,t)\cdot S_{k\alpha}^\beta(0,z_k,0)\equiv_{m+1}\sum_{\beta=1}^r G_k^\beta(z_k,t)\cdot \frac{\partial f_{ik}^\alpha(w_k,z_k)}{\partial w_k^\beta}|_{w_k=0}
\end{align*}
Hence we obtain the equality
\begin{align}\label{33k}
[\Lambda_0,f_{ik}^\alpha(w_k,z_k)-\varphi_i^{\alpha m}(g_{ik}(w_k,z_k),t)]|_{w_k=\varphi_k^m(z_k,t)}+[\Lambda_0, \psi_{ik}^\alpha(z_k,t)]|_{w_k=0}\\
\equiv_{m+1}\sum_{\beta=1}^r G_k^\beta(z_k,t)\cdot S_{k\alpha}^\beta(0,z_k,0)=\sum_{\beta=1}^r G_k^\beta(z_k,t)\cdot \frac{\partial f_{ik}^\alpha(w_k,z_k)}{\partial w_k^\beta}|_{w_k=0}\notag
\end{align}

On the other hand, from (\ref{33g}), we have
\begin{align}\label{33i}
[\Lambda_0,f_{ik}^\alpha(w_k,z_k) -\varphi_i^{\alpha m}(g_{ik}(w_k,z_k),t)]=\sum_{\beta=1}^r(f_{ik}^\beta(w_k,z_k)-\varphi_i^{\beta m}(g_{ik}(w_k,z_k),t))T_{i\alpha}^\beta(w_i,z_i)\\+K_i^{\alpha m}(z_i,t)+\sum_{\beta=1}^r (f_{ik}^\beta(w_k,z_k)-\varphi_i^{\beta m}(g_{ik}(w_k,z_k),t) )L_{i\alpha}^\beta(w_i,z_i,t)\notag
\end{align}
By restricting (\ref{33i}) to $w_k=\varphi_k^m(z_k,t)$, we get, from $(\ref{iii1})$,
\begin{align}\label{33j}
[\Lambda_0,f_{ik}^\alpha(w_k,z_k)-\varphi_i^{\alpha m}(g_{ik}(w_k,z_k),t)]|_{w_k=\varphi_k^m(z_k,t)}\equiv_{m+1} \sum_{\beta=1}^r -\psi_{ik}^\beta(z_k,t)T_{i\alpha}^\beta(0,z_i)+G_i^\alpha(z_i,t)
\end{align}
Hence from (\ref{33k}), (\ref{33j}) and $(\ref{ijk1})$, we obtain
\begin{align}\label{33l}
\sum_{\beta=1}^r -\psi_{ik}^\beta(z_k,t)T_{i\alpha}^\beta(0,z_i)+G_i^\alpha(z_i,t)+[\Lambda_0, \psi_{ik}^\alpha(z_k,t)]|_{w_k=0}=\sum_{\beta=1}^r G_k^\beta(z_k,t)\cdot \frac{\partial f_{ik}^\alpha(w_k,z_k)}{\partial w_k^\beta}|_{w_k=0}
\end{align}

Hence from (\ref{33n}),(\ref{33k}),(\ref{33l}), $\{(\psi_{ik}^1(z_k,t),...,\psi_{ik}^r(z_k,t))\}\oplus \{ (G_i^1(z_i,t),...,G_i^r(z_i,t)) \}$ defines a $1$-cocycle in the above complex so that we get the claim. We call $\psi_{m+1}(t):=\{(\psi_{ik}^1(z_k,t),...,\psi_{ik}^r(z_k,t))\}$ and $G_{m+1}(t):=\{(G_i^1(z_i,t),...,G_i^r(z_i,t))\}$ the $m$-th obstruction so that the coefficients of $(\psi_{m+1}(t),G_{m+1}(t))$ in $t_1,...,t_l$ lies in $\mathbb{H}^1(V,\mathcal{N}_{V/W}^\bullet)$.

On the other hand, by hypothesis, the cohomology group $\mathbb{H}^1(V,\mathcal{N}_{V/W}^\bullet)$ vanishes. Therefore there exists $\varphi_{i|m+1}^\alpha(z_i,t)$ such that $\psi_{ik}(z,t)=F_{ik}(z)\varphi_{k|m+1}(z,t)-\varphi_{i|m+1}(z,t)$ and $\sum_{\beta=1}^r \varphi^\beta_{i|m+1}(z_i,t) T_{i\alpha}^\beta(0,z_i)-[\Lambda_0,\varphi_{i|m+1}^\alpha(z_i,t)]|_{w_i=0}=-G_i^\alpha(z_i,t)$. Then we can show $(\ref{33b})_{m+1}$ (for the detail, see \cite{Kod62} p.154). On the other hand,
\begin{align*}
-[\Lambda_0,\varphi_{i|m+1}^\alpha(z_i,t)]=-\sum_{\beta=1}^r \varphi^\beta_{i|m+1}(z_i,t) T_{i\alpha}^\beta(w_i,z_i)-G_i^\alpha(z_i,t)+\sum_{\beta=1}^r w_i^\beta R_{i\alpha}^\beta(w_i,z_i,t)
\end{align*}
where the degree of $R_{i\alpha}^\beta(w_i,z_i,t)$ is $m+1$ in $t_1,...,t_l$. Then from $(\ref{33f})$, we have
\begin{align}\label{ijk2}
&[\Lambda_0,w_i^\alpha-\varphi_i^{\alpha m}(z_i,t)-\varphi^\alpha_{i|m+1}(z_i,t)]=[\Lambda_0,w_i^\alpha-\varphi_i^{\alpha m}(z_i,t)]+[\Lambda_0, -\varphi_{i|m+1}^\alpha(z_i,t)]\\
&=\sum_{\beta=1}^r(w_i^\beta-\varphi_i^{\beta m}(z_i,t)-\varphi^\beta_{i|m+1}(z_i,t))T_{i\alpha}^\beta(w_i,z_i)+Q_i^{\alpha m}(z_i,t)-G_i^\alpha(z_i,t)+\sum_{\beta=1}^rw_i^\alpha( P_{i\alpha}^{\beta m}(w_i,z_i,t)+ R_{i\alpha}^\beta(w_i,z_i,t))\notag
\end{align}
By setting $\varphi_i^{\alpha(m+1)}(z_i,t):=\varphi_i^{\alpha m}(z_i,t)+\varphi_{i|m+1}^\alpha(z_i,t)$, we show that $[\Lambda_0, w_i^\alpha-\varphi_i^{\alpha (m+1)}(z_i,t)]|_{w_i=\varphi_i^{ m+1}(z_i,t)}\equiv_{m+1} 0$. Indeed, from $(\ref{ijk2})$ and $(\ref{iii1})$,
\begin{align*}
&[\Lambda_0, w_i^\alpha-\varphi_i^{\alpha (m+1)}(z_i,t)]|_{w_i=\varphi_i^{ m+1}(z_i,t)}\\
&\equiv_{m+1} Q_i^{\alpha m}(z_i,t)-G_i^\alpha(z_i,t)+\sum_{\beta=1}^r (\varphi_i^{\alpha m}(z_i,t)+\varphi_{i|m+1}(z_i,t))P_{i\alpha}^{\beta m}(\varphi_i^{ m}(z_i,t)+\varphi_{i|m+1}(z_i,t),z_i,t)\\
&\equiv_{m+1} Q_i^{\alpha m}(z_i,t)-G_i^\alpha(z_i,t)+\sum_{\beta=1}^r \varphi_i^{\alpha m}(z_i,t)P_{i\alpha}^{\beta m}(\varphi_i^{ m}(z_i,t),z_i,t)\equiv_{m+1} K_i^{\alpha m}(z_i,t)-G_i^{\alpha}(z_i,t)\equiv_{m+1} 0
\end{align*}
which shows $(\ref{33c})_{m+1}$.
This completes the inductive construction of the polynomials $\varphi_i^m(z_i,t),i\in I$. 

\subsection{Proof of convergence}\label{3convergence}\

We will show that we can choose $\varphi_{i|m}(z_i,t)$ in each inductive step so that the formal power series $\varphi_i(z_i,t),i\in I$ constructed in the previous subsection, converges absolutely for $|t|<\epsilon$ for a sufficiently small number $\epsilon>0$.

\begin{notation}\label{notation3}
Recall Notation \ref{notation2}. We write 
\begin{align}
A(t)=\frac{a}{16b}\sum_{n=1}^\infty \frac{b^n(t_1+\cdots +t_l)^n}{n^2} 
\end{align}
instead $A(t)$ in Notation \ref{notation2} to keep the notational consistency with \cite{Kod62}. We also note that 
\begin{align}\label{gh5}
A(t)^v\ll \left(\frac{a}{b}\right)^{v-1}A(t) \,\,\,\,\,\text{ for $v=2,3,...$}.
\end{align}
\end{notation}

We may assume that $|F_{ik\mu}^\lambda(0,z)|<c_0$ with $c_0>1$. Then $\varphi_{i|1}(z_i,t)\ll A(t)$ if $b$ is sufficiently large.

We assume that 
\begin{align}\label{yy1}
\varphi_i^m(z_i,t)\ll A(t)
\end{align}
for an integer $m\geq 1$, we shall estimate the coefficients of the homogenous polynomials $\psi_{ik}(z,t)$ and $G_i^\alpha(z,t)$ from $(\ref{33o})$ and $(\ref{33p})$.

Let $W_i^\delta$ be the subdomain of $W_i$ consisting of all points $(w_i,z_i)$, $|w_i|<1-\delta, |z_i|<1-\delta$ for a sufficiently small number $\delta>0$ such that $\{W_i^\delta|i\in I\}$ forms a covering of $W$, and $\{U_i^\delta=W_i^\delta\cap V|i\in I\}$ forms a covering of $V$. 

First we estimate the coefficients of the homogeneous polynomials $\psi_{ik}(z,t)$. We briefly summarize Kodaira's result in the following: we expand $f_{ik}(w_k)=f_{ik}(w_k,z_k)$ and $g_{ik}(w_k)=g_{ik}(w_k,z_k)$ into power series in $w_k^1,...,w_k^r$ whose coefficients are vector-valued holomorphic functions of $z=(0,z_k)$ defined on $U_k\cap U_i$. We assume that $f_{ik}(w_k)\ll \sum_{n=1}^\infty c_1^n(w_k^1+\cdots +w_k^r)^n$ and $g_{ik}(w_k)=\sum_{n=0}^\infty c_1^n(w_k^1+\cdots+w_k^r)^n$. Then we can estimate 
\begin{align}\label{3.6a}
\psi_{ik}(z_k,t)\ll c_3A(t),\,\,\, z_k \in U_k\cap U_i,
\end{align} 
where $c_3=2rc_0\frac{4c_1ra}{b}\left(\frac{2^d}{\delta}+rc_1\right)$ with 
\begin{align}\label{yy5}
b>\max\{2c_1ra, \frac{4c_1 ra}{\delta}\}.
\end{align}
Second we estimate the coefficients of the homogeneous polynomials $G_i^\alpha(z,t),\alpha=1,...,r$. Let $\Lambda_0=\Lambda_{i}(w_i,z_i)=\sum_{p,q=1}^{r+d}\Lambda_{pq}^i(w_i,z_i)\frac{\partial}{\partial x_i^p}\wedge\frac{\partial}{\partial x_i^q}$ with $\Lambda_{pq}^i(w_i,z_i)=-\Lambda_{pq}^i(w_i,z_i)$, where $x_i=(w_i,z_i)$ on $W_i$. By considering coefficients $\Lambda_{pq}^i(w_i,z_i)$ of $\frac{\partial}{\partial x_i^p}\wedge\frac{\partial}{\partial x_i^q}$, we can consider $\Lambda_{i}(w_i,z_i)$ a vector-valued holomorphic function on $W_i$. We expand $\Lambda_i(w_i)=\Lambda_i(w_i,z_i)$ into power series in $w_i^1,...,w_i^r$ whose coefficients are vector valued holomorphic functions of $z=(0,z_i)$ defined on $U_i$ and we may assume, for any $p,q$,
\begin{align}\label{gh1}
\Lambda_{pq}^i(w_i)=\Lambda_{pq}^i(w_i,z_i)\ll \sum_{n=0}^\infty e_1^n(w_i^1+\cdots +w_i^r)^n
\end{align}
for some constant $e_1>0$. Now we estimate 
\begin{align}\label{gh2}
G_i^\alpha(z_i,t)=[[\Lambda_0,w_i^\alpha-\varphi_i^{\alpha m}(z_i,t)]|_{w_i=\varphi_i^m(z_i,t)}]_{m+1}\,\,\,\text{ for $z_i\in U_i^\delta$}.
\end{align}
 First we estimate $[[\Lambda_0, w_i^\alpha]|_{w_i=\varphi_i^m(z_i,t)}]_{m+1}$ in $(\ref{gh2})$. We note that
\begin{align}\label{gh3}
[\Lambda_0,w_i^\alpha]|_{w_i=\varphi_i^m(z_i,t)}=\sum_{p,q=1}^{d+r}2\Lambda_{pq}^i(\varphi_i^m(z_i,t),z_i)\frac{\partial w_i^\alpha}{\partial x_i^p}\frac{\partial}{\partial x_i^q}
\end{align}
Since constant terms and linear terms of $\Lambda_{pq}^i(w_i,z_i)$ with respect to $w_i^1,\cdots,w_i^r$ does not contribute to $[\Lambda_{pq}^i(\varphi_i^m(z_i,t),z_i)]_{m+1}$, we get, from $(\ref{gh1})$ and $(\ref{gh5})$,
\begin{align}\label{gh4}
[\Lambda_{pq}^i(\varphi_i^m(z_i,t),z_i)]_{m+1}\ll\sum_{n=2}^\infty e_1^nr^n A(t)^n= e_1r\sum_{n=1}^\infty e_1^nr^nA(t)^{n+1}\ll e_1rA(t)\sum_{n=1}^\infty  \left( \frac{e_1 ra}{b}\right)^n=\frac{e_1^2r^2a}{b} A(t) \sum_{n=0}^\infty  \left( \frac{e_1 ra}{b}\right)^n
\end{align}
 Assuming that 
 \begin{align}\label{yy6}
 b>2e_1ra, 
 \end{align}
 we obtain, from $(\ref{gh3})$, and $(\ref{gh4})$
\begin{align}\label{yy3}
[[\Lambda_0, w_i^\alpha]|_{w_i=\varphi_i^m(z_i,t)}]_{m+1}\ll 2(d+r)^2\frac{2e_1^2r^2a}{b}A(t)
\end{align}

Second, we estimate $[[\Lambda_0, \varphi_i^{\alpha m}(z_i,t)]|_{w_i=\varphi_i^m(z_i,t)}]_{m+1}$ in $(\ref{gh2})$. We note that  by Cauchy's integral formula and $(\ref{yy1})$,
\begin{align}\label{rr1}
\frac{\partial \varphi_i^{\alpha m}(z_i,t)}{\partial w_i^\beta}=0,\,\,\,\,\,\frac{\varphi_i^{\alpha m}(z_i,t)}{\partial z_i^\gamma}=\frac{1}{2\pi i}\int_{|\xi-z_i^\gamma|=\delta}\frac{\varphi_i^{\alpha m}(z_i^1,\cdots, \xi, \cdots,z_i^d,t)}{(\xi-z_i^\gamma)^2}d\xi\ll \frac{A(t)}{\delta}\,\,\,\text{ for $|z_i|<1-\delta$}.
\end{align}
Since constant term of $\Lambda_{pq}^i(w_i,z_i)$ with respect to $w_i^1,...,w_i^r$ does not contribute to $[[\Lambda_0, \varphi_i^{\alpha m}(z_i,t)]|_{w_i=\varphi_i^m(z_i,t)}]_{m+1}$, we get, from $(\ref{gh1})$, $(\ref{gh5})$ and $(\ref{rr1})$,
\begin{align}
[[\Lambda_0,  \varphi_i^{\alpha m}(z_i,t)]|_{w_i=\varphi_i^{\alpha m}(z_i,t)}]_{m+1}&=\sum_{p,q=1}^{r+d} 2\Lambda_{pq}^i(\varphi_i^{\alpha m}(z_i,t),z_i)\frac{\partial \varphi_i^{\alpha m}(z_i,t)}{\partial x_i^p}\frac{\partial}{\partial x_i^q}\\
&\ll 2(r+d)^2\frac{A(t)}{\delta}\sum_{n=1}^\infty e_1^nr^n A(t)^n,\,\,\,\,\,\,\, z_i\in U_i^\delta \notag\\
&\ll2(r+d)^2\frac{1}{\delta} \sum_{n=1}^\infty e_1^nr^nA(t)^{n+1}\ll \frac{2(r+d)^2}{\delta}\sum_{n=1}^\infty \left(\frac{e_1ra}{b}\right )^nA(t)\notag\\
&\ll  \frac{2(r+d)^2}{\delta}\left(\frac{e_1ra}{b}\right )A(t)\sum_{n=0}^\infty \left(\frac{e_1ra}{b} \right)^n\notag
\end{align}
Assuming that $b> 2e_1ra$, we get
\begin{align}\label{yy2}
[[\Lambda_0,\varphi_i^{\alpha m}(z_i,t)]|_{w_i=\varphi_i^{\alpha m}}]_{m+1}\ll \frac{4(r+d)^2e_1ra}{\delta b}A(t),\,\,\,\,\,z_i\in U_i^\delta.
\end{align}

Hence from (\ref{yy3}), and (\ref{yy2}), we obtain
\begin{align}\label{bc1}
G_i^\alpha(z_i,t)=[[\Lambda_0, w_i^\alpha-\varphi_i^{m\alpha}(z_i,t)]|_{w_i=\varphi_i^m(z_i,t)}]_{m+1} \ll e_2A(t),\,\,\,z_i\in U_i^\delta,
\end{align}
where 
\begin{align}\label{yy7}
e_2=\frac{4(d+r)^2e_1^2r^2a}{b}+\frac{4(r+d)^2e_1ra}{\delta b}.
\end{align}

\begin{lemma}\label{nb2}
We can choose the homogenous polynomials $\varphi_{i|m+1}(z,t),i\in I$ satisfying
\begin{align*}
\psi_{ik}(z,k)&=F_{ik}(z)\varphi_{k|m+1}(z,t)-\varphi_{i|m+1}(z,t)\\
-G_i^\alpha(z,t)&=-[\varphi_{i|m+1}^\alpha(z,t),\Lambda_0]|_{w_i=0} +\sum_{\beta=1}^r \varphi_{i|m+1}^\beta(z,t)T_{i\alpha}^\beta(0,z)
\end{align*}
in such a way that $\varphi_{i|m+1}(z,t)\ll c_4(e_2+c_3)A(t)$, where $c_4$ is independent of $m$.
\end{lemma}
\begin{proof}
For any $0$-cochain $\varphi=\{\varphi_i(z)\}$, $1$-cochain $(\psi,G)=(\{\psi_{ik}(z)\},\{\sum_{\alpha=1}^rG_i^\alpha(z)e_i^\alpha\})$, we define the norms of $\varphi$ and $(\psi, G)$ by
\begin{align*}
||\varphi ||&:=\max_i \sup_{z\in U_i} |\varphi_i(z)|,\\
||(\psi,G)||&:=\max_{i,k} \sup_{z\in U_i\cap U_k} |\psi_{ik}(z)|+\max_{i,\alpha} \sup_{z\in U_i^\delta}|G_i^\alpha(z)|
\end{align*}
The coboundary $\varphi$ is defined by
\begin{align*}
\tilde{\delta}(\varphi):=(-F_{ik}(z)\varphi_k(z,t)+\varphi_i(z), -[\varphi_i^\alpha(z_i),\Lambda_0]|_{w_i=0}+\sum_{\beta=1}^r \varphi_i^\beta(z)T_{i\alpha}^\beta(0,z_i))
\end{align*}

For any coboundary $(\psi,G)$, we define 
\begin{align*}
\iota(\psi)=\inf_{\tilde{\delta} \varphi=(\psi, G)} ||\varphi ||.
\end{align*}
To prove Lemma \ref{nb2}, it suffices to show the existence of a constant $c$ such that $\iota(\psi,G)\leq c||(\psi,G)||$. Assume that such a constant $c$ does not exist. Then we can find a sequence $(\psi',G'),(\psi'',G''),\cdots , (\psi^{(\mu)},G^{(\mu)}),\cdots$ such that there exists $\varphi^{(\mu)}$ with $\delta \varphi^{(\mu)}=(\psi^{(\mu)},G^{(\mu)})$ satisfying $||\varphi^{(\mu)}||<2$. Then we can show that there is a subsequence $\varphi_i^{\mu_1},\varphi_i^{\mu_2},\cdots$ such that $\varphi_i^{(\mu_v)}(z_i)$ converges absolutely and uniformly on $U_i$. Let $\varphi_i(z_i)=\lim_v \varphi_i^{(\mu_v)}(z_i)$ and let $\varphi=\{\varphi_i(z_i)\}$. Then we have $||\varphi^{\mu_v}-\varphi||\to 0$. On the other hand $\tilde{\delta}(\varphi)=(0,G_\varphi)$, where $G_\varphi(z)=0$ for $z\in U_i^\delta$. By identity theorem $G_\varphi(z)=0$ for $z\in U_i$ so that $\tilde{\delta}(\varphi)=(0,0)$. Therefore $\tilde{\delta}(\varphi^{(\mu_v)}-\varphi)=(\varphi^{(\mu_v)},G^{(\mu_v)})$ which contradict to $\iota(\varphi^{(\mu_v)})=1$.
\end{proof}

From $(\ref{3.6a})$ and $(\ref{yy7})$, we have 
\begin{align}\label{yy8}
c_4(c_3+e_2)=c_4c_3+c_4e_2=\frac{8c_4c_0c_1r^2a}{b}\left(\frac{2^d}{\delta}+rc_1\right)+c_4\left(\frac{4(d+r)^2e_1^2r^2a}{b}+\frac{4(r+d)e_1ra}{\delta b}\right)
\end{align}
From $(\ref{yy5}),(\ref{yy6})$, $(\ref{yy8})$ and Lemma \ref{nb2}, by assuming
\begin{align*}
b>\max\{8c_4c_0c_1r^2a\left( \frac{2^d}{\delta}+rc_1  \right)+c_4\left(4(d+r)^2 e_1^2 r^2 a+\frac{4(r+d)e_1ra}{\delta}\right), 2c_1ra,\frac{ 4c_1ra}{\delta},2e_1ra\}
\end{align*}
we can choose $\varphi_{i|m+1}(z_i,t)\ll A(t)$ and so $\varphi_i(z_i,t)\ll A(t)$ so that the power series $\varphi_i(z_i,t)$ converges for $|t|<\frac{1}{lb}$. Then by the argument of \cite{Kod62} p.158, we obtain the equality
\begin{align*}
\varphi_i(g_{ik}&(\varphi_k(z_k,t),z_k),t)=f_{ik}(\varphi_k(z_k,t),z_k),\,\,\,\,\text{for}\,\,|t|<\epsilon,\,\,\,(\varphi_k(z_k,t),z_k)\in W_i^\delta\cap W_k^\delta\\
&[\Lambda_0,w_i^\alpha-\varphi_i^\alpha(z_i,t)]|_{w_i=\varphi_i(z_i,t)}=0
\end{align*}
for a sufficiently small number $\epsilon>0$, which proves Theorem \ref{33a}.
\end{proof}

In the case $\mathbb{H}^1(V,\mathcal{N}_{V/W}^\bullet)\ne 0$, our proof of Theorem $\ref{33a}$ proves the following:
\begin{theorem}
If the obstruction $(\psi_{m+1}(t),G_{m+1}(t))$ vanishes for each integer $m\geq 1$, then there exists a Poisson analytic family $\mathcal{V}$ of compact holomorphic Poisson submanifolds $V_t,t\in M_1$, of $(W,\Lambda_0)$ such that $V_0=V$ and the characteristic map
\begin{align*}
 \sigma_0:T_0(M_1)&\to \mathbb{H}^0(V,\mathcal{N}_{V/W}^\bullet)\\
 \frac{\partial}{\partial t}&\mapsto (\frac{\partial V_t}{\partial t})_{t=0}
\end{align*}
is an isomorphism.
\end{theorem}

\subsection{Maximal families: Theorem of completeness}\

We note that Definition \ref{005} can be extended to arbitrary codimensions.

\begin{theorem}[theorem of completeness]\label{3.5a}
Let $\mathcal{V}$ be a Poisson analytic family of compact holomorphic Poisson submanifolds $V_t$, $t\in M_1$, of $(W,\Lambda_0)$. If the characteristic map
\begin{align*}
\sigma_0:T_0(M_1)&\to \mathbb{H}^0(V_0, \mathcal{N}_{V_0/W}^\bullet)\\
\frac{\partial}{\partial t}&\mapsto \left(\frac{\partial V_t}{\partial t}\right)_{t=0}
\end{align*}
is an isomorphism, then the family $\mathcal{V}$ is maximal at $t=0$.
\end{theorem}
\begin{proof}
We extend the arguments in \cite{Kod62} p.158-160 in the context of holomorphic Poisson deformations.

Consider an arbitrary Poisson analytic family $\mathcal{V}'$ of compact holomorphic Poisson submanifolds $V_s',s\in M'$ of $(W,\Lambda_0)$ such that $V_0'=V_0$, where $M'=\{s=(s_1,...,s_q)\in \mathbb{C}^q||s|<1\}$. We shall construct a holomorphic map $h:s\to t=h(s)$ of a neighborhood $N'$ of $0$ into $M_1$ such that $h(0)=0$ and $V_s'=V_{h(s)}$.

We keep the notations in \ref{3infinitesimal} so that the holomorphic Poisson submanifold $V_t$ is defined in each domain $W_i,i\in I$ by $w_i=\varphi_i(z_i,t)$ and satisfy
\begin{align}\label{uu5}
[\Lambda_0,w_i^\alpha-\varphi_i^\alpha(z_i,t)]=\sum_{\beta=1}^r (w_i^\beta-\varphi_i^\beta(z_i,t))T_{i\alpha}^\beta(w_i,z_i,t).
\end{align}
We may assume that $V_s'$ is defined in each domain $W_i,i\in I$, by $w_i=\theta_i(z_i,t)$, where $\theta_i(z_i,t)$ is a vector-valued holomorphic function of $z_i, s$ with $|z_i|<1,|s|<1$, and satisfy
\begin{align}\label{uu6}
[\Lambda_0, w_i^\alpha-\theta_i^\alpha(z_i,s)]=\sum_{\beta=1}^r (w_i^\beta-\theta_i^\beta(z_i,s))P_{i\alpha}^\beta(w_i,z_i,s)
\end{align}
for some $P_{i\alpha}^\beta(w_i,z_i,s)$ which are power series in $s$ with coefficients in $\Gamma(W_i, T_W)$ and $P_{i\alpha}^\beta(0,z_i,0)=T_{i\alpha}^\beta(0,z_i)$. Then $V_s'=V_{h(s)}$ is equivalent to 
\begin{align}\label{uu3}
\theta_i(z_i,s)=\varphi_i(z_i,h(s))
\end{align}
Recall Notation \ref{notation1} and let us write $h(s)=h_1(s)+h_2(s)+\cdots, \varphi_i(z_i,t)=\varphi_{i|1}(z_i,t)+\varphi_{i|2}(z_i,t)+\cdots$, and $\theta_i(z_i,s)=\theta_{i|1}(z_i,t)+\theta_{i|2}(z_i,t)+\cdots$. We will construct $h(s)$ satisfying $(\ref{uu3})$ by solving the system of congruences by induction on $m$
\begin{align}\label{uu8}
\theta_i(z_i,s)\equiv_m \varphi_i(z_i,h^m(s)),\,\,\,i\in I,\,\,\,m=1,2,3,...
\end{align}
Since $\sigma_0: T_0(M_1)\to \mathbb{H}^0(V_0,\mathcal{N}_{V_0/W}^\bullet)$ is an isomorphism by hypothesis, $\{\frac{\partial \varphi_i(z_i,t)}{\partial t_1}|_{t=0},....,\frac{\partial \varphi_i(z_i,t) }{\partial t_l}|_{t=0}\}$ is a basis of $\mathbb{H}^0(V_0, \mathcal{N}_{V_0/W}^\bullet)$. Since $\{\theta_{i|1}(z_i,s)\}$ of $\theta_{i|1}(z_i,s), i\in I$, represents a linear form in $s$ whose coefficients are in $\mathbb{H}^0(V_0, \mathcal{N}_{V_0/W}^\bullet)$, there exists a linear vector-valued function $h^1(s)$ such that $\theta_{i|1}(z_i,s)=\varphi_{i|1}(z_i, h^1(s))$ which proves $(\ref{uu8})_1$.
Now assume that we have already constructed $h^m(s)$ satisfying $(\ref{uu8})_m$. We will find $h_{m+1}(s)$ such that $h^{m+1}(s)=h^m(s)+h_{m+1}(s)$ satisfy $(\ref{uu8})_{m+1}$. Let $\omega_i(z_i,s)=[\theta_i(z_i,s)-\varphi_i(z_i, h^m(s))]_{m+1}$. We claim that 
\begin{align}
&\omega_i(z_i,s)=F_{ik}(z)\cdot \omega_k(z_k,s),\,\,\,z\in U_i\cap U_k \label{uu1}\\
&-[\omega_i^\alpha(z_i,s),\Lambda_0]|_{w_i=0}+\sum_{\beta=1}^r \omega_i^\beta(z_i,s) T_{i\alpha}^\beta(0,z_i)=0, \label{uu2}
\end{align}
For the proof of $(\ref{uu1})$, see \cite{Kod62} p.160. Let us show $(\ref{uu2})$. From $(\ref{uu5})$ and $(\ref{uu6})$, we have
\begin{align*}
&[\Lambda_0, \omega_i^\alpha(z_i,s)]|_{w_i=0}\equiv_{m+1}[\Lambda_0,\omega_i^\alpha(z_i,s)]|_{w_i=\theta_i(z_i,t)}\equiv_{m+1} [\Lambda_0,\theta_i^\alpha(z_i,s)-w_i^\alpha+w_i^\alpha-\varphi_i^\alpha(z_i, h^m(s))]|_{w_i=\theta_i(z_i,t)}\\
&\equiv_{m+1} -[\Lambda_0,w_i^\alpha-\theta_i^\alpha(z_i,s)]|_{w_i=\theta_i(z_i,t)}+[\Lambda_0, w_i^\alpha- \varphi_i^\alpha(z_i,h^m(s))]|_{w_i=\theta_i(z_i,t)}\equiv_{m+1} [\Lambda_0,w_i^\alpha-\varphi_i^\alpha(z_i,h^m(s))]|_{w_i=\theta_i(z_i,t)}\\
&\equiv_{m+1} \sum_{\beta=1}^r (\theta_i^\beta(z_i,s)-\varphi_i^\beta(z_i,h(s)))T_{i\alpha}^\beta(w_i,z_i,h(s))\equiv_{m+1} \sum_{\beta=1}^r \omega_i^\beta(z_i,s)T_{i\alpha}^\beta(0,z_i)
\end{align*}
This proves $(\ref{uu2})$. From $(\ref{uu1})$ and $(\ref{uu2})$, $\{\omega_i(z_i,s)\}$ is a homogenous polynomial of degree $m+1$ in $s$ with coefficents in $ \mathbb{H}^0(V, \mathcal{N}_{V_0/W}^\bullet)$ so that there exists a homogenous polynomial $h_{m+1}(s)$ of degree $1$ in $s$ such that $\omega_i(z_i,s)=\varphi_{i|1}(z_i,h_{m+1}(s))$. Therefore we have $\varphi_i(z_i, h^{m+1}(s))\equiv_{m+1} \varphi_i(z_i,h^m(s))+\omega_i(z_i,s)\equiv_{m+1} \theta_i(z_i,s)$ which completes the inductive construction of $h^{m+1}(s)$ satisfying $(\ref{uu8})_{m+1}$.

\subsection{Proof of convergence}\

The convergence of the power series $h(s)$ follows from the same arguments in \cite{Kod62} p.160-161. This completes the proof of Theorem \ref{3.5a}.

\end{proof}

\begin{example}\label{mp1}
Let $[\xi_0,\xi_1,\xi_2,\xi_3]$ be the homogenous coordinate on $\mathbb{P}^3_\mathbb{C}$ and a hyperplane $V$ defined by $\xi_3=0$ so that $\mathcal{N}_{V/\mathbb{P}^3_\mathbb{C}}\cong \mathcal{O}_V(1)$. Let $[1,z_1,z_2,z_3]=[1,\frac{\xi_1}{\xi_0},\frac{\xi_2}{\xi_0},\frac{\xi_3}{\xi_0}]$. Consider a Poisson structure $\Lambda_0=z_1\frac{\partial}{\partial z_1}\wedge \frac{\partial}{\partial z_2}$ on $\mathbb{P}^3_\mathbb{C}$. Then $V\cong \mathbb{P}_\mathbb{C}^2$ is a holomorphic Poisson submanifold. We compute $\mathbb{H}^0(V,\mathcal{N}_{V/\mathbb{P}^3_\mathbb{C}}^\bullet)$ which is the kernel $\nabla:H^0(V,\mathcal{O}_V(1))\to H^0(V,T_{\mathbb{P}^3_\mathbb{C}}|_V(1))$. Since $[\Lambda_0,z_3]=0$,
\begin{align*}
\nabla(az_1+bz_2+c)=-[\Lambda_0, a z_1+bz_2+c]|_{z_3=0}= -az_1\frac{\partial}{\partial z_2}+bz_1\frac{\partial}{\partial z_1}=0 \iff a=b=0
\end{align*}
so that $\dim_\mathbb{C} \mathbb{H}^0(V,\mathcal{O}_V(1))=1$. Since $[\Lambda_0, z_3-t]|_{z_3=t}=0$,  holomorphic Poisson deformations of $V$ in $(\mathbb{P}_\mathbb{C}^3,\Lambda_0)$ is unobstructed, and explicitly we have a Poisson analytic family of holomorphic Poisson submanifolds $\mathcal{V}\subset(\mathbb{P}_\mathbb{C}^3\times \mathbb{C},\Lambda_0)$ defined by $\xi_3-t\xi_0=0$ whose characteristic map
\begin{align*}
T_0\mathbb{C}&\to \mathbb{H}^0(V,\mathcal{N}_{V/\mathbb{P}_\mathbb{C}^3}^\bullet)=\mathbb{H}^0(\mathbb{P}_\mathbb{C}^2,\mathcal{O}_{\mathbb{P}_\mathbb{C}^2}(1)^\bullet)\\
a&\mapsto a= 0\cdot z_1+0\cdot z_2+a
\end{align*}
 is an isomorphism so that $\mathcal{V}$ is complete.
\end{example}

\begin{example}
Let us consider the Poisson structure $\Lambda_0=z_1\frac{\partial}{\partial z_1}\wedge \frac{\partial}{\partial z_2}$ on $\mathbb{P}^3_\mathbb{C}$ as in $\textnormal{Example}$ $\ref{mp1}$ and a holomorphic Poisson submanifold $V$ defined by $\xi_1=\xi_3=0$ which is a nonsingular rational curve $\cong \mathbb{P}^1_\mathbb{C}$ and the normal bundle $\mathcal{N}_{V/\mathbb{P}^3_\mathbb{C}}$ is isomorphic to $\mathcal{O}_{\mathbb{P}_\mathbb{C}^1}(1)\oplus \mathcal{O}_{\mathbb{P}_\mathbb{C}^1}(1)$. We compute $\mathbb{H}^0(V,\mathcal{N}_{V/\mathbb{P}_\mathbb{C}^3}^\bullet)$ which is the kernel of $\nabla:H^0(\mathbb{P}_\mathbb{C}^1,\mathcal{O}_{\mathbb{P}_\mathbb{C}^1}(1)\oplus \mathcal{O}_{\mathbb{P}_\mathbb{C}^1}(1))\to H^0(\mathbb{P}_\mathbb{C}^1,T_{\mathbb{P}_\mathbb{C}^3}|_{\mathbb{P}_\mathbb{C}^1}(1)\oplus T_{\mathbb{P}_\mathbb{C}^3}|_{\mathbb{P}_\mathbb{C}^1}(1))$. Since $[\Lambda_0, z_1]=z_1\frac{\partial}{\partial z_2}$, and $[\Lambda_0, z_3]=0$,
\begin{align*}
\nabla((az_2+b)\oplus (cz_2+d))=(-[\Lambda_0, az_2+b]|_{z_1=z_3=0}+(az_2+b)\frac{\partial}{\partial z_2})\oplus (-[\Lambda_0,cz_2+d]|_{z_1=z_3=0})=0 \\
\iff a=b=0.
\end{align*}
so that $\dim_\mathbb{C} \mathbb{H}^0(V, \mathcal{N}_{V/\mathbb{P}_\mathbb{C}^3}^\bullet)=2$. Since $[\Lambda_0, z_1]|_{z_1=z_3-t_1z_2-t_2=0}=0$ and $[\Lambda_0, z_3-t_1z_2-t_2]|_{z_1=z_3-t_1z_2-t_2}=0$, holomorphic Poisson deformations of $V$ in $(\mathbb{P}_\mathbb{C}^3,\Lambda_0)$ is unobstructed, and explicitly we have a Poisson analytic family of holomorphic Poisson submanifolds $\mathcal{V}\subset(\mathbb{P}_\mathbb{C}^3\times \mathbb{C}^2, \Lambda_0))$ defined by $\xi_1=\xi_3-t_1\xi_2-t_2\xi_0=0$ whose characteristic map is
\begin{align*}
T_0\mathbb{C}^2&\to \mathbb{H}^0(V,\mathcal{N}_{V/\mathbb{P}_\mathbb{C}^3}^\bullet)=\mathbb{H}^0(\mathbb{P}_\mathbb{C}^1,(\mathcal{O}_{\mathbb{P}_\mathbb{C}^1}(1)\oplus \mathcal{O}_{\mathbb{P}^1_\mathbb{C}}(1))^\bullet)\\
(a_1,a_2)&\mapsto (0, a_1z_2+a_2)
\end{align*}
which is an isomorphism so that $\mathcal{V}$ is complete.
\end{example}

\begin{example}
Let $(X,\Lambda_0)$ be a non-degenerate Poisson $K3$ surface. Consider $(X\times \mathbb{C}^q,\Lambda_0)$ and let $(w_1,...,w_q)$ be the coordinate of $\mathbb{C}^q$. Then $w_1=\cdots=w_q=0$ defines a holomorphic Poisson submanifold which is $(X,\Lambda_0)$ and the normal bundle $\mathcal{N}_{X/X\times \mathbb{C}^q}$ is $\oplus^q \mathcal{O}_X$. Since $\Lambda_0|_{(x,a)}\in \wedge^2 T_{X}|_x\subset \wedge^2 T_{X\times \mathbb{C}}|_{(x,a)}$ for each $(x,a)\in X\times \mathbb{C}^q$, and $[\Lambda_0, w_i]=0$ for $i=1,...,q$, the obstructions lies in the first cohomology group of the following complex of sheaves
\begin{align*}
\oplus^q \mathcal{O}_X^\bullet:\oplus^q \mathcal{O}_X\xrightarrow{-[-,\Lambda_0]} \oplus^q T_X\xrightarrow{-[-,\Lambda_0]} \oplus^q\wedge^2 T_X\xrightarrow{-[-\Lambda_0]} \cdots
\end{align*}
whose the $0$-th cohomology group is isomorphic to $\oplus^q H^0(X,\mathbb{C})\cong \mathbb{C}^q$ and the first cohomology group is isomorphic to $\oplus^q H^1(X,\mathbb{C})=0$ so that deformations of $X$ in $(X\times \mathbb{C}^q,\Lambda_0)$ is unobstructed. Explicitly we have a Poisson analytic family of holomorphic Poisson submanifolds $\mathcal{V}\subset ((X\times \mathbb{C}^q)\times \mathbb{C}^q,\Lambda_0)$ with $(t_1,...,t_q)$ the last coordinate of $\mathbb{C}^q$ which is defined by $w_1-t_1=\cdots=w_q-t_q=0$ whose characteristic map
\begin{align*}
T_0\mathbb{C}^q&\to \mathbb{H}^0(X,\mathcal{N}_{X/X\times \mathbb{C}}^\bullet)\cong \oplus^q H^0(X,\mathbb{C})\cong \oplus^q \mathbb{C}\\
(a_1,...,a_q)&\mapsto (a_1,...,a_q)
\end{align*}
 is an isomorphism so that $\mathcal{V}$ is complete.
\end{example}

\section{Simultaneous deformations of holomorphic Poisson structures and compact holomorphic Poisson submanifolds}\label{section4}\

We extend the definition of a Poisson analytic family of compact holomorphic Poisson submanifolds in Definition \ref{3a} by deforming holomorphic Poisson structures as well on a fixed complex manifold $W$.
\begin{definition}\label{4a}
Let $W$ be a complex manifold of dimension $d+r$. We denote a point in $W$ by $w$ and a local coordinate of $w$ by $(w^1,...,w^{r+d})$. By an extended Poisson analytic family of compact holomorphic Poisson submanifolds of dimension $d$ of $W$, we mean a holomorphic Poisson submanifold $\mathcal{V}\subset(W\times M, \Lambda)$ of codimension $r$, where $M$ is a complex manifold and $\Lambda$ is a holomorphic Poisson structure on $W\times M$, such that 
\begin{enumerate}
\item the canonical projection $\pi:(W\times M,\Lambda)\to M$ is a holomorphic Poisson fibre manifold as in Definition $\ref{fibre}$ so that $\Lambda\in H^0(W\times M, \wedge^2 T_{W\times M/M})$ and $\pi^{-1}(t):=(W,\Lambda_t)$ is a holomorphic Poisson submanifold of $(W\times M, \Lambda)$ for each point $t\in M$.
\item for each point $t\in M$, $V_t\times t:=\omega^{-1}(t)=\mathcal{V}\cap \pi^{-1}(t)$ is a connected compact holomorphic Poisson submanifold of $(W,\Lambda_t)$ of dimension $d$, where $\omega:\mathcal{V}\to M$ is the map induced from $\pi$.
\item for each point $p\in \mathcal{V}$, there exist $r$ holomorphic functions $f_\alpha(w,t),\alpha=1,...,r$ defined on a neighborhood $\mathcal{U}_p$ of $p$ in $W\times M$ such that $\textnormal{rank} \frac{\partial ( f_1,...,f_r)}{\partial (w^1,...w^{r+d})}=r$, and $\mathcal{U}_p\cap \mathcal{V}$ is defined by the simultaneous equations $f_\alpha(w,t)=0,\alpha=1,...,r$.
\end{enumerate}
We call $\mathcal{V}\subset (W\times M, \Lambda)$ an extended Poisson analytic family of compact holomorphic Poisson submanifolds $V_t,t\in M$ of $(W,\Lambda_t)$. We also call $\mathcal{V}\subset (W\times M,\Lambda)$ an extended Poisson analytic family of simultaneous deformations of a holomorphic Poisson submanifold $V_{t_0}$ of $(W,\Lambda_{t_0})$ for each fixed point $t_0\in M$.
\end{definition}

\subsection{The extended complex associated with the normal bundle of a holomorphic Poisson submanifold of a holomorphic Poisson manifold}\label{43a}\

Let $V$ be a holomorphic Poisson submanifold of a holomorphic Poisson manifold $(W,\Lambda_0)$. We will describe a complex of sheaves to control simultaneous deformations of holomorphic Poisson structures and holomorphic Poisson submanifolds. We recall that the complex associated with the normal bundle (see Definition \ref{332a})
\begin{align}\label{4b}
\mathcal{N}_{V/W}^\bullet:\mathcal{N}_{V/W}\xrightarrow{\nabla} \mathcal{N}_{V/W}\otimes T_W|_V\xrightarrow{\nabla}\mathcal{N}_{V/W}\otimes \wedge^2 T_W|_V\xrightarrow{\nabla}\cdots
\end{align}
controls holomorphic Poisson deformations of $V$ in $(W,\Lambda_0)$, and the complex
\begin{align}\label{4c}
\wedge^2 T_W^\bullet:\wedge^2 T_W\xrightarrow{-[-,\Lambda_0]} \wedge^3 T_W \xrightarrow{-[-,\Lambda_0]}\wedge^4 T_W \xrightarrow{-[-,\Lambda_0]}\cdots
\end{align}
controls deformations of the holomorphic Poisson structure $\Lambda_0$ on the fixed underlying complex manifold $W$ (see Appendix \ref{appendixa}).
By combining two complexes $(\ref{4b})$ and $(\ref{4c})$, we shall define a complex of sheaves on $W$:
\begin{align*}
(\wedge^2 T_W\oplus i_*\mathcal{N}_{V/W})^\bullet:\wedge^2 T_W\oplus i_* \mathcal{N}_{V/W}\xrightarrow{\tilde{\nabla}}\wedge^3 T_W\oplus i_*(\mathcal{N}_{V/W}\otimes T_W|_V)\xrightarrow{\tilde{\nabla}}\wedge^4 T_W\oplus i_*(\mathcal{N}_{V/W}\otimes \wedge^2 T_W|_V)\xrightarrow{\tilde{\nabla}}\cdots
\end{align*}
which controls simultaneous deformations of the holomorphic Poisson structure $\Lambda_0$ and the holomorphic Poisson submanifold $V$ of $(W,\Lambda_0)$, where $i:V\hookrightarrow W$ is the embedding.
We keep the notations in subsection \ref{subsection3.1}.

We note that $\Gamma(W_i,\wedge^{p+2} T_Y\oplus i_*(\mathcal{N}_{V/W}\otimes \wedge^p T_W|_V))=\Gamma(W_i,\wedge^{p+2} T_Y)\oplus \Gamma(U_i,\mathcal{N}_{V/W}\otimes \wedge^p T_W|_V)\cong \Gamma(W_i,\wedge^{p+2}T_W)\oplus (\oplus^r \Gamma(U_i,\wedge^p T_W|_V))$. From these isomorphisms, we define 
\begin{align*}
\tilde{\nabla}:\wedge^{p+2} T_W\oplus i_*(\mathcal{N}_{V/W}\otimes \wedge^p T_W|_V)\to \wedge^{p+3} T_W\oplus i_*( \mathcal{N}_{V/W} \otimes \wedge^{p+1} T_W|_V)
\end{align*}
locally in the following way:
\begin{align*}
\Gamma(W_i,\wedge^{p+2} T_W)\oplus (\oplus^r\Gamma(U_i, \wedge^p T_W|_V))&\xrightarrow{\tilde{\nabla}} \Gamma(W_i,\wedge^{p+3} T_W)\oplus (\oplus^r \Gamma(U_i,\wedge^{p+1} T_W|_V)\\
(\Pi_i,\sum_{\alpha=1}^r g_i^\alpha e_i^\alpha) &\mapsto (-[ \Pi_i,\Lambda_0], \sum_{\alpha=1}^r [\Pi_i,w_i^\alpha]|_{w_i=0}e_i^\alpha+\nabla(\sum_{\alpha=1}^r g_i^\alpha e_i^\alpha))
\end{align*}
In other words,
\begin{center}
\tiny{\begin{align*}
(\Pi_i,(g_i^1,...,g_i^r))\mapsto (-[\Pi_i, \Lambda_0],\left([\Pi_i, w_i^1]|_{w_i=0}-[g_i^1, \Lambda_0]|_{w_i=0}+(-1)^p\sum_{\beta=1}^r g_i^\beta\wedge T_{i1}^\beta(0,z_i),\cdots, [\Pi_i,w_i^r]|_{w_i=0}-[g_i^r,\Lambda_0]|_{w_i=0}+(-1)^p \sum_{\beta=1}^r g_i^\beta \wedge T_{ir}^\beta(0,z_i)\right))
\end{align*}}
\end{center}
First we show that $\tilde{\nabla}$ defines a complex, i.e $\tilde{\nabla}\circ \tilde{\nabla}=0$. Since $\nabla\circ \nabla=0$, we have
\begin{align*}
&\tilde{\nabla}(\tilde{\nabla}(\Pi_i,\sum_{\alpha=1}^r g_i^\alpha e_i^\alpha))=(0,\sum_{\alpha=1}^r -[[ \Pi_i,\Lambda_0],w_i^\alpha]|_{w_i=0}e_i^\alpha+\nabla(\sum_{\alpha=1}^r [\Pi_i,w_i^\alpha]|_{w_i=0}e_i^\alpha))\\
&=(0,\sum_{\alpha=1}^r-[[\Pi_i,\Lambda_0],w_i^\alpha ]|_{w_i=0}e_i^\alpha+\sum_{\alpha=1}^r -[[\Pi_i, w_i^\alpha],\Lambda_0]|_{w_i=0}e_i^\alpha +(-1)^{p+1}\sum_{\alpha,\beta=1}^r [\Pi_i,w_i^\alpha]|_{w_i=0} \wedge T_{i\beta}^\alpha(0,z_i) e_i^\beta)
\end{align*}

Hence $\tilde{\nabla}\circ \tilde{\nabla}=0$ is equivalent to
\begin{align}\label{mm1}
-[[\Pi_i,\Lambda_0],w_i^\alpha]|_{w_i=0}-[[\Pi_i,w_i^\alpha],\Lambda_0]|_{w_i=0}+(-1)^{p+1}\sum_{\beta=1}^r [\Pi_i,w_i^\beta]|_{w_i=0} \wedge T_{i\alpha}^\beta(0,z_i)=0
\end{align}
Let us show (\ref{mm1}). We note that from $(\ref{3d})$,
\begin{align*}
[\Pi_i,[\Lambda_0, w_i^\alpha]]=\sum_{\beta=1}^r[\Pi_i, w_i^\beta T_{i\alpha}^\beta (w_i,z_i) ]=\sum_{\beta=1}^r w_i^\beta[\Pi_i, T_{i\alpha}^\beta(w_i,z_i)]+(-1)^{p+1}\sum_{\beta=1}^r [\Pi_i, w_i^\beta]\wedge T_{i\alpha}^\beta(w_i,z_i)\\
\Longrightarrow [\Pi_i,[\Lambda_0, w_i^\alpha]]|_{w_i=0}=(-1)^{p+1}\sum_{\beta=1}^r [\Pi_i, w_i^\beta]|_{w_i=0}\wedge T_{i\alpha}^\beta(0,z_i)
\end{align*}
Then $(\ref{mm1})$ is equivalent to
\begin{align}\label{mm2}
-[[\Pi_i,\Lambda_0],w_i^\alpha]|_{w_i=0}-[[\Pi_i, w_i^\alpha], \Lambda_0]_{w_i=0}+[\Pi_i,[\Lambda_0, w_i^\alpha]]_{w_i=0}=0
\end{align}
which follows from $-[[\Pi_i,\Lambda_0],w_i^\alpha]-[[\Pi_i, w_i^\alpha], \Lambda_0]+[\Pi_i,[\Lambda_0, w_i^\alpha]]=0$ by the graded Jacobi identity. This proves $\tilde{\nabla}\circ\tilde{\nabla}=0$.

Next we show that $\tilde{\nabla}$ is well-defined. In other words, on $U_i\cap U_k$, the following diagram commutes
\begin{center}
\begin{equation}\label{mm15}
\begin{CD}
\Gamma(W_k, \wedge^{p+2} T_W)\oplus (\oplus^r \Gamma(U_k,\wedge^p T_W|_V))@>\text{on $U_i\cap U_k$}>\cong> \Gamma(W_i, \wedge^{p+2} T_W)\oplus (\oplus^r \Gamma(U_i,\wedge^p T_W|_V))\\
@V\tilde{\nabla}VV @VV\tilde{\nabla}V\\
\Gamma(W_k, \wedge^{p+3} T_X)\oplus (\oplus^r \Gamma(U_k,\wedge^{p+1} T_W|_V))@>\text{on $U_i\cap U_k$}>\cong>\Gamma(W_i, \wedge^{p+3} T_W)\oplus (\oplus^r \Gamma(U_i,\wedge^{p+1} T_W|_V))
\end{CD}
\end{equation}
\end{center}
Let $(\Pi,\sum_{\alpha=1}^r g_k^\alpha e_k^\alpha)\in \Gamma(W_k,\wedge^{p+2}T_W)\oplus (\oplus^r \Gamma(U_k,\wedge^p T_W|_V))$ and let $\nabla(\sum_{\alpha=1}^r g_k^\alpha e_k^\alpha)=\sum_{\alpha=1}^r G_k^\alpha e_k^\alpha$, where $G_k^\alpha =-[g_k^\alpha,\Lambda_0]|_{w_k=0}+(-1)^p\sum_{\beta=1}^r g_k^\beta \wedge T_{k\alpha}^\beta(0,z_k)\in \Gamma(U_k,\wedge^{p+1}T_W|_V)$.

 Then $\tilde{\nabla}((\Pi,\sum_{\alpha=1}^r g_k^\alpha e_k^\alpha))=(-[\Pi,\Lambda_0],\sum_{\alpha=1}^r  [\Pi,w_k^\alpha]|_{w_k=0}e_k^\alpha+\sum_{\alpha=1}^r G_k^\alpha e_k^\alpha)$ is identified on $U_i\cap U_k$ with 
\begin{align}\label{mm5}
(-[\Pi,\Lambda_0],\sum_{\beta=1}^r \left( \sum_{\alpha=1}^r F_{ik \alpha }^\beta(0,z_k) [\Pi, w_k^\alpha]|_{w_k=0}\right)e_i^\beta+\sum_{\beta=1}^r(\sum_{\alpha=1}^r F_{ik\alpha}^\beta(0,z_k) G_k^\alpha) e_i^\beta)
\end{align}
On the other hand, $(\Pi, \sum_{\alpha=1}^r g_k^\alpha e_k^\alpha)$ is identified on $U_i\cap U_k$ with $(\Pi,\sum_{\beta=1}^r \left( \sum_{\alpha=1}^r F_{ik \alpha}^\beta(0,z_k) g_k^\alpha \right)e_i^\beta)\in \Gamma(W_i, \wedge^{p+2} T_W)\oplus (\oplus^r \Gamma(U_i,\wedge^p T_W|_V))$. Then we have 
\begin{align}\label{mm6}
\tilde{\nabla}((\Pi,\sum_{\beta=1}^r \left( \sum_{\alpha=1}^r F_{ik \alpha }^\beta(0,z_k) g_k^\alpha \right)e_i^\beta)=(-[\Pi, \Lambda_0],\sum_{\beta=1}^r [\Pi, w_i^\beta]|_{w_i=0}e_i^\beta+\nabla(\sum_{\alpha,\beta=1}^r F_{ik\alpha }^\beta(0,z_k) g_k^\alpha) e_i^\beta))
\end{align}
Hence in order for the diagram (\ref{mm15}) to commute, we have to show that $(\ref{mm5})$ coincides with $(\ref{mm6})$. By the equality of $(\ref{mm10})$ and (\ref{mm11}), it is enough to show $[\Pi,w_i^\beta]|_{w_i=0}=\sum_{\alpha=1}^r F_{\alpha ik}^\beta(0,z_k)[\Pi, w_k^\alpha]|_{w_k=0}$ which comes from (\ref{3b}), and 
\begin{align*}
w_i^\beta= \sum_{\alpha=1}^r F_{ik\alpha}^\beta(w_k,z_k) w_k^\alpha \Longrightarrow [\Pi,w_i^\beta]=\sum_{\alpha=1}^r (F_{ik\alpha}^\beta(w_k,z_k)[\Pi, w_k^\alpha]+w_k^\alpha[\Pi, F_{ik\alpha}^\beta]).
\end{align*}
Hence $\tilde{\nabla}$ is well-defined.

\begin{definition}\label{033}
We call the complex defined as above
\begin{align*}
(\wedge^2 T_W\oplus i_*\mathcal{N}_{V/W})^\bullet:\wedge^2 T_W\oplus i_*\mathcal{N}_{V/W} \xrightarrow{\tilde{\nabla}} \wedge^3 T_W\oplus i_*(\mathcal{N}_{V/W}\otimes T_W|_V)\to\wedge^4 T_W\oplus i_*(\mathcal{N}_{V/W}\otimes \wedge^2 T_W|_V)\to \cdots
\end{align*}
the extended complex associated with the normal bundle $\mathcal{N}_{V/W}$ of a holomorphic Poisson submanifold $V$ of a holomorphic Poisson manifold $W$ and denote its $i$-th hypercohomology group by $\mathbb{H}^i(W,(\wedge^2 T_W\oplus i_*\mathcal{N}_{V/W})^\bullet)$.
\end{definition}

\subsection{Infinitesimal deformations}\label{027}\

Let $M_1=\{t=(t_1,...,t_l)\in \mathbb{C}^l||t|<1\}$. Consider an extended Poisson analytic family $\mathcal{V}\subset (W\times M_1,\Lambda)$ of compact holomorphic Poisson submanifolds $V_t,t\in M_1$ of $(W,\Lambda_t)$ and let $V=V_0$ as in Definition \ref{4a}. We keep the notations in subsection \ref{3infinitesimal} and subsection \ref{43a} so that for $|t|<\epsilon$ for a sufficiently small number $\epsilon>0$, $V_t$ is defined by $w_i=\varphi_i(z_i,t)$ on each neighborhood $W_i$ and, by setting $w_{ti}^\lambda=w_i^\lambda-\varphi_i^\lambda(z_i,t),\lambda=1,...,r$, $F_{tik}(z_t):=\left(\frac{\partial w_{ti}^\lambda}{\partial w_{tk}^\mu}(z_t)\right)_{\lambda,\mu=1,...,r}$ $\text{for $z_t\in V_t\cap W_i\cap W_k$}$ defines the normal bundle $\mathcal{N}_{V_t/W}$ of $V_t$ in $W$. For an arbitrary tangent vector $\frac{\partial}{\partial t}=\sum_\rho \gamma_\rho \frac{\partial}{\partial t_\rho}$ of $M_1$ at $t,|t|<\epsilon$, we let $\psi_i(z_t,t)=\frac{\partial \varphi_i(z_i,t)}{\partial t}$ for $z_t=(\varphi_i(z_i,t),z_i)$. Then we obtain the equality
\begin{align}\label{ss7}
\psi_i(z_t,t)=F_{tik}(z_t)\cdot \psi_k(z_t,t),\,\,\,\,\text{for $z_t\in V_t\cap W_i\cap W_k$}.
\end{align}

On the other hand, let $\Lambda_i(w_i,z_i,t)$ be the holomorphic Poisson structure $\Lambda$ on $W_i\times M_1$. Then
\begin{align}\label{ss4}
[\Lambda_i(w_i,z_i,t),\Lambda_i(w_i,z_i,t)]=0
\end{align}
and $\Lambda_i(w_i,z_i,t)$ is of the form
\begin{align*}
\Lambda_i(w_i,z_i,t):=\Lambda_i(x_i,t)=\sum_{\alpha,\beta=1}^{r+d}\Lambda_{\alpha\beta}^i(x_i,t)\frac{\partial}{\partial x_i^\alpha}\wedge \frac{\partial}{\partial x_i^\beta},\,\,\,\text{with $\Lambda_{\alpha\beta}^i(x_i,t)=-\Lambda_{\beta\alpha}^i(x_i,t)$},\,\,\,\,\,x_i=(w_i,z_i),
\end{align*}
by which we consider $\Lambda_i(w_i,z_i,t)$ as a vector-valued holomorphic function of $(w_i,z_i,t)$. Let $\pi_i(w_i,z_i,t)=\frac{\partial \Lambda_i(w_i,z_i,t)}{\partial t}$. Then
\begin{align}\label{ss6}
\{\pi_i(w_i,z_i,t)\}\in H^0(W,\wedge^2 T_W)
\end{align}
By taking the derivative of (\ref{ss4}) with respect to $t$, we get 
\begin{align}\label{ss9}
[\Lambda_i(w_i,z_i,t),\frac{\partial \Lambda_i(w_i,z_i,t)}{\partial t}]=0 \iff -[\pi_i(w_i,z_i,t),\Lambda_i(w_i,z_i,t)]=0
\end{align}
Lastly, since $w_i^\lambda-\varphi_i^\lambda(z_i,t)=0,\lambda=1,...,r$ define a holomorphic Poisson submanifold, we have
\begin{align}\label{ss5}
[\Lambda_i(w_i,z_i,t),w_i^\lambda-\varphi_i^\lambda(z_i,t)]=\sum_{\mu=1}^r (w_i^\mu-\varphi_i^\mu(z_i,t))T_{i\lambda}^{\mu}(w_i,z_i,t)
\end{align}
for some $T_{i\lambda}^{\mu}(w_i,z_i,t)$ which is of the form
\begin{align*}
T_{i\lambda}^\mu(w_i,z_i,t)=P_{i1}^\mu(w_i,z_i,t)\frac{\partial}{\partial w_i^1} +\cdots +P_{ir}^\mu(w_i,z_i,t)\frac{\partial}{\partial w_i^r}+Q_{i1}^\mu(w_i,z_i,t)\frac{\partial}{\partial z_i^1}+\cdots+Q_{id}^\mu(w_i,z_i,t)\frac{\partial}{\partial z_i^d}
\end{align*}
by which we consider $T_{i\lambda}^\mu(w_i,z_i,t)$ as a vector valued holomorphic function of $(w_i,z_i,t)$.  

By taking the derivative of (\ref{ss5}) with respect to $t$, we get $[\frac{\partial \Lambda_i(w_i,z_i,t)}{\partial t},w_i^\lambda-\varphi_i^\lambda(z_i,t)]+[\Lambda_i(w_i,z_i,t),-\frac{\partial \varphi_i^\lambda(z_i,t)}{\partial t} ]=\sum_{\mu=1}^r -\frac{\partial \varphi_i^\mu(z_i,t)}{\partial t}T_{i\lambda}^{\mu}(w_i,z_i,t)+\sum_{\mu=1}^n (w_i^\lambda-\varphi_i^\mu(z_i,t))\frac{\partial T_{i\lambda}^{\mu}(w_i,z_i,t)}{\partial t}$. By restricting to $V_t$, equivalently, by setting $w_i=\varphi_i(z_i,t)$, we get, on $\Gamma(W_i\cap V_t,T_W|_{V_t})$,
\begin{align}
[\frac{\partial \Lambda_i(w_i,z_i,t)}{\partial t},w_i^\lambda-\varphi_i^\lambda(z_i,t)]|_{V_t}-[\Lambda_i(w_i,z_i,t),\frac{\partial \varphi_i^\lambda(z_i,t)}{\partial t} ]|_{V_t}=\sum_{\mu=1}^r -\frac{\partial \varphi_i^\mu(z_i,t)}{\partial t}T_{i\lambda}^{\mu}(\varphi_i(z_i,t),z_i,t)\notag\\
\iff [\pi_i(w_i,z_i,t), w_{ti}^\lambda]|_{V_t}-[\psi_i^\lambda(z_t,t),\Lambda_i(w_i,z_i,t)]|_{V_t}+\sum_{\mu=1}^r \psi_i^\mu(z_t,t)T_{i\lambda}^\mu(\varphi_i(z_i,t),z_i,t)=0 \label{ss8}
\end{align}
Hence from (\ref{ss7}), (\ref{ss6}), (\ref{ss9}) and (\ref{ss8}), $\left(\{\pi_i(w_i,z_i,t)\},\{\psi_i(z_t,t)\}\right)$ defines an element of $\mathbb{H}^0(W, (\wedge^2 T_W \oplus i_*\mathcal{N}_{V_t/W})^\bullet)$ so that we have a linear map
\begin{align*}
\sigma_t:T_t(M_1)&\to \mathbb{H}^0(W, (\wedge^2 T_W \oplus i_*\mathcal{N}_{V_t/W})^\bullet)\\
\frac{\partial}{\partial t}&\mapsto \frac{\partial(\Lambda_t,V_t)}{\partial t}:=(\{\pi_i(w_i,z_i,t)\},\{\psi_i(z_t,t)\})
\end{align*}
We call $\sigma_t$ the characteristic map.

\begin{example}
Let $[\xi_0,\xi_1,\xi_2]$ be the homogeneous coordinate on $\mathbb{P}^2_\mathbb{C}$. Let $[1,z_1,z_2]=[1,\frac{\xi_1}{\xi_0},\frac{\xi_2}{\xi_1}]$. Then  $\mathcal{V}\subset(\mathbb{P}_\mathbb{C}^2\times \mathbb{C}^2, (z_1+t_1z_2+t_2)\frac{\partial}{\partial z_1}\wedge \frac{\partial}{\partial z_2}$ defined by $\xi_1+t_1\xi_2+t_2\xi_0=0$ is an extended Poisson analytic family of simultaneous deformations of a holomorphic Poisson submanifold $\xi_1=0$ of $(\mathbb{P}_\mathbb{C}^2, \Lambda_0=z_1\frac{\partial}{\partial z_1}\wedge \frac{\partial}{\partial z_2})$ and we have the characteristic map
\begin{align*}
 T_0\mathbb{C}^2&\to \mathbb{H}^0(\mathbb{P}_\mathbb{C}^2,(\wedge^2 T_{\mathbb{P}_\mathbb{C}^2}\oplus i_*\mathcal{O}_{\mathbb{P}_\mathbb{C}^1}(1))^\bullet)\\
 (a,b)&\mapsto a\left(z_2\frac{\partial}{\partial z_1}\wedge \frac{\partial}{\partial z_2}, -z_2\right)+b\left(\frac{\partial}{\partial z_1}\wedge \frac{\partial}{\partial z_2},-1\right)
\end{align*}

\end{example}

\begin{example}
Let $[\xi_0,\xi_1,\xi_2,\xi_3]$ be the homogenous coordinate on $\mathbb{P}^3_\mathbb{C}$. Let $[1,z_1,z_2,z_3]=[1,\frac{\xi_1}{\xi_0},\frac{\xi_2}{\xi_0},\frac{\xi_3}{\xi_0}]$. Then $\mathcal{V}\subset (\mathbb{P}_\mathbb{C}^3\times \mathbb{C}^2,(z_1+t_1)\frac{\partial}{\partial z_1}\wedge \frac{\partial}{\partial z_2})$ defined by $\xi_1+t_1\xi_0=\xi_3+t_2\xi_0=0$ is an extended Poisson analytic family of simultaneous deformations of a holomorphic Poisson submanifold $\xi_1=\xi_3=0$ of $(\mathbb{P}_\mathbb{C}^3,\Lambda_0=z_1\frac{\partial}{\partial z_1}\wedge \frac{\partial}{\partial z_2})$ and we have the characteristic map
\begin{align*}
T_0\mathbb{C}^2 &\to \mathbb{H}^0(\mathbb{P}_\mathbb{C}^2, (\wedge^2 T_{\mathbb{P}_\mathbb{C}^3}\oplus i_*(\mathcal{O}_{\mathbb{P}_\mathbb{C}^1}(1)\oplus \mathcal{O}_{\mathbb{P}_\mathbb{C}^1}(1))^\bullet)\\
(a,b)&\mapsto a(\frac{\partial}{\partial z_1}\wedge\frac{\partial}{\partial z_2},(-1,0))\oplus b(0,(0,-1))
\end{align*}
\end{example}

\begin{example}
We construct an extended Poisson analytic family of simultaneous deformations of holomorphic Poisson submanifolds of a stable elliptic surface. As in \cite{Bar04} $p.202$, let $z(s)$ be an arbitrary holomorphic function on the unit disk $\Delta=\{s\in \mathbb{C}||s|<1\}$ with $\textnormal{Im}\,z(s)>0 $. Let $G=\mathbb{Z}\times \mathbb{Z}$ act on $\mathbb{C}\times \Delta$ by $(m,n)(c,s)=(c+m+nz(s),s)$. The quotient $X=(\mathbb{C}\times \Delta)/(\mathbb{Z}\times \mathbb{Z})$ is a nonsingular surface fibered over $\Delta$ such that $X_s$ is an elliptic curve with period $1,z(s)$. Let $c'=c+m+nz(s), s'=s$. Then we have $\frac{\partial}{\partial c}=\frac{\partial}{\partial c'}$, and $\frac{\partial}{\partial s}=nz'(s)\frac{\partial}{\partial c'}+\frac{\partial}{\partial s'}$ so that we get $\frac{\partial}{\partial c}\wedge \frac{\partial}{\partial s}=\frac{\partial}{\partial c'}\wedge \frac{\partial }{\partial s'}$ and so $\Lambda=(s-t)\frac{\partial}{\partial c}\wedge \frac{\partial}{\partial s}$ is a $G$-invariant bivector field on $X\times \Delta$ which defines a holomorphic Poisson structure $\Lambda_t $ on $X$ for each $t\in \Delta$, and $X_t:s=t$ is a holomorphic Poisson submanifold of $(X,\Lambda_t)$ since the holomorphic Poisson structure $\Lambda_t$ degenerates along $s=t$. Then $\mathcal{V}\subset (X\times \Delta, \Lambda=(s-t)\frac{\partial}{\partial c}\wedge \frac{\partial}{\partial s})$ defined by $s=t$ is an extended Poisson analytic family of simultaneous deformations of the holomorphic Poisson submanifold $X_0:s=0$ of $(X,s\frac{\partial}{\partial c}\wedge \frac{\partial}{\partial s})$, and we have the characteristic map
\begin{align*}
T_0\mathbb{C}&\to \mathbb{H}^0(X, (\wedge^2 T_X\oplus i_*\mathcal{N}_{X_0/X})^\bullet)\\
a&\mapsto (-a\frac{\partial}{\partial c}\wedge \frac{\partial}{\partial s}, a)
\end{align*}
\end{example}

\subsection{Theorem of existence}

\begin{theorem}[theorem of existence]\label{43c}
Let $V$ be a holomorphic Poisson submanifold of a compact holomorphic Poisson manifold $(W,\Lambda_0)$. If $\mathbb{H}^1(W, (\wedge^2 T_W\oplus i_* \mathcal{N}_{W/V})^\bullet ) =0$, then  there exists an extended Poisson analytic family $\mathcal{V}\subset (W\times M_1,\Lambda)$ of compact holomorphic Poisson submanifolds $V_t,t\in M_1,$ of $(W,\Lambda_t)$ such that $V=V_0\subset (W,\Lambda_0)$ and the characteristic map 
\begin{align*}
\sigma_0:T_0(M_1)&\to \mathbb{H}^0(W, (\wedge^2 T_W\oplus i_*\mathcal{N}_{V/W})^\bullet)\\
\frac{\partial}{\partial t}&\mapsto \left(\frac{\partial (\Lambda_t,V_t)}{\partial t}\right)_{t=0}
\end{align*}
is an isomorphism.
\end{theorem}

\begin{proof}
We extend the argument in the proof of Theorem \ref{33a} in the context of simultaneous deformations. We keep the notations in subsection \ref{43a}.

Let $\{\eta_1,...,\eta_\rho,...,\eta_l\}$ be a basis of $\mathbb{H}^0(W,(\wedge^2 T_W\oplus i_*\mathcal{N}_{W/V})^\bullet)$. On each neighborhood $W_i$ (we recall $U_i=W_i\cap V$), $\eta_\rho$ is represented as
\begin{align*}
\eta_{\rho i}=(\lambda_i^\rho(w_i,z_i))\oplus(\gamma_{\rho i}^1(z_i),...,\gamma_{\rho i}^r(z_i))\in \Gamma(W_i,\wedge^2 T_W)\oplus (\oplus^r \Gamma(U_i,\mathcal{O}_V))
\end{align*}
such that
\begin{align}
&\lambda_i^\rho(w_i,z_i)=\lambda_j^\rho (w_j,z_j)\label{43b}\\
&-[\lambda_i^\rho(w_i,z_i),\Lambda_0]=0\label{43n}\\
&\gamma_{\rho i}(z)=F_{ik}(z)\cdot\gamma_{\rho k}(z),\,\,\,z\in U_i\cap U_k \label{43g}\\
& [\lambda_i^\rho(w_i,z_i),w_i^\alpha]|_{w_i=0}-[\Lambda_0, \gamma_{\rho i}^\alpha(z_i)]|_{w_i=0}+\sum_{\beta=1}^r \gamma_{\rho i}^\beta T_{i\alpha}^\beta(0,z_i)=0,\,\,\,z_i\in U_i,\,\,\,\alpha=1,...,r.\label{43h}
\end{align}
We note that (\ref{43b}) implies $\lambda^\rho:=\{\lambda_i^{\rho}(w_i,z_i)\}\in H^0(W,\wedge^2 T_W)$.

Let $\epsilon$ be a small positive number. In order to prove Theorem \ref{43c}, it suffices to construct vector-valued holomorphic functions $\varphi_i(z_i,t)=(\varphi_i^1(z_i,t),...,\varphi_i^r(z_i,t))$ in $z_i$ and $t$ with $|z_i|<1,|t|<\epsilon$ with $|\varphi_i(z_i,t)|<1$, and $\Lambda(t)$ which is a convergent power series in $t$ with coefficients in $\in H^0(W,\wedge^2 T_W)$ satisfying the boundary condition
\begin{align*}
\varphi_i(z_i,0)&=0,\\
\frac{\partial \varphi_i(z_i,t)}{\partial t_\rho}|_{t=0}&=\gamma_{\rho i}(z)\\
\Lambda(0)&=\Lambda_0\\
\frac{\partial \Lambda(t)}{\partial t_\rho}|_{t=0}&=\lambda^\rho
\end{align*}
such that
\begin{align}
\varphi_i(g_{ik}(\varphi_k(z_k,t),z_k),t)&=f_{ik}(\varphi(z_k,t),z_k),\,\,\,(\varphi_k(z_k,t),z_k)\in W_k\cap W_i, \label{pa1}\\
[\Lambda(t),&w_i^\alpha-\varphi_i^\alpha(z_i,t)]|_{w_i=\varphi_i(z_i,t)}=0, \,\,\,\alpha=1,...,r\label{pa2}\\
[\Lambda(t)&,\Lambda(t)]=0\label{pa3}
\end{align}
Recall Notation \ref{notation1}. Then the equalities $(\ref{pa1}),(\ref{pa2})$, and $(\ref{pa3})$ are equivalent to the system of congruences 
\begin{align}
\varphi_i^m(g_{ik}(\varphi_k^m&(z_k,t),z_k),t)\equiv_m f_{ik}(\varphi_k^m(z_k,t),z_k),\,\,\,\,\,m=1,2,3,\cdots\label{43d}\\
[\Lambda^m(t),&w_i^\alpha-\varphi_i^{\alpha m}(z_i,t)]|_{w_i=\varphi_i^m(z_i,t)}\equiv_m 0,\,\,\,\,\,\,\,m=1,2,3,\cdots,\,\,\,\alpha=1,...,r,\label{43e}\\
[\Lambda^m&(t),\Lambda^m(t)]\equiv_m 0,\,\,\,\,\,\,\,\,\,\,\,\,\,\,\,\,\,\,\,\,\,\,\,\,\,\,\,\,\,\,\,\,\,\,\,\,\,\,\,\,\,\,\,\,\,\,\,\,\,\,\,\,\,m=1,2,3,\cdots.\label{43f}
\end{align}

We will construct the formal power series $\varphi_i^m(z_i,t)$ and $\Lambda^m(t)$ satisfying $(\ref{43d})_m,(\ref{43e})_m$, and $(\ref{43f})_m$ by induction on $m$.

We define $\varphi_{i|1}^\alpha(z_i,t)=\sum_{\rho=1}^l t_\rho \gamma_{\rho i}^\alpha(z_i)$, and $\Lambda_{1}(t)=\sum_{\rho=1}^l t_\rho \lambda^\rho$. Then from $(\ref{43g})$, $\varphi_{i|1}(z_i,t)$ satisfies $(\ref{43d})_1$. On the other hand, from $(\ref{43h})$, $[\Lambda_1(t),w_i^\alpha]|_{w_i=0}-[\Lambda_0,\varphi_{i|1}^\alpha(z_i,t)]|_{w_i=0}+\sum_{\beta=1}^r \varphi_{i|1}^\beta(z_i,t)T_{i\alpha}^\beta(0,z_i)=0$.
Then we get, from $(\ref{3d})$, 
\begin{align*}
[\Lambda_0+\Lambda_1(t),w_i^\alpha-\varphi_{i|1}^\alpha(z_i,t)]=\sum_{\beta=1}^r (w_i^\beta-\varphi_{i|1}^\beta(z_i,t))T_{i\alpha}^\beta(w_i,z_i)-[\Lambda_1(t),\varphi_{i|1}^\alpha(z_i,t)]+\sum_{\beta=1}^r w_i^\beta P_{i\alpha}^\beta(w_i,z_i,t)
\end{align*}
for some $P_{i\alpha}^\beta(w_i,z_i,t)$ which are homogenous polynomials of degree $1$ in $t_1,...,t_l$ with coefficients in $\Gamma(U_i,T_W)$ so that we obtain $[\Lambda^1(t),w_i^\alpha-\varphi_{i|1}^{\alpha}(z_i,t)]|_{w_i=\varphi_{i|1}(z_i,t)}\equiv_1 0$, which implies $(\ref{43e})_1$. Lastly from $(\ref{43n})$,  we have $-[\Lambda_1(t),\Lambda_0]=0$ so that $[\Lambda_0+\Lambda_1(t),\Lambda_0+\Lambda_1(t)]\equiv_1 2[\Lambda_0, \Lambda_1(t)]=0$, which implies $(\ref{43f})_1$. Hence the induction holds for $m=1$.

Now we assume that we have already constructed $\varphi_i^m(z_i,t)=(\varphi_i^{1m}(z_i,t),\cdots, \varphi_i^{\alpha m}(z_i,t),\cdots,\varphi_i^{rm}(z_i,t))$ satisfying $(\ref{43d})_m,(\ref{43e})_m$ and $(\ref{43f})_m$ such that for $\alpha=1,...r$,
\begin{align}\label{43o}
[\Lambda^m(t),w_i^\alpha-\varphi_i^{\alpha m}(z_i,t)]=\sum_{\beta=1}^r(w_i^\beta-\varphi_i^{\beta m}(z_i,t))T_{i\alpha}^\beta(w_i,z_i)+Q_i^{\alpha m}(z_i,t)+\sum_{\beta=1}^r w_i^\beta P_{i\alpha}^{\beta m}(w_i,z_i,t)
\end{align}
such that the degree of $P_{i\alpha}^{\beta m}(w_i,z_i,t)$ is at least $1$ in $t_1,...,t_l$. We note that we can rewrite (\ref{43o})  in the following way.
\begin{align*}
\sum_{\beta=1}^r(w_i^\beta-\varphi_i^{\beta m}(z_i,t))T_{i\alpha}^\beta(w_i,z_i)+Q_i^{\alpha m}(z_i,t)+\sum_{\beta=1}^r \phi_i^{\beta m}(z_i,t)P_{i\alpha}^{\beta m}(\varphi_i^{ m}(z_i,t),z_i,t)+\sum_{\beta=1}^r (w_i^\beta-\varphi_i^{\beta m}(z_i,t) )L_{i\alpha}^\beta(w_i,z_i,t)
\end{align*}
such that the degree of $L_{i\alpha}^\beta(w_i,z_i,t)$ in $t_1,...,t_l$ is at least $1$ so that (\ref{43o}) becomes the following form:
\begin{align}\label{43s}
[\Lambda^m(t), w_i^\alpha-\varphi_i^{\alpha m}(z_i,t)]=\sum_{\beta=1}^r(w_i^\beta-\varphi_i^{\beta m}(z_i,t))T_{i\alpha}^\beta(w_i,z_i)+K_i^{\alpha m}(z_i,t)+\sum_{\beta=1}^r (w_i^\beta-\varphi_i^{\beta m}(z_i,t) )L_{i\alpha}^\beta(w_i,z_i,t)
\end{align}
where $K_i^{\alpha m}(z_i,t):=Q_i^{\alpha m}(z_i,t)+\sum_{\beta=1}^r \phi_i^{\beta m}(z_i,t)P_{i\alpha}^{\beta m}(\varphi_i^{ m}(z_i,t),z_i,t)$. 

We set
\begin{align}
\psi_{ik}(z_k,t)&:=[\varphi_i^m(g_{ik}(\varphi_k^m(z_k,t),z_k),t)-f_{ik}(\varphi_k^m(z_k,t),z_k)]_{m+1}\label{43p}\\
G_i^\alpha(z_i,t)&:=[[\Lambda^m(t),w_i^\alpha-\varphi_i^{\alpha m}(z_i,t)]|_{w_i=\varphi_i^m(z_i,t)}]_{m+1},\,\,\,\,\,\alpha=1,...,r.\label{43q}\\
\Pi(t)&:=[[\Lambda^m(t),\Lambda^m(t)]]_{m+1}\label{433a}
\end{align}

We claim that$(\{(0)\}, (\psi_{ik}^1(z_k,t),...,\psi_{ik}^r(z_k,t))\})\oplus (-\frac{1}{2}\Pi(t),\{(G_i^1(z_i,t),...,G_i^r(z_i,t))\})$ defines a $1$-cocycle in the following \v{C}ech resolution of $(\wedge^2 T_W\oplus i_*\mathcal{N}_{V/W})^\bullet$\

\begin{center}
$\begin{CD}
C^0(\mathcal{U},\wedge^4 T_W\oplus i_*(\mathcal{N}_{W/V}\otimes\wedge^2 T_W|_V))\\
@A\tilde{\nabla}AA\\
C^0(\mathcal{U},\wedge^3 T_W\oplus i_*(\mathcal{N}_{W/V}\otimes T_W|_V))@>\delta>> C^1(\mathcal{U},\wedge^4 T_W\oplus i_*(\mathcal{N}_{W/V}\otimes T_W|_V))\\
@A\tilde{\nabla}AA @A\tilde{\nabla}AA\\
C^0(\mathcal{U},\wedge^2 T_W\oplus i_*\mathcal{N}_{W/V})@>-\delta>>C^1(\mathcal{U},\wedge^2 T_W\oplus i_*\mathcal{N}_{W/V})@>\delta>> C^2(\mathcal{U},\wedge^4 T_W\oplus i_*\mathcal{N}_{W/V})
\end{CD}$
\end{center}\

By defining $\psi_{ik}(z,t)=\psi_{ik}(z_k,t)$ for $(0,z_k)\in U_k\cap U_i$, we have the equality (see \cite{Kod62} p.152-153)
\begin{align}\label{43z}
\psi_{ik}(z,t)=\psi_{ij}(z,t)+F_{ij}\cdot \psi_{jk}(z,t),\,\,\,\,\,\text{for $z\in U_i\cap U_j\cap U_k$}
\end{align}

On the other hand, by applying $[\Lambda^m(t),-]$ on $(\ref{43s})$, we have
\begin{align}\label{43t}
&[\frac{1}{2}[\Lambda^m(t),\Lambda^m(t)],w_i^\alpha-\varphi_i^{\alpha m}(z_i,t)]=\sum_{\beta=1}^r-[\Lambda^m(t),w_i^{\beta}-\varphi_i^{\beta m}(z_i,t)]\wedge T_{i\alpha}^\beta(w_i,z_i)+\sum_{\beta=1}^r(w_i^\beta-\varphi_i^{\beta m}(z_i,t))[\Lambda^m(t),T_{i\alpha}^\beta(w_i,z_i)]\\
&+[\Lambda^m(t),K_i^{\alpha m} (z_i,t)]+\sum_{\beta=1}^r-[\Lambda^m(t),w_i^\beta-\varphi_i^{\beta m}(z_i,t)]\wedge L_{i\alpha}^\beta(w_i,z_i,t)+\sum_{\beta=1}^r (w_i^\beta-\varphi_i^{\beta m}(z_i,t))[\Lambda^m(t),L_{i\alpha}^\beta(w_i,z_i,t)]\notag
\end{align}
By restricting (\ref{43t}) to $w_i=\varphi_i^m(z_i,t)$, since $G_i^\alpha(z_i,t)\equiv_m 0,\alpha=1,...,r$, $\Pi(t)\equiv_m 0$, the degree $L_{i\alpha}^\beta(w_i,z_i,t)$ is at least $1$ in $t_1,...,t_l$ and we have, from $(\ref{43s})$,
\begin{align}\label{iii100}
G_i^\alpha(z_i,t)\equiv_{m+1}K_i^{\alpha m}(z_i,t)=Q_i^{\alpha m}(z_i,t)+\sum_{\beta=1}^r \phi_i^{\beta m}(z_i,t)P_{i\alpha}^{\beta m}(\varphi_i^{ m}(z_i,t),z_i,t),
\end{align}
 we obtain
\begin{align}\label{43u}
[\frac{1}{2}\Pi(t),w_i^\alpha]|_{w_i=0}=\sum_{\beta=1}^r -G_i^\beta(z_i,t)\wedge T_{i\alpha}^\beta(0,z_i)+[\Lambda_0, G_i^\alpha(z_i,t)]|_{w_i=0}
\end{align}

Next, since $f_{ik}^\alpha(w_k,z_k)-\varphi_i^{\alpha m}(g_{ik}(w_k,z_k),t)-(f_{ik}^\alpha(\varphi_k^m(z_k,t),z_k)-\varphi_i^{\alpha m}(g_{ik}(\varphi_k^m(z_k,t),z_k),t)=\sum_{\beta=1}^r(w_k^\beta-\varphi_k^{\beta m}(z_k,t))\cdot S_{k\alpha}^\beta(w_k,z_k,t)$ for some $S_{k\alpha}^\beta(w_k,z_k,t)$. By setting $t=0$, we get $f_{ik}^\alpha(w_k,z_k)-f_{ik}^\alpha(0,z_k)=\sum_{\beta=1}^r w_k^\beta\cdot S_{k\alpha}^\beta(w_k,z_k,0)$, and then by taking the derivative with respect to $w_k^\gamma$ ans setting $w_k=0$, we obtain  $\frac{\partial f_{ik}^\alpha(w_k,z_k)}{\partial w_k^\gamma}|_{w_i=0}=S_{k\alpha}^\gamma(0,z_k,0)$. Then we have
\begin{align*}
&[\Lambda^m(t),f_{ik}^\alpha(w_k,z_k)-\varphi_i^{\alpha m}(g_{ik}(w_k,z_k),t)]|_{w_k=\varphi_k^m(z_k,t)}+[\Lambda_0, \psi_{ik}^\alpha(z_k,t)]|_{w_k=0}\\
&\equiv_{m+1}[\Lambda^m(t),f_{ik}^\alpha(w_k,z_k)-\varphi_i^{\alpha m}(g_{ik}(w_k,z_k),t)]|_{w_k=\varphi_k^m(z_k,t)}+[\Lambda^m(t), \psi_{ik}^\alpha(z_k,t)]|_{w_k=\varphi_k^m(z_k,t)}\\
&\equiv_{m+1}[\Lambda^m(t), f_{ik}^\alpha(w_k,z_k)-\varphi_i^{\alpha m}(g_{ik}(w_k,z_k),t)+\psi_{ik}^\alpha(z_k,t)]|_{w_k=\varphi_k^m(z_k,t)}\\
&\equiv_{m+1}[\Lambda^m(t) ,\sum_{\beta=1}^r (w_k^\beta-\varphi_k^{\beta m}(z_k,t))S_{k\alpha}^\beta(w_k,z_k,t)]|_{w_k=\varphi_k^m(z_k,t)}\\
&\equiv_{m+1}\sum_{\beta=1}^r[\Lambda^m(t), w_k^\beta-\varphi_k^{\beta m}(z_k,t)]|_{w_k=\varphi_k^m(z_k,t)}\cdot S_{k\alpha}^\beta(\varphi_k^m(z_k,t),z_k,t)\\
&\equiv_{m+1}\sum_{\beta=1}^r G_k^\beta(z_k,t)\cdot S_{k\alpha}^\beta(0,z_k,0)\equiv_{m+1}\sum_{\beta=1}^r G_k^\beta(z_k,t)\cdot \frac{\partial f_{ik}^\alpha(w_k,z_k)}{\partial w_k^\beta}|_{w_k=0}
\end{align*}
Hence we obtain the equality
\begin{align}\label{43x}
[\Lambda^m(t),f_{ik}^\alpha(w_k,z_k)-\varphi_i^{\alpha m}(g_{ik}(w_k,z_k),t)]|_{w_k=\varphi_k^m(z_k,t)}+[\Lambda_0, \psi_{ik}^\alpha(z_k,t)]|_{w_k=0}=\sum_{\beta=1}^r G_k^\beta(z_k,t)\cdot \frac{\partial f_{ik}^\alpha(w_k,z_k)}{\partial w_k^\beta}|_{w_k=0}
\end{align}
On the other hand, from $(\ref{43s})$, we have 
\begin{align}\label{43v}
[\Lambda^m(t),f_{ik}^\alpha(w_k,z_k) -\varphi_i^{\alpha m}(g_{ik}(w_k,z_k),t)]=\sum_{\beta=1}^r(f_{ik}^\beta(w_k,z_k)-\varphi_i^{\beta m}(g_{ik}(w_k,z_k),t))T_{i\alpha}^\beta(w_i,z_i)\\+K_i^{\alpha m}(z_i,t)+\sum_{\beta=1}^r (f_{ik}^\beta(w_k,z_k)-\varphi_i^{\beta m}(g_{ik}(w_k,z_k),t) )L_{i\alpha}^\beta(w_i,z_i,t)\notag
\end{align}
By restricting (\ref{43v}) to $w_k=\varphi_k^m(z_i,t)$, we get
\begin{align}\label{43w}
[\Lambda^m(t),f_{ik}^\alpha(w_k,z_k)-\varphi_i^{\alpha m}(g_{ik}(w_k,z_k),t)]|_{w_k=\varphi_k^m(z_k,t)}\equiv_{m+1} \sum_{\beta=1}^r -\psi_{ik}^\beta(z_k,t)T_{i\alpha}^\beta(0,z_i)+G_i^\alpha(z_i,t)
\end{align}
Hence from $(\ref{43x})$ and $(\ref{43w})$, we get
\begin{align*}
\sum_{\beta=1}^r -\psi_{ik}^\beta(z_k,t)T_{i\alpha}^\beta(0,z_i)+G_i^\alpha(z_i,t)+[\Lambda_0, \psi_{ik}^\alpha(z_k,t)]|_{w_k=0}=\sum_{\beta=1}^r G_k^\beta(z_k,t)\cdot \frac{\partial f_{ik}^\alpha(w_k,z_k)}{\partial w_k^\beta}|_{w_k=0}
\end{align*}

Lastly, $[\Lambda_0,\Pi(t)]\equiv_{m+1}[\Lambda_0, [\Lambda^m(t),\Lambda^m(t)]]\equiv_{m+1}[\Lambda^m(t),[\Lambda^m(t),\Lambda^m(t)]]=0$ so that we get
\begin{align}\label{43y}
-[-\frac{1}{2}\Pi(t),\Lambda_0]=0
\end{align}
Hence from $(\ref{43z}),(\ref{43u}),(\ref{43x})$ and $(\ref{43y})$,
\begin{align*}
(\{(0)\}, (\psi_{ik}^1(z_k,t),...,\psi_{ik}^r(z_k,t))\})\oplus (-\frac{1}{2}\Pi(t),\{(G_i^1(z_i,t),...,G_i^r(z_i,t)\}))
\end{align*}
defines a $1$-cocycle in the above complex so that we get the claim. We call $\psi_{m+1}(t):=\{(\psi_{ik}^1(z_k,t),...,\psi_{ik}^r(z_k,t))\}$, $G_{m+1}(t):=\{(G_i^1(z_i,t),...,G_i^r(z_i,t))\}$ and $\Pi_{m+1}(t):=-\frac{1}{2}\Pi(t)$ the $m$-th obstruction so that the coefficients of $(0,\psi_{m+1}(t))\oplus(\frac{1}{2}\Pi_{m+1}(t),G_{m+1}(t))$ in $t_1,...,t_l$ lies in $\mathbb{H}^1(W,(\wedge^2 T_W\oplus i_*\mathcal{N}_{V/W})^\bullet)$.

On the other hand, by hypothesis, the cohomology group $\mathbb{H}^1(W,(\wedge^2 T_W\oplus i_*\mathcal{N}_{V/W})^\bullet)$ vanishes. Therefore there exists $\Lambda_{i|m+1}(t),\varphi_{i|m+1}^\alpha(z_i,t),\alpha=1,...,r$ such that $\psi_{ik}(z,t)=F_{ik}(z)\varphi_{k|m+1}(z,t)-\varphi_{i|m+1}(z,t)$, $[\Lambda_{m+1}(t),w_i^\alpha]|_{w_i=0}+\sum_{\beta=1}^r \varphi^\beta_{i|m+1}(z_i,t)T_{i\alpha}^\beta(0,z_i)-[\Lambda_0,\varphi_{i|m+1}^\alpha(z_i,t)]|_{w_i=0}=-G_i^\alpha(z_i,t)$ and $-[\Lambda_{m+1}(t),\Lambda_0]=\frac{1}{2}\Pi(t)$. Then we can show $(\ref{43d})_{m+1}$ (for the detail, see \cite{Kod62} p.154). On the other hand,
\begin{align}\label{42b}
[\Lambda_{m+1}(t),w_i^\alpha]-[\Lambda_0,\varphi_{i|m+1}^\alpha(z_i,t)]=-\sum_{\beta=1}^r \varphi^\beta_{i|m+1}(z_i,t) T_{i\alpha}^\beta(w_i,z_i)-G_i^\alpha(z_i,t)+\sum_{\beta=1}^r w_i^\beta R_{i\alpha}^\beta(w_i,z_i,t)
\end{align}
where the degree of $R_{i\alpha}^\beta(w_i,z_i,t)$ is $m+1$ in $t$. Let $[\Lambda^m(t)-\Lambda_0,-\varphi_{i|m+1}(z_i,t)]+[\Lambda_{|m+1}(t), -\varphi_i^{\alpha m}(z_i,t)-\varphi_{i|m+1}(z_i,t)]=H_i^{\alpha m}(z_i,t)+\sum_{\beta=1}^r w_i^\beta M_{i\alpha}^\beta(w_i,z_i,t)$, where the degree of $H_i^{\alpha m}(z_i,t)$ and $M_{i\alpha}^\beta(w_i,z_i,t)$ is at least $m+2$. Then from (\ref{43o}) and (\ref{42b}), we have
\begin{align}\label{mmmn1}
&[\Lambda^m(t)+\Lambda_{m+1}(t),w_i^\alpha-\varphi_i^{\alpha m}(z_i,t)-\varphi^\alpha_{i|m+1}(z_i,t)]\\
&=[\Lambda^m(t),w_i^\alpha-\varphi_i^{\alpha m}(z_i,t)]+[\Lambda^m(t)-\Lambda_0,-\varphi_{i|m+1}^\alpha(z_i,t)]+[\Lambda_0, -\varphi_{i|m+1}^\alpha(z_i,t)]\notag\\
&+[\Lambda_{m+1}(t),w_i^\alpha]+[\Lambda_{m+1}(t),-\varphi_i^{\alpha m}(z_i,t)-\varphi_{i|m+1}(z_i,t)]\notag\\
&=\sum_{\beta=1}^r(w_i^\beta-\varphi_i^{\beta m}-\varphi^\beta_{i|m+1})T_{i\alpha}^\beta(w_i,z_i)+Q_i^{\alpha m}(z_i,t)-G_i^\alpha(z_i,t)+H_i^{\alpha m}(z_i,t)\notag\\
&+\sum_{\beta=1}^rw_i^\alpha( P_{i\alpha}^{\beta m}(w_i,z_i,t)+ R_{i\alpha}^\beta(w_i,z_i,t)+M_{i\alpha}^\beta(w_i,z_i,t))\notag
\end{align}
We show that $[\Lambda^{m+1}(t), w_i^\alpha-\varphi_i^{\alpha m+1}]|_{w_i=\varphi_i^{\alpha m+1}}\equiv_{m+1} 0$. Indeed, by restricting $(\ref{mmmn1})$ to $w_i=\varphi_i^{m+1}(z_i,t)$, we get, from $(\ref{iii100})$.
\begin{align*}
&[\Lambda^{m+1}(t), w_i^\alpha-\varphi_i^{\alpha (m+1)}(z_i,t)]|_{w_i=\varphi_i^{ m+1}(z_i,t)}\\
&\equiv_{m+1} Q_i^{\alpha m}(z_i,t)-G_i^\alpha(z_i,t)+\sum_{\beta=1}^r (\varphi_i^{\alpha m}(z_i,t)+\varphi^\alpha_{i|m+1}(z_i,t))P_{i\alpha}^{\beta m}(\varphi_i^{ m}(z_i,t)+\varphi_{i|m+1}(z_i,t),z_i,t)\\
&\equiv_{m+1} Q_i^{\alpha m}(z_i,t)-G_i^\alpha(z_i,t)+\sum_{\beta=1}^r \varphi_i^{\alpha m}(z_i,t)P_{i\alpha}^{\beta m}(\varphi_i^{ m}(z_i,t),z_i,t) \equiv_{m+1} K_i^{\alpha m}(z_i,t)-G_i^\alpha(z_i,t)\equiv_{m+1}0
\end{align*}
which shows $(\ref{43e})_{m+1}$. Lastly $[\Lambda^m(t)+\Lambda_{m+1}(t),\Lambda^m(t)+\Lambda_{m+1}(t)]\equiv_{m+1} [\Lambda^m(t),\Lambda^m(t)]+2[\Lambda_0, \Lambda_{m+1}(t)]\equiv_{m+1} [\Lambda^m(t),\Lambda^m(t)]-\Pi(t)\equiv_{m+1} 0$ which shows $(\ref{43f})_{m+1}$. This completes the inductive constructions of $\varphi_i^m(z_i,t),i\in I$, and $\Lambda^m(t)$.

\subsection{Proof of convergence}\

We will show that we can choose $\varphi_{i|m}(z_i,t)$ and $\Lambda_m(t)$ in each inductive step so that the formal power series $\varphi_i(z_i,t),i\in I$ and $\Lambda(t)$ constructed in the previous subsection, converges absolutely for $|t|<\epsilon$ for a sufficiently small number $\epsilon>0$.

We keep the notations in subsection \ref{3convergence}. For instance, $\Lambda_0=\Lambda_{i}(w_i,z_i)=\sum_{p,q=1}^{r+d}\Lambda_{pq}^i(w_i,z_i)\frac{\partial}{\partial x_i^p}\wedge\frac{\partial}{\partial x_i^q}$ with $\Lambda_{pq}^i(w_i,z_i)=-\Lambda_{pq}^i(w_i,z_i)$, where $x_i=(w_i,z_i)$ on $W_i$, and $W_i^\delta$ is the subdomain of $W_i$ consisting of all points $(w_i,z_i)$, $|w_i|<1-\delta, |z_i|<1-\delta$ for a sufficiently small number $\delta>0$ such that $\{W_i^\delta|i\in I\}$ forms a covering of $W$, and $\{U_i^\delta=W_i^\delta\cap V|i\in I\}$ forms a covering of $V$. Recall Notation \ref{notation3}. We denote $\Lambda^m(t)$ on $W_i$ by $\Lambda_i^m(w_i,z_i,t)$ and $\Pi^m(t)$ on $W_i$ by $\Pi_i^m(w_i,z_i,t)$. Then $\Lambda_i^m(w_i,z_i,t)$ is of the form
\begin{align}
\Lambda_i^m(w_i,z_i,t)=\sum_{p,q=1}^{r+d}\Lambda_{pq}^{im}(w_i,z_i,t)\frac{\partial}{\partial x_i^p}\wedge \frac{\partial}{\partial x_i^q},\,\,\,\,\,\Lambda_{pq}^{im}(w_i,z_i,t)=-\Lambda_{qp}^{im}(w_i,z_i,t)
\end{align}
such that $\Lambda_{pq}^{i0}(w_i,z_i,t)=\Lambda_{pq}^i(w_i,z_i)$. We may assume that $|F_{ik\mu}^\lambda(0,z)|<c_0$ with $c_0>1$. Then $\varphi_{i|1}(z_i,t)\ll A(t)$ and $\Lambda_{i|1}(w_i,z_i,t)\ll A(t)$ if $b$ is sufficiently large. Now, assuming the inequalities
\begin{align}\label{hh2}
\varphi_i^m(z_i,t)\ll A(t),\,\,\,\,\,\,\,\Lambda_i^m(w_i,z_i,t)-\Lambda_i(w_i,z_i)\ll A(t),\,\,\,(w_i,z_i)\in W_i^\delta
\end{align}
for an integer $m\geq 1$, we will estimate the coefficients of the homogenous polynomials $\psi_{ik}(z,t), \Pi_i(w_i,z_i,t)$, and $G_i^\alpha(z_i,t)$ from $(\ref{43p}),(\ref{43q})$, and $(\ref{433a})$.

First we estimate $\psi_{ik}(z,t)$. As in (\ref{3.6a}), we have
\begin{align}\label{zz1}
\psi_{ik}(z_k,t)\ll c_3A(t),\,\,\,z_k \in U_k\cap U_i, 
\end{align}
where $c_3=2rc_0\frac{4c_1ra}{b}\left(\frac{2^d}{\delta}+rc_1\right)$
\begin{align}\label{zz5}
b>\max\{2c_1ra, \frac{4c_1 ra}{\delta}\}.
\end{align}

Next we estimate $G_i^\alpha(z_i,t)$. We note that 
\begin{align}\label{hh1}
G_i^\alpha(z_i,t)&\equiv_{m+1} [\Lambda_i^m(w_i,z_i,t),w_i^\alpha-\varphi_i^{\alpha m}(z_i,t)]|_{w_i=\varphi_i^m(z_i,t)}\\
&\equiv_{m+1}[\Lambda_i^m(w_i,z_i,t)-\Lambda_i(w_i,z_i),w_i^\alpha]|_{w_i=\varphi_i^m(z_i,t)}-[\Lambda_i^m(w_i,z_i,t)-\Lambda_i(w_i,z_i),\varphi_i^{\alpha m}(z_i,t)]|_{w_i=\varphi_i^m(z_i,t)}\notag\\
&\,\,\,\,\,\,\,\,\,\,\,\,\,\,\,\,\,+[\Lambda_i(w_i,z_i),w_i^\alpha-\varphi_i^{\alpha m}(z_i,t)]|_{w_i=\varphi_i^{ m}(z_i,t)}\notag
\end{align}

We estimate each term in $(\ref{hh1})$. First we estimate $[[\Lambda_i^m(w_i,z_i,t)-\Lambda_i(w_i,z_i),w_i^\alpha]|_{w_i=\varphi_i^m(z_i,t)}]_{m+1}$ in $(\ref{hh1})$. We note that
\begin{align}\label{hh20}
[\sum_{p,q=1}^{d+r}(\Lambda_{pq}^{im}(w_i,z_i,t)-\Lambda^i_{pq}(w_i,z_i))\frac{\partial}{\partial x_i^p}\wedge \frac{\partial}{\partial x_i^q},w_i^\alpha]|_{w_i=\varphi_i^m(z_i,t)}=\sum_{p,q=1}^{d+r}2(\Lambda_{pq}^{im}(\varphi_i^m(z_i,t),z_i,t)-\Lambda_{pq}^i(\varphi_i^m(z_i,t),z_i))\frac{\partial w_i^\alpha}{\partial x_i^p}\frac{\partial}{\partial x_i^q}
\end{align}

On the other hand, $\Phi_{pq}^{i m}(w_i,z_i,t):=\Lambda_{pq}^{im}(w_i,z_i,t)-\Lambda_{pq}^i(w_i,z_i)\ll A(t)$ from $(\ref{hh2})$ and has the degree $\leq m$ in $t$. For any $p,q$, we expand $\Phi_{pq}^{im}(w_i,z_i,t)$ into power series in $w_i^1,...w_i^r$ whose coefficients are holomorphic functions of $z=(0,z_i)$ defined on $U_i$:
\begin{align}\label{yy31}
\Phi_{pq}^{im}(w_i,z_i,t)=\sum_{\mu_1,...,\mu_r\geq 0} \Phi_{pq,\mu_1,...,\mu_r}^{im}(z_i,t)w_i^{1\mu_1}\cdots w_i^{r\mu_r}
\end{align}

If $(w_i,z_i)\in W_i^\delta$, we have, by Cauchy's integral formula,
\begin{align*}
\Phi_{pq,\mu_1,...,\mu_r}^{im}(z_i,t)=\left(\frac{1}{2\pi i}\right)^r\int_{|\xi_1-w_i^1|=\delta}\cdots\int_{|\xi_r-w_i^r|=\delta}\frac{\Phi_{pq}^{im}(w_i,z_i,t)}{(\xi_1-w_i^1)^{\mu_1+1}\cdots (\xi_r-w_i^r)^{\mu_r+1}}d\xi_1\cdots d \xi_r
\end{align*}
so that we get
\begin{align}\label{yy30}
\Phi_{pq,\mu_1,...,\mu_r}^{im}(z_i,t)\ll A(t)\frac{1}{\delta^{\mu_1+\cdots+\mu_r}}
\end{align}

Since constant terms of $\Phi_{pq}^{im}(w_i,z_i,t)$ with respect to $w_i^1,...,w_i^r$ does not contribute to $[[\Lambda_i^m(w_i,z_i,t)-\Lambda_i(w_i,z_i),w_i^\alpha]|_{w_i=\varphi_i^m(z_i,t)}]_{m+1}$, from $(\ref{yy31})$ and $(\ref{yy30})$, we have, assuming $\frac{a}{b\delta}<\frac{1}{2}$, (for the detail, see \cite{Kod05} p.300)
\begin{align}\label{hh21}
\Phi_{pq}^{im}(\varphi_i^m(z_i,t),z_i,t)\ll A(t)\sum_{\mu_1+...+\mu_r\geq 1} \left(\frac{A(t)}{\delta}\right)^{\mu_1+\cdots +\mu_r}\ll \frac{2^{r+1}}{\delta}A(t)^2\ll \frac{2^{r+1}a}{b\delta}A(t)
\end{align}

Then from $(\ref{hh20})$ and $(\ref{hh21})$, we obtain
\begin{align}\label{hh22}
[[\Lambda_i^m(w_i,z_i,t)-\Lambda_i(w_i,z_i),w_i^\alpha]|_{w_i=\varphi_i^m(z_i,t)}]_{m+1}\ll2(d+r)^2\frac{2^{r+1}a}{b\delta}A(t)
\end{align}

Next we estimate $[[\Lambda_i^m(w_i,z_i,t)-\Lambda_i(w_i,z_i),\varphi_i^{\alpha m}(z_i,t)]|_{w_i=\varphi_i^m(z_i,t)}]_{m+1}$ in $(\ref{hh1})$. From $(\ref{hh2})$ and $(\ref{rr1})$, we have
\begin{align}
[\sum_{p,q=1}^{d+r}(\Lambda_{pq}^{im}(w_i,z_i,t)-\Lambda_{pq}^i(w_i,z_i))\frac{\partial}{\partial x_i^p}\wedge \frac{\partial}{\partial x_i^q},\varphi_i^{\alpha m}(z_i,t)]|_{w_i=\varphi_i^m(z_i,t)}\\
=\sum_{p,q=1}^{d+r}2(\Lambda_{pq}^{im}(\varphi_i^m(z_i,t),z_i,t)-\Lambda_{pq}^i(\varphi_i^m(z_i,t),z_i))\frac{\partial \varphi_i^{\alpha m}(z_i,t)}{\partial x_i^p}\frac{\partial}{\partial x_i^q}\notag\\
\ll 2(d+r)^2A(t)\frac{A(t)}{\delta}\ll 2(d+r)^2\frac{a}{b\delta}A(t)\notag
\end{align}
Hence we get
\begin{align}\label{bc3}
[[\Lambda_i^m(w_i,z_i,t)-\Lambda_i(w_i,z_i),\varphi_i^{\alpha m}(z_i,t)]|_{w_i=\varphi_i^m(z_i,t)}]_{m+1}\ll 2(d+r)^2\frac{a}{b\delta}A(t)
\end{align}

We estimate $[\Lambda_i(w_i,z_i),w_i^\alpha-\varphi_i^{\alpha m}(z_i,t)]|_{w_i=\varphi_i^m(z_i,t)}$ in $(\ref{hh1})$. This comes from $(\ref{bc1})$:
\begin{align}\label{bc2}
[[\Lambda_i(w_i,z_i),w_i^\alpha-\varphi_i^{\alpha m}(z_i,t)]|_{w_i=\varphi_i^m(z_i,t)}]_{m+1}\ll e_2A(t),\,\,\,z\in U_i^\delta
\end{align}
where $e_2=\frac{4(d+r)^2e_1^2r^2a}{b}+\frac{4(r+d)e_1ra}{\delta b}$ with $b>2e_1ra$.

Hence from $(\ref{hh22}),(\ref{bc3}),(\ref{bc2})$, we get
\begin{align}\label{zz2}
G_i^\alpha(z_i,t)=[[\Lambda_i^m(w_i,z_i,t),w_i^\alpha-\varphi_i^{\alpha m}(z_i,t)]|_{w_i=\varphi_i^m(z_i,t)}]_{m+1}\ll e_3A(t)
\end{align}
where $e_3=\frac{(d+r)^2 2^{r+2}a}{b\delta}+\frac{2(d+r)^2a}{b\delta}+e_2$ with
\begin{align}\label{zz3}
b>\max \{\frac{2a}{\delta}, 2e_1ra \}
\end{align}

Next we estimate $\Pi_i(w_i,z_i,t)=[[\Lambda_i^m(w_i,z_i,t),\Lambda_i^m(w_i,z_i,t)]]_{m+1}$. Since $\Pi_i(w_i,z_i,t)\equiv_m 0$ and
\begin{align*}
\Pi_i(w_i,z_i,t)&\equiv_{m+1} [\Lambda_i^m(w_i,z_i,t),\Lambda_i^m(w_i,z_i,t)]\\
&\equiv_{m+1}[\Lambda_i^m(w_i,z_i,t)-\Lambda_i(w_i,z_i),\Lambda_i^m(w_i,z_i,t)-\Lambda_i(w_i,z_i)]+2[\Lambda_i(w_i,z_i),\Lambda_i^m(w_i,z_i)-\Lambda_i(w_i,z_i)],
\end{align*}
we have
\begin{align}\label{bc10}
\Pi_i(w_i,z_i,t)=[[\Lambda_i^m(w_i,z_i,t)-\Lambda_i(w_i,z_i),\Lambda_i^m(w_i,z_i,t)-\Lambda_i(w_i,z_i)]]_{m+1}
\end{align}

We note the following two remarks.
\begin{remark}\label{4remark1}
Let $\sigma=\sum_{p,q}\sigma^{pq}\frac{\partial}{\partial x_p}\wedge \frac{\partial}{\partial x_q}$ with $\sigma_{pq}=-\sigma_{qp}$ and $\phi=\sum_{l,k}\phi^{lk}\frac{\partial}{\partial x_l}\wedge \frac{\partial}{\partial x_k}$ with $\phi_{lk}=-\phi_{kl}$. Then
$[\sigma,\phi]=\sum_{p,q,l,k} [\sigma^{pq}\frac{\partial}{\partial x_p}\wedge\frac{\partial}{\partial x_q},\phi^{lk}\frac{\partial}{\partial x_l}\wedge \frac{\partial}{\partial x_k}]=\sum_{p,q,l,k} [\sigma^{pq}\frac{\partial}{\partial x_p},\phi^{lk}\frac{\partial}{\partial x_l}]\frac{\partial}{\partial x_q}\wedge \frac{\partial}{\partial x_k}-\phi^{lk}[\sigma^{pq}\frac{\partial}{\partial x_p}, \frac{\partial}{\partial x_k}]\frac{\partial}{\partial x_q}\wedge \frac{\partial}{\partial x_l}-\sigma^{pq}[\frac{\partial}{\partial x_q},\phi^{lk}\frac{\partial}{\partial x_l}]\frac{\partial}{\partial x_p}\wedge \frac{\partial}{\partial x_k}=\sum_{p,q,l,k} \sigma^{pq}\frac{\partial \phi^{lk}}{\partial x_p}\frac{\partial}{\partial x_l}\wedge \frac{\partial}{\partial x_q}\wedge \frac{\partial}{\partial x_k}-\phi^{lk}\frac{\partial \sigma^{pq}}{\partial x_l}\frac{\partial}{\partial x_p}\wedge\frac{\partial}{\partial x_q}\wedge \frac{\partial}{\partial x_k}+\phi^{lk}\frac{\partial \sigma^{pq}}{\partial x_k}\frac{\partial}{\partial x_p}\wedge\frac{\partial}{\partial x_q}\wedge \frac{\partial}{\partial x_l}-\sigma^{pq}\frac{\partial \phi^{lk}}{\partial x_q}\frac{\partial}{\partial x_l}\wedge \frac{\partial}{\partial x_p}\wedge \frac{\partial}{\partial x_k}$
\end{remark}

\begin{remark}\label{4remark2}
\begin{align*}
\frac{\partial (\Lambda_{pq}^{im}(x_i,t)-\Lambda^i_{pq}(x_i))}{\partial x_i^s}=\frac{1}{2\pi i}\int_{|\xi-x_i^s|=\delta} \frac{\Lambda_{pq}^{im}(x_i^1,... \overset{s-th}{\xi},...,x_i^{d+r},t)-\Lambda^i_{pq}(x_i^1,...,\overset{s-th}{\xi},...,x_i^{d+r})}{(\xi-x_i^s)^2}d\xi\ll \frac{A(t)}{\delta},x_i\in W_i^\delta
\end{align*}
\end{remark}

By Remark \ref{4remark1}, Remark \ref{4remark2} and $(\ref{hh2})$, we have
\begin{align}\label{bc11}
&[\Lambda_i^m(w_i,z_i,t)-\Lambda_i(w_i,z_i),\Lambda_i^m(w_i,z_i,t)-\Lambda_i(w_i,z_i)]\\
&=[\sum_{p,q=1}^{d+r}(\Lambda_{pq}^{im}(w_i,z_i,t)-\Lambda_{pq}^i(w_i,z_i))\frac{\partial}{\partial x_i^p}\wedge \frac{\partial}{\partial x_i^q},\sum_{p,q=1}^{d+r}(\Lambda_{pq}^{im}(w_i,z_i,t)-\Lambda_{pq}^i(w_i,z_i))\frac{\partial}{\partial x_i^p}\wedge \frac{\partial}{\partial x_i^q}] \notag\\
&\ll 4(d+r)^4 A(t)\cdot \frac{A(t)}{\delta}\ll  \frac{4(d+r)^4 a}{\delta b}A(t)\notag
\end{align}
Hence from $(\ref{bc10})$ and $(\ref{bc11})$, we obtain
\begin{align}\label{zz4}
\frac{1}{2}\Pi_i(w_i,z_i,t)\ll 4(d+r)^4 A(t)\cdot \frac{A(t)}{\delta}\ll  \frac{2(d+r)^4 a}{\delta b}A(t)=e_4A(t)
\end{align}
where $e_4=\frac{2(d+r)^4 a}{\delta b}$.

\begin{notation}
We consider $P^k=\sum_{\alpha,\beta=1}^{d+r} P_{\alpha\beta}^k(x_k)\frac{\partial}{\partial x_k^\alpha}\wedge \frac{\partial}{\partial x_k^\beta}\in \Gamma(W_i,\wedge^2 T_W)$ to be a vector-valued holomorphic function $P_k(x)=(P_{\alpha\beta}^k(x_k))_{\alpha,\beta=1,...,d+r}$ on $W_k$.  On $W_i\cap W_k$, $P^k$ is translated to $\sum_{\alpha,\beta,p,q=1}^{d+r} P_{\alpha\beta}^k(x_k)\frac{\partial h_{ik}^p}{\partial x_k^\alpha} \frac{\partial h_{ik}^q}{\partial x_k^\beta}\frac{\partial}{\partial x_i^p}\wedge \frac{\partial}{\partial x_i^q}$ which corresponds to a vector-valued holomorphic function $(\sum_{\alpha,\beta=1}^{d+r}P_{\alpha\beta}^k(x_k)\frac{\partial h_{ik}^p}{\partial x_k^\alpha} \frac{\partial h_{ik}^q}{\partial x_k^\beta})_{p,q=1,...,d+r}$ on $W_i$ denoted by $H_{ik}(x)P^k(x)$.
\end{notation}

\begin{lemma}\label{nb1}
We can choose the homogenous polynomials $\varphi_{i|m+1}(z,t),\Lambda_{i|m+1}(w,z,t),i\in I$, satisfying
\begin{align*}
\psi_{ik}(z,t)&=F_{ik}(z)\varphi_{k|m+1}(z,t)-\varphi_{i|m+1}(z,t)\\
-G_i^\alpha(z,t)&=[\Lambda_{i|m+1}(w,z,t),w_i^\alpha]|_{w_i=0}-[\varphi_{i|m+1}^\alpha(z,t),\Lambda_0]|_{w_i=0}+\sum_{\beta=1}^r \varphi_{i|m+1}(z,t)T_{i\alpha}^\beta(0,z)\\
\frac{1}{2}\Pi_i(w_i,z_i,t)&=-[\Lambda_{i|m+1}(w,z,t),\Lambda_0]\\
(\Theta_{ik}(w,z,t)&=H_{ik}(w,z)\Lambda_{k|m+1}(w,z,t)-\Lambda_{i|m+1}(w,z,t)=0)
\end{align*}
in such a way that $\varphi_{i|m+1}(z,t)\ll c_4(c_3+e_3+e_4)A(t)$ and $\Lambda_{i|m+1}(w,z,t)\ll c_4(c_3+e_3+e_4)A(t)$ for $(w,z)\in W_i^\delta$, where $c_4$ is independent of $m$.
\end{lemma}

\begin{proof}
For any $0$-cochain $(\pi,\varphi)=(\{\pi_i\},\{\varphi_i\})$, and  $1$-cochain $(\Theta,\psi)\oplus (\Pi,G)=(\{\Theta_{ik}(w,z)\},\{\psi_{ik}(z)\})\oplus (\{\Pi_i(w,z)\},\{\sum_{\alpha=1}^r G_i^\alpha (z)e_i^\alpha\})$, we define the norms of $(\pi, \varphi)$ and $(\Theta, \psi)\oplus (\Pi, G)$ by
\begin{align*}
||(\pi,\varphi)||&:=\max_i\sup_{x=(w,z)\in W_i} |\pi_i(w,z) |+\max_i\sup_{z\in U_i}|\varphi_i(z)|,\\
||(\Theta,\psi)\oplus (\Pi,G)||&:=\max_{i,k}\sup_{x=(w,z)\in W_i\cap W_k} |\Theta_{ik}(w,z)|+\max_i\sup_{(w,z)\in W_i^\delta} |\Pi_i(w,z)|+\max_{i,k}\sup_{z\in U_i\cap U_k}|\psi_{ik}(z)|+\max_{i,\alpha}\sup_{z\in U_i^\delta}|G_i^\alpha(z)|
\end{align*}

The coboundary of $(\pi,\varphi)$ is defined by
\begin{align*}
\tilde{\delta}(\pi, \varphi):=(\{-H_{ik}(x)&\pi_k(x) +\pi_i(x)\},\{-F_{ik}(z)\varphi_k(z)+\varphi_i(z)\})\\
&\oplus (\{-[\pi_i(x),\Lambda_0]\},\{\oplus_{\alpha=1}^r\left([\pi_i(x),w_i^\alpha]|_{w_i=0}-[\varphi_i^\alpha(z),\Lambda_0]|_{w_i=0}+\sum_{\beta=1}^r \varphi_i^\beta(z)T_{i\alpha}^\beta(0,z)\right)e_i^\alpha\})
\end{align*}

For any coboundary of the form $(0,\psi)\oplus (\Pi, G)$, we define
\begin{align*}
\iota((0, \psi)\oplus (\Pi,G))=\inf_{\tilde{\delta}(\pi,\varphi)=(0,\psi)\oplus (\Pi,G)} ||(\pi,\varphi)||
\end{align*}

To prove Lemma \ref{nb1}, it suffices to show the existence of a constant $c$ such that $\iota((0,\psi)\oplus (\Pi,G))\leq c||(0,\psi)\oplus (\Pi,G)||$. Assume that such a constant does not exist. Then we can find a sequence $(0,\psi^{(\mu)})\oplus(\Pi^{(\mu)},G^{(\mu)})$ such that
\begin{align*}
\iota((0,\psi^{(\mu)})\oplus (\Pi^{(\mu)},G^{(\mu)}))=1,\,\,\,\,\,\,||(0,\psi^{(\mu)})\oplus (\Pi^{(\mu)},G^{(\mu)})||<\frac{1}{\mu}
\end{align*}
Then there exists $(\pi^{(\mu)},\varphi^{(\mu)})$ with $\delta(\pi^{(\mu)},\varphi^{(\mu)})=(0,\psi^{(\mu)})\oplus(\Pi^{(\mu)},G^{(\mu)})$ satisfying $||\pi^{(\mu)}||<2$ and $||\varphi^{(\mu)}||<2$. We note that $\pi^{(\mu)}$ is a global bivector field in $H^0(W,\wedge^2 T_W)$. We take a covering $\{\bar{W}_i^\delta\}$ of $W$ and a covering $\{\bar{U}_i^\delta=\bar{W}_i^\delta\cap U_i\}$ of $V$. Since $|\pi_k^\mu(x)|<2$ for $x\in W_k$ and $|\phi_k^{\mu}(z)|<2$ for $z\in U_k=W_k\cap V$, there exists a subsequence $(\pi^{(\mu_1)},\varphi^{(\mu_1)}),(\pi^{(\mu_2)},\varphi^{(\mu_2)}),\cdots, (\pi^{(\mu_v)},\varphi^{(\mu_v)}),\cdots$ of $(\pi^{(\mu)},\varphi^{(\mu)})$ such that $\pi_k^{(\mu_v)}$ converges absolutely and uniformly on $\bar{W}_k^\delta$ for each $k$, and $\varphi_k^{(\mu_v)}$ converges absolutely and uniformly on $\bar{U}_k^\delta$. Since $W$ is compact, we can choose a subsequence that works for all $k$. On the other hand, since $||(0,\psi^{(\mu)})\oplus(\Pi^{(\mu)},G^{(\mu)})||<\frac{1}{\mu}$, we have
\begin{align}
H_{ik}(x)\pi_k^{(\mu)}(x)=\pi_i^{(\mu)}(x),\,\,\,\,x\in W_i\cap W_k,\,\,\,\,\,|F_{ik}(z)\varphi_k^{(\mu)}(z)-\varphi_i^{(\mu)}(z)|<\frac{1}{\mu},\,\,\,z\in U_i\cap U_k \label{zz11}\\
|-[\pi_i^{(\mu)}(x),\Lambda_0]|<\frac{1}{\mu},\,\,\,\,\,x\in W_i^\delta,\,\,\,\,\,|[\pi_i^{(\mu_v)},w_i^\alpha]|_{w_i=0}-[\varphi_i^{\alpha(\mu_v)},\Lambda_0]|_{w_i=0}+\sum_{\beta=1}^r \varphi_i^{\beta(\mu)}T_{i\alpha}^\beta(0,z)|<\frac{1}{\mu},\,\,\,z\in U_i^\delta\notag
\end{align}
Let $\pi_i(x)=\lim_v \pi_i^{(\mu_v)}(x)$ and $\varphi_k(z)=\lim_v \varphi_i^{(\mu_v)}(z)$. Since $\pi_i^{(\mu_v)}$ converges absolutely and uniformly on $W_i^\delta$, $\pi_i$ is holomorphic on $W_i^\delta$. Since $\{W_i^\delta\}$ covers $W$, and  $\pi_i(x)=H_{ik}(x)\pi_k(x)$ for $x\in W_i\cap W_k$, we get $\{\pi_i(x)\}\in H^0(W,\wedge^2 T_W)$. On the other hand, $\varphi_i^{\mu_v}(z)$ converges absolutely and uniformly on $U_i$. Let $\pi:=\{\pi_i(x)\}$ and $\varphi:=\{\varphi_i(z)\}$. Then we have $||(\pi^{(\mu_v)}-\pi,\varphi^{(\mu_v)}-\varphi)||\to 0$ as $n\to \infty$. On the other hand, by $(\ref{zz11})$, $\tilde{\delta}(\pi,\varphi)=(0,0)\oplus (-[\pi,\Lambda_0],\{G_{i,\pi,\varphi}\})$, where $-[\pi,\Lambda_0](x)=0$ for $x\in W_i^\delta$ (hence $-[\pi,\Lambda_0]=0$) and $G_{i,\pi,\varphi}(z)=0$ for $z\in U_i^\delta$ (hence $G_{i,\pi,\varphi}=0$ by identity theorem) so that $\tilde{\delta}(\pi,\varphi)=(0,0)\oplus (0,0)$. Hence we have $\tilde{\delta}(\pi^{(\mu_v)},\varphi^{(\mu_v)})=\tilde{\delta}(\pi^{(\mu_v)}-\pi,\varphi^{(\mu_v)}-\phi)=(0,\psi^{(\mu_v)})\oplus(\Pi^{(\mu_v)},G^{(\mu_v)})$ which contradicts to $\iota((0,\psi^{(\mu_v)})\oplus(\Pi^{(\mu_v)},G^{(\mu_v)}))=1$.
\end{proof}

From $(\ref{zz1}),(\ref{zz2})$ and $(\ref{zz4})$, we have
\begin{align}
 &c_4(c_3+e_3+e_4)=c_4c_3+c_4e_3+c_4e_4\label{zz10}\\
 &=\frac{8c_4c_0c_1r^2a}{b}\left(\frac{2^d}{\delta}+rc_1\right)+c_4\left(\frac{(d+r)^2 2^{r+2}a}{b\delta}+\frac{2(d+r)^2a}{b\delta}\right)+c_4\left(\frac{4(d+r)^2e_1^2r^2a}{b}+\frac{4(r+d)e_1ra}{\delta b}\right)+c_4\frac{2(d+r)^4 a}{\delta b}.\notag
 \end{align}
  From $(\ref{zz5})$, $(\ref{zz3})$, $(\ref{zz10})$ and Lemma \ref{nb1}, by assuming
\begin{align*}
b>8c_4c_0c_1r^2a\left(\frac{2^d}{\delta}+rc_1\right)+c_4\left(\frac{(d+r)^2 2^{r+2}a}{\delta}+\frac{2(d+r)^2a}{\delta}\right)+c_4\left(4(d+r)^2e_1^2r^2a+\frac{4(r+d)e_1ra}{\delta }\right)+c_4\frac{2(d+r)^4 a}{\delta }\\
+\max\{ 2c_1ra,\frac{4c_1ra}{\delta},\frac{2a}{\delta},2e_1ra\},
\end{align*} 
we can choose $\varphi_{i|m+1}(z_i,t)\ll A(t)$ and $\Lambda_{i|m+1}(w_i,z_i,t)\ll A(t)$ and so $\varphi_i(z_i,t)\ll A(t)$, and $\Lambda_i(w_i,z_i,t)-\Lambda_i(w_i,z_i)\ll A(t)$ so that $\varphi_i(z_i,t)$ and $\Lambda_i(w_i,z_i,t)$ converges for $|t|<\frac{1}{lb}$. Then we obtain the equality
\begin{align*}
\varphi_i(g_{ik}(\varphi_k(z_k,t),z_k),t)&=f_{ik}(\varphi_k(z_k,t),z_k),\,\,\,\,\text{for}\,\,|t|<\epsilon,(\varphi_k(z_k,t),z_k)\in W_i^\delta\cap W_k^\delta\\
&[\Lambda_i(w_i,z_i,t),w_i^\alpha-\varphi_i^\alpha(z_i,t)]|_{w_i=\varphi_i(z_i,t)}=0\\
&[\Lambda_i(w_i,z_i,t),\Lambda_i(w_i,z_i,t)]=0
\end{align*}
for a sufficiently small number $\epsilon>0$. This completes the proof of Theorem \ref{43c}.

\end{proof}

In the case $\mathbb{H}^1(W, (\wedge^2 T_W\oplus i_*\mathcal{N}_{V/W})^\bullet)\ne 0$, our proof of Theorem \ref{43c} also proves the following:
\begin{theorem}
If the obstruction $(0,\psi_{m+1}(t))\oplus( \frac{1}{2}\Pi_{m+1}(t),G_{m+1}(t))$ vanishes for each integer $m\geq 1$, then there exists an extended Poisson analytic family $\mathcal{V}$ of compact holomorphic Poisson submanifolds $V_t,t\in M_1$, of $(W,\Lambda_t)$ such that $V_0=V\subset (W,\Lambda_0)$ and the characteristic map
\begin{align*}
\sigma_0:T_0(M_1)&\to \mathbb{H}^0(W,(\wedge^2 T_W\oplus i_* \mathcal{N}_{V/W})^\bullet)\\
\frac{\partial}{\partial t}&\mapsto \left(\frac{\partial (\Lambda_t, V_t)}{\partial t}\right)_{t=0}
\end{align*}
is an isomorphism.
\end{theorem}

\subsection{Maximal families: Theorem of completeness}

\begin{definition}\label{025}
Let $\mathcal{V}\subset (W\times M,\Lambda)\xrightarrow{\omega} M$ be an extended Poisson analytic family of compact holomorphic Poisson submanifolds of $W$  so that $\omega^{-1}(t)=V_t$ is a compact holomorphic Poisson submanifold of $(W,\Lambda_t),t\in M$ and let $t_0$ be a point on $M$. We say that $\mathcal{V}\xrightarrow{\omega} M$ is maximal at $t_0$ if, for any extended Poisson analytic family $\mathcal{V}'\subset (W\times M',\Lambda')\xrightarrow{\omega'} M'$ of compact holomorphic Poisson submanifolds of $W$ such that $\Lambda_{t_0}=\Lambda'_{t_0'}$ and $\omega^{-1}(t_0)=\omega'^{-1}(t_0'),t_0'\in M'$, there exists a holomorphic map $h$ of a neighborhood $N'$ of $t_0'$ on $M'$ into $M$ which maps $t_0'$ to $t_0$ such that $\omega'^{-1}(t')=\omega^{-1}(h(t'))$ and $\Lambda'_{t'}=\Lambda_{h(t')}$ for $t'\in N'$. We note that if we set a holomorphic map $\hat{h}:W\times N'\to W\times M$ defined by $(w,t')\to (w,h(t'))$, then $\hat{h}$ is a Poisson map $(W\times N',\Lambda')\to (W\times M ,\Lambda)$ and the restriction map of $\hat{h}$ to $\mathcal{V}'|_{N'}=\omega'^{-1}(N')\subset (W\times N',\Lambda')$ defines a Poisson map $\mathcal{V}'|_{N'}\to \mathcal{V}$ so that $\mathcal{V}'|_{N'}$ is the family induced from $\mathcal{V}$ by $h$, which means $\mathcal{V}\xrightarrow{\omega} M$ is complete at $t_0$.  \end{definition}

\begin{theorem}[theorem of completeness]\label{3.5b}
Let $\mathcal{V}\subset(W\times M_1,\Lambda)$ be an extended Poisson analytic family of compact holomorphic Poisson submanifolds $V_t$ of $(W,\Lambda_t)$. If the characteristic map
\begin{align*}
\rho_0:T_0(M_1)&\to \mathbb{H}^0(W, (\wedge^2 T_W\oplus i_*\mathcal{N}_{W/V_0})^\bullet)\\
\frac{\partial}{\partial t}&\mapsto \left(\frac{\partial (\Lambda_t,V_t)}{\partial t}\right)_{t=0}
\end{align*}
is bijective, then the family $\mathcal{V}$ is maximal at $t=0$.
\end{theorem}

\begin{proof}
Consider an arbitrary extended Poisson analytic family $\mathcal{V}'\subset(W\times M, \Lambda')$ of compact holomorphic Poisson submanifolds $V_s',s\in M'$ of $(W,\Lambda_s')$, where $M'=\{s=(s_1,...,s_q)\in \mathbb{C}^q||s|<1\}$. We will construct a holomorphic map $h:s\to t=h(s)$ of a neighborhood $N'$ of $0$ in $M'$ into $M_1$ with $h(0)=0$, $V_s'=V_{h(s)}$ and $\Lambda_s'=\Lambda_{h(s)}$.

We keep the notations in subsection \ref{027} so that the holomorphic Poisson submanifold $V_t$ of $(W,\Lambda_t)$ is defined on each domain $W_i,i\in I$  by the equation $w_i=\varphi_i(z_i,t)$ and satisfy
\begin{align}\label{mm30}
[\Lambda_i(w_i,z_i,t),w_i^\alpha-\varphi_i^\alpha(z_i,t)]=\sum_{\beta=1}^r (w_i^\beta-\varphi_i^\beta(z_i,t))T_{i\alpha}^\beta(w_i,z_i,t)
\end{align}
We may assume that $V_s'$ is defined in each domain $W_i, i\in I$ by $w_i=\theta_i(z_i,s)$ where $\theta_i(z_i,s)$ is a vector-valued holomorphic function of $z_i$ and $s, |z_i|<1,|s|<1$, and let $\Lambda_i'(w_i,z_i,s)$ be the Poisson structure on $W_i\times M'$ induced from $\Lambda'$. Then we have
\begin{align}\label{mm31}
[\Lambda_i'(w_i,z_i,s), w_i^\alpha-\theta_i^\alpha(z_i,s)]=\sum_{\beta=1}^r (w_i^\beta-\theta_i^\beta(z_i,s))P_{i\alpha}^\beta(w_i,z_i,s)
\end{align}
for some $P_{i\alpha}^\beta(w_i,z_i,s)$ which are power series in $s$ with coefficients in $\Gamma(W_i, T_W)$ and $P_{i\alpha}^\beta(0,z_i,0)=T_{i\alpha}^\beta(0,z_i)$. Then $V_s'=V_{h(s)}$ and $\Lambda_s'=\Lambda_{h(s)}$ are equivalent to the simultaneous equations
\begin{align}\label{pp1}
\theta_i(z_i,s)=\varphi_i(z_i, h(s)),\,\,\,\,\, \Lambda_i'(w_i,z_i,s)=\Lambda_i(w_i,z_i,h(s)),\,\,\,\,\,i\in I, \,\,\,m=1,2,3,...
\end{align}
Recall Notation \ref{notation1} and let us write $h(s)=h_1(s)+h_2(s)+\cdots, \varphi_i(z_i,t)=\varphi_{i|1}(z_i,t)+\varphi_{i|2}(z_i,t)+\cdots,\theta_i(z_i,s)=\theta_{i|1}(z_i,s)+\theta_{i|2}(z_i,s)+\cdots, \Lambda_i(w_i,z_i,t)=\Lambda_i(w_i,z_i)+\Lambda_{i|1}(w_i,z_i,t)+\Lambda_{i|2}(w_i,z_i,t)+\cdots$, and $\Lambda_i'(w_i,z_i,s)=\Lambda_i(w_i,z_i)+\Lambda_{i|1}'(w_i,z_i,s)+\Lambda_{i|2}'(w_i,z_i,s)+\cdots$. We will construct $h(s)$ satisfying $(\ref{pp1})$ by solving the system of congruences by induction on $m$
\begin{align}\label{pp2}
\theta_i(z_i,s)\equiv_m \varphi_i(z_i, h^m(s)),\,\,\,\,\, \Lambda_i'(w_i,z_i,s)\equiv_m \Lambda_i(w_i,z_i, h^m(s)),\,\,\,\,\,i\in I, \,\,\,m=1,2,3,\cdots
\end{align}
Since $\sigma_0:T_0(M_1)\to \mathbb{H}^0(W,(\wedge^2 T_W\oplus i_*\mathcal{N}_{V_0/W})^\bullet)$ is an isomorphism by the hypothesis, any element $(\{B_i(w_i,z_i)\},\{\omega_i(z_i)\})\in \mathbb{H}^0(W, (\wedge^2 T_W\oplus i_*\mathcal{N}_{V_0/W})^\bullet)$ can be written uniquely in the form
\begin{align*}
\omega_i(z)=\varphi_{i|1}(z_i,u)=\sum_{\alpha=1}^l \frac{\partial \varphi_i (z_i,t)}{\partial t_\alpha}|_{t=0}u^\alpha,\,\,\,\,\, B_i(w_i,z_i)=\Lambda_{i|1}(w_i,z_i,u)=\sum_{\alpha=1}^l\frac{\partial \Lambda_i(w_i,z_i,t)}{\partial t_\alpha}|_{t=0}u^\alpha
\end{align*}
for some constant $u=(u^1,...,u^l)$. Hence since $(\{\Lambda_{i|1}'(w_i,z_i,s)\},\{\theta_{i|1}(z_i,s)\}),i\in I$ represents a linear form in $s$ whose coefficients are in $\mathbb{H}^0(W, (\wedge^2 T_W\oplus i_*\mathcal{N}_{V_0/W})^\bullet)$, there exists a linear vector-valued function $h^1(s)$ of $s$ such that $\theta_{i|1}(z_i,s)=\varphi_{i|1}(z_i, h^1(s))$, and $\Lambda'_{i|1}(w_i,z_i,s)=\Lambda_{i|1}(w_i,z_i, h^1(s))$. This shows $(\ref{pp2})_1$. Now suppose that we have already constructed $h^m(s)$ satisfying $(\ref{pp2})_m$. We will find $h_{m+1}(s)$ such that $h^{m+1}(s)=h^m(s)+h_{m+1}(s)$ satisfy $(\ref{pp2})_{m+1}$. Let $\omega_i(z_i,s)=[\theta_i(z_i,s)-\varphi_i(z_i, h^m(s))]_{m+1}$, and $B_i(w_i,z_i,s)=[\Lambda_i'(w_i,z_i,s)-\Lambda_i(w_i,z_i, h^m(s))]_{m+1}$. We claim that 
\begin{align}
&\omega_i(z_i,s)=F_{ik}(z)\cdot \omega_k(z_k,s) \label{pp3}\\
&[ B_i(w_i,z_i,s),w_i^\alpha]|_{w_i=0}-[\omega_i^\alpha(z_i,s),\Lambda_0]|_{w_i=0}+\sum_{\beta=1}^r \omega_i^\beta(z_i,s) T_{i\alpha}^\beta(z_i)=0 \label{pp4}\\
&B_i(w_i,z_i,s)-B_j(w_j,z_j,s)=0 \label{pp5}\\
&-[B_i(w_i,z_i,s),\Lambda_0]=0\label{pp6}
\end{align}
(\ref{pp3}) follows from \cite{Kod62} p.160. Since $\Lambda_i'(w_i,z_i,s)=\Lambda_j'(w_j,z_j,s)$ and $\Lambda_i(w_i,z_i,t)=\Lambda_j(w_j,z_j,t)$, we get $(\ref{pp5})$. Since $2[\Lambda_0, B_i(w_i,z_i,s)]\equiv_{m+1} [\Lambda_i'(w_i,z_i,s)+\Lambda_i(w_i,z_i,h^m(s)), \Lambda_i'(w_i,z_i,s)-\Lambda_i(w_i,z_i,h^m(s))]=0$, we get $(\ref{pp6})$. 
It remains to show (\ref{pp4}). From $(\ref{mm30})$ and (\ref{mm31}), we have
\begin{align*}
&[\Lambda_0, \omega_i^\alpha(z_i,s)]|_{w_i=0}\equiv_{m+1}[\Lambda_i'(w_i,z_i,s),\omega_i^\alpha(z_i,s)]|_{w_i=\theta_i(z_i,s)}\\
&\equiv_{m+1} [\Lambda_i'(w_i,z_i,s),\theta_i^\alpha(z_i,s)-w_i^\alpha+w_i^\alpha-\varphi_i^\alpha(z_i, h^m(s))]|_{w_i=\theta_i(z_i,s)}\\
&\equiv_{m+1} -[\Lambda_i'(w_i,z_i,s),w_i^\alpha-\theta_i^\alpha(z_i,s)]|_{w_i=\theta_i(z_i,s)}+[\Lambda_i'(w_i,z_i,s), w_i^\alpha- \varphi_i^\alpha(z_i,h^m(s))]|_{w_i=\theta_i(z_i,s)}\\
&\equiv_{m+1} [B_i(w_i,z_i,s), w_i^\alpha- \varphi_i^\alpha(z_i, h^m(s))]|_{w_i=\theta_i(z_i,s)}+[\Lambda_i(w_i,z_i,h^m(s)),w_i^\alpha-\varphi_i^\alpha(z_i,h^m(s))]|_{w_i=\theta_i(z_i,s)}\\
&\equiv_{m+1} [B_i(w_i,z_i,s), w_i^\alpha]|_{w_i=0}+\sum_{\beta=1}^r (\theta_i^\beta(z_i,s)-\varphi_i^\beta(z_i,h(s)))T_{i\alpha}^\beta(\theta_i(z_i,s),z_i,h(s))\\
&\equiv_{m+1} [B_i(w_i,z_i,s), w_i^\alpha]|_{w_i=0}+\sum_{\beta=1}^r \omega_i^\beta(z_i,s)T_{i\alpha}^\beta(0,z_i)
\end{align*}
This proves $(\ref{pp4})$. From $(\ref{pp3}),(\ref{pp4}),(\ref{pp5})$, and $(\ref{pp6})$, $(\{B_i(w_i,z_i,s)\},\{\omega_i(z_i,s) \})$ is a homogenous polynomial  of degree $m+1$ in $s$ with coefficients in $\mathbb{H}^0(W,(\wedge^2 T_W\oplus i_*\mathcal{N}_{V_0/W})^\bullet)$ so that there exists a homogenous polynomial $h_{m+1}(s)$ of degree $m+1$ in $s$ such that $\omega_i(z_i,s)=\varphi_{i|1}(z_i, h_{m+1}(s))$, and $B_i(w_i,z_i,s)=\Lambda_{i|1}(w_i,z_i,h_{m+1}(s))$ so that we have $\varphi_i(z_i, h^{m+1}(s))\equiv_{m+1} \varphi_i(z_i,h^m(s))+\omega_i(z_i,s)\equiv_{m+1} \theta_i(z_i,s)$, and $\Lambda_i(w_i,z_i,h^{m+1}(s))\equiv_{m+1}\Lambda_i(w_i,z_i, h^m(s))+B_i(w_i,z_i,s)\equiv_{m+1} \Lambda_i'(w_i,z_i,s)$, which completes the inductive construction of $h^{m+1}(s)$ satisfying $(\ref{pp2})_{m+1}$.

\subsection{Proof of convergence}\

The convergence of the power series $h(s)$ follows from the same arguments in \cite{Kod62}. This completes the proof of Theorem \ref{3.5b}.

\end{proof}

\begin{example}\label{example51}
We describe holomorphic Poisson structures on rational ruled surfaces $F_m=\mathbb{P}(\mathcal{O}_{\mathbb{P}_\mathbb{C}^1}(m)\oplus \mathcal{O}_{ \mathbb{P}_\mathbb{C}^1}),m\geq0$ explictly.
$F_m$ can be represented in the following way. Take two copies of $U_i\times \mathbb{P}_\mathbb{C}^1,i=1,2$, where $U_i=\mathbb{C}$ and write the coordinates as $(z,[\xi_0,\xi_1])$ and $(z',[\xi_0',\xi_1'])$. Patch $U_i\times \mathbb{P}_\mathbb{C}^1,i=1,2$ by the relation $z'=\frac{1}{z}$ and $[\xi_0',\xi_1']=[\xi_0,z^m \xi_1]$. We set $\xi=\frac{\xi_1}{\xi_0}$ and $\xi'=\frac{\xi_1'}{\xi_0'}$. Then we have $\frac{\partial}{\partial z'}=-z^2\frac{\partial}{\partial z}+mz\xi\frac{\partial}{\partial \xi}$ and $\frac{\partial}{\partial \xi'}=z^{-m}\frac{\partial}{\partial \xi}$ so that $\frac{\partial}{\partial z'}\wedge \frac{\partial}{\partial \xi'}=-z^{-m+2}\frac{\partial}{\partial z}\wedge \frac{\partial}{\partial \xi}$. We note that  a holomorphic bivector field on $U_1\times \mathbb{P}^1$ is of the form $(d(z)+e(z)\xi+f(z)\xi^2)\frac{\partial}{\partial z}\wedge \frac{\partial}{\partial \xi}$, and a holomorphic bivector field on $U_2\times \mathbb{P}_\mathbb{C}^1$ is of the form $(p(z')+q(z')\xi'+r(z')\xi'^2)\frac{\partial}{\partial z'}\wedge \frac{\partial}{\partial \xi'}$, where $d(z),e(z),f(z)$ are entire functions of $z$ and  $p(z'),q(z'),r(z')$ are entire functions of $z'$. For a holomorphic bivector field on $\mathbb{F}_m$ which has the form on each $U_i\times \mathbb{P}_\mathbb{C}^1$, we must have $d(z)+e(z)\xi+f(z)\xi^2=-(p(\frac{1}{z})+q(\frac{1}{z})z^m \xi+r(\frac{1}{z})z^{2m}\xi^2)z^{-m+2}=-p(\frac{1}{z})z^{-m+2}-q(\frac{1}{z})z^{2}\xi-r(\frac{1}{z})z^{m+2}\xi^2$
so that
\begin{align*}
d(z)=-p\left(\frac{1}{z}\right)z^{-m+2},\,\,\,e(z)=-q\left(\frac{1}{z}\right)z^{2},\,\,\, f(z)=-r\left(\frac{1}{z}\right)z^{m+2}
\end{align*}
\begin{enumerate}
\item In the case of $m=0$, we have $d(z)=a_0+a_1z+a_2z^2 ,e(z)=b_0+b_1z+b_2z^2 ,f(z)=c_0+c_1z+c_2z^2$ so that $H^0(F_0,\wedge^2 T_{F_0})\cong \mathbb{C}^9$.
\item In the case of $m=1$, we have $d(z)=a_0+a_1z, e(z)=b_0+b_1z+b_2z^2, f(z)=c_0+c_1z+c_2z^2+c_3z^3$ so that $H^0(F_1,\wedge^2 T_{F_1})\cong \mathbb{C}^9$.
\item In the case of $m=2$, we have $d(z)=a_0,e(z)=b_0+b_1z+b_2z^2,f(z)=c_0+c_1z+c_2z^2+c_3z^3+c_4z^4$ so that $H^0(F_2,\wedge^2 T_{F_2})\cong \mathbb{C}^9$.
\item In the case of $m\geq 3$, we have $d(z)=0,e(z)=b_0+b_1z+b_2z^2,f(z)=c_0+c_1z+\cdots c_{m+2}z^{m+2}$ so that $H^0(F_m,\wedge^2 T_{F_m})\cong \mathbb{C}^{m+6}$. 
\end{enumerate}

\end{example}
 In the sequel, we keep the notations in Example \ref{example51}.
\begin{example}
Let us consider a rational ruled surface $F_0\cong \mathbb{P}_\mathbb{C}^1\times \mathbb{P}_\mathbb{C}^1$. Let us consider the Poisson structure $\Lambda_0=\xi\frac{\partial}{\partial z}\wedge \frac{\partial}{\partial \xi}$. Then $\xi=0$ defines a holomorphic Poisson submanifold on $F_0$ which is a nonsingular rational curve $\cong \mathbb{P}_\mathbb{C}^1$ and the normal bundle is $\mathcal{N}_{\mathbb{P}_\mathbb{C}^1/F_0}\cong \mathcal{O}_{\mathbb{P}_\mathbb{C}^1}$. We compute $\mathbb{H}^0(F_0,(\wedge^2 T_{F_0}\oplus i_*\mathcal{N}_{\mathbb{P}_\mathbb{C}^1/F_0})^\bullet)$ which is the kernel of $\tilde{\nabla}:H^0(F_0,\wedge^2 T_{F_0})\oplus\mathbb{C}\to H^0(\mathbb{P}_\mathbb{C}^1, T_{F_0}|_{\mathbb{P}^1})$. Since $[\Lambda_0, \xi]=-\xi\frac{\partial}{\partial z}$, we have as the image of $\tilde{\nabla}$,
\begin{align*}
[\left(a_0+a_1z+a_2z^2+(b_0+b_1z+b_2z^2)\xi+(c_0+c_1z+c_2z^2)\xi^2\right)\frac{\partial}{\partial z}\wedge \frac{\partial}{\partial \xi},\xi]|_{\xi=0}-d\frac{\partial}{\partial z}\\
=-(a_0+a_1z+a_2z^2)\frac{\partial}{\partial z}-d\frac{\partial}{\partial z}=-(a_0+d+a_1z+a_2z^2)\frac{\partial}{\partial z}
\end{align*}
where $d$ is a constant.
Hence $\dim \mathbb{H}^0(F_0,(\wedge^2 T_{F_0}\oplus i_*(\mathcal{N}_{\mathbb{P}_\mathbb{C}^1/F_0}))=7$. 
\begin{align*}
&[(-d+(1+b_0+b_1z+b_2z^2)\xi+(c_0+c_1z+c_2z^2)\xi^2)\frac{\partial}{\partial z}\wedge \frac{\partial}{\partial \xi} ,\xi-d]|_{\xi=d}\\
&=-(b_0d+c_0d^2+(b_1d+c_1d^2)z+(b_2d+c_2d^2)z^2)\frac{\partial}{\partial z}=\tilde{\nabla}((b_0d+c_0d^2+(b_1d+c_1d^2)z+(b_2d+c_2d^2)z^2)\frac{\partial}{\partial z}\wedge \frac{\partial}{\partial \xi},0))
\end{align*}
so that obstruction vanishes and an extended Poisson analytic family
\begin{align*}
\mathcal{V}\subset (F_0\times \mathbb{C}^7, (-d-(b_0d+c_0d^2)-(b_1d+c_1d^2)z-(b_2d+c_2d^2)z^2+(1+b_0+b_1z+b_2z^2)\xi+(c_0+c_1z+c_2z^2)\xi^2)\frac{\partial}{\partial z}\wedge \frac{\partial}{\partial \xi})
\end{align*}
defined by $\xi=d$ has the characteristic map
\begin{align*}
T_0\mathbb{C}^7&\to \mathbb{H}^0(F_0,(\wedge^2 T_{F_0}\oplus i_*(\mathcal{N}_{\mathbb{P}_\mathbb{C}^1/F_0}))\\
(a_0,...,a_6)&\mapsto ((-a_0+(a_1+a_2z+a_3z^2)\xi+(a_4+a_5z+a_6z^2)\xi^2)\frac{\partial}{\partial z}\wedge \frac{\partial}{\partial \xi} ,a_0)
\end{align*}
which is an isomorphism so that $\mathcal{V}$ is complete.
\end{example}

\begin{example}
Let us consider a rational ruled surface $F_1$. Let us consider the Poisson structure $\Lambda_0=\xi\frac{\partial}{\partial z}\wedge \frac{\partial}{\partial \xi}$. Then $\xi=0$ defines a holomorphic Poisson submanifold on $F_1$ which is a nonsingular rational curve $\cong \mathbb{P}^1_\mathbb{C}$ and the normal bundle is $\mathcal{N}_{\mathbb{P}_\mathbb{C}^1/F_1}\cong \mathcal{O}_{\mathbb{P}_\mathbb{C}^1}(-1)$ so that $H^0(\mathbb{P}_\mathbb{C}^1,\mathcal{N}_{\mathbb{P}_\mathbb{C}^1/F_1})=0$. We compute $\mathbb{H}^0(F_1, (\wedge^2 T_{F_1}\oplus i_*\mathcal{N}_{\mathbb{P}_\mathbb{C}^1/F_1})^\bullet)$ which is the kernel of $\tilde{\nabla}:H^0(F_1,\wedge^2 T_{F_1})\to H^0(\mathbb{P}_\mathbb{C}^1, T_{F_1}|_{\mathbb{P}_\mathbb{C}^1}(-1))$. Since $[\Lambda_0, \xi]=-\xi\frac{\partial}{\partial z}$, we have as the image of $\tilde{\nabla}$,
\begin{align*}
[\left(a_0+a_1z+(b_0+b_1z+b_2z^2)\xi+(c_0+c_1z+c_2z^2+c_3z^3)\xi^2\right)\frac{\partial}{\partial z}\wedge \frac{\partial}{\partial \xi},\xi]|_{\xi=0}=-(a_0+a_1z)\frac{\partial}{\partial z}
\end{align*}
so that $\mathbb{H}^0(F_1, (\wedge^2 T_{F_1}\oplus i_*\mathcal{N}_{\mathbb{P}_\mathbb{C}^1/F_1})^\bullet)=7$ and an extended Poisson analytic family 
\begin{align*}
\mathcal{V}\subset (F_0\times \mathbb{C}^7, ((1+b_0+b_1z+b_2z^2)\xi+(c_0+c_1z+c_2z^2+c_3z^3)\xi^2)\frac{\partial}{\partial z}\wedge \frac{\partial}{\partial \xi})
\end{align*}
defined by $\xi=d$ has the characteristic map
\begin{align*}
T_0\mathbb{C}^7&\to \mathbb{H}^0(F_1,(\wedge^2 T_{F_1}\oplus i_*(\mathcal{N}_{\mathbb{P}_\mathbb{C}^1/F_1}))\\
(a_0,...,a_6)&\mapsto ((a_0+a_1z+a_2z^2)\xi+(a_3+a_4z+a_5z^2+a_6z^3)\xi^2)\frac{\partial}{\partial z}\wedge \frac{\partial}{\partial \xi},0)
\end{align*}
which is an isomorphism so that $\mathcal{V}$ is complete.
\end{example}

\begin{example}
Let us consider a rational ruled surface $F_m,m\geq 3$. Let us consider the trivial Poisson structure $\Lambda_0=0$ on $F_m$. Then $\xi=0$ defines a holomorphic Poisson submanifold which is a nonsingular rational curve $\cong \mathbb{P}_\mathbb{C}^1$ and the normal bundle is $\mathcal{N}_{\mathbb{P}_\mathbb{C}^1/F_m}\cong \mathcal{O}_{\mathbb{P}^1}(-m)$ so that $H^0(\mathbb{P}^1_\mathbb{C},\mathcal{N}_{\mathbb{P}^1_\mathbb{C}/F_m})=0$. Hence $\mathbb{H}^0(F_m,(\wedge^2 T_{F_m}\oplus i_*\mathcal{N}_{\mathbb{P}_\mathbb{C}^1/F_m})^\bullet)=H^0(F_m, \wedge^2 T_{F_m})=k^{m+6}$. Let us consider an extended Poisson analytic family $\mathcal{V}\subset( F_m\times \mathbb{C}^{m+6},\Lambda(t)=((t_0+t_1z+t_2z^2)\xi+(t_3+t_4z+\cdots+t_{m+5} z^{m+2})\xi^2)\frac{\partial}{\partial z}\wedge \frac{\partial}{\partial \xi})$ defined by $\xi=0$. Then the characteristic map
\begin{align*}
T_0\mathbb{C}^{m+6}&\xrightarrow{\cong}\mathbb{H}^0(F_m,(\wedge^2 T_{F_m}\oplus i_*\mathcal{N}_{\mathbb{P}_\mathbb{C}^1/F_m})^\bullet)\\
(a_0,...,a_{m+5})&\mapsto \left(((a_0+a_1z+a_2z^2)\xi+(a_3+a_4z+\cdots+a_{m+5} z^{m+2})\xi^2)\frac{\partial}{\partial z}\wedge \frac{\partial}{\partial \xi},0\right)
\end{align*}
\end{example}
is an isomorphism so that $\mathcal{V}$ is complete at $0$.
\section{Stability of compact holomorphic Poisson submanifolds}\label{section5}

We extend the definition of a fibre manifold in \cite{Kod63} in the context of the holomorphic Poisson category.
\begin{definition}\label{fibre}
By a holomorphic Poisson fibre manifold, we shall mean a holomorphic Poisson manifold $(\mathcal{W},\Lambda)$ together with a holomorphic map $p$ of $\mathcal{W}$ onto a complex manifold $B$ such that the rank of the  Jacobian of $p$ at each point of $\mathcal{W}$ is equal to the dimension of $B$ and $\Lambda\in H^0(\mathcal{W},\wedge^2 T_{\mathcal{W}/B})$ with $[\Lambda,\Lambda]=0$. For any point $u\in B$, the inverse image $p^{-1}(u)=(W_u, \Lambda_u)$ is a holomorphic Poisson submanifold of $(\mathcal{W},\Lambda)$ and call it the fibre of $(\mathcal{W},\Lambda)$ over $u$. We  will denote the holomorphic Poisson fibre manifold $(\mathcal{W},\Lambda)$ by the quadruple $(\mathcal{W},\Lambda, B, p)$. We note that a holomorphic Poisson fibre manifold is a Poisson analytic family in the sense of \cite{Kim15} when fibres are compact. For any subdomain $N$ of $B$, we call the holomorphic Poisson fibre manifold $(p^{-1}(N),\Lambda|_{p^{-1}(N)},N,p)$ the restriction of $(\mathcal{W},\Lambda)$ to $N$ and denote it by $(\mathcal{W},\Lambda)|_N$. Let $\mathcal{V}$ be a holomorphic  Poisson submanifold of $(\mathcal{W},\Lambda)|_N$ such that $p(\mathcal{V})=N$. We call $\mathcal{V}$ a holomorphic Poisson fibre submanifold of the holomorphic Poisson fibre manifold $(\mathcal{W},\Lambda)|_N$ if and only if $(\mathcal{V},\Lambda|_\mathcal{V},N,p)$ forms a holomorphic Poisson fibre manifold. If, moreover, each fibre $V_u=\mathcal{V}\cap W_u, u\in N$, of $\mathcal{V}$ is compact, we call $\mathcal{V}$ a holomorphic Poisson fibre submanifold with compact fibres of the holomorphic Poisson fibre manifold $(\mathcal{W},\Lambda)|_N$.
\end{definition}

We extend the definition of stability in \cite{Kod63} in the context of the holomorphic Poisson category.
\begin{definition}\label{029}
Let $V$ be a compact holomorphic Poisson submanifold of a holomorphic Poisson manifold $(W,\Lambda_0)$. We call $V$ a stable holomorphic Poisson submanifold of $(W,\Lambda_0)$ if and only if, for any holomorphic Poisson fibre manifold $(\mathcal{W},\Lambda,B,p)$ such that $p^{-1}(0)=(W,\Lambda_0)$ for a point $0\in B$, there exist a neighborhood $N$ of $0$ in $B$ and a holomorphic Poisson fibre submanifold $\mathcal{V}$ with compact fibres of the holomorphic Poisson fibre manifold $(\mathcal{W},\Lambda)|_N$ such that $\mathcal{V}\cap W=V$.
\end{definition}

\subsection{Stability of compact holomorphic Poisson submanifolds}

\begin{theorem}[compare \cite{Kod63} Theorem 1]\label{stability}
Let $V$ be a compact holomorphic Poisson submanifold of a holomorphic Poisson manifold $(W,\Lambda_0)$. Let $\mathcal{N}_{V/W}^\bullet$ be the complex associated with the normal bundle $\mathcal{N}_{V/W}$ as in Definition $\ref{332a}$. If the first cohomology group $\mathbb{H}^1(V,\mathcal{N}_{V/W}^\bullet)$ vanishes, then $V$ is a stable holomorphic Poisson submanifold of $(W,\Lambda_0)$.
\end{theorem}

To prove Theorem \ref{stability}, we extend the argument in \cite{Kod63} p.80-85 in the context of holomorphic Poisson deformations. We tried to maintain notational consistency with \cite{Kod63}.

Let $(\mathcal{W},\Lambda, B,p)$ be a holomorphic Poisson fibre manifold such that $p^{-1}(0)=(W,\Lambda_0)$ for a point $0\in B$ and let $(u^1,...,u^q)$ denote a local coordinate on $B$ with the center $0$. Considering $V\subset W \subset \mathcal{W}$ as a submanifold of $(\mathcal{W},\Lambda_0)$, we cover $V$ by a finite number of coordinate neighborhood $\mathcal{U}_i$ in $\mathcal{W}$ and choose a local coordinate $(z_i,w_i,u)=(z_i^1,...,z_i^d,w_i^1,...,w_i^r,u^1,...,u^q)$ such that the simultaneous equations $w_i^1=\cdots=w_i^r=u^1=\cdots =u^q=0$ define $V$. We assume that each neighborhood $\mathcal{U}_i$ is a polycylinder consisting of all points $(z_i,w_i,u),|z_i|<1,|w_i|<1,|u|<1$. On the intersection $\mathcal{U}_i\cap \mathcal{U}_k$ the local coordinates $z_i^\alpha,w_i^\lambda$ are holomorphic functions of $z_k,w_k,u$: $z_i^\alpha=g_{ik}^\alpha(z_k,w_k,u),\alpha=1,...,d, w_i^\lambda=f_{ik}^\lambda(z_k,w_k,u),\lambda=1,...,r$.
Note that $f_{ik}^\lambda(z_k,0,0)=0$. We set $U_k=V\cap\mathcal{U}_k$. We denote a point on $V$ by $z$ and, if $z=(z_k,0,0)\in U_k$, we call $z_k=(z_k^1,...,z_k^d)$ the local coordinate of $z$ on $U_k$. We define $a_{ik\mu}^\lambda(z)=\left( \frac{\partial f_{ik}^\lambda(z_k,w_k,u)}{\partial w_k^\mu}\right)_{w_k=u=0},b_{ik\rho}^\lambda(z)=\left(\frac{\partial f_{ik}^\lambda(z_k,w_k,u)}{\partial u^\rho}\right)_{w_k=w=0}$. Then the normal bundle of $V$ in $\mathcal{W}$ is defined by the system of transition matrices 
$\left(\begin{matrix}
a_{ik}(z) &b_{iz}(z)\\
0 & 1_q
\end{matrix}\right)$ so that we have the exact sequence $0\to \mathcal{N}_{V/W}\to \mathcal{N}_{V/\mathcal{W}}\to \oplus^q\mathcal{O}_V\to 0$. On the other hand, since $w_i^i=\cdots=w_i^r=u^1=\cdots=u^q=0$ defines a holomorphic Poisson submanifold, $[\Lambda,w_i^\alpha]=\sum_{\beta=1}^r w_i^\alpha T_{i\alpha}^\beta(z_i,w_i,u)+\sum_{\rho=1}^q u^\rho T_{i\alpha}^{q+\rho}(z_i,w_i,u)$ for some $T_{i\alpha}^\gamma(z_i,w_i,u)\in \Gamma(\mathcal{U}_i, T_\mathcal{W})$, $\gamma=1,...,r+q$, and $[\Lambda, u^\rho]=0$. We note that $T_{i\alpha}^\gamma(z_i,0,0)\in \Gamma(U_i, T_W|_V)$. Setting $T_{i(q+\rho)}^{\gamma}(z_i,w_i,u):=0$ for $\rho=1,...,q,\gamma=1,...,r+q$, we can write $[\Lambda, u^\eta]=\sum_{\beta=1}^r w_i^\alpha T_{i(r+\eta)}^\beta(z_i,w_i,u)+\sum_{\rho=1}^q u^\rho T_{i(r+\eta)}^{r+\rho}(z_i,w_i,u)$.

 We note that we can extend $0\to \mathcal{N}_{V/W}\to \mathcal{N}_{V/\mathcal{W}}\to \oplus^q\mathcal{O}_V\to 0$ to obtain an exact sequence of complex of sheaves 
 \begin{align}\label{5h}
 0\to \mathcal{N}_{V/W}^\bullet\to R^\bullet\to \mathcal{Q}^\bullet\to 0:
 \end{align}
\begin{center}
$\begin{CD}
@. \cdots @. \cdots @.\cdots \\
@. @A\nabla_W AA @A\nabla_\mathcal{W} AA @A-[-,\Lambda]|_VAA \\
0@>>>\mathcal{N}_{V/W}\otimes \wedge^2 T_W|_V@>>> R^2:=\mathcal{N}_{V/\mathcal{W}}\otimes \wedge^2 T_W|_V @>>>\mathcal{Q}^2:=\oplus^q\wedge^2 T_{W|V}@>>>0\\
@. @A\nabla_W AA @A\nabla_\mathcal{W} AA @A-[-,\Lambda]|_VAA @.\\
0@>>> \mathcal{N}_{V/W}\otimes T_W|_V@>>> R^1:=\mathcal{N}_{V/\mathcal{W}}\otimes T_W|_V@>>> \mathcal{Q}^1:=\oplus^q T_{W|V}@>>> 0\\
@. @A\nabla_WAA @A\nabla_\mathcal{W} AA @A-[-,\Lambda ]|_V AA @.\\
0@>>> \mathcal{N}_{V/W}@>>> R^0:=\mathcal{N}_{V/\mathcal{W}}@>>> \mathcal{Q}^0:=\oplus^q\mathcal{O}_V@>>>0
\end{CD}$
\end{center}
where the first vertical complex $\mathcal{N}_{V/W}^\bullet$ is the complex associated with the normal bundle $\mathcal{N}_{V/W}$ and the second vertical complex $R^\bullet$ is the subcomplex of the complex $\mathcal{N}_{V/\mathcal{W}}^\bullet$ associated with the normal bundle $\mathcal{N}_{V/\mathcal{W}}$ as in Definition \ref{332a}. First we show that the second vertical complex is well-defined. Indeed, we simply note that $T_{i (r+\rho)}^\beta(z_i,0,0)=0,\rho=1,...,q,\beta=1,..,r+q$, $T_{i\alpha}^\beta(z_i,0,0)\in \Gamma(U_i, T_W|_V),\alpha=1,...,r,\beta=1,...,r+q$, and $-[g,\Lambda]|_V\in \Gamma(U_i,\wedge^{p+1}T_W|_V)$ for $g\in \Gamma(U_i,\wedge^p T_W|_V)$.

Next we show that the sequence of complex of sheaves (\ref{5h}) is well-defined. In other words, the above diagram commutes. The commutativity of the first two complexes follows from the following local commutativity:
\begin{center}
\tiny{$\begin{CD}
(-[f_i^1,\Lambda]+(-1)^p\sum_{\beta=1}^r f_i^\beta T_{i1}^\beta(z_i),...,-[f_i^r,\Lambda]+(-1)^p\sum_{\beta=1}^r f_i^\beta T_{ir}^\beta(z_i)@>>> (-[f_i^1,\Lambda]+(-1)^p\sum_{\beta=1}^r f_i^\beta T_{i1}^\beta(z_i),...,-[f_i^r,\Lambda]+(-1)^p\sum_{\beta=1}^r f_i^\beta T_{ir}^\beta(z_i),0,...,0)\\
@AAA @AAA\\
(f_i^1,...,f_i^r)@>>>(f_i^1,...,f_i^r,0,...,0)\\
\end{CD}$}
\end{center}
where $(f_i^1,...,f_i^r)\in \oplus^r \Gamma(U_i, \wedge^p T_W|_V)$. On the other hand, the commutativity of the last two complexes follows from the following local commutativity:
\begin{center}
\tiny{$\begin{CD}
(-[f_i^1,\Lambda]|_V+(-1)^p\sum_{\beta=1}^{r+q} h_i^\beta T_{i1}^\beta(z_i),...,-[f_i^r,\Lambda]|_V+(-1)^p\sum_{\beta=1}^{r+q} h_i^\beta T_{ir}^\beta(z_i),-[g_i^1,\Lambda]|_V,...,-[g_i^q,\Lambda]|_V)@>>> (-[g_i^1,\Lambda]|_V,...,-[g_i^q,\Lambda]|_V)\\
@AAA @AAA\\
(h_i^1,...,h_i^{r+q}):=(f_i^1,...,f_i^r,g_i^1,...,g_i^q)@>>>(g_i^1,...,g_i^q)\\
\end{CD}$}
\end{center}
where $(f_i^1,...,f_i^r,g_i^1,...,g_i^q)\in \oplus^{r+q}\Gamma(U_i,\wedge^p T_W|_V)$.

Since $V$ is compact, $\mathbb{H}^0(V,\mathcal{Q}^\bullet)=H^0(V,\mathcal{O}_V^q)\cong \mathbb{C}^q$ so that from the hypothesis $\mathbb{H}^1(V,\mathcal{N}_{V/W}^\bullet)=0$ and the exact sequence $0\to \mathcal{N}_{V/W}^\bullet\to R^\bullet\to Q^\bullet\to 0$, we obtain an exact sequence 
\begin{align}\label{5g}
0\to \mathbb{H}^0(V,\mathcal{N}_{W/V}^\bullet)\to \mathbb{H}^0(V,R^\bullet)\xrightarrow{\kappa} \mathbb{C}^q\to 0
\end{align}

Note that any element of $\mathbb{H}^0(V,R^\bullet)$ is a collection $\{\psi_i(z)\}$ of vector-valued holomorphic functions
\begin{align*}
\psi_i(z)=(\psi_i^1(z),...,\psi_i^r(z),\psi^{r+1},...,\psi^{r+q})
\end{align*}
defined respectively on $U_i$ satisfying $\psi_i^\lambda(z)=\sum_{\mu=1}^r a_{ik\mu}^\lambda \psi_k^\mu(z)+\sum_{\rho=1}^q b_{ik\rho}^\lambda(z)\psi^{r+\rho}$ and $-[\psi_i^\lambda(z),\Lambda]|_V+\sum_{\beta=1}^r \psi_i^\beta T_{i\lambda}^\beta(z_i,0,0)+\sum_{\rho=1}^q\psi^{r+\rho} T_{i\lambda}^{r+\rho}(z_i,0,0)=0,\lambda=1,...,r
$
and we have $\kappa \psi=(\psi^{r+1},...,\psi^{r+q})$.

Let $M_\epsilon=\{t=(t_1,...,t_n)\in \mathbb{C}^n||t|<\epsilon\}$, where $\epsilon$ is a small positive number. Consider a Poisson analytic family $\mathcal{F}$ of compact holomorphic Poisson submanifolds $V_t,t\in M_\epsilon$, of $(\mathcal{W},\Lambda)$ such that $V_0=V$ (see Definition \ref{3a}). Then $\mathcal{F}$ is a holomorphic Poisson submanifold of $(\mathcal{W}\times M_\epsilon,\Lambda)$ such that $\mathcal{F}\cap (\mathcal{W}\times t)=V_t\times t$. We assume that $\mathcal{F}$ is covered by the neighborhoods $\mathcal{U}_i\times M_\epsilon$. On each neighborhood $\mathcal{U}_i\times M_\epsilon$, the holomorphic Poisson submanifold $\mathcal{F}$ is defined by simultaneous holomorphic equations of the form $w_i^\lambda=\theta_i^\lambda(z_i,t),\,\,\, \lambda=1,...,r,u^\rho=\theta^{r+\rho}(t),\,\,\, \rho=1,...,q$. For any tangent vector $\frac{\partial}{\partial t}=\sum_v c_v\frac{\partial}{\partial t_v}$ of $M_\epsilon $ at $t=0$, we set $\psi_i^\lambda(z)=\left(\frac{\partial \theta_i^\lambda(z_i,t)}{\partial t}\right)_{t=0},\lambda=1,...,r,\psi^{r+\rho}=\left(\frac{\partial \theta^{r+\rho}(t)}{\partial t}\right)_{t=0},\rho=1,...,q$ and let $\psi_i(z)=(\psi_i^1(z),...,\psi_i^r(z),\psi^{r+1},...,\psi^{r+q})$. Then the collection $\{\psi_i(z)\}$ of $\psi_i(z)$ represents an element of $\mathbb{H}^0(V,\mathcal{N}_{\mathcal{W}/V}^\bullet)=\mathbb{H}^0(V,R^\bullet)$ (see subsection \ref{3infinitesimal}). With this preparation, we prove
\begin{theorem}[compare \cite{Kod63} Theorem 2]\label{5theorem1}
There exists a Poisson analytic family $\mathcal{F}$ of compact holomorphic Poisson submanifolds $V_t$, where $t\in M_\epsilon,\epsilon>0$, of $(\mathcal{W},\Lambda)$ such that $V_0=V$ and the characteristic map: $\frac{\partial}{\partial t}\to \frac{\partial V_t}{\partial t}|_{t=0}$ maps the tangent space $T_0(M_\epsilon)$ isomorphically onto $\mathbb{H}^0(V,\mathcal{N}_{\mathcal{W}/V}^\bullet)=\mathbb{H}^0(V,R^\bullet)$ provided that cohomology group $\mathbb{H}^1(V,\mathcal{N}_{W/V}^\bullet)$ vanishes. 
\end{theorem}

\begin{proof}
Let $n=\mathbb{H}^0(V, R^\bullet)$. We can choose a base $\{\beta_1,...,\beta_n\}$ of $\mathbb{H}^0(V,R^\bullet)$ such that $\beta_v^{r+\rho}=1$ if $v=n-q+\rho$ or $\beta_v^{r+\rho}=0$ otherwise for $\rho=1,...,q$. 

We shall construct on each neighborhood $\mathcal{U}_i\times M_\epsilon$ a vector-valued holomorphic function of the form
\begin{align*}
\phi_i(z_i,t)&=(\theta_i^1(z_i,t),...,\theta_i^r(z_i,t),t_{n-q+1},...,t_n),\,\,\,\,\,\text{where $t=(t_1,...,t_n)$}
\end{align*}
satisfying the boundary conditions
\begin{align*}
\phi_i(z_i,0)&=0,\\
\left( \frac{\partial \phi_i(z_i,t)}{\partial t_v}\right)_{t=0}&=\beta_{vi}(z),\,\,\,\,\,v=1,...,n.
\end{align*}
such that 
\begin{align}
\phi_i(g_{ik}(z_k,\phi_k(z_k,t)),t)&=f_{ik}(z_k,\phi_k(z_k,t)),\label{5c}\\
[\Lambda,w_i^\alpha-\theta_i^\alpha(z_i,t)]&|_{w_i^\alpha=\theta_i^\alpha(z_i,t),u^\rho=t_{n-q+\rho}}=0,\,\,\,\alpha=1,...,r.\label{5d}
\end{align}
(Note that we have $[\Lambda, u^{\rho}-t_{n-q+\rho}]=0,\rho=1,...,q$ so that we do not need to consider them.) 

Recall Notation \ref{notation1}. Then (\ref{5c}) and (\ref{5d}) are equivalent to the system of congruences
\begin{align}
\phi_i^m(g_{ik}(z_k,\phi_k^m(z_k,t)),t)\equiv_{m} f_{ik}(z_k,\phi_k^m(z_k,t)),\,\,\,\,\,m=1,2,3,\cdots \label{5e}\\
[\Lambda, w_i^\alpha-\theta_i^{\alpha m}(z_i,t)]|_{w_i=\theta_i^{\alpha m}(z_i,t),u^\rho=t_{n-q+\rho}}\equiv_{m} 0,\,\,\,\,\,m=1,2,3,\cdots,\alpha=1,...,r\label{5f}
\end{align}
As in the proof of Theorem $\ref{33a}$, we will construct the formal power series 
\begin{align*}
\phi_i^m(z_i,t)=(\theta_i^{1 m}(z_i,t),\cdots, \theta_i^{r m}(z_i,t),t_{n-q+1},...,t_n)
\end{align*}
satisfying $(\ref{5e})_m$ and $(\ref{5f})_m$ by induction on $m$.

We define $\phi_{i}^1(z_i,t)=\sum_{v=1}^n \beta_{vi}(z)t_v$. Then $(\ref{5e})_1$ holds. On the other hand, since $[\Lambda, \beta_{vi}^\alpha(z)]|_{w_1=u=0}=\sum_{\beta=1}^{r+q} \beta_{vi}^\beta (z) T_{i\alpha}^\beta(z_i,0,0),\alpha=1,...,r$, $[\Lambda,w_i^\alpha-\theta_i^{\alpha 1}(z_i,t)]$ is of the form 
\begin{align*}
\sum_{\beta=1}^r (w_i^\beta-\theta_i^{\beta 1}(z_i,t)) T_{i\alpha}^\beta(z_i,w_i,u)+\sum_{\rho=1}^q (u^\rho-t_{n-q+\rho})T_{i\alpha}^{n-q+\rho}(z_i,w_i,u)+\sum_{\beta=1}^r w_i^\beta P_{i\alpha}^\beta(z_i,w_i,u,t)+\sum_{\rho=1}^q u^\rho P_{i\alpha}^{n-q+\rho}(z_i,w_i,u,t),
\end{align*}
where the degree of $P_{i\alpha}^\gamma(z_i,w_i,u,t),\gamma=1,...,r+q$ in $t$ is $1$ so that $(\ref{5f})_1$ holds.
Now we assume that we have already constructed $\phi_i^m(z_i,t)$ satisfying $(\ref{5e})_m$ and $(\ref{5f})_m$ such that $[\Lambda, w_i^\alpha-\theta_i^{\alpha m}(z_i,t)]$ is of the form (as in (\ref{33f}))
\begin{align}\label{53m}
[\Lambda, w_i^\alpha-\theta_i^{\alpha m}(z_i,t)]=\sum_{\beta=1}^r (w_i^\beta-\theta_i^{\beta m}(z_i,t)) T_{i\alpha}^\beta(z_i,w_i,u)+\sum_{\rho=1}^q (u^\rho-t_{n-q+\rho})T_{i\alpha}^{n-q+\rho}(z_i,w_i,u)+ Q_i^{\alpha m}(z_i,t)\\
+\sum_{\beta=1}^r w_i^\beta P_{i\alpha}^{\beta m}(z_i,w_i,u,t)+\sum_{\rho=1}^q u^\rho P_{i\alpha}^{(n-q+\rho) m}(z_i,w_i,u,t). \notag
\end{align}
such that the degree of $P_{i\alpha}^{\gamma m}(z_i,w_i,u,t),\gamma=1,...,r+q$ is at least $1$ in $t$. Let
\begin{align*}
\psi_{ik}(z,t)&\equiv_{m+1}\phi_i^m(g_{ik}(z_k, \phi_k^m(z_k,t),t)-f_{ik}(z_k,\phi_k^m(z_k,t))\\
G_i^\alpha(z_i,t)&\equiv_{m+1} [\Lambda, w_i^\alpha-\theta_i^{\alpha m}(z_i,t)]|_{w_i^\alpha=\theta_i^{\alpha m}(z_i,t),u^\rho=t_{n-q+\rho}}
\end{align*}
As in the proof of Theorem \ref{33a}, we can show that the collection $\{\psi_{ik}(z,t)\}\in C^1(\mathcal{U}\cap V, \mathcal{N}_{V/\mathcal{W}})$ and $\{(G_i^1(z_i,t),...,G_i^r(z_i,t),0,...,0)\}\in C^0(\mathcal{U}\cap V, \mathcal{N}_{V/\mathcal{W}}\otimes T_{\mathcal{W}}|_V)$ define a $1$-cocycle in the \v{C}ech resolution of $\mathcal{N}_{V/\mathcal{W}}^\bullet$.

We note that $\{\psi_{ik}(z,t)\}$ are in $C^1(\mathcal{U}\cap V,\mathcal{N}_{V/W})$ and $\{(G_i^1(z_i,t),...,G_i^r(z_i,t))\}$ are in $C^0(\mathcal{U}\cap V, \mathcal{N}_{V/W}\otimes T_W|_V)$ so that $(\{\psi_{ik}(z,t)\},\{(G_i^1(z_i,t),...,G_i^r(z_i,t))\})$ define an element in $\mathbb{H}^1(V,\mathcal{N}_{V/W}^\bullet)$. Since $\mathbb{H}^1(V,\mathcal{N}_{V/W}^\bullet)=0$, there exists $\{\chi_i(z,t)=(\chi_i^1(z,t),...,\chi_i^r(z,t))\}$ which are homogenous polynomials of degree $m+1$ in $t_1,...,t_n$ whose coefficients are in $C^0(\mathcal{U}\cap V, \mathcal{N}_{V/W})$ such that
\begin{align}
\psi_{ik}^\lambda(z,t)&=\sum_{\mu=1}^r a_{ik\mu}^\lambda(z)\chi_k^\mu(z,t)-\chi_i^\lambda(z,t),\,\,\,\,\,\lambda=1,...,r.\notag\\
-G_i^\alpha(z_i,t)&=-[ \chi_i^\alpha(z_i,t),\Lambda_0]|_{w_i=0}+\sum_{\beta=1}^r \chi_i^\beta(z_i,t) T_{i\alpha}^\beta(z_i,0,0),\,\,\,\alpha=1,...,r.\notag\\
                           &=-[ \chi_i^\alpha(z_i,t),\Lambda]|_{w_i=u=0}+\sum_{\beta=1}^r \chi_i^\beta(z_i,t) T_{i\alpha}^\beta(z_i,0,0)\label{53l}
\end{align}
We define $\phi_{i|m+1}(z_i,t)=(\chi_i^1(z_i,t),...,\chi_i^r(z_i,t),0,...,0)$. Then $\phi_i^m(z_i,t)=\phi_i^m(z_i,t)+\phi_{i|m+1}(z_i,t)$ satisfy $(\ref{5e})_{m+1}$. On the other hand, from (\ref{53l}), we have
\begin{align*}
-[ \chi_i^\alpha(z_i,t),\Lambda]=-G_i^\alpha(z_i,t)-\sum_{\beta=1}^r \chi_i^\beta(z_i,t) T_{i\alpha}^\beta(z_i,w_i,u)+\sum_{\beta=1}^r w_i^\beta R_{i\alpha}^\beta(z_i,w_i,u,t)+\sum_{\rho=1}^q u^{n-q+\rho} R_{i\alpha}^{q+\rho}(z_i,w_i,u,t)
\end{align*}
where the degree of $R_{i\alpha}^\gamma(z_i,w_i,u,t)$ is $m+1$ in $t$. Therefore we have, from $(\ref{53m})$, 
\begin{align*}
&[\Lambda, w_i^\alpha-\theta_i^{\alpha m}(z_i,t)-\chi_i^\alpha(z_i,t)]=\sum_{\beta=1}^r (w_i^\beta-\theta_i^{\beta m}(z_i,t)-\chi_i^\beta(z_i,t))T_{i\alpha}^\beta(z_i,w_i,u)+\sum_{\rho=1}^q(u^\rho-t_{n-q+\rho})T_{i\alpha}^{r+\rho}(z_i,w_i,u)\\
&-G_i^\alpha(z_i,t)+Q_i^{\alpha m}(z_i,t)+\sum_{\beta=1}^r w_i^\beta  (P_{i\alpha}^{\beta m}(z_i,w_i,u,t)+R_{i\alpha}^\beta(z_i,w_i,u,t))+\sum_{\rho=1}^q u^\rho (P_{i\alpha}^{(r+\rho) m}(z_i,w_i,u,t) +R_{i\alpha}^{n-q+\rho}(z_i,w_i,u,t))
\end{align*}
Lastly, we note that  from $(\ref{53m})$, we have 
\begin{align}\label{jj2}
G_i^\alpha(z_i,t)\equiv_{m+1} Q_i^{\alpha m}(z_i,t)+\sum_{\beta=1}^r \theta_i^{\beta m}(z_i,t) P_{i\alpha}^\beta(z_i,\phi_i^m(z_i,t),t)+\sum_{\rho=1}^q t_{n-q+\rho} P_{i\alpha}^{(n-q+\rho)m} (z_i,\phi_i^m(z_i,t),t)
\end{align}      
so that we obtain, from $(\ref{53m})$ and $(\ref{jj2})$,
\begin{align*}
&[\Lambda,w_i^\alpha-\theta_i^{\alpha m}(z_i,t)-\chi_i^\alpha(z_i,t)]|_{w_i^\alpha=\theta_i^{\alpha m}(z_i,t)+\chi_i^\alpha(z_i,t),u^\rho=t_{n-q+\rho}}\\
&\equiv_{m+1}[\Lambda, w_i^\alpha-\theta_i^{\alpha m}(z_i,t)]|_{w_i^\alpha=\theta_i^{\alpha m}(z_i,t)+\chi_i^\alpha(z_i,t),u^\rho=t_{n-q+\rho}}-[\Lambda, \chi_i^\alpha(z_i,t)]|_{w_i=u=0}\\
&\equiv_{m+1}\sum_{\beta=1}^r\chi_i^\alpha(z_i,t) T_{i\alpha}^\beta(z_i,0,0)+G_i^\alpha(z_i,t)-[\Lambda, \chi_i^\alpha(z_i,t)]|_{w_i=u=0}=0
\end{align*}
Hence $(\ref{5f})_{m+1}$ holds. By subsection \ref{3convergence}, the formal power series $\phi_i(z_i,t)$ converges for $|t|<\epsilon$ for sufficiently small positive number $\epsilon$. Let $M_\epsilon=\{t=(t_1,...,t_n)\in \mathbb{C}^n||t|<\epsilon\}$. Then on each neighborhood $\mathcal{U}_i\times M_\epsilon$ of $\mathcal{W}\times M_\epsilon$, the simultaneous equation $w_i^\alpha-\theta_i^\alpha(z_i,t)=u^\rho-t_{n-q+r}=0,\alpha=1,...,r,\rho=1,...,q$ defines the desired Poisson analytic family $\mathcal{F}$ of compact holomorphic Poisson submanifolds of $V_t,t\in M_\epsilon$ of $(\mathcal{W},\Lambda)$ such that $V_0=V$. This completes the proof of Theorem \ref{5theorem1}.
\end{proof}

\begin{proof}[Proof of Theorem \ref{stability}]
Let $N_\epsilon=\{u=(u^1,...,u^q)\in B||u|<\epsilon\}$, where $\epsilon$ is a small positive number. Let $\mathcal{F}\subset (\mathcal{W}\times M_\epsilon, \Lambda)$ be the Poisson analytic family of compact holomorphic Poisson submanifolds $V_t,t\in M_\epsilon$ of $(\mathcal{W},\Lambda)$ defined by $w_i^\lambda=\theta_i^\lambda(z_i,t),\lambda=1,2,...,r,u^\rho=t_{n-q+\rho},\rho=1,2,...,q$ on $\mathcal{U}_i\times M_\epsilon$ as in the proof of Theorem \ref{5theorem1}. If we define a linear map $u\to t(u)=(0,...,0,u^1,...,u^q)$ of $N_\epsilon$ into $M_\epsilon$, then the union $\mathcal{V}=\bigcup_u V_{t(u)}$ of the compact holomorphic Poisson submanifolds $V_{t(u)},u\in N_\epsilon$ of $(\mathcal{W},\Lambda)$ which is defined by $w_i^\lambda=\theta_i^\lambda(z_i,0,\cdots,0,u^1,\cdots, u^q):=\eta_i^\lambda(z_i,u^1,\cdots, u^q)$ on $\mathcal{U}_i$ forms a holomorphic Poisson fibre submanifold with compact fibres of the holomorphic Poisson fibre manifold $(\mathcal{W},\Lambda )|_{N_\epsilon}$ with $\mathcal{V}\cap W=V$ so that $V$ is a stable holomorphic Poisson submanifold of $(W,\Lambda_0)$.
\end{proof}

As in Theorem 3 in \cite{Kod63}, by combining the exact sequence (\ref{5g}) and Theorem \ref{stability}, we can show
\begin{theorem}
Let $(\mathcal{W},\Lambda, B,p)$ be a holomorphic Poisson fibre manifold such that $(W,\Lambda_0)=p^{-1}(0)$ is the fibre of $(\mathcal{W},\Lambda)$ over a point $0\in B$. Let $V$ be a compact holomorphic Poisson submanifold of $(W,\Lambda_0)$. Assume that $\mathbb{H}^1(V,\mathcal{N}_{W/V}^\bullet)=\mathbb{H}^0(V,\mathcal{N}_{W/V}^\bullet)=0$. Then, for a sufficiently small neighborhood $N$ of $0\in B$, there exists the unique holomorphic Poisson fibre submanifold $\mathcal{V}$ with compact fibers of the holomorphic Poisson fibre manifold $(\mathcal{W},\Lambda)|_N$ such that $\mathcal{V}\cap W=V$.
\end{theorem}

\begin{example}
We keep the notations in Example $\ref{example51}$. Let us consider a Poisson rational ruled surface $(F_2, z^2\xi\frac{\partial}{\partial z}\wedge \frac{\partial}{\partial \xi})$. We show that the holomorphic Poisson submanifold $\xi=\xi'=0$ of $(F_2, z^2\xi\frac{\partial}{\partial z}\wedge \frac{\partial}{\partial \xi})$ is unstable. Take two copies of $U_i\times \mathbb{P}_\mathbb{C}^1\times \mathbb{C}$, where $U_i=\mathbb{C}$ and write the coordinates as $(z,[\xi_0,\xi_1],t)$ and $(z',[\xi_0',\xi_1'],t')$. Patch $U_i\times \mathbb{P}_\mathbb{C}^1\times \mathbb{C},i=1,2$ by the relation $z'=\frac{1}{z},t=t'$ and $[\xi_0',\xi_1']=[\xi_0,z^2 \xi_1+tz \xi_0]$ and denote it by $X$. Then the projection $\pi:X\to \mathbb{C}$ define a complex analytic family of deformations of $F_2$. We give a holomorphic Poisson structure on $X$ to make a Poisson analytic family. We set $\xi=\frac{\xi_1}{\xi_0}$ and $\xi'=\frac{\xi_1'}{\xi_0'}$. Since $-\xi'\frac{\partial}{\partial z'}\wedge\frac{\partial}{\partial \xi'}=(z^2\xi+tz)\frac{\partial}{\partial z}\wedge \frac{\partial}{\partial \xi}$, $\pi:(X,\Lambda=(z^2\xi+tz)\frac{\partial}{\partial z}\wedge \frac{\partial}{\partial \xi})\to \mathbb{C}$ defines a Poisson analytic family. Since $\xi=\xi'=0$ can not be extended to a complex analytic family as in \cite{Kod63} p.86, it can not be extended to a Poisson analytic family as a  holomorphic Poisson fibre submanifold of $(X,\Lambda)$ so that it is not stable.

\end{example}

\begin{example}
Let us consider $F_0\cong \mathbb{P}_\mathbb{C}^1\times \mathbb{P}_\mathbb{C}^1$ and a Poisson structure $\Lambda_0=\xi \frac{\partial}{\partial z}\wedge \frac{\partial}{\partial \xi}$ on $F_0$. We keep the notations in Example $\ref{example51}$. We show that the holomorphic Poisson submanifold $V:\xi=0$ is unstable. Let us consider a Poisson analytic family $(F_0\times \mathbb{C}, \Lambda=(\xi-tz)\frac{\partial}{\partial z}\wedge \frac{\partial}{\partial \xi})$. Assume that there is a holomorphic Poisson fibre manifold of $\mathcal{V}\subset (F_0\times B,\Lambda)$, where $B=\{t\in \mathbb{C}||t|<\epsilon\}$ for a sufficiently small number $\epsilon>0$ such that $\mathcal{V}|_{t=0}$ is $\xi=\xi'=0$. We may assume that $\mathcal{V}$ is defined by $\xi-\varphi_1(z,t)=0$ on $U_1\times \mathbb{P}_\mathbb{C}^1\times \mathbb{C}$ and $\xi'-\varphi_2(z',t)=0$ on $U_2\times \mathbb{P}_\mathbb{C}^1\times \mathbb{C}$, where $U_i=\mathbb{C},i=1,2$ so that we have a relation $\varphi_1(z,t)-\varphi_2(\frac{1}{z},t)=0$. Hence $\varphi_i(z,t)$ is of the from $\varphi_1(z,t)=f(t)$. On the other hand, since $\xi-\varphi_1(z,t)$ defines a holomorphic Poisson submanifold of $F_0\times \mathbb{C}$, $[\Lambda,\xi-\varphi_1(z,t)]|_{\xi=\varphi_1(z,t)}=0$ so that $-(\xi-tz)\frac{\partial}{\partial z}|_{\xi-\varphi_1(z,t)}=0\iff \varphi_1(z,t)=tz$ which contradicts to $\varphi_1(z,t)=f(t)$. Hence $V:\xi=0$ is not stable as a holomorphic Poisson submanifold while it is stable as a complex submanifold since $H^1(V,\mathcal{N}_{V/F_0})\cong H^1(\mathbb{P}_\mathbb{C}^1,\mathcal{O}_{\mathbb{P}_\mathbb{C}^1})=0$.
\end{example}

\appendix
\section{Deformations of Poisson structures}\label{appendixa}
We denote by $\bold{Art}$ the category of local artinian $k$-algebras with residue field $k$, where $k$ is an algebraically closed field with characteristic $0$, and by $k[\epsilon]$ by the ring of dual numbers.

\begin{definition}
Let $(Y,\Lambda_0)$ be a nonsingular Poisson variety. An infinitesimal deformation of $\Lambda_0$ over $A\in \bold{Art}$ is an algebraic Poisson scheme $(Y\times_{Spec(k)} A,\Lambda)$ which induces $(Y,\Lambda_0)$, where $\Lambda\in H^0(Y,\wedge^2 T_Y)\otimes A$. Then for each $A\in \bold{Art}$, we can define a functor of Artin rings
\begin{align*}
Def_{\Lambda_0}:\bold{Art}&\to (sets)\\
A&\mapsto \{\text{infinitesimal deformations of $\Lambda_0$ over $A$}\}
\end{align*}
We will denote by $\mathbb{H}^i(Y,\wedge^2 T_Y^{\bullet-1})$ the $i$-th hypercohomology group of the following complex of sheaves
\begin{align}
\wedge^2 T_Y^{\bullet-1}: \wedge^2 T_Y\xrightarrow{-[-,\Lambda_0]}\wedge^3 T_Y\xrightarrow{-[-,\Lambda_0]}\wedge^4 T_Y\xrightarrow{-[-,\Lambda_0]}\cdots
\end{align}
\end{definition}
\begin{proposition}
Let $(Y,\Lambda_0)$ be a nonsingular Poisson variety. Then
\begin{enumerate}
\item There is a natural identification $Def_{\Lambda_0}(k[\epsilon])\cong\mathbb{H}^0(Y, \wedge^2 T_Y^{\bullet-1})$.
\item Given an infinitesimal deformation $\eta$ of $\Lambda_0$ over $A\in \bold{Art}$ and a small extension $0\to(t)\to \tilde{A}\to A\to 0$, we can associate  an element $o_\eta(e)\in \mathbb{H}^1(Y,\wedge^2 T_Y^{\bullet-1})$, which is zero if and only if there is a lifting of $\eta$ to $\tilde{A}$.
\end{enumerate}
\end{proposition}
\begin{proof}
Let $\Lambda\in H^0(Y,\wedge^2 T_Y)\otimes k[\epsilon]$ be an infinitesimal deformation of $\Lambda_0$ over $Spec(k[\epsilon])$ so that $\Lambda=\Lambda_0+\epsilon \Lambda'$ for some $\Lambda'\in H^0(Y,\wedge^2 T_Y)$. Since $[\Lambda_0+\epsilon \Lambda',\Lambda_0+\epsilon \Lambda']=0$ so that $-[\Lambda',\Lambda_0]=0$. Hence $\Lambda'\in \mathbb{H}^0(Y,\wedge^2 T_Y^{\bullet-1})$.

Now we identify obstructions. Consider a small extension $e:0\to (t)\to \tilde{A}\to A\to 0$. Let $\eta$ be an infinitesimal deformations of $\Lambda_0$ over $A$, in other words, a Poisson structure $\Lambda\in H^0(Y,\wedge^2 T_Y)\otimes A$ on $Y\times_{Spec(k)} A$ over $A$. Let $\mathcal{U}=\{U_i\}$ be an affine open covering of $Y$.  Then $\Lambda$ is locally expressed as $\Lambda_i\in \Gamma(U_i,\wedge^2 T_Y)\otimes A$ with $[\Lambda_i,\Lambda_i]=0$. Let $\tilde{\Lambda}_i\in H^0(Y,\wedge^2 T_Y)\otimes \tilde{A}$ be an arbitrary lifting of $\Lambda_i$ to $\tilde{A}$ so that $[\tilde{\Lambda}_i,\tilde{\Lambda}_i]=t \Pi_i$ for some $\Pi_i\in \Gamma(U_i,\wedge^3 T_Y)$ and $\tilde{\Lambda}_i-\tilde{\Lambda}_j= t\Lambda_{ij}'$ for some $\Lambda_{ij}'\in \Gamma(U_i\cap U_j,\wedge^2 T_Y)$. Then we have 
\begin{align}
&t[\Lambda_0,\Pi_i]=[\tilde{\Lambda}_i,[\tilde{\Lambda}_i,\tilde{\Lambda}_i]]=0 \iff -[\frac{1}{2}\Pi_i,\Lambda_0]=0\label{a10}\\
&(\frac{1}{2}\Pi_i-\frac{1}{2}\Pi_j)=\frac{1}{2}[\tilde{\Lambda}_i,\tilde{\Lambda}_i]-\frac{1}{2}[\tilde{\Lambda}_j,\tilde{\Lambda}_j]=t[\Lambda_0, \Lambda_{ij}'] \iff \delta(\frac{1}{2}\Pi_i)-[-\Lambda_{ij}',\Lambda_0]=0 \label{a11}\\
 &t(\Lambda_{jk}'-\Lambda_{ik}'+\Lambda_{ij}')=\tilde{\Lambda}_j-\tilde{\Lambda}_k-\tilde{\Lambda}_i+\tilde{\Lambda}_k+\tilde{\Lambda}_i-\tilde{\Lambda}_j=0\iff
-\delta(-\Lambda_{ij}')=0.\label{a12}
\end{align}
Hence $(\{\frac{1}{2}\Pi_i\},\{ -\Lambda_{ij}'\})\in C^0(\mathcal{U}, \wedge^3 T_Y)\oplus C^1(\mathcal{U},\wedge^2 T_Y)$ define a $1$-cocycle in the following \v{C}ech resolution of $\wedge^2 T_X^{\bullet-1}$:
\begin{center}
$\begin{CD}
C^0(\mathcal{U},\wedge^4 T_Y)\\
@A-[-,\Lambda_0]AA \\
C^0(\mathcal{U},\wedge^3 T_Y)@>\delta>> C^1(\mathcal{U},\wedge^3 T_Y)\\
@A-[-,\Lambda_0]AA @A-[-,\Lambda_0]AA\\
C^0(\mathcal{U},\wedge^2 T_Y)@>-\delta>> C^1(\mathcal{U},\wedge^2 T_Y)@>\delta>> C^2(\mathcal{U},\wedge^2 T_Y)
\end{CD}$
\end{center} 
Now we choose another arbitrary lifting $\tilde{\Lambda}_i'\in\Gamma(U_i, \wedge^2 T_Y)\otimes \tilde{A}$ of $\Lambda_i$. We show that the associated cohomology class $b:=(\{\frac{1}{2}\Pi_i'\},\{-\Lambda_{ij}''\})$ is cohomologous to $a:=(\{\frac{1}{2}\Pi_i\}, \{-\Lambda_{ij}'\})$. We note that $\tilde{\Lambda}_i'=\tilde{\Lambda}_i+tD_i$ for some $D_i\in \Gamma(U_i, \wedge^2 T_Y)$. Then
\begin{align}
t\frac{1}{2}\Pi_i'-t\frac{1}{2}\Pi_i=\frac{1}{2}[\tilde{\Lambda}_i',\tilde{\Lambda}_i']-\frac{1}{2}[\tilde{\Lambda}_i,\tilde{\Lambda}_i]=[tD_i,\Lambda_0]\iff \frac{1}{2}\Pi_i-\frac{1}{2}\Pi'=-[D_i,\Lambda_0]\label{aa1}\\
t\Lambda_{ij}''-t\Lambda_{ij}'=\tilde{\Lambda}_i'-\tilde{\Lambda}_j'-\tilde{\Lambda}_i+\tilde{\Lambda}_j=t(D_i-D_j)\iff -\Lambda_{ij}'-(-\Lambda_{ij}'')=-\delta(D_i)\label{aa2}
\end{align}
Hence $\{D_i\}\in C^0(\mathcal{U},\wedge^2 T_Y)$ is mapped to $a-b$ so that $a$ is cohomologous to $b$. So given a small extension $e:0\to (t)\to \tilde{A}\to A\to 0$, we can associate an element $o_\eta(e):=$ the cohomology class of $a\in \mathbb{H}^1(Y,\wedge^2 T_Y^{\bullet-1})$. We note that $o_\eta(e)=0$ if and only if there exists a collection $\{\tilde{\Lambda}_i\}$ such that $\Pi_i=0$ (which means $[\tilde{\Lambda}_i,\tilde{\Lambda}_i]=0$) and  $\Lambda_{ij}'=0$ (which means $\{\tilde{\Lambda}_i\}$ glues together to define a Poisson structure on $Y\times Spec(\tilde{A})$) if and only if there is a lifting of $\eta$ to $\tilde{A}$.

\end{proof}

\begin{remark}
We have an exact sequence of complex of sheaves $0\to \wedge^2 T_X^{\bullet -1}\to T_X^\bullet\to T_X\to 0:$
\begin{center}
$\begin{CD}
@. \cdots @. \cdots @.\cdots\\
@.@A-[-.\Lambda_0]AA @A-[-,\Lambda_0]AA @AAA\\
0@>>> \wedge^3 T_X@>>> \wedge^3 T_X@>>> 0@>>>0\\
@. @A-[-,\Lambda_0]AA @A-[-,\Lambda_0]AA @AAA\\
0@>>>\wedge^2 T_X@>>> \wedge^2 T_X@>>> 0@>>>0\\
@.@AAA @A-[-,\Lambda_0]AA @AAA\\
0@>>> 0@>>> T_X@>>> T_X@>>>0
\end{CD}$
\end{center}
which induces 
\begin{align}
\mathbb{H}^0(X,\wedge^2 T_X^{\bullet-1})\to \mathbb{H}^1 (X,T_X^\bullet)\to H^1(X,T_X) \label{a1}\\
\mathbb{H}^1(X,\wedge^2 T_X^{\bullet -1})\to \mathbb{H}^2(X,T_X^\bullet)\to \mathbb{H}^2(X,T_X)\label{a2}
\end{align}

We also have morphisms of deformation functors
\begin{align}
Def_{\Lambda_0}\to Def_{(X,\Lambda_0)}\to Def_X \label{a3}
\end{align}
where $Def_{(X,\Lambda_0)}$ is the functor of flat Poisson deformations of $(X,\Lambda_0)$ $($see \cite{Kim16}$)$ and $Def_X$ is the functor of flat deformations of $X$ $($see \cite{Ser06} $p.64$$)$. Then $(\ref{a1})$ represents the morphisms of tangent spaces for $(\ref{a3})$ $:Def_{\Lambda_0}(k[\epsilon])\to Def_{(X,\Lambda_0)}(k[\epsilon])\to Def_X(k[\epsilon])$, and $(\ref{a2})$ represents the obstruction maps for $(\ref{a3})$.
\end{remark}

\section{Deformations of Poisson closed subschemes}\label{appendixb}

\subsection{The local Poisson Hilbert functor}\

Let $X\subset (Y,\Lambda_0)$ be a closed embedding of algebraic Poisson schemes, where $(Y,\Lambda_0)$ is a nonsingular Poisson variety. An infinitesimal deformation of $X$ in $(Y,\Lambda_0)$ over $A\in \bold{Art}$ is a cartesian diagram of morphisms of schemes
\begin{center}
$\begin{CD}
X@>>> \mathcal{X}\subset (Y\times Spec(A), \Lambda_0)\\
@VVV @VV\pi V\\
Spec(k)@>>> S=Spec(A)
\end{CD}$
\end{center}
where $\pi$ is flat and induced by a projection from $Y\times S$ to $S$, and $\mathcal{X}$ is a Poisson closed subscheme of $(Y\times S,\Lambda_0)$. Then we can define a functor of Artin rings (called the local Poisson Hilbert functor of $X$ in $(Y,\Lambda_0)$)
\begin{align*}
H_X^{(Y,\Lambda_0)}:\bold{Art}&\to (Sets)\\
A&\mapsto \{\text{ infinitesimal deformations of $X$ in $(Y,\Lambda_0)$ over $A$}\}
\end{align*}

\subsection{The complex associated with the normal bundle of a Poisson closed subscheme of a nonsingular Poisson variety}\label{b1}\

Let $(Y,\Lambda_0)$ be a nonsingular Poisson variety and $X$ be a Poisson closed subscheme of $(Y,\Lambda)$ defined by a Poisson ideal sheaf $\,\mathcal{I}$. Assume that $i:X\hookrightarrow Y$ be a regular embedding. Let $\{U_i\}$ be an affine open cover of $Y$ such that $I_i=(f_i^1,...,f_i^N)$ be a Poisson ideal of $\Gamma(U_i,\mathcal{O}_Y)$ defining $X\cap U_i$ and $\{f_i^1,...,f_i^N\}$ is a regular sequence. Since $(f_i^1,...,f_i^N)=(f_j^1,...,f_j^N)$, $f_i^\alpha=\sum_{\beta=1}^N r_{ij\beta}^\alpha f_j^\beta$ for some $r_{ij\beta}^\alpha\in \Gamma(U_i\cap U_j,\mathcal{O}_Y)$. Since $I_i/I_i^2$ is free $\Gamma(U_i\cap X,\mathcal{O}_X)$-module and generated by $\{f_i^\alpha+I_i^2\},\alpha=1,...,N$, $\bar{r}_{ij\alpha}^\beta$ is uniquely determined, where $\bar{r}_{ij\beta}^\alpha$ is the restriction of $r_{ij\beta}^\alpha$ to $\Gamma(U_i\cap U_j\cap X,\mathcal{O}_X)$. Then the normal sheaf $\mathcal{N}_{X/Y}:=\mathscr{H}om_{\mathcal{O}_X}(\mathcal{I}/\mathcal{I}^2,\mathcal{O}_X)$ is locally described in the following way: $\textnormal{Hom}_{\Gamma(U_i\cap X,\mathcal{O}_X)}(I_i/I_i^2,\Gamma(U_i\cap X,\mathcal{O}_X)\cong \oplus^N \Gamma(U_i\cap X,\mathcal{O}_X),\phi\mapsto (\phi(\bar{f}_i^1),...,\phi(\bar{f}_i^N))$, where $\bar{f}_i^\alpha$ is the image of $f_i^\alpha\in I_i$ in $I_i/I_i^2$, and on $U_i\cap U_j$,  $(g_j^1,...,g_j^N)\in \oplus^N \Gamma(U_j\cap X,\mathcal{O}_X)$ is identified with $(\sum_{\beta=1}^N \bar{r}_{ij\beta}^1 g_j^\beta,...,\sum_{\beta=1}^N \bar{r}_{ij\beta}^N g_j^\beta)\in \oplus^N \mathcal{O}_X(U_i\cap X)$.

On the other hand, $(f_i^1,...,f_i^N)$ is a Poisson ideal, in other words, $\{f_i^\alpha,\mathcal{O}_Y\}\subset (f_i^1,...,f_i^N),\alpha=1,...,N$ so that $[\Lambda_0, f_i^\alpha]=\sum_{\beta=1}^n f_i^\beta T_{i\alpha}^\beta $ for some $T_{i\alpha}^\beta\in \Gamma(U_i,T_Y)$. Let $\bar{T}_{i\alpha}^\beta$ be the image of $T_{i\alpha}^\beta$ in $\Gamma(U_i\cap X,T_Y|_X)$. Then we have
\begin{enumerate}
\item We note that $\sum_{\beta,\gamma=1}^N  r_{ij\gamma}^\beta f_j^\gamma T_{i\alpha}^\beta=\sum_{\beta=1}^N f_i^\beta T_{i\alpha}^\beta=[\Lambda_0, f_i^\alpha]=[\Lambda_0,\sum_{\beta=1}^N r_{ij\beta}^\alpha f_j^\beta]=\sum_{\beta=1}^N [\Lambda_0, r_{ij\beta}^\alpha]f_j^\beta+\sum_{\beta=1}^N r_{ij\beta}^\alpha[\Lambda_0, f_j^\beta] =\sum_{\beta=1}^N [\Lambda_0, r_{ij\beta}^\alpha]f_j^\beta+\sum_{\beta,\gamma=1}^N r_{ij\beta}^\alpha f_j^\gamma T_{j\beta}^\gamma $.  Then $\sum_{\gamma=1}^N f_j^\gamma(\sum_{\beta=1}^N r_{ij\gamma}^\beta T_{i\alpha}^\beta-[\Lambda_0, r_{ij\gamma}^\alpha]-\sum_{\beta=1}^N r_{ij\beta}^\alpha T_{j\beta}^\gamma)=0$ so that we get 
\begin{align}\label{b2}
\sum_{\beta=1}^N \bar{r}_{ij\gamma}^\beta \bar{T}_{i\alpha}^\beta =\overline{[\Lambda_0, r_{ij\gamma}^\alpha]}+\sum_{\beta=1}^N \bar{r}_{ij\beta}^\alpha\bar{T}_{j\beta}^\gamma
\end{align}
where $\overline{[\Lambda_0, r_{ij\gamma}^\alpha]}$ is the image of $[\Lambda_0, r_{ij\gamma}^\alpha]$ in $\Gamma(U_i\cap X, T_Y|_X)$.
\item By taking $[\Lambda_0,-]$ on $[\Lambda_0,f_i^\alpha]=\sum_{\beta=1}^N f_i^\beta T_{i\alpha}^\beta$, we get $\sum_{\beta=1}^N f_i^\beta [\Lambda_0, T_{i\alpha}^\beta]-\sum_{\beta=1}^N [\Lambda_0, f_i^\beta]\wedge T_{i\alpha}^\beta=0$. Then $\sum_{\gamma=1}^N f_i^\gamma[\Lambda_0, T_{i\alpha}^\gamma]-\sum_{\beta,\gamma=1}^N f_i^\gamma T_{i\beta}^\gamma \wedge T_{i\alpha}^\beta =\sum_{\gamma=1}^N f_i^\gamma([\Lambda_0,T_{i\alpha}^\gamma]-\sum_{\beta=1}^N T_{i\beta}^\gamma \wedge T_{i\alpha}^\beta )=0$ so that we get 
\begin{align}\label{b3}
\overline{[\Lambda_0, T_{i\alpha}^\gamma]}-\sum_{\beta=1}^N \bar{T}_{i\beta}^\gamma\wedge \bar{T}_{i\alpha}^\beta=0,
\end{align}
where $\overline{[\Lambda_0, T_{i\alpha}^\gamma]}$ is the image of $[\Lambda_0, T_{i\alpha}^\gamma]$ in $\Gamma(U_i\cap X,\wedge^2T_Y|_X)$. 
\end{enumerate}
Now we define the complex $\mathcal{N}_{X/Y}^\bullet$ associated with the normal bundle $\mathcal{N}_{X/Y}$:
\begin{align}
\mathcal{N}_{X/Y}^\bullet:\mathcal{N}_{X/Y}\xrightarrow{\nabla} \mathcal{N}_{X/Y}\otimes T_Y|_X\xrightarrow{\nabla} \mathcal{N}_{X/Y}\otimes \wedge^2 T_Y|_X\xrightarrow{\nabla}\cdots
\end{align}
The complex is defined locally in the following way:
\begin{align*}
\nabla:\oplus^N \Gamma(U_i\cap X, \wedge^p T_Y|_X) &\to \oplus^N \Gamma(U_i\cap X,\wedge^{p+1} T_Y|_X)\\
(g_i^1,...,g_i^N)&\mapsto (-\overline{[g_i^1,\Lambda_0]}+(-1)^p\sum_{\beta=1}^N g_i^\beta \wedge \bar{T}_{i1}^\beta,...,-\overline{[g_i^N,\Lambda_0]}+(-1)^p\sum_{\beta=1}^N g_i^\beta \wedge \bar{T}_{iN}^\beta)
\end{align*}
We denote the $i$-th hypercohomology  group of $\mathcal{N}_{X/Y}^\bullet$ by $\mathbb{H}^i(X,\mathcal{N}_{X/Y}^\bullet)$.

\begin{proposition}[compare \cite{Ser06} Proposition $3.2.1$ and Proposition $3.2.6$] \label{b15}
Given a regular closed embedding of algebraic Poisson schemes $i:X\hookrightarrow (Y,\Lambda_0)$, where $(Y,\Lambda_0)$ is a nonsingular Poisson variety, then
\begin{enumerate}
\item There is a natural identification
\begin{align*}
H_X^{(Y,\Lambda_0)}(k[\epsilon])\cong \mathbb{H}^0(X, \mathcal{N}_{X/Y}^\bullet)
\end{align*}
\item Given an infinitesimal deformation $\eta$ of $X$ in $(Y,\Lambda_0)$ over $A\in \bold{Art}$ and a small extension $e:0\to (t)\to\tilde{A}\to A\to 0$, we can associate an element $o_\eta(e)\in \mathbb{H}^1(X,\mathcal{N}_{X/Y}^\bullet)$, which is zero if and only if there is a lifting of $\eta$ to $\tilde{A}$.
\end{enumerate}
\end{proposition}

\begin{proof}
Let $\mathcal{U}=\{U_i\}$ be an affine open covering of $Y$ and let $I_i=(f_i^1,...,f_i^N)$ be a Poisson ideal of $\Gamma(U_i,\mathcal{O}_Y)$ defining $U_i\cap X$  such that $\{f_i^1,\cdots f_i^N\}$ is a regular sequence. We keep the notations in subsection \ref{b1}. 

A first-order deformation of $X$ in $(Y,\Lambda_0)$ is a flat family
\begin{center}
$\begin{CD}
X@>>> \mathcal{X}\subset (Y\times Spec(k[\epsilon]),\Lambda_0)\\
@VVV @VVV\\
Spec(k)@>>> Spec(k[\epsilon])
\end{CD}$
\end{center}
so that $\mathcal{X}$ is determined by a Poisson ideal sheaf $\mathcal{I}$ generated by $\{f_i^\alpha+\epsilon g_i^\alpha\},\alpha=1,...,N$ for some $g_i^\alpha\in \Gamma(U_i,\mathcal{O}_Y)$. Let $(\bar{g}_i^1,...,\bar{g}_i^N)\in \oplus^N \Gamma(X\cap U_i,\mathcal{O}_X)$ be the image of $(g_i^1,...,g_i^N)\in \oplus^N \Gamma(U_i,\mathcal{O}_Y)$. Since $(f_i^1+\epsilon g_i^1,...,f_i^N+\epsilon g_i^N)=(f_j^1+\epsilon g_j^1,...,f_i^N+\epsilon g_i^N)$, $f_i^\alpha+\epsilon g_i^\alpha=\sum_{\beta=1}^N (r_{ij\beta}^\alpha+\epsilon h_{ij\beta}^\alpha)(f_j^\beta+\epsilon g_j^\beta)$ for some $h_{ij\beta}^\alpha\in \Gamma(U_i\cap U_j, \mathcal{O}_Y)$ so that  we have $\bar{g}_i^\alpha=\sum_{\beta=1}^n \bar{r}_{ij\beta}^\alpha \bar{g}_i^\beta$. Hence $\{(\bar{g}_i^1,...,\bar{g}_i^N)\}$ define a global section in $H^0(X,\mathcal{N}_{Y/X})$. On the other hand, since $(f_i^1+\epsilon g_i^1,...,f_i^N+\epsilon g_i^N)$ is a Poisson ideal, we have $[\Lambda_0, f_i^\alpha+\epsilon g_i^\alpha]=\sum_{\beta=1}^N (f_i^\beta+\epsilon g_i^\beta)(T_{i\alpha}^\beta+\epsilon W_{i\alpha}^\beta)$ for some $W_{i\alpha}^\beta\in \Gamma(U_i,T_Y)$. Then $\overline{[\Lambda_0,g_i^\alpha]}=\sum_{\beta=1}^N \bar{g}_i^\beta \bar{T}_{i\alpha}^\beta$ so that $\nabla(\{(\bar{g}_i^1,...,\bar{g}_i^N)\})=0$. Hence $\{(\bar{g}_i^1,...,\bar{g}_i^N)\}\in \mathbb{H}^0(X,\mathcal{N}_{X/Y}^\bullet)$.

Now we identify obstructions. Consider a small extension $e:0\to \tilde{A}\to A\to 0$. Let $\eta:=(\mathcal{X}\subset (Y\times Spec(A),\Lambda_0))$ be an infinitesimal deformation of $X$ in $(Y,\Lambda_0)$ over $A$. Then $\mathcal{X}$ is determined by a Poisson ideal sheaf $\mathcal{I}_A$ generated by $(F_i^1,...,F_i^N)$ in $\Gamma(U_i,\mathcal{O}_Y)\otimes A$ such that $F_i^\alpha\equiv f_i^\alpha\otimes 1 \mod (\mathfrak{m}_A)$ and $\{ F_i^1,...,F_i^N\}$ is a regular sequence. Since $(F_i^1,...,F_i^N)=(F_j^1,...,F_j^N)$, we have $F_i^\alpha=\sum_{\beta=1}^N R_{ij\beta}^\alpha F_j^\beta$ for some $R_{ij\beta}^\alpha\in \Gamma(U_i,\mathcal{O}_Y)\otimes A$. On the other hand, since $(F_i^1,...,F_i^N)$ is a Poisson ideal,  we have $[\Lambda_0, F_i^\alpha]=\sum_{\beta=1}^N F_i^\beta W_{i\alpha}^\beta $ for some $W_{i\alpha}^\beta\in \Gamma(U_i,T_Y)\otimes A$. Let $\tilde{F}_i^\alpha\in \Gamma(U_i, \mathcal{O}_Y)\otimes A$ be an arbitrary lifting of $F_i^\alpha$, $\tilde{T}_{i\alpha}^\beta\in \Gamma(U_i, T_Y)\otimes \tilde{A}$ be an arbitrary lifting of $W_{i\alpha}^\beta$, and $\tilde{R}_{ij\alpha}^\beta\in \Gamma(U_i\cap U_j, \mathcal{O}_Y)\otimes \tilde{A}$ be an arbitrary lifting of $R_{ij\beta}^\alpha$. Then $[\Lambda_0,\tilde{F}_i^\alpha]-\sum_{\beta=1}^N \tilde{F}_i^\beta \tilde{T}_{i\alpha}^\beta =t G_i^\alpha$ for some $G_i^\alpha\in \Gamma(U_i,T_Y)$, $\tilde{F}_i^\alpha-\sum_{\beta=1}^N \tilde{R}_{ij\beta}^\alpha \tilde{F}_j^\beta=t h_{ij}^\alpha$ for some $h_{ij}^\alpha\in \Gamma(U_i\cap U_j,\mathcal{O}_Y)$, and $\tilde{R}_{ik\gamma}^\alpha-\sum_{\beta=1}^N \tilde{R}_{ij\beta}^\alpha \tilde{R}_{jk\gamma}^\beta=tP_{ijk\gamma}^\alpha$ for some $P_{ijk\gamma}^\alpha\in \Gamma(U_i\cap U_j\cap U_k,\mathcal{O}_Y)$. We will show that $\{(-\bar{G}_i^1,...,-\bar{G}_i^N)\}\oplus\{ (\bar{h}_{ij}^1,..., \bar{h}_{ij}^N)\}\in C^0(\mathcal{U},\mathcal{N}_{X/Y}\otimes T_Y|_X)\oplus C^1(\mathcal{U}, \mathcal{N}_{X/Y})$ define a $1$-cocycle in the following \v{C}ech resolution of $\mathcal{N}_{X/Y}^\bullet$:
\begin{center}
$\begin{CD}
C^0(\mathcal{U}\cap X,\mathcal{N}_{X/Y}\otimes \wedge^2 T_{Y}|_X)\\
@A\nabla AA \\
C^0(\mathcal{U}\cap X,\mathcal{N}_{X/Y}\otimes T_{Y}|_X)@>\delta>> C^1(\mathcal{U}\cap X,\mathcal{N}_{X/Y}\otimes T_Y|_X)\\
@A\nabla AA @A\nabla AA\\
C^0(\mathcal{U}\cap X,\mathcal{N}_{X/Y})@>-\delta>> C^1(\mathcal{U}\cap X,\mathcal{N}_{X/Y})@>\delta>> C^2(\mathcal{U}\cap X,\mathcal{N}_{X/Y})
\end{CD}$
\end{center}

First we show that $\nabla(\{(-\bar{G}_i^1,...,-\bar{G}_i^N)\})=0$. As in (\ref{b3}), we can show $\sum_{\gamma=1}^N F_i^\gamma([\Lambda_0, W_{i\alpha}^\gamma]-\sum_{\beta=1}^N W_{i\beta}^\gamma\wedge W_{i\alpha}^\beta)=0$ so that we have $\sum_{\gamma=1}^N \tilde{F}_i^\gamma([\Lambda_0,\tilde{T}_{i\alpha}^\gamma]-\sum_{\beta=1}^N \tilde{T}_{i\beta}^\gamma\wedge \tilde{T}_{i\alpha}^\beta)=\sum_{\gamma=1}^N \tilde{F}_i^\gamma tQ_{i\alpha}^\gamma=\sum_{\gamma=1}^N f_i^\gamma tQ_{i\alpha}^\gamma$ for some $Q_{i\alpha}^\gamma\in \Gamma(U_i, \wedge^2 T_Y)$. Then we have
\begin{align}\label{b4}
&t[\Lambda_0,G_i^\alpha]=-\sum_{\beta=1}^N[\Lambda_0, \tilde{F}_i^\beta \tilde{T}_{i\alpha}^\beta ]=-\sum_{\beta=1}^N \tilde{F}_i^\beta[\Lambda_0,\tilde{T}_{i\alpha}^\beta] +\sum_{\beta=1}^N [\Lambda_0, \tilde{F}_i^\beta]\wedge \tilde{T}_{i\alpha}^\beta\\
&=-\sum_{\gamma=1}^N \tilde{F}_i^\gamma[\Lambda_0,\tilde{T}_{i\alpha}^\gamma]+\sum_{\beta=1}^N tG_i^\beta\wedge T_{i\alpha}^\beta+\sum_{\beta,\gamma=1}^N\tilde{F}_i^\gamma \tilde{T}_{i\beta}^\gamma \wedge \tilde{T}_{i\alpha}^\beta=-\sum_{\gamma=1}^N \tilde{F}_i^\gamma([\Lambda_0, \tilde{T}_{i\alpha}^\gamma]-\tilde{T}_{i\beta}^\gamma\wedge \tilde{T}_{i\alpha}^\beta)+\sum_{\beta=1}^N tG_i^\beta\wedge T_{i\alpha}^\beta\notag
\end{align}
By taking $-$ on (\ref{b4}), we obtain
\begin{align}\label{b11}
t\overline{[\Lambda_0, G_i^\alpha]}=-\sum_{\gamma=1}^N \bar{f}_i^\gamma t\bar{Q}_{i\alpha}^\gamma +\sum_{\beta=1}^N t\bar{G}_i^\beta\wedge \bar{T}_{i\alpha}^\beta=\sum_{\beta=1}^N t\bar{G}_i^\beta\wedge \bar{T}_{i\alpha}^\beta \iff -\overline{[ -G_i^\alpha,\Lambda_0]}+(-1)^1\sum_{\beta=1}^N -\bar{G}_i^\beta\wedge \bar{T}_{i\alpha}^\beta=0
\end{align}

Next we show that $\delta(\{(-\bar{G}_i^1,...,-\bar{G}_i^N)\})+\nabla(\{\bar{h}_{ij}^1,...,\bar{h}_{ij}^N\})=0$. We have 
\begin{align}\label{b6}
t(G_i^\alpha-\sum_{\beta=1}^N r_{ij\beta}^\alpha G_j^\beta )=[\Lambda_0, \tilde{F}_i^\alpha]-\sum_{\beta=1}^N \tilde{F}_i^\beta\tilde{T}_{i\alpha}^\beta -\sum_{\beta=1}^N \tilde{R}_{ij\beta}^\alpha [\Lambda_0,\tilde{F}_j^\beta]+\sum_{\beta,\gamma=1}^N \tilde{R}_{ij\beta}^\alpha \tilde{F}_j^\gamma \tilde{T}_{j\beta}^\gamma 
\end{align}
On the other hand, we have
\begin{align}\label{b7}
t[\Lambda_0, h_{ij}^\alpha]-t\sum_{\beta=1}^N  h_{ij}^\beta T_{i\alpha}^\beta=[\Lambda_0,\tilde{F}_i^\alpha]-\sum_{\beta=1}^N[\Lambda_0, \tilde{R}_{ij\beta}^\alpha\tilde{F}_j^\beta]-\sum_{\beta=1}^N \tilde{F}_i^\beta\tilde{T}_{i\alpha}^\beta +\sum_{\beta,\gamma=1}^N \tilde{R}_{ij\gamma}^\beta \tilde{F}_j^\gamma \tilde{T}_{i\alpha}^\beta\\
=[\Lambda_0,\tilde{F}_i^\alpha]-\sum_{\beta=1}^N \tilde{F}_j^\beta[\Lambda_0, \tilde{R}_{ij\beta}^\alpha]-\sum_{\beta=1}^N  \tilde{R}_{ij\beta}^\alpha [\Lambda_0,\tilde{F}_j^\beta]-\sum_{\beta=1}^N \tilde{F}_i^\beta \tilde{T}_{i\alpha}^\beta+\sum_{\beta,\gamma=1}^N \tilde{R}_{ij\gamma}^\beta \tilde{F}_j^\gamma \tilde{T}_{i\alpha}^\beta\notag
\end{align}

As in (\ref{b2}), we can show $\sum_{\gamma=1}^N F_j^\gamma(\sum_{\beta=1}^N R_{ij\gamma}^\beta W_{i\alpha}^\beta- [\Lambda_0, R_{ij\gamma}^\alpha]-\sum_{\beta=1}^N R_{ij\beta}^\alpha W_{j\beta}^\gamma)=0$ so that we have $\sum_{\gamma=1}^N \tilde{F}_j^\gamma(\sum_{\beta=1}^N \tilde{R}_{ij\gamma}^\beta \tilde{T}_{i\alpha}^\beta- [\Lambda_0, \tilde{R}_{ij\gamma}^\alpha]-\sum_{\beta=1}^N \tilde{R}_{ij\beta}^\alpha \tilde{T}_{j\beta}^\gamma)=\sum_{\gamma=1}^N \tilde{F}_j^\gamma tS_{ij\gamma}^\alpha=\sum_{\gamma=1}^N f_j^\gamma tS_{ij\gamma}^\alpha$ for some $S_{ij\gamma}^\alpha\in \Gamma(U_i\cap U_j,T_Y)$. Then from (\ref{b6}) and (\ref{b7}), we obtain

\begin{align}\label{b5}
&t(G_i^\alpha-\sum_{\beta=1}^N r_{ij\beta}^\alpha G_j^\beta)-t([\Lambda_0, h_{ij}^\alpha]-\sum_{\beta=1}^N T_{i\alpha}^\beta h_{ij}^\beta)=\sum_{\beta,\gamma=1}^N \tilde{F}_j^\gamma \tilde{R}_{ij\beta}^\alpha \tilde{T}_{j\beta}^\gamma+\sum_{\beta=1}^N \tilde{F}_j^\gamma [\Lambda_0, \tilde{R}_{ij\gamma}^\alpha]-\sum_{\beta,\gamma=1}^N \tilde{F}_j^\gamma\tilde{R}_{ij\gamma}^\beta \tilde{T}_{i\alpha}^\beta 
\end{align}
By taking $-$ on (\ref{b5}) , we get
\begin{align}\label{b10}
&t(\bar{G}_i^\alpha-\sum_{\beta=1}^N \bar{r}_{ij\beta}^\alpha \bar{G}_j^\beta)-t(\overline{[\Lambda_0, h_{ij}^\alpha]}-\sum_{\beta=1}^N \bar{h}_{ij}^\beta \bar{T}_{i\alpha}^\beta )=-\sum_{\gamma=1}^N \bar{f}_i^\gamma t\bar{S}_{ij\gamma}^\alpha=0 \iff (\bar{G}_i^\alpha-\sum_{\beta=1}^N \bar{r}_{ij\beta}^\alpha \bar{G}_j^\beta)+(-\overline{[ h_{ij}^\alpha,\Lambda_0]}+\sum_{\beta=1}^N \bar{h}_{ij}^\beta \bar{T}_{i\alpha}^\beta )=0
\end{align}
Lastly, we show that $\delta(\{(\bar{h}_{ij}^1,...,\bar{h}_{ij}^N)\})=0$.
\begin{align}\label{b8}
t(\sum_{\beta=1}^N r_{ij\beta}^\alpha h_{jk}^\beta-h_{ik}^\alpha+h_{ij}^\alpha)=\sum_{\beta=1}^N \tilde{R}_{ij\beta}^\alpha\tilde{F}_j^\beta-\sum_{\beta,\gamma=1}^N\tilde{R}_{ij\beta}^\alpha \tilde{R}_{jk\gamma}^\beta\tilde{F}_k^\gamma-\tilde{F}_i^\alpha+\sum_{\beta=1}^N \tilde{R}_{ik\beta}^\alpha \tilde{F}_k^\beta+\tilde{F}_i^\alpha-\sum_{\beta=1}^N \tilde{R}_{ij\beta}^\alpha \tilde{F}_j^\beta\\
=-\sum_{\beta,\gamma=1}^N\tilde{R}_{ij\beta}^\alpha \tilde{R}_{jk\gamma}^\beta\tilde{F}_k^\gamma +\sum_{\beta=1}^N \tilde{R}_{ik\beta}^\alpha \tilde{F}_k^\beta=\sum_{\gamma=1}^N (\tilde{R}_{ik\gamma}^\alpha-\sum_{\beta=1}^N \tilde{R}_{ij\beta}^\alpha \tilde{R}_{jk\gamma}^\beta) \tilde{F}_k^\gamma=\sum_{\gamma=1}^N tP_{ijk\gamma}^\alpha \tilde{F}_k^\gamma=\sum_{\gamma=1}^N tP_{ijk\gamma}^\alpha f_{k\gamma}\notag
\end{align}
By taking $-$ on (\ref{b8}), we get
\begin{align}\label{b9}
t(\sum_{\beta=1}^N \bar{r}_{ij\beta}^\alpha \bar{h}_{jk}^\beta-\bar{h}_{ik}^\alpha+\bar{h}_{ij}^\alpha)=\sum_{\gamma=1}^N t\bar{P}_{ijk\gamma}^\alpha \bar{f}_{k\gamma}=0\iff \sum_{\beta=1}^N \bar{r}_{ij\beta}^\alpha \bar{h}_{jk}^\beta-\bar{h}_{ik}^\alpha+\bar{h}_{ij}^\alpha=0
\end{align}

Hence from (\ref{b11}), (\ref{b10}) and (\ref{b9}), $\{(-\bar{G}_i^1,...,-\bar{G}_i^N)\}\oplus \{(\bar{h}_{ij}^1,..., \bar{h}_{ij}^N)\}\in C^0(\mathcal{U}\cap X,\mathcal{N}_{X/Y}\otimes T_Y|_X)\oplus C^1(\mathcal{U}\cap X, \mathcal{N}_{X/Y})$ define a $1$-cocycle in the above \v{C}ech resoultion.

Now we choose another arbitrary lifting $F'^\alpha_i\in \Gamma(U_i,\mathcal{O}_Y)\otimes \tilde{A}$ of $F_i^\alpha$, another arbitrary lifting $\tilde{T}'^\beta_{i\alpha}\in \Gamma(U_i, T_Y)\otimes \tilde{A}$ of $W_{i\alpha}^\beta$ and another arbitrary lifting $R'^\alpha_{ij\beta}\in \Gamma(U_i\cap U_j, \mathcal{O}_Y)\otimes \tilde{A}$. We show that the associated cohomology class $b:=\{(-\bar{G}'^1_i,..., -\bar{G}'^N_i)\}\oplus\{ (\bar{h}'^1_{ij},...,\bar{h}'^N_{ij})\}$ is cohomologous to $a:=\{(-\bar{G}^1_i,..., -\bar{G}^N_i)\}\oplus \{(\bar{h}^1_{ij},...,\bar{h}^N_{ij})\}$. We note that $\tilde{F}'^\alpha_i=\tilde{F}_i^\alpha+t A_i^\alpha$ for some $A_i^\alpha\in \Gamma(U_i, \mathcal{O}_Y)$, $\tilde{T}'^\beta_{i\alpha}=\tilde{T}_{i\alpha}^\beta +tB_{i\alpha}^\beta$ for some $B_{i\alpha}^\beta\in \Gamma(U_i, T_Y)$ and $\tilde{R}'^\alpha_{ij\beta}=\tilde{R}_{ij\beta}^\alpha+tC_{ij\beta}^\alpha$ for some $C_{ij\beta}^\alpha\in \Gamma(U_i\cap U_j, \mathcal{O}_Y)$.

\begin{align}\label{b11}
t(G'^\alpha_i-G_i^\alpha)=[\Lambda_0,  \tilde{F}'^\alpha_i]-\sum_{\beta=1}^N \tilde{F}'^\beta_i \tilde{T}'^\beta_{i\alpha}-[\Lambda_0, \tilde{F}_i^\alpha]+\sum_{\beta=1}^N \tilde{F}_i^\beta \tilde{T}_{i\alpha}^\beta=[\Lambda_0, t A_i^\alpha]-\sum_{\beta=1}^N tA_i^\beta T_{i\alpha}^\beta-\sum_{\beta=1}^N f_i^\beta tB_{i\alpha}^\beta
\end{align}
By taking $-$ on (\ref{b11}), we get  $-\bar{G}'^\alpha_i-(-\bar{G}_i^\alpha)=-[\bar{A}_i^\alpha,\Lambda_0]+\sum_{\beta=1}^N \bar{A}_i^\beta \bar{T}_{i\alpha}^\beta$.

On the other hand,
\begin{align}\label{b12}
t(h'^\alpha_{ij}-h_{ij}^\alpha)=\tilde{F}'^\alpha_i-\sum_{\beta=1}^N \tilde{R}'^\alpha_{ij\beta} \tilde{F}'^\beta_j-\tilde{F}^\alpha_i+\sum_{\beta=1}^N \tilde{R}^\alpha_{ij\beta} \tilde{F}^\beta_j=tA_i^\alpha-\sum_{\beta=1}^N r_{ij\beta}^\alpha tA_j^\beta-\sum_{\beta=1}^N tC_{ij\beta}^\alpha f_j^\beta.
\end{align}
By taking $-$ on (\ref{b12}), we get $h'^\alpha_{ij}-h^\alpha_{ij}=\bar{A}_i^\alpha-\sum_{\beta=1}^N \bar{r}_{ij\beta}^\alpha \bar{A}_j^\beta.$ Hence $\{(\bar{A}_i^1,...,\bar{A}_i^N)\}$ is mapped to $b-a$ so that $a$ is cohomologous to $a$. So given a small extension $e:0\to (t)\to\tilde{A}\to A\to 0$, we can associate an element $o_\eta(e):=$ the cohomology class $a\in \mathbb{H}^1(X,\mathcal{N}_{X/Y}^\bullet)$. We note that $o_\eta(e)=0$ if and only if there exists a collection $\{\tilde{F}_i^\alpha\},\{\tilde{T}_{i\alpha}^\beta\}$ and $\{\tilde{R}_{ij\beta}^\alpha\}$ such that $\bar{h}_{ij}^\alpha=0$ and $\bar{G}_i^\alpha=0,\alpha=1,...,N$:
\begin{enumerate}
\item If $\bar{h}_{ij}^\alpha=0$, then $h_{ij}^\alpha=f_j^1L_j^1+\cdots f_j^NL_j^N$ for some $L_j^\beta\in \Gamma(U_i\cap U_j, \mathcal{O}_Y)$. Then $\tilde{F}_i^\alpha-\sum_{\beta=1}^N \tilde{R}_{ij\beta}^\alpha \tilde{F}_j^\beta=t(\sum_{\beta=1}^N \tilde{F}_j^\beta L_j^\beta)$. Hence $\tilde{F}_i^\alpha-\sum_{\beta=1}^N (\tilde{R}_{ij\beta}^\alpha+tL_j^\beta)\tilde{F}_j^\beta=0$ so that $(\tilde{F}_i^1,...,\tilde{F}_i^N)=(\tilde{F}_j^1,...,\tilde{F}_j^N)$. Hence $\{(\tilde{F}_i^1,...,\tilde{F}_i^N)\}$ is an ideal sheaf on $Y\times Spec(\tilde{A})$.
\item If $\bar{G}_i^\alpha=0$, then $G_i^\alpha=f_i^1P_i^1+\cdots +f_i^NP_i^N$ for some $P_i^\beta\in \Gamma(U_i,T_Y)$, then $[\Lambda_0,\tilde{F}_i^\alpha]-\sum_{\beta=1}^N \tilde{F}_i^\beta \tilde{T}_{i\alpha}^\beta=t\sum_{\beta=1}^N f_i^\beta P_i^\beta=t\sum_{\beta=1}^N \tilde{F}_i^\beta P_i^\beta$. Hence $[\Lambda_0,\tilde{F}_i^\alpha]=\sum_{\beta=1}^N \tilde{F}_i^\beta (\tilde{T}_{i\alpha}^\beta+t P_{i\beta})$ so that $(\tilde{F}_i^1,...,\tilde{F}_i^N)$ is a Poisson ideal.
\end{enumerate}
Hence $o_\eta(e)=0$ if and only if there is a lifting of $\eta$ to $\tilde{A}$.

\end{proof}

\subsection{Deformations of Poisson closed subschemes of codimension $1$ and Poisson semi-regularity}(compare \cite{Ser06} p.143-144)\

Let $(L,\nabla)$ be a Poisson invertible sheaf on a nonsingular Poisson variety $(X,\Lambda_0)$, where $\nabla$ is a Poisson connection on $L$, and we denote by $\mathbb{H}^i(X,L^\bullet)$ the $i$-th hypercohomology group of the following complex of sheaves (see \cite{Kim16})
\begin{align}
L^\bullet:L\xrightarrow{\nabla}L\otimes T_X\xrightarrow{\nabla} L\otimes \wedge^2 T_X\xrightarrow{\nabla} L\otimes \wedge^2 T_X\xrightarrow{\nabla}\cdots
\end{align}
Let $s\in \mathbb{H}^0(X,L^\bullet)$ and $D=div(s)\subset (X,\Lambda_0)$ be the Poisson divisor associated with $s$. Then we have a morphism of functors of Artin rings:
\begin{align}
H_D^{(X,\Lambda_0)}\to Def_{(L,\nabla)} \label{s1}
\end{align}
where $Def_{(L,\nabla)}$ is the functor of deformations of $(L,\nabla)$ (see \cite{Kim16}).
Let us consider the following exact sequence $0\to \mathcal{O}_X^\bullet\xrightarrow{s} L^\bullet\to L_D^\bullet\to 0:$
\begin{center}
$\begin{CD}
@.\cdots@.\cdots @.\cdots @.\\
@. @A-[-,\Lambda_0]AA @A\nabla AA @AAA @.\\
0@>>>\wedge^2 T_X@>s>> L\otimes \wedge^2 T_X @>>> L_D\otimes \wedge^2 T_X|_D@>>>0\\
@. @A-[-,\Lambda_0]AA @A\nabla AA @AAA @.\\
0@>>> T_X@>s>> L\otimes T_X@>>> L_D\otimes T_X|_D@>>>0\\
@. @A-[-,\Lambda_0]AA @A\nabla AA @AAA @.\\
0@>>> \mathcal{O}_X@>s>> L@>>> L_D@>>>0
\end{CD}$
\end{center}
so that we have a long exact sequence
\begin{align*}
0\to\mathbb{H}^0(X,\mathcal{O}_X^\bullet)\to \mathbb{H}^0(X,L^\bullet)\to \mathbb{H}^0(D,L_D^\bullet)\xrightarrow{\delta_0} \mathbb{H}^1(X,\mathcal{O}_X^\bullet)\to \mathbb{H}^1(X,L^\bullet)\to \mathbb{H}^1(D,L_D^\bullet)\xrightarrow{\delta_1} \mathbb{H}^2(X,\mathcal{O}_X^\bullet)\to \cdots
\end{align*}
Then $\delta_0$ represents the morphism of tangent spaces for (\ref{s1}): $H_D^{(X,A)}(k[\epsilon])\to Def_{(L,\nabla)}(k[\epsilon])$ and $\delta_1$ represents the obstruction map for (\ref{s1}).
\begin{remark}[Poisson semi-regularity]
Let $(X,\Lambda_0)$ be a nonsingular Poisson projective variety. A Poisson Cartier divisor $D$ on $X$ is called Poisson semi-regular if the natural map
\begin{align*}
\mathbb{H}^1(X, \mathcal{O}_X(D)^\bullet)\to \mathbb{H}^1(D,\mathcal{O}_D(D)^\bullet)
\end{align*}
is zero. If $D\subset (X,\Lambda_0)$ is Poisson semi-regular and $Def_{(L,\nabla)}$ is smooth, then $H_D^{(X,\Lambda_0)}$ is smooth so that $D\subset (X,\Lambda_0)$ is unobstructed. 
\end{remark}

\section{Simultaneous deformations of Poisson structures and Poisson closed subschemes}\label{appendixc}

\subsection{The local extended Poisson Hilbert functor}\

Let $X\subset (Y,\Lambda_0)$ be a closed embedding of algebraic Poisson schemes where $(Y,\Lambda_0)$ is a nonsingular Poisson variety, and $A\in \bold{Art}$. An infinitesimal simultaneous deformation of $X$ in $(Y,\Lambda_0)$ over $A\in \bold{Art}$ is a cartesian diagram of morphisms of schemes
\begin{center}
$\eta:\begin{CD}
X@>>> \mathcal{X}\subset(Y\times Spec(A),\Lambda)\\
@VVV @VV\pi V\\
Spec(k)@>>> S=Spec(A)
\end{CD}$
\end{center}
where $\pi$ is flat, and it is induced by the projection from $Y\times S$ to $S$, $\Lambda\in \Gamma(Y\times S, \mathscr{H}om(\wedge^2 \Omega_{Y\times S/S}^1,\mathcal{O}_{Y\times S}))$ defines a Poisson structure on $Y\times S$ over $S$, $\mathcal{X}$ is a Poisson closed subscheme of $(Y\times S,\Lambda)$, and $(Y\times S,\Lambda)$ induces $(Y,\Lambda_0)$. Then we can define a functor of Artin rings (called the local extended Poisson Hilbert functor of $X$ in $(Y,\Lambda_0)$)
\begin{align*}
EH_X^{(Y,\Lambda_0)}:\bold{Art}&\to (Sets)\\
A&\mapsto\{\text{infinitesimal simultaneous deformations of $X$ in $(Y,\Lambda_0)$}\}
\end{align*}

\subsection{The extended complex associated with the normal bundle $\mathcal{N}_{X/Y}$ of a Poisson closed subschemes of a nonsingular Poisson variety}\

Let $(Y,\Lambda_0)$ be a nonsingular Poisson variety and $X$ be a Poisson closed subscheme of $(Y,\Lambda_0)$ defined by a Poisson ideal sheaf $\mathcal{I}$. Assume that $i:X\hookrightarrow Y$ be a regular embedding. We keep the notations in subsection \ref{b1}.

We define the extended complex  $(\wedge^2 T_Y\oplus i_* \mathcal{N}_{X/Y})^\bullet$ associated with the normal bundle $\mathcal{N}_{X/Y}$:
\begin{align}
(\wedge^2 T_Y\oplus i_*\mathcal{N}_{X/Y})^\bullet:\wedge^2 T_Y\oplus i_*\mathcal{N}_{X/Y}\xrightarrow{\tilde{\nabla}}\wedge^3 T_Y\oplus (i_*\mathcal{N}_{X/Y}\otimes T_Y|_X)\xrightarrow{\tilde{\nabla}}\wedge^4 T_Y\oplus (i_*\mathcal{N}_{X/Y}\otimes \wedge^2 T_Y|_X)\xrightarrow{\tilde{\nabla}}\cdots
\end{align}
The complex is defined locally in the following way:
\begin{align*}
\tilde{\nabla}&:\Gamma(U_i, \wedge^{p+2} T_Y)\oplus (\oplus^r \Gamma(U_i\cap X,\wedge^p T_Y|_X)\to \Gamma(U_i, \wedge^{p+3} T_Y)\oplus (\oplus^r \Gamma(U_i\cap X, \wedge^{p+1} T_Y|_X)\\
&(\Pi_i,(g_i^1,...,g_i^N))\mapsto (-[\Pi_i, \Lambda_0], \overline{[\Pi_i, f_i^1]}-\overline{[g_i^1, \Lambda_0]}+(-1)^p\sum_{\beta=1}^N g_i^\beta\wedge \bar{T}_{i1}^\beta,...,\overline{[\Pi_i, f_i^N]}-\overline{[g_i^N,\Lambda_0]}+(-1)^p \sum_{\beta=1}^N g_i^\beta \wedge \bar{T}_{iN}^\beta)
\end{align*}

\begin{proposition}\label{cp}
Given a regular closed embedding of algebraic Poisson schemes $i:X\hookrightarrow (Y,\Lambda_0)$, where $(Y,\Lambda_0)$ is a nonsingular Poisson variety, 
\begin{enumerate}
\item There is a natural identification
\begin{align*}
EH_X^{(Y,\Lambda_0)}(k[\epsilon])\cong \mathbb{H}^0(Y,(\wedge^2 T_Y\oplus i_* \mathcal{N}_{X/Y})^\bullet)
\end{align*}
\item Given an infinitesimal simultaneous deformation $\eta$ of $X$ in $(Y,\Lambda_0)$ over $A\in \bold{Art}$ and a small extension $e:0\to (t)\to\tilde{A}\to A\to 0$, we can associate an element $o_\eta(e)\in \mathbb{H}^1(Y,(\wedge^2 T_Y\oplus i_*\mathcal{N}_{X/Y})^\bullet)$, which is zero if and only if there is a lifting of $\eta$ to $\tilde{A}$.
\end{enumerate}
\end{proposition}

\begin{proof}
Let $\mathcal{U}=\{U_i\}$ be an affine open covering of $Y$ and let $I_i=(f_i^1,...,f_i^N)$ be a Poisson ideal of $\Gamma(U_i,\mathcal{O}_Y)$ defining $U_i\cap X$  such that $\{f_i^1,\cdots f_i^N\}$ is a regular sequence. We keep the notations in subsection \ref{b1}.

A first-order simultaneous deformation of $X$ in $(Y,\Lambda_0)$ is a flat family
\begin{center}
$\begin{CD}
X@>>> \mathcal{X}\subset (Y\times Spec(k[\epsilon]),\Lambda_0+\epsilon \Lambda')\\
@VVV @VVV\\
Spec(k[\epsilon])@>>> Spec(k[\epsilon])
\end{CD}$
\end{center}
Since $\Lambda_0+\epsilon \Lambda'$ with $\Lambda'\in H^0(Y,\wedge^2 T_Y)$ define a Poisson structure on $Y\times Spec(k[\epsilon])$, we have $[\Lambda_0,\Lambda']=0$.
Assume that $\mathcal{X}$ is determined by a Poisson ideal sheaf $\mathcal{I}$ generated by $\{f_i^\alpha+\epsilon g_i^\alpha\},\alpha=1,...,N$ for some $g_i^\alpha\in \Gamma(U_i,\mathcal{O}_Y)$. Then $(\bar{g}_i^1,...,\bar{g}_i^N)$ define a global section in $H^0(X,\mathcal{N}_{X/Y})$ as in the proof of Proposition \ref{b15}. On the other hand, $(f_i^1+\epsilon g_i^1,...,f_i^N+\epsilon g_i^N)$ is a Poisson ideal, we have $[\Lambda_0+\epsilon \Lambda',f_i^\alpha+\epsilon g_i^\alpha]=\sum_{\beta=1}^N(f_i^\beta+\epsilon g_i^\beta)(T_{i\alpha}^\beta+\epsilon W_{i\alpha}^\beta)$ for some $W_{i\alpha}^\beta\in \Gamma(U_i,T_Y)$. Then $\overline{[\Lambda_0,g_i^\alpha]} +\overline{[\Lambda',f_i^\alpha]}=\sum_{\beta=1}^N \bar{g}_i^\beta\bar{T}_{i\alpha}^\beta$ so that $\tilde{\nabla}(-\Lambda',(\bar{g}_i^1,...,\bar{g}_i^N))=0$. Hence $(-\Lambda',\{ (\bar{g}_i^1,...,\bar{g}_i^N) \} )\in \mathbb{H}^0(Y ,(\wedge^2 T_Y\oplus i_* \mathcal{N}_{X/Y})^\bullet)$.

Next we identify obstructions. Consider a small extension $e:0\to (t)\to \tilde{A}\to A\to 0$. Let $\eta:=(\mathcal{X}\subset (Y\times Spec(A),\Lambda))$ be an infinitesimal simultaneous deformation of $X$ in $(Y,\Lambda_0)$ over $A$. Then $\mathcal{X}$ is determined by a Poisson ideal sheaf $\mathcal{I}_A$ generated by $(F_i^1,...,F_i^N)$ in $\Gamma(U_i,\mathcal{O}_Y)\otimes A$ such that $F_i^\alpha\equiv f_i^\alpha\otimes 1 \mod \mathfrak{m}_A$, and $\{F_i^1,...,F_i^N\}$ is a regular sequence. Let $\Lambda_i\in \Gamma(U_i,\wedge^2 T_Y)\otimes A$ be the restriction of $\Lambda$ on $U_i$. Since $(F_i^1,...,F_i^N)=(F_j^1,..,F_j^N)$, we have $F_i^\alpha=\sum_{\beta=1}^N R_{ij\beta}^\alpha F_j^\beta$ for some $R_{ij\beta}^\alpha\in\Gamma(U_i\cap U_j, \mathcal{O}_Y)\otimes A$. Since $(F_i^1,...,F_i^N)$ is a Poisson ideal, we have $[\Lambda_i,F_i^\alpha]=\sum_{\beta=1}^N F_i^\beta W_{i\alpha}^\beta$ for some $W_{i\alpha}^\beta\in \Gamma(U_i,T_Y)\otimes A$. Let $\tilde{\Lambda}_i\in \Gamma(U_i,\wedge^2 T_Y)\otimes \tilde{A}$ be an arbitrary lifting of $\Lambda_i$, $\tilde{F}_i^\alpha\in \Gamma(U_i,\mathcal{O}_Y)\otimes \tilde{A}$ be an arbitrary lifting of $F_i^\alpha$, $\tilde{T}_{i\alpha}^\beta\in \Gamma(U_i,T_Y)\otimes \tilde{A}$ be an arbitrary lifting of $W_{i\alpha}^\beta$, and $\tilde{R}_{ij\beta}^\alpha\in \Gamma(U_i,\mathcal{O}_Y)\otimes \tilde{A}$ be an arbitrary lifting of $R_{ij\beta}^\alpha$. Then $[\tilde{\Lambda}_i,\tilde{\Lambda}_i]=t\Pi_i$ for some $\Pi_i\in \Gamma(U_i,\wedge^3 T_Y)$, $\tilde{\Lambda}_i-\tilde{\Lambda}_j=t\Lambda_{ij}'$ for some $\Lambda_{ij}'\in \Gamma(U_i\cap U_j,\wedge^2 T_Y)$, $[\tilde{\Lambda}_i,\tilde{F}_i^\alpha]-\sum_{\beta=1}^N \tilde{F}_i^\beta \tilde{T}_{i\alpha}^\beta =t G_i^\alpha$ for some $G_i^\alpha\in \Gamma(U_i,T_Y)$, $\tilde{F}_i^\alpha-\sum_{\beta=1}^N \tilde{R}_{ij\beta}^\alpha \tilde{F}_j^\beta=t h_{ij}^\alpha$ for some $h_{ij}^\alpha\in \Gamma(U_i\cap U_j,\mathcal{O}_Y)$, and $\tilde{R}_{ik\gamma}^\alpha-\sum_{\beta=1}^N \tilde{R}_{ij\beta}^\alpha \tilde{R}_{jk\gamma}^\beta=tP_{ijk\gamma}^\alpha$ for some $P_{ijk\gamma}^\alpha\in \Gamma(U_i\cap U_j\cap U_k,\mathcal{O}_Y)$. We will show that $(\{\frac{1}{2}\Pi_i\},\{(-\bar{G}_i^1,...,-\bar{G}_i^N)\})\oplus (\{-\Lambda_{ij}'\},\{\bar{h}_{ij}^1,...,\bar{h}_{ij}^N\})\in C^0(\mathcal{U},\wedge^3 T_Y\oplus i_*(\mathcal{N}_{X/Y}\otimes T_Y|_X))\oplus \mathcal{C}^1(\mathcal{U},\wedge^2 T_Y\oplus i_*\mathcal{N}_{X/Y})$ define a $1$-cocycle in the following \v{C}ech resolution of $(\wedge^2 T_Y\oplus i_*\mathcal{N}_{X/Y})^\bullet$:
 
 \begin{center}
$\begin{CD}
C^0(\mathcal{U},\wedge^4T_Y\oplus i_*(\mathcal{N}_{X/Y}\otimes \wedge^2 T_{Y}|_X))\\
@A\tilde{\nabla}AA \\
C^0(\mathcal{U},\wedge^3 T_Y\oplus i_*(\mathcal{N}_{X/Y}\otimes T_{Y}|_X))@>\delta>> C^1(\mathcal{U},\wedge^3T_Y\oplus i_*(\mathcal{N}_{X/Y}\otimes T_Y|_X))\\
@A\tilde{\nabla}AA @A\tilde{\nabla}AA\\
C^0(\mathcal{U},\wedge^2 T_Y\oplus i_*\mathcal{N}_{X/Y})@>-\delta>> C^1(\mathcal{U},\wedge^2 T_Y\oplus i_*\mathcal{N}_{X/Y})@>\delta>> C^2(\mathcal{U}, \wedge^2 T_Y\oplus i_*\mathcal{N}_{X/Y})
\end{CD}$
\end{center}

First we show that $\tilde{\nabla}((\{\frac{1}{2}\Pi_i\},\{(-\bar{G}_i^1,...,-\bar{G}_i^N)\}))=0$. From (\ref{a10}), we have 
\begin{align}\label{c2}
-[\frac{1}{2}\Pi_i,\Lambda_0]=0
\end{align}
 On the other hand, as in (\ref{b3}), we can show $\sum_{\gamma=1}^N F_i^\gamma([\Lambda_i,W_{i\alpha}^\gamma]-\sum_{\beta=1}^N W_{i\beta}^\gamma\wedge W_{i\alpha}^\beta)$ so that $\sum_{\gamma=1}^N \tilde{F}_i^\gamma([\Lambda_i,\tilde{T}_{i\alpha}^\gamma]-\sum_{\beta=1}^N \tilde{T}_{i\beta}^\gamma \wedge\tilde{T}_{i\alpha}^\beta)=\sum_{\gamma=1}^N \tilde{F}_i^\gamma tQ_{i\alpha}^\gamma=\sum_{\gamma=1}^N f_i^\gamma tQ_{i\alpha}^\gamma$ for some $Q_{i\alpha}^\gamma\in \Gamma(U_i, \wedge^2 T_Y)$. Then we have
\begin{align}\label{c1}
t[\Lambda_0, G_i^\alpha]&=[\tilde{\Lambda}_i,[\tilde{\Lambda}_i,\tilde{F}_i^\alpha]]-\sum_{\beta=1}^N[\tilde{\Lambda}_i, \tilde{F}_i^\beta\tilde{T}_{i\alpha}^\beta]=t[\frac{1}{2}\Pi_i,f_i^\alpha]-\sum_{\beta=1}^N \tilde{F}_i^\beta[\tilde{\Lambda}_i,\tilde{T}_{i\alpha}^\beta]+\sum_{\beta=1}^N[\tilde{\Lambda}_i,\tilde{F}_i^\beta]\wedge \tilde{T}_{i\alpha}^\beta\\
&=t[\frac{1}{2}\Pi_i,f_i^\alpha]-\sum_{\gamma=1}^N \tilde{F}_i^\gamma [\tilde{\Lambda}_i,\tilde{T}_{i\alpha}^\gamma]+\sum_{\beta=1}^N tG_i^\beta \wedge T_{i\alpha}^\beta+\sum_{\beta,\gamma=1}^N \tilde{F}_i^\gamma \tilde{T}_{i\beta}^\gamma\wedge\tilde{T}_{i\alpha}^\beta \notag 
\end{align}
By taking $-$ on (\ref{c1}), we obtain
\begin{align}
t\overline{[\Lambda_0, G_i^\alpha]}=t\overline{[\frac{1}{2}\Pi_i,f_i^\alpha]}-\sum_{\gamma=1}^N \bar{f}_i^\gamma tQ_{i\alpha}^\gamma+\sum_{\beta=1}^N t\bar{G}_i^\beta \wedge  \bar{T}_{i\alpha}^\beta=t\overline{[\frac{1}{2}\Pi_i,f_i^\alpha]}+\sum_{\beta=1}^N t\bar{G}_i^\beta \wedge  \bar{T}_{i\alpha}^\beta \notag\\
\iff \overline{[\frac{1}{2}\Pi_i,f_i^\alpha]}-\overline{[-\bar{G}_i^\alpha,\Lambda_0]}+(-1)^1\sum_{\beta=1}^N -\bar{G}_i^\beta\wedge \bar{T}_{i\alpha}^\beta=0\label{c3}
\end{align}

Next we show that $\delta((\{\frac{1}{2}\Pi_i\},\{(-\bar{G}_i^1,...,-\bar{G}_i^N)\}))+\tilde{\nabla}((\{-\Lambda_{ij}'\},\{\bar{h}_{ij}^1,...,\bar{h}_{ij}^N\}))=0$. From (\ref{a11}), we have
\begin{align}\label{c4}
\delta(\frac{1}{2}\Pi_i)-[-\Lambda_{ij}',\Lambda_0]=0
\end{align}

On the other hand, we have 
\begin{align}\label{c5}
t(G_i^\alpha-\sum_{\beta=1}^N r_{ij\beta}^\alpha G_j^\beta)=[\tilde{\Lambda}_i,\tilde{F}_i^\alpha]-\sum_{\beta=1}^N \tilde{F}_i^\beta \tilde{T}_{i\alpha}^\beta-\sum_{\beta=1}^N \tilde{R}_{ij\beta}^\alpha [\tilde{\Lambda}_j,\tilde{F}_j^\beta]+\sum_{\beta,\gamma=1}^N \tilde{R}_{ij\beta}^\alpha \tilde{F}_j^\gamma\tilde{T}_{j\beta}^\gamma
\end{align}
\begin{align}\label{c6}
& t[\Lambda_0,h_{ij}^\alpha]-t\sum_{\beta=1}^N h_{ij}^\beta T_{i\alpha}^\beta=[\tilde{\Lambda}_i,\tilde{F}_i^\alpha]-\sum_{\beta=1}^N [\tilde{\Lambda}_i,\tilde{R}_{ij\beta}^\alpha\tilde{F}_j^\beta]-\sum_{\beta=1}^N \tilde{F}_i^\beta\tilde{T}_{i\alpha}^\beta+\sum_{\beta,\gamma=1}^N \tilde{R}_{ij\gamma}^\beta\tilde{F}_j^\gamma \tilde{T}_{i\alpha}^\beta\\
& =[\tilde{\Lambda}_i,\tilde{F}_i^\alpha]-\sum_{\beta=1}^N \tilde{F}_j^\beta[\tilde{\Lambda}_i,\tilde{R}_{ij\beta}^\alpha]-\sum_{\beta=1}^N \tilde{R}_{ij\beta}^\alpha[\tilde{\Lambda}_j+t \Lambda_{ij}',\tilde{F}_j^\beta]-\sum_{\beta=1}^N\tilde{F}_i^\beta\tilde{T}_{i\alpha}^\beta+\sum_{\beta,\gamma=1}^N \tilde{R}_{ij\gamma}^\beta\tilde{F}_j^\gamma \tilde{T}_{i\alpha}^\beta\notag
 \end{align}
  As in (\ref{b2}), we can show $\sum_{\gamma=1}^N F_j^\gamma (\sum_{\beta=1}^N R_{ij\gamma}^\beta W_{i\alpha}^\beta-[\Lambda_i, R_{ij\gamma}^\alpha]-\sum_{\beta=1}^N R_{ij\beta}^\alpha W_{j\beta}^\gamma)=0$ so that we have $\sum_{\gamma=1}^N \tilde{F}_j^\gamma(\sum_{\beta=1}^N\tilde{R}_{ij\gamma}^\beta \tilde{T}_{i\alpha}^\beta-[\tilde{\Lambda}_i, \tilde{R}_{ij\gamma}^\alpha]-\sum_{\beta=1}^N \tilde{R}_{ij\beta}^\alpha \tilde{T}_{j\beta}^\gamma)=\sum_{\gamma=1}^N \tilde{F}_j^\gamma tS_{ij\gamma}^\alpha$ for some $S_{ij\gamma}^\alpha\in \Gamma(U_i\cap U_j,T_Y)$.
 Then from (\ref{c5}) and (\ref{c6}), we get
 \begin{align}\label{c7}
 &t(G_i^\alpha-\sum_{\beta=1}^N r_{ij\beta}^\alpha G_j^\beta)- t[\Lambda_0,h_{ij}^\alpha]+t\sum_{\beta=1}^N h_{ij}^\beta T_{i\alpha}^\beta\\
 &=\sum_{\beta,\gamma=1}^N \tilde{R}_{ij\beta}^\alpha \tilde{F}_j^\gamma\tilde{T}_{j\beta}^\gamma+\sum_{\gamma=1}^N \tilde{F}_j^\gamma [\tilde{\Lambda}_i,\tilde{R}_{ij\gamma}^\alpha] +\sum_{\beta=1}^N r_{ij\beta}^\alpha[t\Lambda_{ij}',f_j^\beta]-\sum_{\beta,\gamma=1}^N \tilde{R}_{ij\gamma}^\beta\tilde{F}_j^\gamma \tilde{T}_{i\alpha}^\beta\notag
 \end{align}
 By taking $-$ on (\ref{c7}), we get
 \begin{align}
 t(\bar{G}_i^\alpha-\sum_{\beta=1}^N \bar{r}_{ij\beta}^\alpha \bar{G}_j^\beta)- t\overline{[\Lambda_0,h_{ij}^\alpha]}+t\sum_{\beta=1}^N \bar{h}_{ij}^\beta \bar{T}_{i\alpha}^\beta=-\sum_{\gamma=1}^N t\bar{S}_{ij\gamma} \tilde{f}_i^\gamma+\sum_{\beta=1}^N t \bar{r}_{ij\beta}^\alpha\overline{[\Lambda_{ij}',f_j^\beta]}=\sum_{\beta=1}^N t\bar{r}_{ij\beta}^\alpha \overline{[ \Lambda_{ij}', f_j^\beta]}=t\overline{[\Lambda_{ij}', f_i^\beta]} \notag\\
 \iff  (\sum_{\beta=1}^N \bar{r}_{ij\beta}^\alpha\cdot(-\bar{G}_j^\beta)-(-\bar{G}_i^\alpha))+ \overline{[-\Lambda_{ij}', f_i^\beta]}-\overline{[h_{ij}^\alpha, \Lambda_0]}+\sum_{\beta=1}^N \bar{h}_{ij}^\beta \bar{T}_{i\alpha}^\beta=0\label{c8}
 \end{align}

 Lastly, from (\ref{a12}) and (\ref{b9}), we have
 \begin{align}\label{c99}
 \delta(\{-\Lambda_{ij}'\},\{(\bar{h}_{ij}^1,...,\bar{h}_{ij}^N)\})=0.
 \end{align}
 Hence from (\ref{c2}),(\ref{c3}),(\ref{c4}), (\ref{c8}), and (\ref{c99}), $(\{\frac{1}{2}\Pi_i\},\{(-\bar{G}_i^1,...,-\bar{G}_i^N)\})\oplus (\{-\Lambda_{ij}'\},\{(\bar{h}_{ij}^1,...,\bar{h}_{ij}^N)\})\in C^0(\mathcal{U},\wedge^3 T_Y\oplus i_*(\mathcal{N}_{X/Y}\otimes T_Y|_X))\oplus \mathcal{C}^1(\mathcal{U},\wedge^2 T_Y\oplus i_*\mathcal{N}_{X/Y})$ define a $1$-cocycle in the above \v{C}ech resolution.

Now we choose another arbitrary lifting $F'^\alpha_i\in \Gamma(U_i,\mathcal{O}_Y)\otimes \tilde{A}$ of $F_i^\alpha$, another arbitrary lifting $\tilde{T}'^\beta_{i\alpha}\in \Gamma(U_i, T_Y)\otimes \tilde{A}$ of $W_{i\alpha}^\beta$, another arbitrary lifting $R'^\alpha_{ij\beta}\in \Gamma(U_i\cap U_j, \mathcal{O}_Y)\otimes \tilde{A}$ of $R_{ij\beta}^\alpha$ and another arbitrary lifting $\tilde{\Lambda}'_i\in \Gamma(U_i, \wedge^2 T_Y)\otimes \tilde{A}$ of $\Lambda_i$. We show that the associated cohomology class $b:=(\{\frac{1}{2}\Pi_i'\},\{(-\bar{G}'^1_i,..., -\bar{G}'^N_i)\})\oplus (\{-\Lambda_{ij}''\},\{(\bar{h}'^1_{ij},...,\bar{h}'^N_{ij})\})$ is cohomologous to $a:=(\{\frac{1}{2}\Pi_i\},\{(-\bar{G}^1_i,..., -\bar{G}^N_i)\})\oplus (\{-\Lambda_{ij}'\},\{ (\bar{h}^1_{ij},...,\bar{h}^N_{ij})\})$. We note that $\tilde{F}'^\alpha_i=\tilde{F}_i^\alpha+t A_i^\alpha$ for some $A_i^\alpha\in \Gamma(U_i, \mathcal{O}_Y)$, $\tilde{T}'^\beta_{i\alpha}=\tilde{T}_{i\alpha}^\beta +tB_{i\alpha}^\beta$ for some $B_{i\alpha}^\beta\in \Gamma(U_i, T_Y)$, $\tilde{R}'^\alpha_{ij\beta}=\tilde{R}_{ij\beta}^\alpha+tC_{ij\beta}^\alpha$ for some $C_{ij\beta}^\alpha\in \Gamma(U_i\cap U_j, \mathcal{O}_Y)$ and $\tilde{\Lambda}'_i=\tilde{\Lambda}_i+t D_i$ for some $D_i\in \Gamma(U_i, \wedge^2 T_Y)$. Then

\begin{align}\label{c20}
t(G'^\alpha_i-G_i^\alpha)=[\tilde{\Lambda}_i',  \tilde{F}'^\alpha_i]-\sum_{\beta=1}^N \tilde{F}'^\beta_i \tilde{T}'^\beta_{i\alpha}-[\tilde{\Lambda}_i, \tilde{F}_i^\alpha]+\sum_{\beta=1}^N \tilde{F}_i^\beta \tilde{T}_{i\alpha}^\beta=[tD_i,f_i^\alpha]+[\Lambda_0, t A_i^\alpha]-\sum_{\beta=1}^N tA_i^\beta T_{i\alpha}^\beta-\sum_{\beta=1}^N f_i^\beta tB_{i\alpha}^\beta
\end{align}
By taking $-$ on (\ref{c20}), we get  
\begin{align}
\bar{G}'^\alpha_i-\bar{G}_i^\alpha=\overline{[D_i,f_i^\alpha]}+\overline{[\Lambda_0,  \bar{A}_i^\alpha]}-\sum_{\beta=1}^N \bar{A}_i^\beta \bar{T}_{i\alpha}^\beta\iff -\bar{G}^\alpha_i-(-\bar{G}'^\alpha_i)=\overline{[D_i,f_i^\alpha]}-\overline{[\Lambda_0,-  A_i^\alpha]}+\sum_{\beta=1}^N -\bar{A}_i^\beta \bar{T}_{i\alpha}^\beta
\end{align}
 and from (\ref{b12}), we have $\bar{h}'^\alpha_{ij}-\bar{h}_{ij}^\alpha=\bar{A}_i^\alpha-\sum_{\beta=1}^N \bar{r}_{ij\beta}^\alpha \bar{A}_j^\beta$ so that $\{\bar{h}^\alpha_{ij}-\bar{h}'^\alpha_{ij}\}=-\delta(\{-\bar{A}_i^\alpha\})$. On the other hand, from (\ref{aa1}) and (\ref{aa2}), we have $\frac{1}{2}\Pi_i-\frac{1}{2}\Pi'_i=-[D_i,\Lambda_0]$ and $ -\Lambda_{ij}'-(-\Lambda_{ij}'')=-\delta(D_i)$. 
 Hence $(\{D_i\}\oplus \{(-\bar{A}_i^1,...,-\bar{A}_i^N)\})$ is mapped to $a-b$ so that $a$ is cohomologous to $b$. So given a small extension $e:0\to (t)\to \tilde{A}\to A \to 0$, we can associate an element $o_\eta(e):=$ the cohomology class $a\in \mathbb{H}^1(Y,(\wedge^2 T_Y\oplus i_*\mathcal{N}_{X/Y})^\bullet)$. We note that $o_\eta(e)=0$ if and only if there exists collections $\{\tilde{F}_i^\alpha\},\{\tilde{T}_{i\alpha}^\beta\},\{\tilde{R}_{ij\beta}^\alpha\}$ and $\{\tilde{\Lambda}_i\}$ such that $\bar{h}_{ij}^\alpha=0,\bar{G}_i^\alpha=0,\alpha=1,...,N, \Pi_i=0$, and $\Lambda_{ij}'=0$:
\begin{enumerate}
\item If $\bar{h}_{ij}^\alpha=0$, then $(\tilde{F}_i^1,...,\tilde{F}_i^N)=(\tilde{F}_j^1,...,\tilde{F}_j^N)$ so that $\{(\tilde{F}_{i1},...,\tilde{F}_{iN})\}$ define an ideal sheaf on $Y\times Spec(\tilde{A})$.
\item If $\Pi_i=0$ and $\Lambda_{ij}'=0$, then $[\tilde{\Lambda}_i,\tilde{\Lambda}_i]=0$, and $\{\tilde{\Lambda}_i\}$ glues together to define a Poisson structure on $Y\times Spec(\tilde{A})$.
\item  If $\bar{G}_i^\alpha=0$, then $G_i^\alpha=\sum_{\beta=1}^N f_i^\beta P_i^\beta$ for some $P_i^\beta\in \Gamma(U_i,T_Y)$. Then $[\tilde{\Lambda}_i,\tilde{F}_i^\alpha]=\sum_{\beta=1}^N(\tilde{T}_i^\beta+tP_i^\beta)\tilde{F}_i^\beta$ so that $(\tilde{F}_i^1,...\tilde{F}_i^N)$ defines a Poisson ideal.
\end{enumerate}
Hence $o_\eta(e)=0$ if and only if there is a lifting of $\eta$ to $\tilde{A}$.
\end{proof}

\bibliographystyle{amsalpha}
\bibliography{References-Rev9}

\end{document}